\documentclass[reqno,english]{amsart}
\usepackage{amsfonts,amsmath,latexsym,verbatim,amscd,mathrsfs,color,array}
\usepackage[colorlinks=true]{hyperref}

\hypersetup{
	colorlinks=true,
	linkcolor=red,
    citecolor=blue
}

\usepackage{orcidlink}

\newcommand{\rme}{\mathrm{e}}
\newcommand{\rmd}{\mathrm{d}}

\allowdisplaybreaks

\usepackage{float}
\usepackage{amsmath,amssymb,amsthm,amsfonts,graphicx,color,mathtools}
\usepackage[hmargin=2.5cm, vmargin=2.5cm]{geometry}
\usepackage{bbm}
\usepackage{amssymb}
\usepackage{pdfsync}
\usepackage{epstopdf}
\usepackage{cite}
\usepackage{graphicx}
\usepackage{multirow}
\usepackage[font=bf,aboveskip=15pt]{caption}
\usepackage[toc,page]{appendix}
\usepackage[colorlinks=true]{hyperref}
\renewcommand{\arraystretch}{1.3}

\allowdisplaybreaks[4]

\newcommand{\1}{\mathbf{1}}

\newcommand{\pp}{ {\partial} }

\newcommand{\RR}{{{\mathbb R}}}

\newcommand{\R} {\mathbb R}

\newcommand{\cuad}{{\sqcap\kern-.68em\sqcup}}

\newcommand{\be}{\begin{equation}}
\newcommand{\ee}{\end{equation}}

\newtheorem{definition}{Definition}[section]
\newtheorem{lemma}{Lemma}[section]
\newtheorem{proposition}{Proposition}[section]
\newtheorem{theorem}{Theorem}[section]
\newtheorem{corollary}{Corollary}[section]
\newtheorem{remark}{Remark}[section]
\newcommand{\bremark}{\begin{remark} \em}
	\newcommand{\eremark}{\end{remark} }

\numberwithin{equation}{section}
\begin{document}
\title[One-bubble solution with maximum points escaping to infinity]{One-bubble solution with maximum points escaping to infinity for the energy-critical heat equation in $\mathbb{R}^5$}

\author[Junichi Harada]{Junichi Harada}

\address[Junichi Harada]{Faculty of Education and Human Studies, Akita University}
\email{harada-j@math.akita-u.ac.jp}

\author[Qidi Zhang]{Qidi Zhang\orcidlink{0000-0002-0486-9868} }


\address[Qidi Zhang]{ 
Institut de Math\'{e}matiques de Jussieu, Sorbonne Universit\'{e}, Universit\'{e} Paris Cit\'{e}, 4 place, Jussieu, 75005 Paris, France
}
\email{qzhang@imj-prg.fr}



\begin{abstract}
	
We investigate the existence of a one-bubble solution with maximum points escaping to infinity for the energy-critical heat equation in dimension five. The key point of the construction is a refined inner linear theory as in \cite[Section 8]{Wei-Zhang-Zhou2022LLG}.

\end{abstract}

\maketitle

\tableofcontents

\section{Introduction and main results}

Given an integer $n\ge 3$, consider the energy-critical heat equation
\begin{equation}\label{critical-heat-eq}
\partial_{t} u = \Delta u + |u|^{\frac{4}{n-2}} u \mbox{ \ in \ } \mathbb{R}^n \times (0, \infty),
\quad
u(\cdot, 0) = u_0 \mbox{ \ in \ } \mathbb{R}^n.
\end{equation}
It is well-known that for any initial value $u_0\in\dot H^1(\R^n)$, \eqref{critical-heat-eq} admits a unique local-in-time solution $u(x,t)$ satisfying
 \[
 u(t)\in C([0,T);\dot H^1(\R^n))\cap C((0,T);L^\infty(\R^n))
 \]
 for some $T>0$.
We are interested in the rich and intricate behavior of solutions to this equation. Typically, depending on the size of the initial value,
solutions either decay to zero at a linear rate or blow up in finite time in a self-similar manner (Type I blow-up).
Beyond this classical dichotomy, however, the presence of bubbles gives rise to a much wider variety of dynamics.
By Caffarelli-Gidas-Spruck \cite{Caffarelli-Gidas-Spruck1989},
up to scaling and translation,
the unique positive steady-state solution of \eqref{critical-heat-eq} is the well-known Aubin-Talenti bubble
\begin{equation}\label{26June6-U-def}
	U(x) :=\alpha_n (1+|x|^2)^{-\frac{n-2}{2}},
	\quad \alpha_n :=[ n(n-2) ]^{\frac{n-2}{4}}.
\end{equation}

A basic example of such bubbling dynamics appears in the work of Fila and King \cite{FilaKing12}, who predicted, at a formal level,
global-in-time solutions in which a rescaled copy of $U$ spreads at
various growth rates. More precisely, they predicted that, for a given positive, radially
symmetric function $f_0(x)$ satisfying
\begin{equation}\label{move-26Sep9-4}
\lim_{|x| \to \infty} |x|^\gamma f_0(x) = A > 0
~\mbox{ and }~ \gamma > (n-2)/2,
\end{equation}
there exists a constant $k = k(f_0) > 0$ such that the solution
$u(x,t)$ of \eqref{critical-heat-eq} with $u_0 = kf_0$ behaves asymptotically
as
\begin{equation}\label{INT_Aug27_1}
u(x,t)
 \approx
 \begin{cases}
 \lambda(t)^{-\frac{n-2}{2}}
 U\big( \tfrac{x}{\lambda(t)} \big)
 & \text{for } |x| \ll \sqrt{t}
 \\
 \Psi(x,t)
 & \text{for } |x|>\sqrt{t},
 \end{cases}
 \end{equation}
where $\Psi(x,t)$ denotes a solution of the linear heat equation $\Psi_t = \Delta_x \Psi$ in $\R^n \times(0,\infty)$, with an initial value suitably chosen, depending on $\gamma$ (typically of self-similar form),
and the specific leading term of $\lambda(t)$ given in Table \ref{table_1}.
 \begin{table}[h]
 \begin{center}
 \renewcommand{\arraystretch}{1.0}
 \begin{tabular}{p{3em}|p{6em}p{5em}p{3em}c}
 & \centering $\frac{n-2}{2}<\gamma<2$ & \centering $\gamma=2$ & \centering $\gamma>2$ &
 \\ \hline
 $n=3$ & \centering $t^{1-\gamma}$ & \centering $t^{-1}(\ln t)^{2}$
 & \centering $t^{-1}$ &
 \\
 $n=4$ & \centering $t^{\frac{2-\gamma}{2}} (\ln t)^{-1}$ & \centering $1$ &
 \centering $(\ln t)^{-1}$ &
 \\
 $n=5$ & \centering $t^{2-\gamma}$ & \centering $(\ln t)^{2}$
 & \centering $1$ &
 \\ \hline
 \end{tabular}
 \\[2mm]
 \hspace{-25mm}
 If $n\geq6$ and $\gamma>\frac{n-2}{2}$,
 then
 $\lambda(t)=1$.
 \caption{Fila-King Conjecture (Rate of $\lambda(t)$)}
 \label{table_1}
 \end{center}
 \end{table}

Subsequently,
thanks to the development of the gluing method in Cort\'{a}zar-del Pino-Musso \cite{Green16JEMS} and D\'{a}vila-del Pino-Wei \cite{17HMF},
the predictions of Fila and King have been rigorously established in a series of works,
which can be summarized as follows:
\begin{itemize}
\item 
$n=3$, $\gamma>1$ (infinite-time blow-up): del Pino-Musso-Wei \cite{173D}.
\item
$n=4$, $\gamma>2$ (infinite-time blow-up): Wei-Zhang-Zhou \cite{infi4d}.
A counterpart for the 1-equivariant harmonic map heat flow including decaying and asymptotically stable cases: Wei-Zhang-Zhou \cite{TriHMF2026}.
\item
$n=5$, $\gamma>\frac{3}{2}$: Li-Wei-Zhang-Zhou
\cite{decay5d}.
\item
$n=6$, $\gamma >2$: Wei-Zhou \cite{Wei-Zhou-2025}.
\end{itemize}
For $n = 6$, Harada \cite{harada2025oscillatory} constructed radial
solutions $u \in C([0,\infty); \dot{H}^1(\mathbb{R}^6))$ that are
(i) positive and decaying, (ii) sign-changing and growing,
or (iii) sign-changing and oscillatory,
with initial data $u_0$ satisfying
$|u_0(x)| \sim |x|^{-2} (\ln |x|)^{-\beta}$ for $|x| \gg 1$, where
$\beta \in (\frac12, 1)$.
This construction is inspired by Gustafson-Nakanishi-Tsai \cite{GNT10CMP}.
In contrast with Harada's classification result
\cite{Harada2026-6DnearGround}, this construction demonstrates that
the dynamical behavior in $\dot{H}^{1}(\mathbb{R}^6)$ is much
different from the behavior in $H^{1}(\mathbb{R}^6)$.
As a counterpart of the energy-critical heat equation in $n=6$,
Sire, Wei, Zheng, and Zhou \cite{Sire-Wei-Zheng-Zhou-2026global} gave more general dynamics for the Yang-Mills heat flow in dimension $4$
with a wider class of initial values in the radial class.
For $n = 6$ and $u_0 \notin \dot{H}^{1}(\mathbb{R}^6)$, Hisa,
Takahashi, and Zhanpeisov \cite{Hisa-Takahashi-Zhanpeisov-2026infinite}
constructed
(i) sign-changing infinite-time blow-up solutions and
(ii) nonnegative, slowly decaying solutions, using a forward self-similar
solution, instead of a solution of the linear heat equation, as the
leading term.
Away from one bubble,
del Pino, Musso, and Wei \cite{TowerNLH} constructed bubble-tower solutions for $n\geq6$.
For infinite-time blow-ups in bounded domains, we refer to Galaktionov-King \cite{King03JDE},
Cort\'azar-del Pino-Musso \cite{Green16JEMS}, del Pino-Musso-Wei-Zheng \cite{del2018sign}, and Ageno-del Pino \cite{ageno2023infinite},
where the Dirichlet boundary plays an essential role in the mechanism.

In a related direction, Collot, Merle, and Rapha\"el
\cite{Collot17CMP} established a classification of the dynamics for
$n \ge 7$ when $u_0 \in \dot{H}^{1}(\mathbb{R}^n)$ is close to $U$ in
$\dot{H}^{1}$; see Harada \cite{Harada2026-6DnearGround} for the case $n = 6$, where $u_0 \in H^1(\mathbb{R}^6)$.
In \cite{Kihyun-Merle2025, kim-Merle-2026rigidity},
Kim and Merle gave rigid classification results for global solutions with $n \ge 7$ in radial and non-radial cases, respectively.

We now return to the solution constructed by Fila and King \cite{FilaKing12}. More
generally, in the absence of radial symmetry, the solution is
expected to behave as
\begin{equation*}
u(x,t)
=
\lambda(t)^{-\frac{n-2}{2}}
U\big( \tfrac{x-\xi(t)}{\lambda(t)} \big)
+
\Psi(x,t)
+
h_{\rm{error}}(x,t),
\end{equation*}
where $\xi(t) \in \R^n$ denotes the (possibly time-dependent)
center of the bubble, and $\Psi(x,t)$ denotes a certain solution of the linear heat equation:
$\Psi_t = \Delta_x \Psi$ in $\R^n \times(0,\infty)$;
we do not specify $\Psi$ further at this level of generality. The shape of the solution $u$ is maintained through a delicate interaction between
$\lambda(t)^{-\frac{n-2}{2}} U( \tfrac{x-\xi(t)}{\lambda(t)} )$
and $\Psi(x,t)$.

In the radially symmetric case treated in \cite{FilaKing12},
the center of the bubble is fixed at the origin, i.e.,
$\xi(t) \equiv 0$ for $t>0$. Fila and King \cite{FilaKing12} use a radial solution $\Psi(x,t)$ of $\Psi_t = \Delta_x \Psi$ in $\R^n \times(0,\infty)$ of the form
\begin{equation*}
\Psi(x,t) \approx t^{-\gamma/2} g(|x|/\sqrt{t})
\end{equation*}
when $n=5$ and $\gamma \in (3/2,2)$, where $g=g(r)$ is expected to be positive and strictly decreasing, attaining its maximum at $r=0$. Their solution moreover satisfies
$\|\Psi(\cdot,t)\|_{L_x^\infty(\R^n)}\ll\lambda(t)^{-\frac{n-2}{2}}$
(see \cite{FilaKing12} for details), so that it holds that
\begin{equation*}
\| u(\cdot,t) \|_{L_x^\infty(\R^n)}
\sim
\lambda(t)^{-\frac{n-2}{2}}
\big\| U(\tfrac{\cdot}{\lambda(t)}) \big\|_{L_x^\infty(\R^n)}
=
\lambda(t)^{-\frac{n-2}{2}}
U(0).
\end{equation*}
Consequently, in the radially symmetric case of \cite{FilaKing12},
for large $t>0$ the maximum point of $u(\cdot,t)$ is expected to be
located in a neighborhood of the maximum point of
$\lambda(t)^{-\frac{n-2}{2}}U(\tfrac{\cdot}{\lambda(t)})$,
namely $x=0$ (recall $\xi(t)\equiv 0$ in this case).
On the other hand, $\Psi(\cdot,t)$ attains its maximum at
$x=0$ for $t>0$, as noted above.
Thus, in the radially symmetric case, the
maximum point of $u(\cdot,t)$ lies in a neighborhood of the maximum point of $\Psi(\cdot,t)$ for large $t>0$. This observation, however, relies crucially on the radial symmetry. It is far from clear whether the similar phenomenon
persists generally. The present work is motivated
by the following question.

\medskip
\noindent
\textbf{Question.}
\textit{
Is it always the case that the maximum point of a global-in-time solution $u(\cdot,t)$ of
\eqref{critical-heat-eq} is located in a neighborhood of the maximum point of the corresponding solution $\Psi(\cdot,t)$ of the linear
heat equation in $\mathbb{R}^n$ for all
sufficiently large $t>0$?
}
\medskip

The main result, Theorem \ref{5d-main-th}, answers this question in
the negative when $n = 5$: there exists a global solution $u$ of
\eqref{critical-heat-eq} whose maximum points of $u(\cdot, t)$ are far away from those of the
corresponding solution $\Psi(\cdot,t)$ of the linear heat equation
(defined in \eqref{move26Sep9-3})
for sufficiently large $t$ (see Corollary \ref{move26Sep8-1-cor} (i)).

For $\gamma_{i}<5$, $i=1,2$, $c_{\sharp} \ne 0$, we take
\begin{equation}\label{move26Sep9-3}
\partial_t \Psi = \Delta \Psi \mbox{ \ in \ } \mathbb{R}^5 \times (0,\infty),
\quad 
\Psi(x,0) = |x|^{-\gamma_1} + c_{\sharp} x_1 |x|^{-\gamma_2 - 1} \mbox{ \ in \ } \mathbb{R}^5.
\end{equation}
The location of the maximum points of $\Psi(\cdot, t)$ will be given in Lemma \ref{Psi-maxpoint-lem} for general dimensions. Due to the time-translation invariance of \eqref{critical-heat-eq}, for technical convenience, we initiate our analysis at a fixed time $t_0$. The solution of \eqref{critical-heat-eq} can then be recovered by applying a time translation. The main theorem is stated below.
\begin{theorem}\label{5d-main-th}  
	Consider
	\begin{equation}\label{u-eq-5d} 
		\partial_t u=\Delta u+ |u|^{\frac{4}{3}} u
		\mbox{ \ in \ } 
		\mathbb{R}^5 \times (t_0,\infty).
	\end{equation}  
	Given constants $c_{\sharp} \ne 0$, $\zeta \in (0, 1)$, 
\begin{equation}\label{move-26Aug23-1}
\frac{3}{2} < \gamma_1 < 2, 
\quad
\frac{1}{11} (3+9\gamma_1) < \gamma_2 < \frac{1}{3}(-3+5\gamma_1),
    \end{equation}
then for $t_0$ sufficiently large, 
	there exists a solution $u$ of the form 
	\begin{equation}\label{u-behavior}
		u = 
		15^{\frac{3}{4}}
		\lambda^{-\frac{3}{2}} 
		\Big( 1+
		\Big|  \frac{x-(\xi_1,0,0,0,0) }{\lambda} \Big|^2 \Big)^{-\frac{3}{2}}
		\eta\Big( \frac{x- (\xi_1,0,0,0,0)}{\sqrt{t}} \Big)
        +
        \Psi
		+
		h_{\rm{error}}
	\end{equation}
with the initial value
\begin{equation}\label{move-26Sep7-3}
\begin{aligned}
		u(x,t_0) = \ &
		15^{\frac{3}{4}}
		\lambda(t_0)^{-\frac{3}{2}} 
		\Big( 1+
		\Big|  \frac{x-(\xi_1(t_0),0,0,0,0) }{\lambda(t_0)} \Big|^2 \Big)^{-\frac{3}{2}}
		\eta\Big( \frac{x- (\xi_1(t_0),0,0,0,0)}{\sqrt{t_0}} \Big)
        +
        \Psi(x,t_0)
\\
&
		+
		\lambda(t_0)^{-\frac{3}{2}}
    g_{0}(t_0) \eta\Big( \frac{2}{\ln t_0} \frac{x- (\xi_1(t_0),0,0,0,0)}{\lambda(t_0)} \Big)
    Z_{0}\Big( \frac{x- (\xi_1(t_0),0,0,0,0)}{\lambda(t_0)} \Big).
\end{aligned}
	\end{equation}
Here $g_{0}(t_0)$ is a scalar satisfying $|g_{0}(t_0)| \lesssim (\ln t_0) t_0^{3-2\gamma_1}$, and $Z_0$ is a radially symmetric, smooth, and exponentially decaying function.
\begin{itemize}
    \item[(i)] $u$ is even with respect to the $i$-th component of $x$ for $i =2,3,4,5$.
    
    \item[(ii)]  $\lambda=\lambda(t) , \xi_1 = \xi_1(t) \in C^1([t_0,\infty))$ satisfy
\begin{equation}\label{move-26Aug23-6}
    \begin{aligned}
&
		\lambda =
		D_{5,\gamma_1,1} [ 1
		+
		O( (\ln t_0)^{-\frac{1}{2} \zeta } ) ]
		t^{2-\gamma_1},  
\\
&
\xi_1 =
        c_{\sharp} C_{\xi^{[0]}} 
\begin{cases}
	\frac{2}{7-3\gamma_1 - \gamma_2} [ 1 + O( (\ln t_0)^{-\frac{1}{4} \zeta } ) ] t^{\frac{7}{2} - \frac{3}{2} \gamma_1 - \frac{1}{2}\gamma_2 }
	&
	\mbox{ \ if \ } 3 \gamma_1 + \gamma_2 \ne 7
	\\
	[ 1 + O( (\ln t_0)^{-\frac{1}{4}\zeta } ) ] \ln t
	&
	\mbox{ \ if \ } 3 \gamma_1 + \gamma_2 = 7
\end{cases}
    \end{aligned}
	\end{equation}
with positive constants $D_{5,\gamma_1, 1}$, $C_{\xi^{[0]}}$ given in \eqref{qd25Jan4-3}, \eqref{qd25Jan4-4} respectively. 

\item[(iii)]
$h_{\rm{error}} \in C( \mathbb{R}^5 \times [t_0,\infty) )$ is a remainder term satisfying $|h_{\rm{error}}| \lesssim t^{-\frac{1}{2}\gamma_1 } \ln t (\ln t_0)^{6} \langle t^{-1} |x|^2 \rangle^{-1}$.

\end{itemize}
\end{theorem}

The center of the bubble $(\xi_1,0,0,0,0)$ appearing in the solution constructed in Theorem \ref{5d-main-th} goes to $\infty$ as $t \to \infty$, provided that $3\gamma_1+\gamma_2 \le 7$. It is natural to expect that the maximum point of $u(\cdot,t)$ also goes to $\infty$ as $t \to \infty$. This is a subtle point, since the bubble's center and the maximum point of $u(\cdot,t)$ are a priori different quantities. Nevertheless, at least when $\gamma_1 > \gamma_2$, it can be deduced from Theorem \ref{5d-main-th} that the maximum point of $u(\cdot,t)$ also goes to $\infty$, as we show in the following corollary.
\begin{corollary}\label{move26Sep8-1-cor}

Let constants $c_{\sharp} \ne 0$, $\zeta \in (0, 1)$, $\gamma_1,\gamma_2$ satisfy \eqref{move-26Aug23-1}, and let $u$ be the solution constructed in Theorem \ref{5d-main-th}. Then
\begin{itemize}
\item[(i)]
The set of maximum points of $u(\cdot,t)$ (respectively, $\Psi(\cdot, t)$) is nonempty for $t>t_0$. There exists a constant $C>0$ and a sufficiently small constant $\epsilon_1 >0$ such that for any maximum point $z(t)$ of $u(\cdot,t)$, and any maximum point $w(t)$ of $\Psi(\cdot, t)$, it holds that $|z(t)| \le C t^{-\epsilon_1} |w(t)|$ for $t>t_0$.

\item[(ii)] 
Moreover, under the additional assumption $\gamma_1 > \gamma_2$,
any maximum point $z(t)$ of $u(\cdot,t)$ satisfies
\begin{equation}\label{move-26Aug23-12}
|z(t) - (\xi_1(t),0,0,0,0)| \le \frac{1}{2} |\xi_1(t)|.
\end{equation}
In particular, $|z(t)| \to \infty$ as $t\to \infty$ if $3\gamma_1 + \gamma_2 \le 7$.
\end{itemize}
\end{corollary}

The following remark gives some features of the solution constructed in Theorem \ref{5d-main-th}.
\begin{remark}\label{move-26Aug24-2-rmk}

\begin{enumerate}

\item\label{move26Sep7-4-rmk} When $\gamma_1 = \gamma_2$ and $0<|c_{\sharp}|<1$, we have $u>0$ by $\Psi(x,0) \ge (1-|c_{\sharp}|) |x|^{-\gamma_1}$ and $\Psi > 0$. See the proof in Subsection \ref{move-26Sep8-2-subsec}. However, when $\gamma_1 \ne \gamma_2$ and $c_{\sharp} \ne 0$, since $\Psi(x,0)$ is sign-changing, then $u$ may also be sign-changing.

\item Both $u$ and $\Psi$ are even with respect to the $i$-th component of $x$ for $i =2,3,4,5$. This suggests that the constructed solution could inherit certain symmetry properties from the leading-order terms that drive its dynamics.

\item 
By solving the linearized elliptic problem as Harada \cite{harada2025oscillatory} for instance, we may get a next-order expansion of the solution $u$, and then the error term will be smaller.
\end{enumerate}
\end{remark}

The construction is based on the parabolic gluing method originated from Cort\'{a}zar-del Pino-Musso \cite{Green16JEMS} and D\'{a}vila-del Pino-Wei \cite{17HMF}. The key point of this paper is to use the re-gluing method, that is, using the gluing method again in the inner problem, to establish a refined inner linear theory for mode $0$, mode $1$, and higher modes to eliminate the loss of some positive power of $R$ due to the energy method and comparison theorem, where $R$ is the radius of the region of the inner problem up to a multiplicity of a constant. Since the linear theory established in Cort\'{a}zar-del Pino-Musso \cite[Proposition 7.1]{Green16JEMS} loses some positive power of $R$, the direct application of \cite[Proposition 7.1]{Green16JEMS} to this paper will lead to the failure of the parameter choice to close the gluing process. Instead, we modify  \cite[Proposition 7.1]{infi4d} as \cite[Section 8]{Wei-Zhang-Zhou2022LLG} and \cite[Proposition 7.2]{TriHMF2026} to deduce a refined linear theory in Section \ref{re-gluing-Sec}. The re-gluing process was first used in the analysis of linearization of the harmonic map heat flow at mode $0$ in D\'{a}vila-del Pino-Wei \cite[Subsection 7.5]{17HMF}. Additionally, we supplement the fixed-point argument with the continuity argument in Section \ref{conti-Sec}.

For clarity, we present the proof scheme here. In Section \ref{gluing-sys-sec}, we give the approximate solution and formulate the gluing system by the standard method. As a relatively independent section in the paper, Section \ref{re-gluing-Sec} gives the key linear theory for the inner problem by re-gluing. In Section \ref{Proj-spherical-HM-Sec}, we present spherical harmonic functions and project the inner problem into these spaces. In Section \ref{lambda-xi-leading-in-topo-Sec}, we use the dominant parts of reduced orthogonality equations \eqref{qd25Jan5} to give the leading terms of the scaling parameter $\lambda$ and the translation parameter $\xi$. Then we design the topology for the inner problem.
In Section \ref{inner-map-sec}, we use the linear theory in Section \ref{re-gluing-Sec} to present the mapping $\mathcal{T}^{\rm{in}}[\mathcal{J}]$ for solving the inner problem \eqref{inner-tau-eq} with $n=5$ in a new time variable $\tau$ defined in \eqref{tau-def},
where we extend the right-hand side of the inner problem in \eqref{calJ-def} for technical convenience to achieve continuity of mappings in Section \ref{conti-Sec} later.
In Section \ref{modi-orth-eq-sec}, we formulate the reduced orthogonality equations and the solution mapping $\mathcal{S}_6$, $\vec{\mathcal{S}}$ in \eqref{mu1-xi-sys} for $\dot{\lambda}_1$, $\dot{\xi}^{[1]}$, which are the derivative of the minor terms of $\lambda$ and $\xi$.
In Section \ref{out-solve-sec}, we give the fixed-point scheme, and $\mathcal{T}_{\rm{o}} [ \mathcal{G} \1_{|x| \le T^{9}, t \le T} ]$ as a solution mapping for the outer problem in some bounded domain and corresponding estimates by convolution estimates.
In Section \ref{inner-orth-subsec}, we give the estimate of $\mathcal{T}^{\rm{in}}[\mathcal{J}](y,\tau(t)) \times \mathcal{S}_6 \times \vec{\mathcal{S}}$ as the solution mappings for the inner problem and reduced orthogonality equations. In Section \ref{conti-Sec}, we prove the continuity of $\mathcal{T}_{\rm{o}} [ \mathcal{G} \1_{|x| \le T^{9}, t \le T} ]$, $\mathcal{T}^{\rm{in}}[\mathcal{J}](y,\tau(t))$, $\mathcal{S}_6$, $\vec{\mathcal{S}}$ in a norm space $\mathcal{X}$ defined in \eqref{Xspace-def}. 
In Section \ref{Final-u-Sec},  combining Sections \ref{out-solve-sec}, \ref{inner-orth-subsec}, \ref{conti-Sec}, and applying the Schauder fixed-point theorem in $\mathcal{X}$, we find a solution for the gluing system and the reduced orthogonality equations in some bounded domain. So we get a desired solution of $\partial_t u=\Delta u+ |u|^{\frac{4}{3}} u$ in some bounded domain. Finally, the compactness argument leads to a solution $u$ in $\mathbb{R}^5 \times (t_0, \infty)$, and we complete the proof.

\section*{Notations}

\begin{itemize} 

\item  $\eta(x)$ is a radially symmetric, smooth cut-off function satisfying  $\eta(x)=1$ for $|x|\le 1$, $\eta(x)=0$ for $|x|\ge 2$, and $0\le \eta(x) \le 1$ for all $x\in \mathbb{R}^n$, where $n$ denotes the spatial dimension corresponding to the variable $x$.

\item For integers $i \le j$, denote a set as $\overline{i,j} := \{i,i+1,\dots, j\}$. Denote the natural number set as $\mathbb{N} = \{ 0,1,2,\dots \}$.
	
	\item 
    Unless otherwise specified, all constants in this paper are independent of 
$t_0$ and $T$. We write $a\lesssim b$ (respectively $a \gtrsim b$) if there exists a  constant $C > 0$ such that $a \le  Cb$ (respectively $a \ge  Cb$). Set $a \sim b$ if $b \lesssim a \lesssim b$. Denote $f_1=O(f_2)$ if $|f_1| \lesssim f_2$. 
	
	\item For non-negative quantities $C_1, C_2$, the symbol $C_1 \gg (\ll) C_2$ means that there exists a sufficiently large (small) positive constant $c$ such that $C_1 \ge (\le) c \, C_2$.
	
	\item 
	For $r>0$, denote $B_r = \{ x \in \mathbb{R}^n \mid |x|<r \}$.

	\item For $x\in \mathbb{R}^n$, denote $|x|=\big( \sum\limits_{i=1}^n x_i^2 \big)^{1/2}$, and the Japanese bracket $\langle x \rangle = \sqrt{|x|^2+1}$.

\item Given a vector $\mathbf{k} = (k_1, k_2, \dots, k_n) \in \mathbb{R}^n$, denote its $\ell_1$ norm as $\| \mathbf{k} \|_{\ell_1} = \sum_{i=1}^n |k_i|$.

\item 
For $c \in \mathbb{R}$, $c-$ denotes a number that is less than $c$ and can be chosen arbitrarily close to $c$.

\item Denote $\1_{\Omega}(x)$ as the indicator function with $\1_{\Omega}(x)=1$ if $x\in \Omega$ and $\1_{\Omega}(x)=0$ if $x\not\in \Omega$. We will use $\1_{\Omega}$ to denote $\1_{\Omega}(x)$ if no ambiguity.
\end{itemize}

\medskip

{\textbf{Statement on the Use of AI Tools.}} We used AI tools solely for language polishing. All mathematical arguments and proofs were written by the authors.

\section{Approximate solution and gluing system}\label{gluing-sys-sec}
 
Given an integer $n\ge 3$, consider the energy-critical heat equation 
\begin{equation}\label{u-eq} 
		\pp_t u=\Delta u+\left|u\right|^{\frac{4}{n-2}}u
		\mbox{ \ in \ } 
		\mathbb{R}^n \times (t_0,\infty). 
\end{equation} 
Throughout this paper, we always assume $t_0 \gg 1$. For $U$ given in \eqref{26June6-U-def}, the linearized operator $\Delta+\frac{n+2}{n-2}U^{\frac{4}{n-2}}$ has bounded kernels 
\begin{equation}\label{Zi-def}
\begin{aligned}
&
Z_i(x) := \partial_{x_i} U(x) = \alpha_n (2-n) (1+|x|^2)^{-\frac n 2} x_i,
 \quad
   i \in \overline{1, n}, 
\\
& Z_{n+1}(x) := \frac{n-2}{2}U(x)+
   x\cdot \nabla U(x) = \frac{n-2}{2} \alpha_n \frac{1-|x|^2 }{(1+|x|^2 )^{\frac n 2} }.
\end{aligned}
\end{equation}
We take the leading term of the solution to \eqref{u-eq} to be of the following form 
\begin{equation*}
	u_1(x,t) := \lambda^{-\frac{n-2}{2}} U (y)
	\eta(\tilde{y})
	+\Psi(x,t),
	\mbox{ \ where \ } 
y:=\frac{x-\xi}{\lambda},
\quad \tilde{y} := \frac{x-\xi}{\sqrt{t}},
\end{equation*}
$\lambda=\lambda(t) \in C^1((t_0,\infty), \mathbb{R}_+)$, $\xi= (\xi_1(t), \xi_2(t),\dots, \xi_n(t) ) \in C^1 ( (t_0,\infty), \mathbb{R}^n )$ will be determined later, and
\begin{equation*}
\partial_t \Psi = \Delta \Psi \mbox{ \ in \ } \mathbb{R}^n \times (0,\infty),
\quad 
\Psi(x,0) = |x|^{-\gamma_1} + c_{\sharp} x_1 |x|^{-\gamma_2 -1} \mbox{ \ in \ } \mathbb{R}^n
\end{equation*}
with constants $\gamma_{i}<n$, $i=1,2$ and $c_{\sharp} \ne 0$. Note that when $\gamma_1 = \gamma_2$ and $0<|c_{\sharp}|<1$, $\Psi(x,0) > 0$. When $\gamma_1 \ne \gamma_2$ and $c_{\sharp} \ne 0$, $\Psi(x,0)$ is sign-changing. We formulate
\begin{equation}\label{move-26Aug30-1}
	\Psi(x,t) :=  ( 4\pi t )^{-\frac n2} \int_{\RR^n} 
	\rme^{-\frac{|x-z|^2}{4t} } (|z|^{-\gamma_1} + c_{\sharp} z_1 |z|^{-\gamma_2-1}) \rmd z.
\end{equation} 
It is readily seen that $\Psi(x,t)$ is even with respect to the $i$-th component of $x$ for $i \in \overline{2,n}$. That is,
\begin{equation}
\Psi((x_1,x_2,\dots, -x_i, \dots, x_n),t) = \Psi((x_1,x_2,\dots, x_i, \dots, x_n),t)
\mbox{ \ for \ } i \in \overline{2,n}.
\end{equation}
We will ensure that the constructed solution $u$ inherits this symmetry.

The following basic lemma is useful in this paper.
\begin{lemma}\label{qd25Oct22-3-lem}

	Given an integer $n>0$, $C_0>0$, $b>-n$, then for all $t>0$,
	\begin{equation*}
		\int_{\mathbb{R}^n}	\rme^{-C_0 \frac{|x-z|^2}{t} } |z|^{b} \rmd z
		\sim t^{\frac{b+n}{2}} \1_{|x| \le t^{\frac{1}{2}}} + t^{\frac{n}{2}} |x|^{b} \1_{|x| > t^{\frac{1}{2}}},
	\end{equation*}
	where the ``$\sim$'' only depends on $n, C_0, b$.
	
\end{lemma}

\begin{proof}
We postpone the proof of Lemma \ref{qd25Oct22-3-lem} to Section \ref{lem-proof-26June10-sec}.
\end{proof}

Given $\gamma_{i}<n$, $i=1,2$, for any $\mathbf{m} \in \mathbb{N}^n$,
\begin{equation}\label{qd25Dec25-4}
\begin{aligned}
&
	\big| \partial_{x}^{\mathbf{m}}
	\Psi(x,t) \big|
	= 
	\Big|
	(4\pi t)^{-\frac n2} \int_{\mathbb{R}^n} 
	\partial_{x}^{\mathbf{m}} \Big(	\rme^{-\frac{|x-z|^2}{4t} } \Big) (|z|^{-\gamma_1} + c_{\sharp} z_1 |z|^{-\gamma_2 -1}) \rmd z \Big|
	\\
	\lesssim \ & t^{-\frac n2 - \frac{\| \mathbf{m} \|_{\ell_1}}{2} } \int_{\mathbb{R}^n}	\rme^{-\frac{|x-z|^2}{5t} } ( |z|^{-\gamma_1} + |c_{\sharp}| |z|^{-\gamma_2} ) \rmd z
    \\
\stackrel{{\text{Lemma }} \ref{qd25Oct22-3-lem}}{\sim} \ &
t^{- \frac{\| \mathbf{m} \|_{\ell_1}}{2} }
\big[
(t^{-\frac{\gamma_1}{2}} +
|c_{\sharp}| t^{-\frac{\gamma_2}{2}}
) 
\1_{|x| \le t^{\frac{1}{2}}} + (|x|^{-\gamma_1} + |c_{\sharp}| |x|^{-\gamma_2}) \1_{|x| > t^{\frac{1}{2}}} \big].
\end{aligned}
\end{equation}
Thus, $\Psi(x,t)$ is smooth in $\mathbb{R}^n \times (0,\infty)$. For $|x| \ge t^{\frac{1}{2}}$, by Lemma \ref{move-26Spe2-1-lem}, a better estimate holds
\begin{equation}\label{move-26Sep2-2}
\big| \partial_{x}^{\mathbf{m}}
	\Psi(x,t) \big| \lesssim
|x|^{-\gamma_1 - \| \mathbf{m} \|_{\ell_1}} + |c_{\sharp}| |x|^{-\gamma_2 - \| \mathbf{m} \|_{\ell_1} }.
\end{equation}
In particular, if $\gamma_i>\frac{n}{2} - 1$, $i=1,2$, $\Psi(\cdot, t) \in \dot{H}^1(\mathbb{R}^n)$ for $t > 0$. Although the construction is carried out in some weighted $L^{\infty}$ spaces, we will eventually see that both $\Psi(\cdot, t) $ and $u(\cdot, t_0)$ belong to $\dot{H}^1(\mathbb{R}^n)$ with $n=5$.

The change of variables gives
\begin{equation*}
\Psi(x,t)
= (4\pi)^{-\frac n2} t^{-\frac{\gamma_1}{2}} \int_{\mathbb{R}^n} 
\rme^{-4^{-1} |t^{-\frac{1}{2}} x-a|^2 }  (|a|^{-\gamma_1} + c_{\sharp} t^{\frac{\gamma_1 - \gamma_2}{2}} a_1 |a|^{-\gamma_2-1}) \rmd a.
\end{equation*}
Using the parity in the integrand, we have
\begin{equation}\label{Psi0t}
\Psi(0,t)
= C_{n,\gamma_1,1} t^{-\frac{\gamma_1}{2}},
\mbox{ \ where \ } 
C_{n,\gamma_1,1} := (4\pi)^{-\frac n2}  \int_{\mathbb{R}^n} 
\rme^{-4^{-1} |a|^2 } |a|^{-\gamma_1} \rmd a > 0.
\end{equation}
Besides,
\begin{equation}\label{nabla-Psi}
\nabla \Psi(x,t)
= (4\pi)^{-\frac n2} t^{-\frac{\gamma_1}{2}} \int_{\mathbb{R}^n} 
\rme^{-4^{-1} |t^{-\frac{1}{2}} x-a|^2 }  
(-2^{-1}) (t^{-\frac{1}{2}} x - a) t^{-\frac{1}{2}}
(|a|^{-\gamma_1} + c_{\sharp} t^{\frac{\gamma_1 - \gamma_2}{2}} a_1 |a|^{-\gamma_2-1}) \rmd a,
\end{equation}
\begin{equation}\label{partial-Psi0t}
\partial_{x_i} \Psi(0,t)
=
c_{\sharp} t^{- \frac{\gamma_2 + 1}{2}} 2^{-1} (4\pi)^{-\frac n2} \int_{\mathbb{R}^n} 
\rme^{-4^{-1} |a|^2 }  
a_i a_1 |a|^{-\gamma_2-1} \rmd a
=
\begin{cases}
c_{\sharp} C_{n, \gamma_2, 2} t^{- \frac{\gamma_2 + 1}{2}} 
& 
\mbox{ \ if \ } i=1
\\
0
& \mbox{ \ if \ } i \in \overline{2,n},
\end{cases}
\end{equation}
where we denote
\begin{equation}
C_{n, \gamma_2, 2} := 2^{-1} (4\pi)^{-\frac n2} \int_{\mathbb{R}^n} 
\rme^{-4^{-1} |a|^2 } a_1^2 |a|^{-\gamma_2-1} \rmd a > 0.
\end{equation}

For $\gamma_i \in [0,n)$, $i=1,2$, we have an alternative estimate
\begin{equation}\label{move-26Aug19-1}
|\Psi(x,t)| \lesssim t^{-\frac{\gamma_1}{2}} + | c_{\sharp} | t^{-\frac{\gamma_2}{2} - \frac{1}{2}} |x|,
\end{equation}
which gives better time decay in the orthogonality equations later. Indeed, for $\gamma_1 \in [0, n)$,
\begin{equation*}
\Big| ( 4\pi t )^{-\frac n2} \int_{\RR^n} 
	\rme^{-\frac{|x-z|^2}{4t} } |z|^{-\gamma_1} \rmd z \Big|
\lesssim t^{-\frac{\gamma_1}{2}},
\end{equation*}
and for $\gamma_2 \in [0, n)$, using $\int_{\RR^n} \rme^{-\frac{|-z|^2}{4t} } z_1 |z|^{-\gamma_2-1} \rmd z = 0$, we have
\begin{align*}
& \Big| ( 4\pi t )^{-\frac n2} \int_{\RR^n} 
\big( 
\rme^{-\frac{|x-z|^2}{4t} } 
-
\rme^{-\frac{|-z|^2}{4t} }
\big) c_{\sharp} z_1 |z|^{-\gamma_2-1} \rmd z \Big|
\\
= \ & \Big| ( 4\pi t )^{-\frac n2} \int_{\RR^n} 
\rme^{-\frac{|\theta x-z|^2}{4t} }
(-1) \frac{1}{2t} (\theta x - z) \cdot x c_{\sharp} z_1 |z|^{-\gamma_2-1} \rmd z \Big|
\\
\lesssim \ & | c_{\sharp} | t^{-\frac{n}{2} - \frac{1}{2}} |x|  \int_{\RR^n} 
\rme^{-\frac{|\theta x-z|^2}{5t} } |z|^{-\gamma_2} \rmd z
\stackrel{\gamma_2 \ge 0}{\lesssim}
| c_{\sharp} | t^{-\frac{\gamma_2}{2} - \frac{1}{2}} |x|
\end{align*}
for some $\theta \in [0,1]$. We analyze the location of the maximum points of $\Psi(\cdot,t)$ in the following lemma.
\begin{lemma}\label{Psi-maxpoint-lem}

Given an integer $n\ge 1$, $\gamma_i \in (0, n)$, $i=1,2$, $c_{\sharp} \ne 0$, $t>0$, for $\Psi$ given in \eqref{move-26Aug30-1}, denote $ S_{\Psi(\cdot, t), {\rm{max}}} := \{ x \mid \Psi(x, t) = \sup_{w \in \mathbb{R}^n} \Psi(w, t) \} $.
Then,
$S_{\Psi(\cdot, t), {\rm{max}}}$ is a nonempty compact set, and there exists a large constant $C \ge 1$ only depending on $n, \gamma_1, \gamma_2$ such that
\begin{equation}
C^{-1} \min\{ t^{\frac{1}{2}},  |c_{\sharp}| t^{\frac{\gamma_1 - \gamma_2}{2} + \frac{1}{2}}  \} \le |x| \le C \max\{ t^{\frac{1}{2}},  |c_{\sharp}|^{\frac{1}{\gamma_2}} t^{\frac{\gamma_1}{2 \gamma_2}}\}
\mbox{ \ for \ } x \in S_{\Psi(\cdot, t), {\rm{max}}}.
\end{equation}

\end{lemma}

\begin{proof}

For $\gamma_{i} \in (0,n)$, $i=1,2$, and $t>0$, by \eqref{qd25Dec25-4}, we have $\lim\limits_{|x| \to \infty} \Psi(x,t) = 0$. For $t>0$, since $\Psi(0,t)>0$ by \eqref{Psi0t}, $S_{\Psi(\cdot, t), {\rm{max}}} \ne \emptyset$ is a compact set. For $x$ as a critical point of $\Psi(\cdot, t)$, it holds that $\nabla \Psi(x,t) = 0 $.
Denote the key part of $\nabla \Psi$ in \eqref{nabla-Psi} as
\begin{equation*}
\vec{f}(x,t)
= \int_{\mathbb{R}^n} 
\rme^{-4^{-1} |t^{-\frac{1}{2}} x-a|^2 } (t^{-\frac{1}{2}} x - a)
(|a|^{-\gamma_1} + c_{\sharp} t^{\frac{\gamma_1 - \gamma_2}{2}} a_1 |a|^{-\gamma_2-1}) \rmd a
\end{equation*}
with $\vec{f}=0$ is and only if $\nabla \Psi =0$. It is ready to get
\begin{equation*}
\vec{f}(0,t)
= 
\big( 
- c_{\sharp} t^{\frac{\gamma_1 - \gamma_2}{2}}  \int_{\mathbb{R}^n} 
\rme^{-4^{-1} |a|^2} a_1^2 |a|^{-\gamma_2-1} \rmd a ,\underbrace{0,0, \dots, 0}_{n-1 \text{ zeros}} \big),
\end{equation*}
where the $1$st component is nonzero.
\begin{equation*}
\begin{aligned}
&
|\vec{f}(x,t) - \vec{f}(0,t)|
\\
= \ & 
\Big|
\int_{\mathbb{R}^n} 
\Big[ 
\rme^{-4^{-1} |t^{-\frac{1}{2}} x-a|^2 } (t^{-\frac{1}{2}} x - a)
-
\rme^{-4^{-1} |-a|^2 } (- a)
\Big]
(|a|^{-\gamma_1} + c_{\sharp} t^{\frac{\gamma_1 - \gamma_2}{2}} a_1 |a|^{-\gamma_2-1}) \rmd a
\Big|
\\
= \ & \Big| \int_{\mathbb{R}^n} 
\Big[ \rme^{-4^{-1} |\theta t^{-\frac{1}{2}} x - a|^2 } (-2^{-1}) 
\big[ 
(\theta t^{-\frac{1}{2}} x - a) \cdot t^{-\frac{1}{2}} x \big]  (\theta t^{-\frac{1}{2}} x - a) 
+
\rme^{-4^{-1} |\theta t^{-\frac{1}{2}} x - a|^2 } t^{-\frac{1}{2}} x
\Big]
\\
& \times
(|a|^{-\gamma_1} + c_{\sharp} t^{\frac{\gamma_1 - \gamma_2}{2}} a_1 |a|^{-\gamma_2-1}) \rmd a \Big|
\\
\lesssim \ & t^{-\frac{1}{2}} |x| \int_{\mathbb{R}^n} 
\rme^{-5^{-1} |\theta t^{-\frac{1}{2}} x - a|^2 } 
(|a|^{-\gamma_1} + |c_{\sharp}| t^{\frac{\gamma_1 - \gamma_2}{2}} |a|^{-\gamma_2}) \rmd a
\lesssim 
t^{-\frac{1}{2}} |x| 
(1 + |c_{\sharp}| t^{\frac{\gamma_1 - \gamma_2}{2}})
\end{aligned}
\end{equation*}
for some $\theta \in [0,1]$, where for the last step, we use $\gamma_i \in [0,n)$, $i=1,2$, and Lemma \ref{qd25Oct22-3-lem}. If
$ |x| \ll \min\{ t^{\frac{1}{2}},  |c_{\sharp}| t^{\frac{\gamma_1 - \gamma_2}{2} + \frac{1}{2}}  \} $, we have $
t^{-\frac{1}{2}} |x| 
(1 + |c_{\sharp}| t^{\frac{\gamma_1 - \gamma_2}{2}}) \ll |c_{\sharp}| t^{\frac{\gamma_1 - \gamma_2}{2}}
$. Then $\vec{f}(x,t) = \vec{f}(0,t) + \vec{f}(x,t) - \vec{f}(0,t) \ne 0$, which implies $\nabla \Psi(x,t) \ne 0$.

Let us estimate the upper bound for the maximum point. By $\gamma_i<n$, $i=1,2$, Lemma \ref{qd25Oct22-3-lem}, we have
\begin{equation*}
\begin{aligned}
&
|\Psi(x,t)| \lesssim  t^{-\frac n2} \int_{\mathbb{R}^n} 
	\rme^{-\frac{|x-z|^2}{4t} } (|z|^{-\gamma_1} + |c_{\sharp}| |z|^{-\gamma_2}) \rmd z
\\
\sim \ & (t^{-\frac{\gamma_1}{2}} + |c_{\sharp}| t^{- \frac{\gamma_2}{2}}) \1_{|x| \le t^{\frac{1}{2}}} + (|x|^{-\gamma_1} + |c_{\sharp}| |x|^{-\gamma_2}) \1_{|x| > t^{\frac{1}{2}}}.
\end{aligned}
\end{equation*}
Let us compare with 
$
\Psi(0,t)
\sim t^{-\frac{\gamma_1}{2}}
$. In the region $|x| > t^{\frac{1}{2}}$, 
for $\gamma_i >0$, $i=1,2$, $|x|
\gg \max\{ t^{\frac{1}{2}},  |c_{\sharp}|^{\frac{1}{\gamma_2}} t^{\frac{\gamma_1}{2 \gamma_2}}\} $ implies $|x|^{-\gamma_1} + |c_{\sharp}| |x|^{-\gamma_2} \ll t^{-\frac{\gamma_1}{2}}$, which means that these $\Psi(x,t)$ are not the maximum value.

In sum, we conclude the necessary condition for $x$ to be the maximum point of $\Psi(\cdot, t)$.
\end{proof}

Next, we will give the inner-outer gluing system. Define the error of $u$ as
\begin{equation*}
	\mathcal{E}[u]:=- \pp_t u+\Delta u+|u|^{\frac{4}{n-2}} u.
\end{equation*}
We look for an exact solution $u$ of \eqref{u-eq} in the form
\begin{equation}\label{u-def}
	u = \lambda^{-\frac{n-2}{2}} U (y)
	\eta(\tilde{y})
	+ \Psi(x,t) + \psi(x,t)+\lambda^{-\frac{n-2}{2}}\phi\Big(\frac{x-\xi}{\lambda},t \Big) \eta_{R}(y),
\quad 
\eta_{R}(y) := \eta(\frac{y}{R}) = \eta\Big(\frac{x-\xi}{\lambda R }\Big)
\end{equation}
with 
\begin{equation}
R= R(t) := t^{\beta}
\end{equation}
and a constant $\beta>0$ to be determined later.
Straightforward calculation gives
\begin{align*}
	\mathcal{E}[u]
= \ &
\big(
\dot{\lambda} \lambda^{-\frac{n}{2}}  Z_{n+1}(y) 
+
\lambda^{-\frac{n}{2}} \dot{\xi} \cdot 
(\nabla U)( y )
\big) \eta(\tilde{y} )
+
\mathcal{E}_{U}^{\rm{cut}}
+
\lambda^{-\frac{n+2}{2}} U(y)^{\frac{n+2}{n-2}}
\big( \eta(\tilde{y})^{\frac{n+2}{n-2}} 
-
\eta( \tilde{y} )
\big)
\\
& - \pp_t \psi + \Delta \psi
- \lambda^{-\frac{n-2}{2}} \pp_t \phi (y,t) \eta_R(y) 
+
\lambda^{-\frac{n+2}{2}} \Delta_y \phi(y, t) \eta_R(y) 
\\
& + \Lambda_1 + \Lambda_2
+ \mathcal{N}
+ \frac{n+2}{n-2} \lambda^{-2}
U(y)^{\frac{4}{n-2}}
\eta(\tilde{y})^{\frac{4}{n-2}} 
\big( \Psi + \psi +\lambda^{-\frac{n-2}{2}}\phi(y,t) \eta_{R}(y) \big),
\end{align*}
where
\begin{equation}\label{E-eta-def} 
	\mathcal{E}_{U}^{\rm{cut}}:=   \lambda^{-\frac{n-2}{2}} U(y) 
	( 2^{-1} t^{-1} \tilde{y}  + t^{-\frac{1}{2}} \dot{\xi} )\cdot ( \nabla \eta ) ( \tilde{y} )  
	+
	2\lambda^{-\frac{n}{2}} t^{-\frac{1}{2}} 
	( \nabla U ) (y) \cdot 
	( \nabla \eta ) ( \tilde{y} )
	+
	\lambda^{-\frac{n-2}{2}} t^{-1} U ( y ) ( \Delta \eta ) ( \tilde{y} ), 
\end{equation}
\begin{equation}\label{Lambda1-phi}
	\begin{aligned}  
		\Lambda_1 = \Lambda_1[\phi,\lambda,\xi ] := \ & \lambda^{-\frac{n+2}{2}} R^{-2} \phi(y, t)  ( \Delta 
		\eta )(\frac{y}{R})
		+ 
		2 \lambda^{-\frac{n+2}{2}}  R^{-1}  \nabla_y \phi(y, t) \cdot (\nabla \eta )(\frac{y}{R}) 
		\\
		&
		+
		\lambda^{-\frac{n-2}{2}} \phi( y, t) 
		( \nabla \eta )(\frac{y}{R}) \cdot \Big[ \frac{\dot{\xi}}{\lambda R}
		+ \frac{y}{R} \frac{1}{\lambda R } \frac{\rmd (\lambda R)}{\rmd t}  \Big],
	\end{aligned}
\end{equation}
\begin{equation}\label{Lambda2-phi}
	\Lambda_2 = \Lambda_2[\phi,\lambda,\xi] := 
	\dot{\lambda} \lambda^{-\frac n2}   \Big( \frac{n-2}{2} \phi(y,t)
	+ y \cdot \nabla_y \phi (y,t) \Big) \eta_R(y)
	+ \lambda^{-\frac{n}{2}} \dot{\xi} \cdot \nabla_y \phi(y, t) \eta_R(y),
\end{equation}
\begin{equation}\label{N-def} 
\mathcal{N} =	\mathcal{N} [\psi,\phi,\lambda,\xi ] :=  
	|u|^{\frac{4}{n-2}} u 
	-
	\lambda^{-\frac{n+2}{2}} ( U(y) \eta(\tilde{y}) )^{\frac{n+2}{n-2}}
	-
	\frac{n+2}{n-2} \lambda^{-2} 
	(U(y) \eta(\tilde{y}) )^{\frac{4}{n-2}} 
	\big(\Psi + \psi + \lambda^{-\frac{n-2}{2}} \phi(y,t) \eta_R(y) \big).
\end{equation}
Here we use the symbol `` $\dot{}$ '' to denote $\frac{\rmd}{\rmd t}$ for brevity.

To make $\mathcal{E}[u] =0$, it suffices to solve the following inner-outer gluing system.
\\
\textbf{The outer problem:}
\begin{equation}\label{out-eq-original}
		\pp_t \psi 
		= \Delta \psi + \mathcal{G} \mbox{ \ in \ } \mathbb{R}^n \times (t_0,\infty),
		\quad 
		\psi(\cdot,t_0)=0 \mbox{ \ in \ } \mathbb{R}^n,
\end{equation}
where
\begin{equation}\label{g}
\begin{aligned}
&
\mathcal{G} = \mathcal{G}[\psi,\phi,\lambda,\xi](x,t) := \big(
\dot{\lambda} \lambda^{-\frac{n}{2}}  Z_{n+1}(y) 
+
\lambda^{-\frac{n}{2}} \dot{\xi} \cdot 
(\nabla U)( y )
\big) ( \eta(\tilde{y} ) - \eta_R(y) )
+
\mathcal{E}_{U}^{\rm{cut}}
\\
&
+
\lambda^{-\frac{n+2}{2}} U(y)^{\frac{n+2}{n-2}}
\big( \eta(\tilde{y})^{\frac{n+2}{n-2}} 
-
\eta( \tilde{y} )
\big)
+ \Lambda_1 + \Lambda_2
+ \mathcal{N}
\\
&
+ 
\frac{n+2}{n-2} \lambda^{-2}
U(y)^{\frac{4}{n-2}}
\big(
\eta(\tilde{y})^{\frac{4}{n-2}} - \eta_{R}(y) \big) (\Psi + \psi)
+ 
\frac{n+2}{n-2} \lambda^{-2}
U(y)^{\frac{4}{n-2}}
\big(
\eta(\tilde{y})^{\frac{4}{n-2}} - 1 \big) \lambda^{-\frac{n-2}{2}}\phi(y,t) \eta_{R}(y).
\end{aligned}
\end{equation}
\textbf{The inner problem:}
\begin{equation}
		\lambda^{2} \partial_t \phi (y,t) 
		= \Delta_y \phi(y, t) + 
		\frac{n+2}{n-2}
		U(y)^{\frac{4}{n-2}} \phi(y,t) + \mathcal{H}
\mbox{ \ for \ } t \in (t_0, \infty), y \in B_{2R},
\end{equation}
where
\begin{equation}\label{H-def}  
	\mathcal{H} = \mathcal{H}[\psi,\lambda,\xi ](y,t) :=  
	\dot{\lambda} \lambda Z_{n+1}(y)
	+\lambda \dot{\xi} \cdot (\nabla U )(y) 
	+\frac{n+2}{n-2}\lambda^{\frac{n-2}{2}} U(y)^{\frac{4}{n-2}}
	\big(
	\Psi(\lambda y+\xi,t)
	+
	\psi(\lambda y+\xi, t)
	\big). 
\end{equation}

For convenience of the continuity argument, given any $T \in (t_0,\infty)$, we consider
the following equations with a suitably chosen solution $(\psi, \phi)$ to
\begin{equation}\label{outer-problem}
		\pp_t \psi 
		= \Delta \psi + \mathcal{G} \mbox{ \ in \ } B_{T^{9}}  \times (t_0,T],
		\quad 
		\psi(\cdot,t_0)=0 \mbox{ \ in \ } B_{T^{9}},
\end{equation}
\begin{equation}\label{inner-problem}
	\lambda^{2} \partial_t \phi (y,t) 
	= \Delta_y \phi(y, t) + 
	\frac{n+2}{n-2}
	U(y)^{\frac{4}{n-2}} \phi(y,t) + \mathcal{H} \eta\Big(\frac{y}{2R(t)}\Big)
	\mbox{ \ for \ } t \in (t_0, T], y \in B_{2R},
\end{equation}
and $\lambda>0$ will be solved in $(t_0, T]$. Once solving \eqref{outer-problem} and \eqref{inner-problem}, we find a solution $u$ in $B_{T^{9}}  \times (t_0,T]$. Finally, we will take $T\to \infty$ to find a solution $u$ in $\mathbb{R}^n \times (t_0,\infty)$.
Set a new time variable
\begin{equation}\label{tau-def}
	\tau=\tau[\lambda](t) :=\int_{t_0}^t \lambda(s)^{-2} \rmd s + \tau_0 \mbox{ \ for \ } t \in [t_0, T], \quad
	 \tau(t_0) = \tau_0 :=  C_{\tau} t_0 (t_0^{2-\gamma_1})^{-2},
\end{equation}
where $C_\tau$ is a sufficiently large constant independent of $t_0$ to be determined later, and the choice of $\tau_0$ will be used for deducing \eqref{tau-est} later. Then,
\eqref{inner-problem} is rewritten as  
\begin{equation}\label{inner-tau-eq}
	\partial_{\tau}\phi (y, t(\tau)) = \Delta_y \phi(y, t(\tau))
	+\frac{n+2}{n-2}U(y)^{\frac{4}{n-2}}\phi(y,t(\tau))
	+ \mathcal{H}(y,t(\tau)) \eta\Big(\frac{y}{2R(t(\tau))}\Big)
	\mbox{ \ for \ }
	\tau \in (\tau_0, \tau(T)], y \in B_{2 R(t(\tau))}.
\end{equation}

\section{Projection on spherical harmonics functions}\label{Proj-spherical-HM-Sec}

We establish mappings to solve the inner problem in different modes by projecting onto spherical harmonic functions. Set $\mathbf{SH}$ as an orthonormal basis in $L^2(S^{n-1})$ made up of spherical harmonic functions
\begin{equation}\label{SH-def}
\mathbf{SH} :=
\bigg\{
\Upsilon_{i,j} \ \Big| \ i \in \mathbb{N}, 
\begin{cases}
	j=1 & \mbox{ \ if \ } i=0
	\\
	j =1,2,\dots, n
	& \mbox{ \ if \ } i=1
	\\
	j=1,2,\dots, \binom{n+i-1}{i} - \binom{n+i-3}{i-2}
	& \mbox{ \ if \ } i \ge 2
\end{cases}
\bigg\},
\end{equation} 
where $\Upsilon_{i,j} \in \mathbf{SH}$ satisfy
\begin{equation}\label{move-26Sep5-2}
\begin{aligned}
&
	\Delta_{S^{n-1} } \Upsilon_{i,j} + i (n-2+i) \Upsilon_{i,j} = 0 
	\mbox{ \ in \ }  S^{n-1},
\quad
\int_{ S^{n-1}} \Upsilon_{i_1, j_1}(w) \Upsilon_{i_2, j_2}(w) \rmd w =\delta_{(i_1, j_1), (i_2, j_2)},
\\
&
\Upsilon_{0,1}(w)= C_{{\rm{har}},0}, \ \Upsilon_{1,j}(w)= C_{{\rm{har}},1} w_j, \ 
j = 1,2,\dots, n
\mbox{ \ for \ } w \in S^{n-1} 
\end{aligned}
\end{equation}
with some nonzero constants $C_{{\rm{har}},0}$, $C_{{\rm{har}},1}$.

Denote a label set of the first $n+1$ spherical harmonic functions as
\begin{equation}
\mathbf{M_{0,1}} := \{ (i,j) \mid i=0,j=1 \mbox{ or } i=1, j=1,2,\dots,n \}.
\end{equation}
For any admissible function $f(y)$ defined in $B_{C}$, $C \in (0, \infty]$ with the notation $B_{\infty} = \mathbb{R}^n$, we define 
\begin{equation}\label{qd26Jan3-1}
\begin{aligned}
&
f_{i,j}^{\sharp}(|y|) := \int_{ S^{n-1} } f(|y| \theta) \Upsilon_{i,j}(\theta) \rmd \theta,
	\quad
	\tilde{f}_{i,j}(y) := f_{i,j}^{\sharp}(|y|) \Upsilon_{i,j} \big( \frac{y}{|y|} \big)
\mbox{ \ for \ } (i,j) \in \mathbf{M_{0,1}},
\\
&
	\tilde{f}_{\perp}(y) := f - \sum\limits_{(i,j) \in \mathbf{M_{0,1}} } \tilde{f}_{i,j}.
\end{aligned}
\end{equation}
In this way, for the spatial variable, we define $\mathcal{H}_{i,j}^{\sharp}(|y|,t)$ for $(i,j) \in \mathbf{M_{0,1}}$; $\tilde{\mathcal{H}}_{i,j}(y,t)$ for $(i,j) \in \mathbf{M_{0,1}} \cup \{ \perp \}$, where we mean $\tilde{\mathcal{H}}_{i,j} = \tilde{\mathcal{H}}_{\perp}$ when $(i,j) \in \{ \perp \}$. To solve \eqref{inner-tau-eq}, it suffices to solve
\begin{equation}\label{qd26Jan3-8}
\begin{aligned}
\partial_{\tau} \tilde{\phi}_{i,j} = \ & \Delta_y \tilde{\phi}_{i,j}
	+
	\frac{n+2}{n-2} U(y)^{\frac{4}{n-2}} \tilde{\phi}_{i,j}
    \\
    & 
	+ \tilde{\mathcal{H}}_{i,j}(y,t(\tau)) \eta\Big(\frac{y}{2R(t(\tau))}\Big)
	\mbox{ \ for \ }
	\tau \in (\tau_0, \tau(T)], y \in B_{2 R(t(\tau))},
	\quad (i,j) \in \mathbf{M_{0,1}} \cup \{ \perp \}.
\end{aligned}
\end{equation}

The following useful identities are prepared for reduced orthogonality equations later.

Claim: Given an integer $n\ge 3$, $C \in (0,\infty]$, suppose $f(y) \in L^{\infty}(B_C)$ if $C\in (0,\infty)$ and $|f(y)| \lesssim \langle y \rangle^{(-2)-}$ if $C = \infty$, then
\begin{equation}\label{qd25Jan4}
	\begin{aligned}
		&
		\int_{B_C}
		\tilde{f}_{1,i}(y) Z_i(y) \rmd y = \int_{B_C} f(y) Z_i(y) \rmd y
		\mbox{ \ for \ } i \in \overline{1,n};
		\quad
		\int_{B_C}
		\tilde{f}_{0,1}(y) Z_{n+1}(y) \rmd y
		= \int_{B_C} f(y) Z_{n+1}(y) \rmd y;
		\\
        &
        \int_{B_C}
		\tilde{f}_{1,i}(y) Z_{j}(y) \rmd y = 0
        \mbox{ \ for \ } i \in \overline{1,n}, j\in \overline{1,n+1} \setminus \{ i \};
        \quad
        \int_{B_C}
		\tilde{f}_{0,1}(y) Z_{j}(y) \rmd y
		= 0 \mbox{ \ for \ } j \in \overline{1,n};
        \\
        &
		\int_{B_C}
		\tilde{f}_{\perp}(y) Z_i(y) \rmd y = 0
		\mbox{ \ for \ } i \in \overline{1,n+1}.
	\end{aligned}
\end{equation}

\begin{proof}[Proof of \eqref{qd25Jan4}]

By assumption, the integrals above are well-defined. For $i, j \in \overline{1,n}$, since $\frac{Z_j(y)}{\Upsilon_{1,j} ( \frac{y}{|y|} )}$ is radially symmetric with a natural continuous extension at $y=0$, then
\begin{align*}
			&
			\int_{B_C}
			\tilde{f}_{1,i}(y) Z_{j}(y) \rmd y
				=
				\int_{B_C}
				f_{1,i}^{\sharp}(|y|) \Upsilon_{1,i} \big( \frac{y}{|y|} \big) \frac{Z_j(y)}{\Upsilon_{1,j} \big( \frac{y}{|y|} \big)} \Upsilon_{1,j} \big( \frac{y}{|y|} \big) \rmd y
			\\ 
			= \ & \int_{0}^{C} \int_{\partial B_r}
			f_{1,i}^{\sharp}(|y|) \Upsilon_{1,i} \big( \frac{y}{|y|} \big) \Big[ \frac{Z_j(y)}{\Upsilon_{1,j} \big( \frac{y}{|y|} \big)} \Big](|y|) \Upsilon_{1,j} \big( \frac{y}{|y|} \big) \rmd S \rmd |y|
			\\
			= \ & \int_{0}^{C} \int_{S^{n-1}}
			f_{1,i}^{\sharp}(r) \Upsilon_{1,i} (\theta) \Big[ \frac{Z_j(y)}{\Upsilon_{1,j} \big( \frac{y}{|y|} \big)} \Big](r) \Upsilon_{1,j}(\theta) r^{n-1} \rmd \theta \rmd r
			\\
			= \ & \delta_{i,j} \int_{0}^{C} f_{1,i}^{\sharp}(r) \Big[ \frac{Z_j(y)}{\Upsilon_{1,j} \big( \frac{y}{|y|} \big)} \Big](r) r^{n-1} \rmd r
			= \delta_{i,j} \int_{0}^{C} \int_{ S^{n-1} } f(r\theta) \Upsilon_{1,i}(\theta)  \Big[ \frac{Z_j(y)}{\Upsilon_{1,j} \big( \frac{y}{|y|} \big)} \Big](r) r^{n-1} \rmd \theta \rmd r
			\\
			= \ & \delta_{i,j} \int_{0}^{C} \int_{ \partial B_r } f(y) \Upsilon_{1,i}\big( \frac{y}{|y|} \big) \frac{Z_j(y)}{\Upsilon_{1,j} \big( \frac{y}{|y|} \big)} \rmd S \rmd |y|
			= \delta_{i,j} \int_{B_C} f(y) Z_i(y) \rmd y.
\end{align*}
The other identities follow similarly, and we omit details.
\end{proof}

\section{Leading terms of $\lambda$, $\xi$, and topology of the inner problem}\label{lambda-xi-leading-in-topo-Sec}

{\textbf{We take $n=5$ hereafter except for Section \ref{re-gluing-Sec}.}}
We will first give the leading terms of $\lambda$ and $\xi$, which are determined by the dominant parts of the reduced orthogonality equations \eqref{qd25Jan5} later. Set
\begin{equation}
R_0 = \ln t_0.
\end{equation}
Denote $\lambda_0$ as the leading term of $\lambda$. $\lambda_0$ is determined by
\begin{equation}\label{mu0-eq-original}
\begin{aligned}
&
	\int_{B_{R_0}} \Big( \dot{\lambda}_0 \lambda_0 Z_{6}(y)+\frac{7}{3}\lambda_0^{\frac{3}{2}} U(y)^{\frac{4}{3}} \Psi(0, t) \Big) Z_{6}(y) \rmd y = 0
\\
\Leftrightarrow \ & \dot{\lambda}_0
= - \frac{7}{3} \Big( \int_{B_{R_0}} Z_{6}(y)^2 \rmd y \Big)^{-1}
\int_{B_{R_0}} U(y)^{\frac{4}{3}}  Z_{6}(y) \rmd y \, \lambda_0^{\frac{1}{2}} \Psi(0, t).
\end{aligned}
\end{equation}
Denote $\xi^{[0]} = (\xi_1^{[0]}, \xi_2^{[0]},\dots, \xi_5^{[0]})$ as the leading term of $\xi$. For $i\in \overline{1,5}$, $\xi_i^{[0]}$ is determined by
\begin{equation*}
\begin{aligned}
& \int_{B_{R_0}} \Big[ \lambda_0 \dot{\xi}_i^{[0]} Z_i(y) 
+\frac{7}{3} \lambda_0^{\frac{3}{2}} U(y)^{\frac{4}{3}}
(\nabla \Psi)(0,t) \cdot (\lambda_0 y+\xi^{[0]})  \Big] Z_i(y) \rmd y = 0
\\
\Leftrightarrow \ &
\int_{B_{R_0}} \lambda_0 \dot{\xi}_i^{[0]} Z_i(y)^2 \rmd y
+
\int_{B_{R_0}} \frac{7}{3} \lambda_0^{\frac{3}{2}} U(y)^{\frac{4}{3}}
(\partial_{x_i} \Psi)(0,t) \lambda_0 y_i Z_i(y) \rmd y = 0
\\
\Leftrightarrow \ &
\dot{\xi}_i^{[0]}
= - \frac{7}{3}
\Big( \int_{B_{R_0}} Z_1(y)^2 \rmd y \Big)^{-1}
\int_{B_{R_0}} U(y)^{\frac{4}{3}}
y_1 Z_1(y) \rmd y
\, 
\lambda_0^{\frac{3}{2}} (\partial_{x_i} \Psi)(0,t),
\end{aligned}
\end{equation*}
where we use the formulae of $U, Z_{i}$. It is equivalent to
\begin{equation}\label{mu0-eq}
\dot{\lambda}_0 = A_1(R_0) \lambda_0^{\frac{1}{2}} \Psi(0, t),
\quad
\dot{\xi}_i^{[0]}
= A_2(R_0) \lambda_0^{\frac{3}{2}} (\partial_{x_i} \Psi)(0,t) \mbox{ \ for \ } i \in \overline{1,5},
\end{equation}
where
\begin{align}
& 
A_1(R_0):= - \frac{7}{3} \Big( \int_{B_{R_0}} Z_{6}(y)^2 \rmd y \Big)^{-1}
\int_{B_{R_0}} U(y)^{\frac{4}{3}}  Z_{6}(y) \rmd y
\notag
\\
= \ &
- \frac{7}{3} \Big( \int_{\mathbb{R}^5} Z_{6}(y)^2 \rmd y + O(R_0^{-1}) \Big)^{-1}
\Big[
-\frac{9}{14} \int_{\mathbb{R}^5} U(y)^{\frac{7}{3}} \rmd y
+
O(R_0^{-2})
\Big]
\notag
\\
= \ &
\frac{3}{2} \Big( \int_{\mathbb{R}^5} Z_{6}(y)^2 \rmd y \Big)^{-1}
\int_{\mathbb{R}^5} U(y)^{\frac{7}{3}} \rmd y
\big(
1+
O( R_0^{-1} )
\big)
\sim 1,
\label{AR-def}
\\ 
& 
A_2(R_0):= - \frac{7}{3}
\Big( \int_{B_{R_0}} Z_1(y)^2 \rmd y \Big)^{-1}
\int_{B_{R_0}}  U(y)^{\frac{4}{3}}
y_1 Z_1(y) \rmd y
\notag
\\
= \ & - \frac{7}{3}
\Big( \int_{\mathbb{R}^{5}} Z_1(y)^2 \rmd y + O(R_0^{-3}) \Big)^{-1}
\Big(
\int_{\mathbb{R}^{5}}  U(y)^{\frac{4}{3}}
y_1 Z_1(y) \rmd y
+
O(R_0^{-2})
\Big)
\notag
\\
= \ & - \frac{7}{3}
\Big( \int_{\mathbb{R}^{5}} Z_1(y)^2 \rmd y \Big)^{-1}
\Big(
\int_{\mathbb{R}^{5}}  U(y)^{\frac{4}{3}}
y_1 Z_1(y) \rmd y
\Big)
\big( 1 + O(R_0^{-2}) \big) \sim 1
\notag
\end{align}
for $t_0 \gg 1$. Here for $A_1(R_0)$, we have used 
$
\int_{\mathbb{R}^5} U(y)^{\frac{4}{3}} Z_{6}(y) \rmd y = -\frac{9}{14} \int_{\mathbb{R}^5} U(y)^{\frac{7}{3}} \rmd y $.
For $\lambda_0$, we take
\begin{equation}\label{mu-0-explicit}
\lambda_0 =
\Big( \frac{1}{2} \int_{1}^{t} A_1(R_0)\Psi(0, s) \rmd s
\Big)^{2}.
\end{equation}  
By \eqref{Psi0t},
\begin{equation*}
\lambda_0
= 
\Big[ C_{5,\gamma_1,1} \frac{3}{4}
\frac{\int_{\mathbb{R}^5} U(y)^{\frac{7}{3}} \rmd y}{\int_{\mathbb{R}^5} Z_{6}(y)^2 \rmd y} \Big]^{2}
\Big[ \int_{1}^{t}  
\big(
1+
O( R_0^{-1} )
\big)  s^{-\frac{\gamma_1}{2}} \rmd s
\Big]^{2}.
\end{equation*}
When $\gamma_1<2$,
\begin{equation*}
\int_{1}^{t}  
\big(
1+
O( R_0^{-1} )
\big)  s^{-\frac{\gamma_1}{2}} \rmd s
=
\frac{2}{2-\gamma_1} ( t^{1-\frac{\gamma_1}{2}} - 1 )
+
O\big( t^{1-\frac{\gamma_1}{2}}
R_0^{-1}
\big)
= \frac{2}{2-\gamma_1} t^{1-\frac{\gamma_1}{2}}
[ 1
+
O(R_0^{-1}) ].
\end{equation*}
Hence,
\begin{equation}
		\lambda_0
= D_{5,\gamma_1,1} [ 1
+
O( R_0^{-1} ) ]
t^{2-\gamma_1},
\end{equation}
where
\begin{equation}\label{qd25Jan4-3}
D_{5,\gamma_1,1} :=
\Big[ C_{5,\gamma_1,1} \frac{3}{2(2-\gamma_1)}
\frac{\int_{\mathbb{R}^5} U(y)^{\frac{7}{3}} \rmd y}{\int_{\mathbb{R}^5} Z_{6}(y)^2 \rmd y} \Big]^{2} >0.
\end{equation}
Then
\begin{equation}
	\begin{aligned}
		&
		\dot{\lambda}_0 \stackrel{\eqref{mu0-eq}}{=} A_1(R_0) \lambda_0^{\frac{1}{2}} \Psi(0, t)
		\stackrel{\eqref{AR-def} \eqref{Psi0t}}{=} 
		 \frac{3}{2}
		\frac{\int_{\mathbb{R}^5} U(y)^{\frac{7}{3}} \rmd y}{\int_{\mathbb{R}^5} Z_{6}(y)^2 \rmd y} 
		[
		1+
		O( R_0^{-1} )
		] D_{5,\gamma_1,1}^{\frac{1}{2}}
		t^{\frac{2-\gamma_1}{2}}
		[ 1 + O( R_0^{-1} ) ]^{\frac{1}{2}} C_{5,\gamma_1,1} t^{-\frac{\gamma_1}{2}}
		\\
		= \ & C_{5,\gamma_1,1} D_{5,\gamma_1,1}^{\frac{1}{2}} \frac{3}{2}
		\frac{\int_{\mathbb{R}^5} U(y)^{\frac{7}{3}} \rmd y}{\int_{\mathbb{R}^5} Z_{6}(y)^2 \rmd y} 
		[
		1+
		O( R_0^{-1} ) ] 
		t^{1-\gamma_1}
\stackrel{\eqref{qd25Jan4-3}}{=}  D_{5,\gamma_1,1} (2-\gamma_1) [
1 + O( R_0^{-1} ) ]  
t^{1-\gamma_1}.
	\end{aligned}
\end{equation}

By \eqref{partial-Psi0t},
\begin{equation}\label{xi0-value}
\begin{aligned}
&
\dot{\xi}_i^{[0]} = 0 \mbox{ \ for \ } i \in \overline{2,5},
\\
&
\dot{\xi}_1^{[0]}
= A_2(R_0) \lambda_0^{\frac{3}{2}} (\partial_{x_1} \Psi)(0,t)
= A_2(R_0) \lambda_0^{\frac{3}{2}} c_{\sharp} C_{5,\gamma_2,2} t^{-\frac{\gamma_2 +1}{2}}
\stackrel{\eqref{AR-def}}{=}  c_{\sharp} C_{\xi^{[0]}} [ 1 + O( R_0^{-1} ) ]
t^{\frac{5}{2}-\frac{3\gamma_1}{2} - \frac{\gamma_2}{2}},
\end{aligned}
\end{equation}
where we denote
\begin{equation}\label{qd25Jan4-4}
C_{\xi^{[0]}} := - \frac{7}{3} C_{5,\gamma_2,2} D_{5,\gamma_1,1}^{\frac{3}{2}}
\Big( \int_{\mathbb{R}^{5}} Z_1(y)^2 \rmd y \Big)^{-1}
\int_{\mathbb{R}^{5}}  U(y)^{\frac{4}{3}}
y_1 Z_1(y) \rmd y
\stackrel{\eqref{Zi-def}}{>} 0.
\end{equation}

We take
\begin{equation}
\xi_i^{[0]} = 0 \mbox{ \ for \ } i \in \overline{2,5},
\quad
\xi_1^{[0]} = c_{\sharp} C_{\xi^{[0]}} 
\begin{cases}
	\frac{2}{7-3\gamma_1 - \gamma_2} [ 1 + O(R_0^{-1}) ] t^{\frac{7}{2} - \frac{3\gamma_1}{2} - \frac{\gamma_2}{2}}
	&
	\mbox{ \ if \ } 3 \gamma_1 + \gamma_2 \ne 7
	\\
	[ 1 + O(R_0^{-1}) ] \ln t
	&
	\mbox{ \ if \ } 3 \gamma_1 + \gamma_2 = 7.
\end{cases}
\end{equation}
We note that the extra addition of a constant vector in $\xi^{[0]}$ will only contribute a translation to the final solution. In sum, $\lambda_0$ and $\xi^{[0]}$ are smooth functions in $(t_0,\infty)$. Under the parameter restrictions
\begin{equation}\label{mov-26Aug11-1}
\gamma_1 \in (\frac{3}{2}, 2), 
\quad 
\gamma_1 + \gamma_2 \ge 3,
\quad 
t_0 \gg 1,
\end{equation}
then
\begin{equation}
\begin{aligned}
&
\lambda_0 \sim t^{2-\gamma_1},
\quad
\dot{\lambda}_0 \sim t^{1-\gamma_1},
\quad
\xi^{[0]} = (\xi_1^{[0]}, 0, 0, 0, 0),
\\
&
|\xi^{[0]}| 
\sim 
\begin{cases}
t^{\frac{7}{2} - \frac{3\gamma_1}{2} - \frac{\gamma_2}{2}}
	&
	\mbox{ \ if \ } 3 \gamma_1 + \gamma_2 \ne 7
	\\
 \ln t
	&
	\mbox{ \ if \ } 3 \gamma_1 + \gamma_2 = 7
\end{cases},
\quad
|\dot{\xi}^{[0]}|
\sim t^{\frac{5}{2}-\frac{3\gamma_1}{2} - \frac{\gamma_2}{2}};
\quad
|\xi^{[0]}| \lesssim \lambda_0 \ll t^{\frac{1}{2}},
\quad
|\dot{\xi}^{[0]}| \lesssim \dot{\lambda}_0.
\end{aligned}
\end{equation}

For the technical convenience of the continuity argument, we will first construct a solution for $t\in [t_0, T]$ with an arbitrary constant $T>2 t_0$. 
And we make the following ansatz about $\lambda$ and $\xi$: 
\begin{align}
		&
		\lambda=\lambda_0+\lambda_1,
		\quad
		\xi = \xi^{[0]} + \xi^{[1]},
		\mbox{ \ where \ } 
		\lambda_1=\lambda_1(t) \in C^1 ([t_0,T], \mathbb{R}),
        \quad
        \xi^{[1]} =
        (\xi_1^{[1]}(t), 0, 0, 0 , 0)
		\in C^1 ([t_0,T], \mathbb{R}^5),
        \notag
		\\
		&
		|\lambda_1 |\le \lambda_0/9 , 
		\quad
		|\dot{\lambda}_1 |\le \dot{\lambda}_0 / 9,
		\quad
		|\xi^{[1]}| \le |\xi^{[0]}|/9,
		\quad
		|\dot{\xi}^{[1]}| \le |\dot{\xi}^{[0]}| / 9,
        \label{mu1-ansatz}
\end{align}
which implies 
\begin{equation}
	\begin{aligned}
		&
		\frac{8}{9} \lambda_0 \le \lambda \le \frac{10}{9} \lambda_0,
		\quad
		\frac{8}{9} \dot{\lambda}_0 \le \dot{\lambda} \le \frac{10}{9} \dot{\lambda}_0,
		\quad
		\frac{8}{9} |\xi^{[0]}|
		\le |\xi| \le \frac{10}{9} |\xi^{[0]}|,
		\quad
		\frac{8}{9} |\dot{\xi}^{[0]}|
		\le
		|\dot{\xi}| \le
		\frac{10}{9} |\dot{\xi}^{[0]}|;
		\\
		&
		|\xi| \lesssim \lambda \ll t^{\frac{1}{2}},
		\quad
		|\dot{\xi}| \lesssim \dot{\lambda}.
	\end{aligned}
\end{equation}

Recall \eqref{tau-def}. For the convenience of the continuity argument for the inner problem, we extend 
\begin{equation}\label{qd26May1-7}
	\tau=\tau[\lambda](t) :=\int_{t_0}^t \Big\{ \lambda(s) \1_{s\le T} + \big[ (T+1-s)\lambda(T) + (s-T) s^{2-\gamma_1} \big] \1_{T < s\le T+1} + s^{2-\gamma_1} \1_{s > T+1} \Big\}^{-2} \rmd s + \tau_0
\end{equation}
for $t \in [t_0, \infty)$. $\tau(t) \in C^1([t_0, \infty))$ is strictly monotonically increasing.
Denote $t = t[\lambda](\tau)$ as the inverse function well-defined for $\tau \in [\tau_0,\infty)$. For $\gamma_1 > \frac{3}{2}$ and $C_\tau \gg 1$, there exists a constant $C_{tm} \ge 1$ such that
\begin{equation}\label{tau-est}
\begin{aligned}
&
C_{tm}^{-1} t^{2\gamma_1 -3} \le \tau(t)  
\le C_{tm} t^{2\gamma_1 -3}
\mbox{ \ for \ } t \in [t_0, \infty); \quad
\tau_0 = \tau(t_0);
\\
&
(C_{tm}^{-1} \tau)^{\frac{1}{2\gamma_1 -3}} \le  t(\tau) \le (C_{tm} \tau)^{\frac{1}{2\gamma_1 -3}}
\mbox{ \ for \ } \tau \in [\tau_0, \infty); \quad t_0 = t(\tau_0).
\end{aligned}
\end{equation}

For $n=5$, under the ansatz \eqref{mu1-ansatz}, the upper bound for the leading term of $\mathcal{H}$ is given by
\begin{equation}\label{qd25Jan4-2}
	| \dot{\lambda} \lambda Z_{6}(y)| + | \lambda^{\frac{3}{2}} U(y)^{\frac{4}{3}}
			\Psi(0,t) |
	\lesssim t^{3-2\gamma_1} \langle y \rangle^{-3} \sim (\tau(t))^{-1} \langle y \rangle^{-3}.
\end{equation}
Inspired by Propositions \ref{mode0-regluing-prop}, \ref{regluing-mode1-prop}, \ref{regluing-higher-mode-prop} about the inner linear theory later, we define the norm
\begin{equation}\label{inner-norm}
	\|g\|_{\rm{in}}:= \sup_{t \in [t_0,T], y\in \overline{B_{4 R}} }
	[ R_0^{6} t^{3-2\gamma_1} \ln t \langle y \rangle^{-1} ]^{-1} ( \langle y\rangle |\nabla g(y,t)| + |g(y,t)| )
\end{equation}
for the inner problem \eqref{inner-problem} with $n=5$. We will solve \eqref{inner-problem} in the space
\begin{equation}\label{Bin-def}
\begin{aligned}
	B_{\rm{in}} := 
	\big\{ g(y,t)  \mid & \ g(y,t)\in C\big( [t_0,T]; C^1 ( \overline{B_{4 R}} ) \big), 
	\quad
	\|g\|_{\rm{in}} \le 1, 
    \\
    &
    \mbox{ $g(y,t)$ is even with respect to the $i$-th component of $y$ for $i \in \overline{2,5}$}
    \big\}.
\end{aligned}
\end{equation}
The larger domain $B_{4 R}$ provides technical convenience in the continuity argument in \eqref{qd26May1-2} later.

\section{Mappings for solving the inner problem in different modes}\label{inner-map-sec}

To solve the inner problem \eqref{inner-tau-eq}, it suffices to solve the components projected onto spherical harmonic functions, $\tilde{\phi}_{i,j}$  in \eqref{qd26Jan3-8} with right-hand side $\tilde{\mathcal{H}}_{i,j}(y,t(\tau)) \eta( \frac{y}{2R(t(\tau))} )$, where $(i,j) \in \mathbf{M_{0,1}} \cup \{ \perp \}$ label the spherical harmonic functions. We will give $\mathcal{T}_{i,j}^{\rm{in}}$ as the inverse mapping of $\tilde{\phi}_{i,j}$ in this section.

For the continuity argument for the mappings to solve the inner problem in Section \ref{conti-Sec} later, we extend $\mathcal{H}(y,t(\tau)) \eta(\frac{y}{2R(t(\tau))})$ to $\tau > \tau(T)$ in the following form
\begin{equation}\label{calJ-def}
\begin{aligned}
&
	\mathcal{J} = \mathcal{J}[\psi,\lambda,\xi ](y,\tau) 
    \\
    := \ & \mathcal{H}(y,t(\tau)) \eta\Big(\frac{y}{2R(t(\tau))}\Big) \1_{\tau_0 < \tau \le \tau(T)}
	+
	\mathcal{H}(y,T) \eta\Big(\frac{y}{2R(T)}\Big) \1_{ \tau > \tau(T)}
	\mbox{ \ for \ } (y,\tau) \in \mathbb{R}^5 \times (\tau_0, \infty).
\end{aligned}
\end{equation}
Since $\eta$ is radially symmetric, by \eqref{qd26Jan3-1}, we have that for $(j,k) \in \mathbf{M_{0,1}}$,
\begin{equation}\label{move-26Aug21-5}
\begin{aligned}
&
\mathcal{J}_{j,k}^{\sharp}(|y|, \tau) = \int_{ S^{4} } \mathcal{J}(|y| \theta, \tau) \Upsilon_{j,k}(\theta) \rmd \theta
= 
\mathcal{H}_{j,k}^{\sharp}(|y| ,t(\tau)) \eta\Big(\frac{|y|}{2R(t(\tau))}\Big) \1_{\tau_0 < \tau \le \tau(T)}
	+ \mathcal{H}_{j,k}^{\sharp}(|y|,T) \eta\Big(\frac{|y|}{2R(T)}\Big) \1_{ \tau > \tau(T)},
\\
&
\tilde{\mathcal{J}}_{j,k}(y, \tau) = \mathcal{J}_{j,k}^{\sharp}(|y|, \tau) \Upsilon_{j,k}(y/|y|)
= 
\tilde{\mathcal{H}}_{j,k}(y ,t(\tau)) \eta\Big(\frac{|y|}{2R(t(\tau))}\Big) \1_{\tau_0 < \tau \le \tau(T)}
	+ \tilde{\mathcal{H}}_{j,k}(y,T) \eta\Big(\frac{|y|}{2R(T)}\Big) \1_{ \tau > \tau(T)},
\\
&
\tilde{\mathcal{J}}_{\perp}(y, \tau) = \mathcal{J} - \sum\limits_{(j,k) \in \mathbf{M_{0,1}} } \tilde{\mathcal{J}}_{j,k}
= 
\tilde{\mathcal{H}}_{\perp}(y ,t(\tau)) \eta\Big(\frac{|y|}{2R(t(\tau))}\Big) \1_{\tau_0 < \tau \le \tau(T)}
	+ \tilde{\mathcal{H}}_{\perp}(y,T) \eta\Big(\frac{|y|}{2R(T)}\Big) \1_{ \tau > \tau(T)}.
\end{aligned}
\end{equation}

Denote
\begin{equation}
	R_{*}(s) := 8 (C_{tm} s)^{\frac{\beta}{2\gamma_1 -3}} \mbox{ \ for \ } s>0.
\end{equation}
For any $\lambda_1$ satisfying \eqref{mu1-ansatz}, by \eqref{tau-est},
\begin{equation}
(C_{tm}^{-1} \tau)^{\frac{\beta}{2\gamma_1 -3}} \le	R(t(\tau)) = (t(\tau))^{\beta} \le (C_{tm} \tau)^{\frac{\beta}{2\gamma_1 -3}},
\quad
R_{*}(\tau) \ge 8 R(t(\tau)) 
\mbox{ \ for \ } \tau \in (\tau_0, \infty);
\quad
\tau(T) \le C_{tm} T^{2\gamma_1 -3}. 
\end{equation}

We will use the re-gluing inner linear theory to give inverse mappings for $\tilde{\mathcal{J}}_{i,j}$, $(i,j) \in \mathbf{M_{0,1}} \cup \{ \perp \}$.
Formally, by Proposition \ref{mode0-regluing-prop} (with $R_0=\ln t_0$), there exist a mapping
\begin{equation}\label{qd26Apr12-3}
	(\mathcal{T}_{0,1}^{\rm{in}}, g_0, \varrho_{0,1}(\tau) ) = (\mathcal{T}_{0,1}^{\rm{in}}(y,\tau), g_0, \varrho_{0,1}(\tau) )[ \tilde{\mathcal{J}}_{0,1}]
\end{equation}
linearly depending on $\tilde{\mathcal{J}}_{0,1}$ such that
\begin{equation}\label{qd26Apr12-4}
	\begin{cases}
		\begin{aligned}
			\partial_{\tau} \mathcal{T}_{0,1}^{\rm{in}} = \ & \big( \Delta + \frac{7}{3} U^{\frac{4}{3}} \big) \mathcal{T}_{0,1}^{\rm{in}} + \tilde{\mathcal{J}}_{0,1}
			+ 
			\varrho_{0,1}(\tau) \eta(y) Z_{6}(y) 
			\mbox{ \ for \ } \tau \in (\tau_0, C_{tm} T^{2\gamma_1 -3}], y\in B_{R_{*}(\tau)},
		\end{aligned}
		\\
		\mathcal{T}_{0,1}^{\rm{in}}(y,\tau_0)
		=
		g_0
		\eta( \frac{2 y}{R_0} ) Z_0(y) 
		\mbox{ \ in \ } B_{R_{*}(\tau_0)},
	\end{cases}
\end{equation}
where
\begin{equation}\label{varrho01-def}
		\varrho_{0,1}(\tau)
		=
		-\Big(\int_{B_{2} } \eta(y) Z_{6}^2(y) \rmd y\Big)^{-1} \Big(\int_{B_{R_0}} \tilde{\mathcal{J}}_{0,1}(y,\tau)  Z_{6}(y) \rmd y
		+ \varrho_{0,1}^{*}(\tau)  \Big)
\end{equation}
with a small quantity $\varrho_{0,1}^{*}(\tau) = \varrho_{0,1}^{*}[\tilde{\mathcal{J}}_{0,1}](\tau)$ linearly depending on $\tilde{\mathcal{J}}_{0,1}$. For $\tau \in (\tau_0, \tau(T)]$, $R_0 < 2R$, 
\begin{equation}
	\begin{aligned}
		\varrho_{0,1}(\tau)
		= \ &
		-\Big(\int_{B_{2} } \eta(y) Z_{6}^2(y) \rmd y\Big)^{-1} \Big(\int_{B_{R_0}} \tilde{\mathcal{H}}_{0,1}(y,t(\tau))
		\eta\Big(\frac{y}{2R(t(\tau))}\Big)  Z_{6}(y) \rmd y
		+ \varrho_{0,1}^{*}(\tau)  \Big)
		\\
		= \ &
		-\Big(\int_{B_{2} } \eta(y) Z_{6}^2(y) \rmd y\Big)^{-1} \Big(\int_{B_{R_0}} \tilde{\mathcal{H}}_{0,1}(y,t(\tau))  Z_{6}(y) \rmd y
		+ \varrho_{0,1}^{*}(\tau)  \Big).
	\end{aligned}
\end{equation}
Compared with \eqref{qd26Jan3-8}, the purpose of extending the spatial-time region in \eqref{qd26Apr12-4} is to make the mappings $\mathcal{T}_{0,1}^{\rm{in}}[\cdot]$,
$g_0[\cdot]$, $\varrho_{0,1}[\cdot]$ themselves independent the choice of $\lambda_1$. We use the same extension for other modes in this section.

Formally, by Proposition \ref{regluing-mode1-prop} (with $R_0=\ln t_0$), for $i\in \overline{1,5}$, there exist a mapping
\begin{equation}\label{qd26Apr12-5}
	(\mathcal{T}_{1,i}^{\rm{in}},  \varrho_{1,i}(\tau) ) = (\mathcal{T}_{1,i}^{\rm{in}}(y,\tau), \varrho_{1,i}(\tau) ) [ \tilde{\mathcal{J}}_{1,i}]
\end{equation}
linearly depending on $ \tilde{\mathcal{J}}_{1,i}$ such that
\begin{equation}\label{qd26Apr12-6}
\begin{cases}
			\partial_{\tau} \mathcal{T}_{1,i}^{\rm{in}} = \big( \Delta + \frac{7}{3} U^{\frac{4}{3}} \big) \mathcal{T}_{1,i}^{\rm{in}} +  \tilde{\mathcal{J}}_{1,i}
			+ 
			\varrho_{1,i}(\tau) \eta(y) Z_{i}(y) 
			\mbox{ \ for \ } \tau \in (\tau_0, C_{tm} T^{2\gamma_1 -3}], y\in B_{R_{*}(\tau)},
\\
		\mathcal{T}_{1,i}^{\rm{in}}(y,\tau_0)
		=
		0 
		\mbox{ \ in \ } B_{R_{*}(\tau_0)},
\end{cases}
\end{equation}
where 
\begin{equation}\label{varrho1i-def}
		\varrho_{1,i}(\tau)
		=
		-\Big(\int_{B_{2} } \eta(y) Z_{i}^2(y) \rmd y\Big)^{-1} \Big(\int_{B_{R_0}} \tilde{\mathcal{J}}_{1,i}(y,\tau)  Z_{i}(y) \rmd y
		+ \varrho_{1,i}^{*}(\tau)  \Big)
\end{equation}
with a small quantity $\varrho_{1,i}^{*}(\tau) = \varrho_{1,i}^{*}[\tilde{\mathcal{J}}_{1,i}](\tau)$ linearly depending on $\tilde{\mathcal{J}}_{1,i}$. For $\tau \in (\tau_0, \tau(T)]$, $R_0 < 2R$, 
\begin{equation}\label{move-26Aug20-4}
		\varrho_{1,i}(\tau)
		= 
		-\Big(\int_{B_{2} } \eta(y) Z_{i}^2(y) \rmd y\Big)^{-1} \Big(\int_{B_{R_0}} \tilde{\mathcal{H}}_{1,i}(y,t(\tau))  Z_{i}(y) \rmd y
		+ \varrho_{1,i}^{*}(\tau)  \Big).
\end{equation}

Formally, by Proposition \ref{regluing-higher-mode-prop} (with $R_0$ as a large constant independent of $t_0, T$), there exist a mapping
\begin{equation}
	\mathcal{T}_{\perp}^{\rm{in}} = \mathcal{T}_{\perp}^{\rm{in}} [ \tilde{\mathcal{J}}_{\perp} ](y,\tau) 
\end{equation}
linearly depending on $ \tilde{\mathcal{J}}_{\perp} $ such that
\begin{equation}\label{move-26Sep6-1}
	\partial_{\tau} \mathcal{T}_{\perp}^{\rm{in}} = \big( \Delta + \frac{7}{3} U^{\frac{4}{3}} \big) \mathcal{T}_{\perp}^{\rm{in}} + \tilde{\mathcal{J}}_{\perp} 
	\mbox{ \ for \ } \tau \in (\tau_0, C_{tm} T^{2\gamma_1 -3} ], y\in B_{R_{*}(\tau)},
	\quad
	\mathcal{T}_{\perp}^{\rm{in}}(y,\tau_0)
	=
	0 
	\mbox{ \ in \ } B_{R_{*}(\tau_0)}.
\end{equation}

We will impose suitable assumptions to make the above mappings well-defined. Denote 
\begin{equation}\label{Tin-def}
	\mathcal{T}^{\rm{in}} = \mathcal{T}^{\rm{in}}[\mathcal{J}](y,\tau) :=
	\mathcal{T}_{0,1}^{\rm{in}} + \mathcal{T}_{\perp}^{\rm{in}} +
	\sum_{i=1}^{5} \mathcal{T}_{1,i}^{\rm{in}}
\end{equation}
as the inverse mapping for $\phi$, where $\mathcal{T}^{\rm{in}}$ linearly depends on $\mathcal{J}$. We emphasize that since the domain $ \{ (y,\tau) \mid \tau \in (\tau_0, C_{tm} T^{2\gamma_1 -3}], y \in B_{R_{*}(\tau)} \}$ and $\ln t_0$ are independent of the choice of $\lambda_1, \xi^{[1]}, \psi, \phi$, then
\begin{equation}\label{qd26May2-6}
\begin{aligned}
&
\mbox{The mappings $\mathcal{T}_{0,1}^{\rm{in}}[\cdot]$,
$g_0[\cdot]$, $\varrho_{0,1}[\cdot]$, 
$\varrho_{0,1}^{*}[\cdot]$,
$\mathcal{T}_{1,i}^{\rm{in}}[\cdot]$, 
$\varrho_{1,i}[\cdot]$, $\varrho_{1,i}^{*}[\cdot]$, $i\in\overline{1,5}$, $\mathcal{T}_{\perp}^{\rm{in}}[\cdot]$, $\mathcal{T}^{\rm{in}}[\cdot]$}
\\
&
\mbox{are independent of the choice of $\lambda_1, \xi^{[1]}, \psi, \phi$.}
\end{aligned}
\end{equation}

\section{Reduced orthogonality equations}\label{modi-orth-eq-sec}

To find a solution for the inner problem, we need to solve the reduced orthogonality equations
\begin{equation}\label{26June10-2}
\varrho_{0,1}(\tau) = 0,
\quad
\varrho_{1,i}(\tau) = 0, \ i \in \overline{1,5}
\mbox{ \ for \ }
\tau \in (\tau_0, \tau(T)]
\end{equation}
with $\varrho_{0,1}$ and $\varrho_{1,i}$ defined in \eqref{varrho01-def} and \eqref{varrho1i-def} respectively.
\eqref{26June10-2} is equivalent to that for $\tau \in (\tau_0, \tau(T)]$, 
\begin{equation}\label{orth-eq}
	\int_{B_{R_0}} \tilde{\mathcal{H}}_{0,1}(y,t(\tau))  Z_{6}(y) \rmd y
		+ \varrho_{0,1}^{*}(\tau)  = 0,
		\quad \int_{B_{R_0}} \tilde{\mathcal{H}}_{1,i}(y,t(\tau))  Z_{i}(y) \rmd y
		+ \varrho_{1,i}^{*}(\tau) = 0, \ i \in \overline{1,5}.
\end{equation}
In time variable $t$, \eqref{orth-eq} is rewritten as that for $t\in (t_0, T]$,
\begin{equation}\label{qd26Jan4-1}
	\int_{B_{R_0 }} \tilde{\mathcal{H}}_{0,1}(y,t)  Z_{6}(y) \rmd y
	+ \varrho_{0,1}^{*}(\tau(t))  = 0,
	\quad 
	\int_{B_{R_0}} \tilde{\mathcal{H}}_{1,i}(y,t)  Z_{i}(y) \rmd y
	+ \varrho_{1,i}^{*}(\tau(t)) = 0, \  i \in \overline{1,5}.
\end{equation}
By \eqref{qd25Jan4}, we know that \eqref{qd26Jan4-1} is equivalent to that for $t\in (t_0, T]$,
\begin{equation}\label{qd25Jan5}
	\int_{B_{R_0 }} \mathcal{H}(y,t)  Z_{6}(y) \rmd y
	+ \varrho_{0,1}^{*}(\tau(t))  = 0,
	\quad 
	\int_{B_{R_0}} \mathcal{H}(y,t)  Z_{i}(y) \rmd y
	+ \varrho_{1,i}^{*}(\tau(t)) = 0, \  i \in \overline{1,5}.
\end{equation}

We will reformulate \eqref{qd25Jan5} for the minor terms $\lambda_1, \xi^{[1]}$. Using the formulae of $U, Z_{j}, j \in \overline{1,6}$, we have 
\begin{equation*}
	\begin{aligned}
		&
		\dot{\lambda} \int_{B_{R_0}} Z_{6}(y)^2 \rmd y
		+
		\frac{7}{3} \lambda^{\frac{1}{2}} \int_{B_{R_0}} U(y)^{\frac{4}{3}} Z_{6}(y)
		\big(
		\Psi(\lambda y+\xi,t)
		+
		\psi(\lambda y+\xi, t)
		\big) \rmd y
		+ \lambda^{-1} \varrho_{0,1}^{*}(\tau(t))
		= 0
		\\
		\Leftrightarrow \ & \dot{\lambda}
		= - 
		\frac{7}{3} \lambda^{\frac{1}{2}}
		\Big(
		\int_{B_{R_0}} Z_{6}(y)^2 \rmd y
		\Big)^{-1}
		\int_{B_{R_0}} U(y)^{\frac{4}{3}} Z_{6}(y)
		\big(
		\Psi(\lambda y+\xi,t)
		+
		\psi(\lambda y+\xi, t)
		\big) \rmd y
		\\
		& -	\Big( \int_{B_{R_0}} Z_{6}(y)^2 \rmd y \Big)^{-1}
		\lambda^{-1} \varrho_{0,1}^{*}(\tau(t)).
	\end{aligned}
\end{equation*}
Subtracting the above equation with the one of $\lambda_0$ in  \eqref{mu0-eq}, that is, $
		\dot{\lambda}_0 =  - \frac{7}{3} \lambda_0^{\frac{1}{2}} \Psi(0, t) ( \int_{B_{R_0}} Z_{6}(y)^2 \rmd y )^{-1} $ $
		\int_{B_{R_0}} U(y)^{\frac{4}{3}}  Z_{6}(y) \rmd y $, we have
\begin{equation}\label{qd25Dec25-2}
	\begin{aligned}
		& \dot{\lambda}_1
		=  
		- \frac{7}{3} \Big(
		\int_{B_{R_0}} Z_{6}(y)^2 \rmd y
		\Big)^{-1}
		\Big[
		\lambda^{\frac{1}{2}}
		\int_{B_{R_0}} U(y)^{\frac{4}{3}} Z_{6}(y)
		\big(
		\Psi(\lambda y+\xi,t)
		+
		\psi(\lambda y+\xi, t)
		\big) \rmd y
		\\
		& - \lambda_0^{\frac{1}{2}} \Psi(0, t)
		\int_{B_{R_0}} U(y)^{\frac{4}{3}}  Z_{6}(y) \rmd y \Big]
		- 
		\Big( \int_{B_{R_0}} Z_{6}(y)^2 \rmd y \Big)^{-1}
		\lambda^{-1} \varrho_{0,1}^{*}(\tau(t))
		= \mathcal{F}(t) - \omega(t) \lambda_1,
	\end{aligned}
\end{equation}
where
\begin{equation}\label{qd25Dec26-1}
	\begin{aligned}
		&
		\omega(t) :=  \frac{7}{6} \Big(
		\int_{B_{R_0}} Z_{6}(y)^2 \rmd y
		\Big)^{-1} 
		\int_{B_{R_0}} U(y)^{\frac{4}{3}}  Z_{6}(y) \rmd y \lambda_0^{-\frac{1}{2}} \Psi(0, t)
		\\
		\stackrel{\eqref{AR-def} \eqref{mu-0-explicit}}{=} \ &  - \Psi(0, t) \Big( \int_{1}^{t} \Psi(0, s) \rmd s
		\Big)^{-1}
		\stackrel{\eqref{Psi0t}}{=} - \frac{t^{-\frac{\gamma_1}{2}} }{\int_{1}^t s^{-\frac{\gamma_1}{2}} \rmd s}
		\stackrel{\gamma_1 \ne 2}{=}
		- \frac{2-\gamma_1}{2} t^{-1} (1-t^{\frac{\gamma_1}{2}-1})^{-1},
	\end{aligned}
\end{equation}
\begin{equation}\label{qd26Mar3-1}
	\begin{aligned}
		&
		\mathcal{F}(t) = \mathcal{F}[\psi, \lambda_1, \xi^{[1]}](t) := - \frac{7}{3} \Big(
		\int_{B_{R_0}} Z_{6}(y)^2 \rmd y
		\Big)^{-1}
		\Big[
		\lambda^{\frac{1}{2}}
		\int_{B_{R_0}} U(y)^{\frac{4}{3}} Z_{6}(y)
		\psi(\lambda y+\xi, t) \rmd y
		\\
		& +
		\lambda^{\frac{1}{2}}
		\int_{B_{R_0}} U(y)^{\frac{4}{3}} Z_{6}(y)
		\big( \Psi(\lambda y+\xi,t) - \Psi(0,t) \big) \rmd y
		\\
		& + (\lambda^{\frac{1}{2}} - \lambda_0^{\frac{1}{2}} - 2^{-1} \lambda_0^{-\frac{1}{2}} \lambda_1 ) \Psi(0, t)
		\int_{B_{R_0}} U(y)^{\frac{4}{3}}  Z_{6}(y) \rmd y \Big]
		 - 
		\Big( \int_{B_{R_0}} Z_{6}(y)^2 \rmd y \Big)^{-1}
		\lambda^{-1} \varrho_{0,1}^{*}(\tau(t)).
	\end{aligned}
\end{equation}

For $i \in \overline{1,5}$, we have
\begin{align*}
	&
		\int_{B_{R_0}}
		\Big[ \dot{\lambda} \lambda Z_{6}(y)
		+\lambda \dot{\xi} \cdot (\nabla U )(y) 
		+\frac{7}{3}\lambda^{\frac{3}{2}} U(y)^{\frac{4}{3}}
		\big(
		\Psi(\lambda y+\xi,t)
		+
		\psi(\lambda y+\xi, t)
		\big)
		\Big] Z_{i}(y) \rmd y + \varrho_{1,i}^{*}(\tau(t)) = 0
		\\
		\Leftrightarrow \ 
		& \int_{B_{R_0}}
		\Big[ \lambda \dot{\xi}_i Z_i(y) 
		+\frac{7}{3}\lambda^{\frac{3}{2}} U(y)^{\frac{4}{3}}
		\big(
		\Psi(\lambda y+\xi,t) - \Psi(0,t)
		+
		\psi(\lambda y+\xi, t) - \psi(0, t)
		\big) \Big] Z_{i}(y) \rmd y
		+ \varrho_{1,i}^{*}(\tau(t)) 
		= 0
		\\
		\Leftrightarrow \ &
		\dot{\xi}_i \int_{B_{R_0}} Z_1(y)^2 \rmd y
		+ \frac{7}{3} \lambda^{\frac{1}{2}}
		\int_{B_{R_0}} U(y)^{\frac{4}{3}} Z_{i}(y)
		\big(
		\Psi(\lambda y+\xi,t) - \Psi(0,t)
		+
		\psi(\lambda y+\xi, t) - \psi(0, t)
		\big) \rmd y
		\\
		& + \lambda^{-1} \varrho_{1,i}^{*}(\tau(t)) 
		= 0
		\\
		\Leftrightarrow \ & \dot{\xi}_i 
		= 
		- \frac{7}{3} \lambda^{\frac{1}{2}}
		\Big( \int_{B_{R_0}} Z_1(y)^2 \rmd y \Big)^{-1}
		\int_{B_{R_0}} U(y)^{\frac{4}{3}} Z_{i}(y)
		\\
		& \times
		\big[
		\Psi(\lambda y+\xi,t) - \Psi(0,t)
		-
		(\nabla \Psi)(0,t) \cdot (\lambda y+\xi)
		+
		(\nabla \Psi)(0,t) \cdot (\lambda y+\xi)
		+
		\psi(\lambda y+\xi, t) - \psi(0, t)
		\big] \rmd y
		\\
		& - \Big( \int_{B_{R_0}} Z_1(y)^2 \rmd y \Big)^{-1} \lambda^{-1}
		\varrho_{1,i}^{*}(\tau(t))
		\\
		= \ & - \frac{7}{3} \lambda^{\frac{1}{2}}
		\Big( \int_{B_{R_0}} Z_1(y)^2 \rmd y \Big)^{-1}
		\int_{B_{R_0}} U(y)^{\frac{4}{3}} Z_{i}(y)
		\\
		& \times
		\big[
		\Psi(\lambda y+\xi,t) - \Psi(0,t)
		-
		(\nabla \Psi)(0,t) \cdot (\lambda y+\xi)
		+
		\lambda y_i (\partial_{x_i} \Psi)(0,t)
		+
		\xi \cdot (\nabla \Psi)(0,t) 
		+
		\psi(\lambda y+\xi, t) - \psi(0, t)
		\big] \rmd y
		\\
		& - \Big( \int_{B_{R_0}} Z_1(y)^2 \rmd y \Big)^{-1} \lambda^{-1} 
		\varrho_{1,i}^{*}(\tau(t))
		\\
		= \ & - \frac{7}{3} \lambda^{\frac{1}{2}}
		\Big( \int_{B_{R_0}} Z_1(y)^2 \rmd y \Big)^{-1}
		\Big\{ \lambda (\partial_{x_i} \Psi)(0,t)
		\int_{B_{R_0}} U(y)^{\frac{4}{3}} y_1 Z_{1}(y) \rmd y
		\\
		& +
		\int_{B_{R_0}} U(y)^{\frac{4}{3}} Z_{i}(y)
		\big[
		\Psi(\lambda y+\xi,t) - \Psi(0,t)
		-
		(\nabla \Psi)(0,t) \cdot (\lambda y+\xi)
		+
		\psi(\lambda y+\xi, t) - \psi(0, t)
		\big] \rmd y
		\Big\}
		\\
		& - \Big( \int_{B_{R_0}} Z_1(y)^2 \rmd y \Big)^{-1} \lambda^{-1} 
		\varrho_{1,i}^{*}(\tau(t)).
\end{align*}
Subtracting the above equation with the one of $\xi_i^{[0]}$ in \eqref{mu0-eq}, that is,
$ \dot{\xi}_i^{[0]}
	= - \frac{7}{3}
	( \int_{B_{R_0}} Z_1(y)^2 \rmd y )^{-1} \lambda_0^{\frac{3}{2}} (\partial_{x_i} \Psi)(0,t)$ $
	\int_{B_{R_0}} U(y)^{\frac{4}{3}}
	y_1 Z_1(y) \rmd y $,
we have
\begin{align}
		&
		\dot{\xi}_i^{[1]} = \mathcal{S}_i = \mathcal{S}_i[\psi, \lambda_1,\xi^{[1]}](t) := - \frac{7}{3}
		\Big( \int_{B_{R_0}} Z_1(y)^2 \rmd y \Big)^{-1}
		\Big\{ ( \lambda^{\frac{3}{2}} - \lambda_0^{\frac{3}{2}} ) (\partial_{x_i} \Psi)(0,t)
		\int_{B_{R_0}} U(y)^{\frac{4}{3}} y_1 Z_{1}(y) \rmd y
        \notag
		\\
		& + \lambda^{\frac{1}{2}}
		\int_{B_{R_0}} U(y)^{\frac{4}{3}} Z_{i}(y)
		\big[
		\Psi(\lambda y+\xi,t) - \Psi(0,t)
		-
		(\nabla \Psi)(0,t) \cdot (\lambda y+\xi)
		+
		\psi(\lambda y+\xi, t) - \psi(0, t)
		\big] \rmd y
		\Big\}
        \label{qd25Dec25-3}
		\\
		& - \Big( \int_{B_{R_0}} Z_1(y)^2 \rmd y \Big)^{-1} \lambda^{-1} \varrho_{1,i}^{*}(\tau(t)).
        \notag
\end{align}

To find a solution of the system of equations \eqref{qd25Dec25-2} and \eqref{qd25Dec25-3}, it suffices to solve
\begin{equation}\label{mu1-xi-sys}
	\begin{aligned}
		& \dot{\lambda}_1 = \mathcal{S}_6 = \mathcal{S}_6[\psi, \lambda_1,\xi^{[1]}](t)
		:=
		\frac{\rmd}{\rmd t} \Big(
		\int_{t_0}^t
		\rme^{\int_t^s \omega(a) \rmd a}  \mathcal{F}(s) \rmd s \Big)
		= - \omega(t) \int_{t_0}^t
		\rme^{\int_t^s \omega(a) \rmd a}  \mathcal{F}(s) \rmd s
		+
		\mathcal{F}(t),
		\\
		&
		\dot{\xi}^{[1]} = \vec{\mathcal{S}} = \vec{\mathcal{S}}[\psi,\lambda_1,\xi^{[1]}](t)
		:= (\mathcal{S}_1, \mathcal{S}_2,\dots, \mathcal{S}_5)(t),
		\\
		&
		\lambda_1 = \lambda_1[\dot{\lambda}_1](t) :=
		\int_{t_0}^t \dot{\lambda}_1(s) \rmd s, \quad
		\xi^{[1]} = \xi^{[1]}[\dot{\xi}^{[1]}](t) :=
\begin{cases} 
\int_{t_0}^t \dot{\xi}^{[1]}(s) \rmd s
			& \mbox{ \ if \ } 3 \gamma_1 + \gamma_2 \le 7
			\\
\int_{\infty}^t \dot{\xi}^{[1]}(s) \rmd s
			& \mbox{ \ if \ } 3 \gamma_1 + \gamma_2 > 7
		\end{cases}
	\end{aligned}
\end{equation}
about $(\dot{\lambda}_1, \dot{\xi}^{[1]})$ if these integrals are well-defined. The choice of $\xi^{[1]}$ is due to the topology of $\dot{\xi}^{[1]}$ in \eqref{mu-xi-norm} later. 
Indeed, given $ (\psi, \dot{\lambda}_1, \dot{\xi}^{[1]} )$, and $\lambda_1[\dot{\lambda}_1], \xi^{[1]}[\dot{\xi}^{[1]}]$ in \eqref{mu1-xi-sys},
$\mathcal{S}_6\big[ \psi, \lambda_1[\dot{\lambda}_1],\xi^{[1]}[\dot{\xi}^{[1]}] \big]$ will be regraded as a mapping from a product normed space of $(\psi, \dot{\lambda}_1, \dot{\xi}^{[1]})$ into a normed space of $\dot{\lambda}_1$. A similar perspective is applied for $\vec{\mathcal{S}}$.

\section{The fixed-point scheme and a mapping for solving the outer problem in $\overline{B_{T^{9}}}  \times [t_0, T]$}\label{out-solve-sec}

We introduce a symbol $\mathbf{P}_1[\cdot]$.
\begin{definition}[\cite{5dinfinite-time-2025}]\label{AP-def}
	
	Define $g(t) = \langle \ln t\rangle$, and let $g^{(k)} = \underbrace{g \circ g \circ \dots \circ g}_{\mbox{\rm k-fold composition}}$ be the $k$-fold composition of $g$. Given a function $f(t)$ defined in an unbounded set $\Omega \subset \mathbb{R}_{>0}$, we say that $f(t)$ is of {\it algebraic power type in $\Omega$} if
	\begin{equation*}
		C^{-1} f_*(t) \le f(t) \le C f_*(t)\ \mbox{ for }  \ t \in \Omega,\quad   f_*(t) := c_0 t^{c_1} \prod_{i=1}^{\infty} \big( g^{(i)}(t) \big)^{c_{i+1}},
	\end{equation*}
	where  $C\ge 1$ is a constant, $c_0 > 0$ ($c_0$ possibly depends on $t_0$ in this paper), $c_i\in \mathbb{R}$ for $i\ge 1$, and $c_{i} = 0$ for $i\ge N_0$ for some integer $N_0\ge 1$. 
	Then, we define
	\begin{equation}\label{Pi-def}
		\mathbf{P}_i[f] := c_i, \quad i \ge 1.
	\end{equation}
	Let $\mathbf{AP}(\Omega)$ denote the set of algebraic power-type functions defined on $\Omega \subset \mathbb{R}_{>0}$.
	
\end{definition}
Notice that for $i\ge 1$, $\mathbf{P}_i[\cdot]$ is uniquely determined for any $f \in \mathbf{AP}(\Omega)$ with  $\Omega \subset \mathbb{R}_{>0}$ unbounded. Denote
\begin{equation}
	\mathcal{T}_{\rm{o}} [f] = \mathcal{T}_{\rm{o}} [f](x,t)
	:=
	\int_{t_0}^t \int_{\R^5} 
	[ 4\pi(t-s) ]^{-\frac{5}{2}}
	\rme^{ -\frac{|x-z|^2}{4(t-s)}  } f(z,s) \rmd z \rmd s.
\end{equation}
For solving the outer problem \eqref{outer-problem} with $n=5$,
\begin{equation}
		\pp_t \psi 
		= \Delta \psi + \mathcal{G} \mbox{ \ in \ } B_{T^{9}}  \times (t_0,T],
		\quad 
		\psi(\cdot,t_0)=0 \mbox{ \ in \ } B_{T^{9}},
\end{equation}
where $\mathcal{G}$ is given in \eqref{g} with $n=5$,
\begin{equation}\label{g-n=5}
\begin{aligned}
&
\mathcal{G} = \mathcal{G}[\psi,\phi,\lambda,\xi](x,t) = \frac{7}{3} \lambda^{-2}
U(y)^{\frac{4}{3}}
\big(
\eta(\tilde{y})^{\frac{4}{3}} - 1 \big) \lambda^{-\frac{3}{2}}\phi(y,t) \eta_{R}(y)
+ \Lambda_1 + \Lambda_2
\\
& + 
\underbrace{
\big(
\dot{\lambda} \lambda^{-\frac{5}{2}}  Z_{6}(y) 
+
\lambda^{-\frac{5}{2}} \dot{\xi} \cdot 
(\nabla U)( y )
\big) ( \eta(\tilde{y} ) - \eta_R(y) )
+
\mathcal{E}_{U}^{\rm{cut}}
+
\lambda^{-\frac{7}{2}} U(y)^{\frac{7}{3}}
\big( \eta(\tilde{y})^{\frac{7}{3}} 
-
\eta( \tilde{y} )
\big)
}_{=: \mathcal{G}_1}
\\
&
+ 
\underbrace{
\frac{7}{3} \lambda^{-2}
U(y)^{\frac{4}{3}}
\big(
\eta(\tilde{y})^{\frac{4}{3}} - \eta_{R}(y) \big) (\Psi + \psi)
}_{=: \mathcal{G}_2} \ 
+ \mathcal{N},
\end{aligned}
\end{equation}
we give a mapping
\begin{equation}
\mathcal{T}_{\rm{o}} [ \mathcal{G} \1_{|x| \le T^{9}, t \le T} ](x,t)
	\mbox{ \ for \ } (x,t) \in \mathbb{R}^5 \times [t_0, \infty).
\end{equation}

Before detailed calculations,
we first present the global picture for the fixed-point scheme.

Given $(\psi, \phi, \dot{\lambda}_1, \dot{\xi}^{[1]})$ in some suitable topology. 
We regard 
$\lambda_1 = \lambda_1[\dot{\lambda}_1](t)$ and $\xi^{[1]} = \xi^{[1]}[\dot{\xi}^{[1]}](t)$ as functionals about $\dot{\lambda}_1$ and $\dot{\xi}^{[1]}$ in the form of \eqref{mu1-xi-sys}.
Let $\tilde{\psi} = \mathcal{T}_{\rm{o}} [ \mathcal{G}[\psi,\phi,\lambda_0 + \lambda_1,\xi^{[0]} + \xi^{[1]}] \1_{|x| \le T^{9}, t \le T} ](x,t)$, which satisfies
\begin{equation}
		\partial_{t} \tilde{\psi} 
		= \Delta \tilde{\psi} +  \mathcal{G}[\psi,\phi,\lambda_0 + \lambda_1,\xi^{[0]} + \xi^{[1]}] \1_{|x| \le T^{9}, t \le T}  \mbox{ \ in \ } B_{T^{9}}  \times (t_0,T],
		\quad 
		\tilde{\psi}(\cdot,t_0)=0 \mbox{ \ in \ } B_{T^{9}}.
\end{equation}
Recall \eqref{calJ-def} and \eqref{Tin-def}. Let $\tilde{\phi}(y,\tau) = \mathcal{T}^{\rm{in}}\big[ \mathcal{J}[\psi, \lambda_{0} + \lambda_{1}, \xi^{[0]} + \xi^{[1]} ] \big](y,\tau)$. It holds that
\begin{equation}
	\begin{cases}
		\begin{aligned}
        \begin{aligned}
			\partial_{\tau} \tilde{\phi} = \ & \big( \Delta + \frac{7}{3} U^{\frac{4}{3}} \big) \tilde{\phi} + \mathcal{J}[\psi, \lambda_{0} + \lambda_{1}, \xi^{[0]} + \xi^{[1]} ]
            \\
            &
			+ 
			\varrho_{0,1}(\tau) \eta(y) Z_{6}(y) 
            +
            \sum_{i=1}^{5} \varrho_{1,i}(\tau) \eta(y) Z_{i}(y)
			\mbox{ \ for \ } \tau \in (\tau_0, C_{tm} T^{2\gamma_1 -3}], y\in B_{R_{*}(\tau)},
            \end{aligned}
		\end{aligned}
		\\
		\tilde{\phi}(y,\tau_0)
		=
		g_0
		\eta( \frac{2 y}{R_0} ) Z_0(y) 
		\mbox{ \ in \ } B_{R_{*}(\tau_0)}.
	\end{cases}
\end{equation}
Recall \eqref{mu1-xi-sys}. Let 
\begin{equation}
\tilde{\dot{\lambda}}_1 = \mathcal{S}_6[\psi, \lambda_1,\xi^{[1]}](t),
\quad
\tilde{\dot{\xi}}^{[1]} = \vec{\mathcal{S}}[\psi,\lambda_1,\xi^{[1]}](t).
\end{equation}
From this perspective,
our problem is reduced to finding a fixed point $(\psi, \phi, \dot{\lambda}_1, \dot{\xi}^{[1]})$ in some space $\mathcal{X}$ for 
\begin{equation}\label{fixed-point-problem}
\begin{aligned}
\big( & \mathcal{T}_{\rm{o}} [ \mathcal{G}[\psi,\phi,\lambda_0 + \lambda_1,\xi^{[0]} + \xi^{[1]}] \1_{|x| \le T^{9}, t \le T} ](x,t),  
\mathcal{T}^{\rm{in}}\big[ \mathcal{J}[\psi, \lambda_{0} + \lambda_{1}, \xi^{[0]} + \xi^{[1]} ] \big](y,\tau[\lambda_0 + \lambda_1](t)),
\\
&
\mathcal{S}_6[\psi, \lambda_1,\xi^{[1]}](t),
\vec{\mathcal{S}}[\psi,\lambda_1,\xi^{[1]}](t)
\big)
\end{aligned}
\end{equation}
with $\lambda_1 = \lambda_1[\dot{\lambda}_1](t)$ and $\xi^{[1]} = \xi^{[1]}[\dot{\xi}^{[1]}](t)$ given in \eqref{mu1-xi-sys}. Indeed, $\mathcal{X}$ will be defined in \eqref{Xspace-def}, which we explain in the following sections. By the argument in Section \ref{modi-orth-eq-sec}, once the fixed-point problem \eqref{fixed-point-problem} is solved, then for $t\in (t_0, T]$, we have $\varrho_{0,1}(\tau(t)) = 0$, $\varrho_{1,i}(\tau(t)) = 0$, $i\in \overline{1,5}$.

\medskip

Hereafter in this section, we devote ourselves to the detailed calculations for the outer problem.
We set
\begin{equation}
	w_{\rm o}(x,t)
	:=
	t^{-\frac{\gamma_1}{2} } (\ln t)^2 R_0^6 R^{-1} 
	( \1_{|x| \le t^{\frac{1}{2}}} + t |x|^{-2} \1_{|x| > t^{\frac{1}{2}} } ),
\end{equation}
and define the norm 
\begin{equation}
	\|f\|_{\rm o} := \sup_{t \in [t_0, T], x \in \overline{ B_{T^{9}} } } (w_{\rm o}(x,t) )^{-1} |f(x,t)|.
\end{equation}
We also define the following semi-norms and norm
\begin{equation}
[f]_{{\rm{o}},\nabla} := \sup_{t \in [t_0, T], x \in \overline{ B_{T^{9}} } } 
        \big[ 
		t^{-2+\frac{\gamma_1}{2}} (\ln t)^{\frac{3}{2} + \frac{1}{100}} R_0^6 R(t)^{-2} \big]^{-1} |\nabla f(x,t)|;
\end{equation}
given a constant $\alpha \in (0,1)$,
\begin{equation}
\begin{aligned}
[f]_{{\rm{o}},\alpha} := & \sup_{ t_*\in (t_0, T], x_* \in \overline{ B_{T^{9}} },  t_1, t_2 \in [ \max\{t_0, t_* -1\} ,t_*], t_1 \ne t_2 }
\\
&
\Big\{ \ln t_0
\big[
t_*^{-\frac{\gamma_1}{2} } (\ln t_*)^2 R_0^6 R(t_*)^{-1} 
+
t_*^{\frac{3}{2} \gamma_1 - 4} (\ln t_*) R_0^{6} R(t_*)^{-3} \big] \Big\}^{-1}
\frac{|f(x_*, t_1) - f(x_*, t_2) |}{| t_1 - t_2|^{\alpha}};
\end{aligned}
\end{equation}
\begin{equation}
\| f \|_{B_{\rm o}} := \max\{ \|f\|_{\rm o}, [f]_{{\rm{o}},\nabla}, [f]_{{\rm{o}},\alpha} \}.
\end{equation}
The outer problem \eqref{outer-problem} with $n=5$ will be solved in the space
\begin{equation}\label{Bo-def}
\begin{aligned}
B_{\rm o} := \big\{ f(x,t)  \mid & \ 
f(x,t) \in C\big( [t_0,T]; C^{1}( \overline{ B_{T^{9}} } ) \big),
\quad
\| f \|_{B_{\rm o}} \le 1,  
\\
&
\mbox{ $f(x,t)$ is even with respect to the $i$-th component of $x$ for $i \in \overline{2,5}$}
\big\},
\end{aligned}
\end{equation}
where the gradient estimate in $B_{\rm o}$ will be used for \eqref{qd26Apr12-7}, \eqref{move-26Sep5-4}, while the H\"older estimate will only be used for the compactness argument for $\mathcal{S}_{6} \times
\vec{\mathcal{S} }$ in Subsection \ref{inner-orth-subsec} later. Thus, the H\"older estimate of the outer problem is pretty rough. The following lemma gives the estimate for the mapping of the outer problem.

\begin{lemma}\label{outer-exist} 
	
For any $\psi$ defined in $\overline{B_{T^{9}}}  \times [t_0, T] $ satisfying $\|\psi\|_{\rm o} \le 1$ and $\psi(x,t)$ is even with respect to the $i$-th component of $x$ for $i \in \overline{2,5}$, $\phi \in B_{\rm{in}}$, $\lambda_1, \dot{\lambda}_1,  \xi^{[1]}, \dot{\xi}^{[1]}$ satisfying \eqref{mu1-ansatz}, then $\mathcal{T}_{\rm{o}} [ \mathcal{G} \1_{|x| \le T^{9}, t \le T} ](x,t)$ is even with respect to the $i$-th component of $x$, $i \in \overline{2,5}$.

Moreover, suppose the parameter assumption \eqref{mov-26Aug11-1} and
\begin{align}
&
3\gamma_1 + \gamma_2 > 6,
\quad
\frac{5}{2} -\beta -\frac{\gamma_1}{2} > 0,
\quad
\gamma_1 \in (\frac{3}{2}, 2),
\quad
6 + 2 \beta < 3 \gamma_1 + \min\{ \gamma_1, \gamma_2 \}, \quad  \gamma_1 <\min\{ \gamma_1, \gamma_2 \} + 2 \beta,
\notag
\\
&
\frac{3}{2} + \beta - \min\{ \gamma_1, \gamma_2 \} < 0,
\quad
0 < -\frac{15}{2} - 3 \beta + 5\gamma_1,
\quad
0 < -1 -\beta + \frac{7}{6} \min\{ \gamma_1, \gamma_2 \} - \frac{\gamma_1}{2},
\label{qd25Dec24-2}
\\
&
\min\{ \gamma_1, \gamma_2 \} \ge \frac{6}{7},
\quad
0 < - \frac{17}{2} - \frac{11}{3} \beta + \frac{17}{3} \gamma_1,
\quad
0 < -4 - \beta + \frac{8}{3} \gamma_1,
\notag
\end{align}
then for $t_0$ sufficiently large,
\begin{equation}\label{psi-est}
|\mathcal{T}_{\rm{o}} [ \mathcal{G} \1_{|x| \le T^{9}, t \le T} ](x,t) | \le  w_{\rm o}(x,t);
\end{equation}
there exists a constant $D_1 >0$ independent of $t_0, T$ such that for any $(x,t) \in \mathbb{R}^5 \times [t_0, \infty)$,
\begin{equation}\label{qd25Oct30-3}
|\nabla \mathcal{T}_{\rm{o}} [ \mathcal{G} \1_{|x| \le T^{9}, t \le T} ](x,t) |  
\le D_1
t^{-2+\frac{\gamma_1}{2}} (\ln t)^{\frac{3}{2}} R_0^6 R^{-2};
\end{equation}  
given any $\alpha_1 \in (0,1)$, there exists a constant $D_2 >0$ independent of $t_0, T$ such that for any $t_* \in (t_0,T]$,
\begin{equation}\label{out-holder-est}
\begin{aligned}
&
\sup\limits_{x_* \in \mathbb{R}^5, t_1, t_2 \in [ \max\{t_0, t_* -1\} ,t_*], t_1 \ne t_2}
\frac{|\mathcal{T}_{\rm{o}}[ \mathcal{G} \1_{|x| \le T^{9}, t \le T} ](x_*, t_1) - \mathcal{T}_{\rm{o}}[ \mathcal{G} \1_{|x| \le T^{9}, t \le T} ](x_*, t_2) |}{| t_1 - t_2|^{\alpha_1}}
\\
\le \ & D_2 \big[ 
t_*^{-\frac{\gamma_1}{2} } (\ln t_*)^2 R_0^6 R(t_*)^{-1} 
+
t_*^{\frac{3}{2} \gamma_1 - 4} (\ln t_*) R_0^{6} R(t_*)^{-3} \big].
\end{aligned}
\end{equation}

Furthermore, for any compact set $K \subset \mathbb{R}^5 \times [t_0,\infty)$, there exists $C_{K}$ sufficiently large such that for all $T > C_{K}$, then 
$K\subset B_{T^9} \times [t_0, T]$ and $\mathcal{T}_{\rm{o}} [ \mathcal{G} \1_{|x| \le T^{9}, t \le T} ]$, $\nabla \mathcal{T}_{\rm{o}} [ \mathcal{G} \1_{|x| \le T^{9}, t \le T} ]$ have uniform H\"older continuity in $K$ independent of the choice of $\psi$, $\phi$, $\lambda_1, \dot{\lambda}_1, \xi^{[1]}, \dot{\xi}^{[1]}$.
\end{lemma}

\begin{proof}

By $\xi = (\xi_1, 0,0,0,0)$ from \eqref{mu1-ansatz}, $\eta$ is radially symmetric, and parity of $\psi, \phi$ in the assumption, $\mathcal{G}$ given in \eqref{g-n=5} is even with respect to the $i$-th component of $x$, $i \in \overline{2,5}$. So is $ \mathcal{G} \1_{|x| \le T^{9}, t \le T} $, which implies that $\mathcal{T}_{\rm{o}} [ \mathcal{G} \1_{|x| \le T^{9}, t \le T} ](x,t)$ is even with respect to the $i$-th component of $x$, $i \in \overline{2,5}$. We check two typical terms in $\mathcal{G}$, which are included in $\Lambda_1$. The other terms in $\mathcal{G}$ can be handled similarly. Consider
\begin{equation}\label{move-26Sep4-9}
		2 \lambda^{-\frac{7}{2}}  R^{-1}  \nabla_y \phi(y, t) \cdot (\nabla \eta )(\frac{y}{R}) 
		+
		\lambda^{-\frac{3}{2}} \phi( y, t) 
		( \nabla \eta )(\frac{y}{R}) \cdot \Big[ \frac{\dot{\xi}}{\lambda R}
		+ \frac{y}{R} \frac{1}{\lambda R } \frac{\rmd (\lambda R)}{\rmd t}  \Big]
\mbox{ \ with \ } y =\frac{(x_1-\xi_1, x_2, x_3, x_4, x_5)}{\lambda}.
\end{equation}
Since
\begin{equation*}
\begin{aligned}
&
2 \lambda^{-\frac{7}{2}}  R^{-1}  (\nabla_y \phi)\Big( \frac{(x_1-\xi_1, x_2, x_3, x_4, x_5)}{\lambda}, t \Big) \cdot (\nabla \eta )\Big( \frac{(x_1-\xi_1, x_2, x_3, x_4, x_5)}{\lambda R} \Big),
\\
&
( \nabla \eta )\Big( \frac{(x_1-\xi_1, x_2, x_3, x_4, x_5)}{\lambda R} \Big) \cdot \frac{\dot{\xi}}{\lambda R}
=
( \partial_{x_1} \eta )\Big( \frac{(x_1-\xi_1, x_2, x_3, x_4, x_5)}{\lambda R} \Big) \frac{\dot{\xi}_1}{\lambda R},
\\
&
( \nabla \eta )\Big( \frac{(x_1-\xi_1, x_2, x_3, x_4, x_5)}{\lambda R} \Big) \cdot \frac{(x_1-\xi_1, x_2, x_3, x_4, x_5)}{\lambda R}
\end{aligned}
\end{equation*}
are even with respect to the $i$-th component of $x$, $i \in \overline{2,5}$, then so is \eqref{move-26Sep4-9}. Under the assumption 
\begin{equation}\label{qd25Dec23-2}
\frac{3}{2}  +\beta -\gamma_1 <0,
\end{equation}
then $\lambda R \ll t^{\frac{1}{2}}$. Under the assumption \eqref{mu1-ansatz},
for $t>t_0 \gg 1$, we have the following useful relationships
\begin{equation}\label{qd25Oct19-1}
\begin{aligned}
&
		\{ x \mid |x| \le \lambda_{0} R/2 \}
		\subset
		\{ x \mid |x-\xi| \le \lambda R \}
		\subset
		\{ x \mid |x-\xi| \le 2 \lambda R \}
		\subset
		\{ x \mid |x| \le 3 \lambda_0 R \} 
\\
&
\subset
\{ x \mid |x| \le 2^{-1} t^{1/2} \}
\subset
\{ x \mid |x-\xi| \le t^{1/2} \}
\subset
\{ x \mid |x-\xi| \le 2 t^{1/2} \}
\stackrel{3\gamma_1 + \gamma_2 > 6}{\subset}
\{ x \mid |x| \le 3 t^{1/2} \}
\\
&
		\mbox{and } 
		\langle y \rangle \sim \langle \lambda_{0}^{-1} |x| \rangle,
\end{aligned}
\end{equation}
since $|\xi| \lesssim \lambda_0$, and then $
\langle y \rangle 
\sim 
1+ \lambda_{0}^{-1} | x- \xi |
\lesssim 
1+ \lambda_{0}^{-1} |x| $, 
$ \langle y \rangle
\sim C_1 + \lambda_{0}^{-1} | x - \xi |
\ge 
C_1 + \lambda_{0}^{-1} \big( |x| - |\xi| \big)
\sim 
1 + \lambda_{0}^{-1}|x|$
with a large constant $C_1$.

When estimating $\mathcal{G}$ in \eqref{g-n=5} term by term in order, we always assume $|x| \le T^{9}, t \le T$.

{\textbf{Estimate for $\lambda^{-2}
U(y)^{\frac{4}{3}}
(
\eta(\tilde{y})^{\frac{4}{3}} - 1 ) \lambda^{-\frac{3}{2}}\phi(y,t) \eta_{R}(y)$}.}
Since 
$ | \eta(\tilde{y})^{\frac{4}{3}} - 1 | \le \1_{|\tilde{y}| > 1} = \1_{|x-\xi| > t^{\frac{1}{2}}} \le \1_{|x| > \frac{3}{4} t^{\frac{1}{2}} } $, and $\eta_{R}(y) \le \1_{|x| \le 2^{-1} t^{\frac{1}{2}}} $ by \eqref{qd25Oct19-1}, we have
\begin{equation}\label{move-26Sep6-5}
\lambda^{-2}
U(y)^{\frac{4}{3}}
(
\eta(\tilde{y})^{\frac{4}{3}} - 1 ) \lambda^{-\frac{3}{2}}\phi(y,t) \eta_{R}(y) = 0.
\end{equation}

{\textbf{Estimate for $\Lambda_1$.}}
Hereafter in this proof, convolution estimates \cite[Lemma A.1, Lemma A.2]{infi4d} will be used repetitively to estimate $\mathcal{T}_{\rm{o}} [\cdot]$.  
Recall $\Lambda_1$ given in \eqref{Lambda1-phi}. Under the assumption \eqref{mu1-ansatz}, we have 
\begin{equation*}
	\Big| \frac{\dot{\xi}}{\lambda R} \Big|
	+
	\Big| \frac{1}{\lambda R} \frac{\rmd (\lambda R)}{\rmd t} \Big|
	= 
	\Big| \frac{\dot{\xi}}{\lambda R} \Big|
	+
	\Big| \frac{\dot{\lambda}}{\lambda} + \frac{\dot{R}}{R} \Big|
	\lesssim t^{-1} \lesssim \lambda^{-2} R^{-2}.
\end{equation*}
For $\phi \in B_{\rm{in}}$, by \eqref{Bin-def},  
\begin{equation}\label{qd25Oct22-6}
\langle y\rangle |\nabla_y \phi(y,t)| + |\phi(y,t)| 
\le R_0^{6} t^{3-2 \gamma_1} \ln t \langle y \rangle^{-1}
\mbox{ \ for \ }
t \in (t_0,T], y\in B_{4 R}.
\end{equation} 
Thus, 
\begin{align}
		&
		|\Lambda_1| 
= \Big| \lambda^{-\frac{7}{2}} R^{-2} \phi(y, t)  ( \Delta 
\eta )(\frac{y}{R})
+ 
2 \lambda^{-\frac{7}{2}}  R^{-1}  \nabla_y \phi(y, t) \cdot (\nabla \eta )(\frac{y}{R}) 
\notag
\\
&
+
\lambda^{-\frac{3}{2}} \phi( y, t) 
( \nabla \eta )(\frac{y}{R}) \cdot \Big[ \frac{\dot{\xi}}{\lambda R}
+ \frac{y}{R} \frac{1}{\lambda R} \frac{\rmd (\lambda R)}{\rmd t} \Big] \Big|
\notag
\\
\lesssim \ & \lambda^{-\frac{7}{2}} R^{-2} R_0^{6} t^{3-2\gamma_1} \ln t R^{-1}
\1_{R\le |y| \le 2R} 
+ 
\lambda^{-\frac{7}{2}} R^{-1}
R_0^{6} t^{3-2\gamma_1} \ln t R^{-2} \1_{R\le |y| \le 2R}
\label{Lambda1-upp}
\\
&
+
\lambda^{-\frac{3}{2}}
R_0^{6} t^{3-2\gamma_1} \ln t R^{-1} 
\1_{R\le |y| \le 2R} \Big[ \Big| \frac{\dot{\xi}}{\lambda R} \Big|
+ \Big| \frac{1}{\lambda R} \frac{\rmd (\lambda R)}{\rmd t} \Big| \Big]
\notag
\\
\sim \ & \lambda^{-\frac{7}{2}}  t^{3-2\gamma_1} \ln t R_0^{6} R^{-3}
\1_{R\le |y| \le 2R}
\sim t^{\frac{3}{2} \gamma_1 - 4} \ln t R_0^{6} R^{-3} \1_{R\le |y| \le 2R}
\le t^{\frac{3}{2} \gamma_1 - 4} \ln t R_0^{6} R^{-3} \1_{|x| \le 3 \lambda_0 R}.
\notag
\end{align}
Then,
\begin{align}
		&  |\mathcal{T}_{\rm{o}} [\Lambda_1]| \lesssim 
		\mathcal{T}_{\rm{o}} [ t^{\frac{3}{2} \gamma_1 - 4} \ln t R_0^6 R^{-3} \1_{|x| \le 3 \lambda_0 R} ]
		\lesssim 
		t^{-\frac 52}
		\rme^
		{-\frac{|x|^2}{16 t } } 
		\int_{t_0}^{ \max\{t_0, t/2 \} }
		s^{\frac{3}{2} \gamma_1 - 4} \ln s (R_0^6 R^{-3})(s)
		(\lambda_0 R)^5(s)
		\rmd s 
        \notag
		\\
		&
		+ t^{\frac{3}{2} \gamma_1 - 4} \ln t R_0^6 R^{-3}
		\big[ (\lambda_0 R)^2 \1_{|x| \le \lambda_0 R} + (\lambda_0 R)^5 |x|^{-3} 
		\rme^{-\frac{|x|^2}{16 t}} \1_{|x| > \lambda_0 R} \big]
		\\
		\lesssim \ &
		t^{-\frac 52}
		\rme^
		{-\frac{|x|^2}{16 t } } 
		\int_{t_0}^{ \max\{t_0, t/2 \} }
		s^{\frac{3}{2} \gamma_1 - 4} \ln s (R_0^6 R^{-3})(s)
		(\lambda_0 R)^5(s)
		\rmd s 
        \notag
		\\
		&
		+ t^{\frac{3}{2} \gamma_1 - 4} \ln t R_0^6 R^{-3} 
		\big[ (\lambda_0 R)^2 \1_{|x| \le t^{\frac{1}{2}}} + (\lambda_0 R)^2 t |x|^{-2} \1_{|x| > t^{\frac{1}{2}} } \big].
        \notag
\end{align}
To make
\begin{equation*}
	\begin{aligned}
		&
		t^{-\frac 52}
		\rme^
		{-\frac{|x|^2}{16 t } } 
		\int_{t_0}^{ \max\{t_0, t/2 \} }
		s^{\frac{3}{2} \gamma_1 - 4} \ln s (R_0^6 R^{-3})(s)
		(\lambda_0 R)^5(s)
		\rmd s 
		\\
		\lesssim \ & t^{\frac{3}{2} \gamma_1 - 4} \ln t R_0^6 R^{-3}
		\big[ (\lambda_0 R)^2 \1_{|x| \le t^{\frac{1}{2}}} + (\lambda_0 R)^2 t |x|^{-2} \1_{|x| > t^{\frac{1}{2}} } \big],
		\\
		&
		\mbox{it suffices to ensure  } 
		\int_{t_0}^{ \max\{t_0, t/2 \} }
		s^{\frac{3}{2} \gamma_1 - 4} \ln s (R_0^6 R^{-3})(s)
		(\lambda_0 R)^5(s)
		\rmd s
		\lesssim t^{\frac{3}{2} \gamma_1 - \frac{3}{2}} \ln t R_0^6 R^{-3} (\lambda_0 R)^2
		\\
\Leftrightarrow \ & \int_{t_0}^{ \max\{t_0, t/2 \} }
s^{ - \frac{7}{2} \gamma_1 +6} \ln s (R_0^6 R^{2})(s)
\rmd s
\lesssim t^{-\frac{\gamma_1}{2} + \frac{5}{2}} \ln t R_0^6 R^{-1}.
	\end{aligned}
\end{equation*}
We take the almost necessary condition up to a multiplicity of power of $\ln t$:
\begin{equation}\label{qd25Dec22-1}
	\mathbf{P}_1[t^{-\frac{\gamma_1}{2} + \frac{5}{2}} \ln t R_0^6 R^{-1}] > 0
	\Leftrightarrow
	\frac{5}{2} -\beta -\frac{\gamma_1}{2} >0.
\end{equation}

When $\mathbf{P}_1[ s^{ - \frac{7}{2} \gamma_1 +6} \ln s (R_0^6 R^{2})(s) ] >-1 $, it suffices to make
\begin{equation}\label{qd25Dec22-2}
		t^{-\frac{7}{2} \gamma_1 + 7} \ln t R_0^6 R^{2} 
		\lesssim t^{-\frac{\gamma_1}{2}  + \frac{5}{2}} \ln t R_0^6 R^{-1}
		\Leftrightarrow t^{-3 \gamma_1 + \frac{9}{2}} R^{3} 
		\lesssim 1
		\Leftrightarrow \frac{3}{2} + \beta - \gamma_1 <0.
\end{equation}

When $\mathbf{P}_1[s^{-\frac{7}{2} \gamma_1 + 6} \ln s (R_0^6 R^{2})(s)] =-1$, it suffices to make
$
(\ln t)^m
\lesssim t^{-\frac{\gamma_1}{2} + \frac{5}{2}} \ln t R_0^6 R^{-1} 
$
with a large constant $m>0$, which is true by \eqref{qd25Dec22-1}. Moreover,
$
1 + \mathbf{P}_1[
s^{ - \frac{7}{2} \gamma_1 +6} \ln s ( R_0^6 R^{2} )(s) ]
< \mathbf{P}_1[ t^{-\frac{\gamma_1}{2} + \frac{5}{2}} \ln t R_0^6 R^{-1} ]
$
implies \eqref{qd25Dec22-2}.

When $\mathbf{P}_1[s^{-\frac{7}{2} \gamma_1 + 6} \ln s ( R_0^6 R^{2})(s)] <-1$, it suffices to make
\begin{equation*}
	\begin{aligned}
		&
		t_0^{-\frac{7}{2} \gamma_1 + 7} \ln t_0 ( R_0^6 R^{2})(t_0)
		\lesssim  t^{-\frac{\gamma_1}{2}  + \frac{5}{2}} \ln t R_0^6 R^{-1}.
		\\
		&
		\mbox{By \eqref{qd25Dec22-1}, it suffices to make } 
		t_0^{-\frac{7}{2} \gamma_1 + 7} \ln t_0 (R_0^6 R^{2})(t_0)
		\lesssim  t_0^{-\frac{\gamma_1}{2} + \frac{5}{2}} \ln t_0 (R_0^6 R^{-1})(t_0),
	\end{aligned}
\end{equation*}
which is same as \eqref{qd25Dec22-2}. In sum, under the parameter restrictions
\begin{equation}
\frac{5}{2} -\beta -\frac{\gamma_1}{2} > 0,
\quad  
\frac{3}{2} + \beta - \gamma_1 < 0,
\end{equation}
we have
\begin{align*}
		& |\mathcal{T}_{\rm{o}} [\Lambda_1]| \lesssim
		\mathcal{T}_{\rm o} [ t^{\frac{3}{2} \gamma_1 - 4} \ln t R_0^6 R^{-3} \1_{|x| \le 3 \lambda_0 R} ] 
        \lesssim t^{\frac{3}{2} \gamma_1 - 4} \ln t R_0^6 R^{-3}
		[ (\lambda_0 R)^2 \1_{|x| \le t^{\frac{1}{2}}} + (\lambda_0 R)^2 t |x|^{-2} \1_{|x| > t^{\frac{1}{2}} } ]
        \\
		\sim \ &
		t^{-\frac{\gamma_1}{2} } \ln t R_0^6 R^{-1}
		( \1_{|x| \le t^{\frac{1}{2}}} + t |x|^{-2} \1_{|x| > t^{\frac{1}{2}} } ) = (\ln t)^{-1} w_{\rm{o}}.
\end{align*}

{\textbf{Estimate for $\Lambda_2$.}} Hereafter in this proof, we always
assume $\epsilon>0$ is a sufficiently small constant varying from line to line. For $\Lambda_2$ given in \eqref{Lambda2-phi}, we have
\begin{align}
|\Lambda_2|
= \ & \Big|
\dot{\lambda} \lambda^{-\frac 52}   \Big( \frac{3}{2} \phi(y,t)
+ y \cdot \nabla_y \phi (y,t) \Big) \eta_R(y)
+ \lambda^{-\frac{5}{2}} \dot{\xi} \cdot \nabla_y \phi(y, t) \eta_R(y) \Big|
\lesssim
| \dot{\lambda}_0 \lambda_0^{-\frac 52} | R_0^6 t^{3-2\gamma_1} \ln t
\langle y \rangle^{-1} \eta_R(y)
\notag
\\
\lesssim \ &
t^{-\frac{\gamma_1}{2}-1} \ln t R_0^6
\langle \lambda_{0}^{-1} |x| \rangle^{-1} \1_{|x| \le 3 \lambda_0 R}
\sim
t^{1-\frac{3}{2} \gamma_1} \ln t R_0^6
(\lambda_{0} + |x|)^{-1} \1_{|x| \le 3 \lambda_0 R}.
\end{align}
The convolution estimate gives
\begin{align}
|\mathcal{T}_{\rm o}[\Lambda_2]| \lesssim \ &
\mathcal{T}_{\rm o} \big[ t^{1-\frac{3}{2}\gamma_1} \ln t R_0^6
(\lambda_{0} + |x|)^{-1} \1_{|x| \le 3 \lambda_0 R} \big]
\lesssim 
t^{-\frac{5}{2}} \rme^{-\frac{|x|^2}{16 t}} \int_{t_0}^{\max\{t_0, t/2\}} s^{1-\frac{3}{2}\gamma_1} \ln s R_0^6 (\lambda_0 R)^{4}(s) \rmd s
\notag
\\
& \qquad + t^{1-\frac{3}{2}\gamma_1} \ln t R_0^6
\big[ \lambda_0 R \1_{|x| \le \lambda_0 R}
+
(\lambda_0 R)^{4}
|x|^{-3} \rme^{-\frac{|x|^2}{16 t}} \1_{ |x| > \lambda_0 R}
\big].
\end{align}
We will impose suitable parameter restrictions to make this upper bound smaller than $t^{-\epsilon} w_{\rm o} =
t^{-\epsilon-\frac{\gamma_1}{2} } $ $(\ln t)^2 R_0^6 R^{-1} 
( \1_{|x| \le t^{\frac{1}{2}}} + t |x|^{-2} \1_{|x| > t^{\frac{1}{2}} } )$. Consider the 1st part. To make
\begin{equation*}
\begin{aligned}
&
t^{-\frac{5}{2}} \rme^{-\frac{|x|^2}{16 t}} \int_{t_0}^{\max\{t_0, t/2\}} s^{1-\frac{3}{2}\gamma_1} \ln s R_0^6 (\lambda_0 R)^{4}(s) \rmd s
\lesssim t^{-\epsilon-\frac{\gamma_1}{2} } (\ln t)^2 R_0^6 R^{-1} 
( \1_{|x| \le t^{\frac{1}{2}}} + t |x|^{-2} \1_{|x| > t^{\frac{1}{2}} } )
\\
\Leftrightarrow \ &  \int_{t_0}^{\max\{t_0, t/2\}} s^{1-\frac{3}{2}\gamma_1} \ln s R_0^6 (\lambda_0 R)^{4}(s) \rmd s
\lesssim t^{-\epsilon-\frac{\gamma_1}{2} +\frac{5}{2} } (\ln t)^2 R_0^6 R^{-1} 
( \1_{|x| \le t^{\frac{1}{2}}} + t |x|^{-2} \rme^{\frac{|x|^2}{16 t}} \1_{|x| > t^{\frac{1}{2}} } ),
\\
&
\mbox{it suffices to make }
\int_{t_0}^{\max\{t_0, t/2\}} s^{9-\frac{11}{2}\gamma_1} \ln s (R_0^6 R^{4})(s) \rmd s
\lesssim t^{-\epsilon-\frac{\gamma_1}{2} +\frac{5}{2} } (\ln t)^2 R_0^6 R^{-1}.
\end{aligned}
\end{equation*}
We take the almost necessary condition for the above inequality:
\begin{equation}\label{qd25Dec23-4}
\mathbf{P}_1[t^{-\epsilon-\frac{\gamma_1}{2} +\frac{5}{2} } (\ln t)^2 R_0^6 R^{-1}] > 0
\Leftrightarrow
\frac{5}{2} - \beta -\frac{\gamma_1}{2} > \epsilon.
\end{equation}

When $\mathbf{P}_1[s^{9-\frac{11}{2}\gamma_1} \ln s (R_0^6 R^{4})(s)] > -1$, it suffices to make
\begin{equation}\label{qd25Dec23-5}
\begin{aligned}
& t^{10-\frac{11}{2}\gamma_1} \ln t R_0^6 R^{4}
\lesssim t^{-\epsilon-\frac{\gamma_1}{2} +\frac{5}{2} } (\ln t)^2 R_0^6 R^{-1}
\Leftrightarrow
1
\lesssim 
t^{-\epsilon + 5 \gamma_1 - \frac{15}{2} } \ln t R^{-5},
\\
&
\mbox{which can be guaranteed by } 
\frac{3}{2} + \beta - \gamma_1 < -\frac{\epsilon}{5}.
\end{aligned}
\end{equation}

When $\mathbf{P}_1[s^{9-\frac{11}{2}\gamma_1} \ln s (R_0^6 R^{4})(s)] = -1$, it suffices to make
$
(\ln t)^m
\lesssim 
t^{-\epsilon-\frac{\gamma_1}{2} +\frac{5}{2} } (\ln t)^2 R_0^6 R^{-1}
$
with a large constant $m>0$, which is true by \eqref{qd25Dec23-4}. Notice that
$
1 + 
\mathbf{P}_1[s^{9-\frac{11}{2}\gamma_1} \ln s (R_0^6 R^{4})(s)]
< 
\mathbf{P}_1 [ t^{-\epsilon-\frac{\gamma_1}{2} +\frac{5}{2} } (\ln t)^2 R_0^6 R^{-1} ]
$
implies the parameter assumption in \eqref{qd25Dec23-5}.

When $\mathbf{P}_1[s^{9-\frac{11}{2}\gamma_1} \ln s (R_0^6 R^{4})(s)] < -1$, it suffices to make
\begin{equation*}
\begin{aligned}
& 
t_0^{10-\frac{11}{2}\gamma_1} \ln t_0 (R_0^6 R^{4})(t_0)
\lesssim 
t^{-\epsilon-\frac{\gamma_1}{2} +\frac{5}{2} } (\ln t)^2 R_0^6 R^{-1}.
\\
&
\mbox{By \eqref{qd25Dec23-4}, it suffices to ensure } t_0^{10-\frac{11}{2}\gamma_1} \ln t_0 (R_0^6 R^{4})(t_0)
\lesssim 
t_0^{-\epsilon-\frac{\gamma_1}{2} +\frac{5}{2} } (\ln t_0)^2 (R_0^6 R^{-1})(t_0),
\end{aligned}
\end{equation*}
which can be dealt with similarly to \eqref{qd25Dec23-5}.

Consider the 2nd part. When $|x| \le t^{\frac{1}{2}}$, to make
\begin{equation*}
	\begin{aligned}
		& t^{1-\frac{3}{2}\gamma_1} \ln t R_0^6
		\big[ \lambda_0 R \1_{|x| \le \lambda_0 R}
		+
		(\lambda_0 R)^{4}
		|x|^{-3} \rme^{-\frac{|x|^2}{16 t}} \1_{\lambda_0 R < |x| \le t^{\frac{1}{2}} }
		\big]
		\lesssim t^{-\epsilon-\frac{\gamma_1}{2} } (\ln t)^2 R_0^6 R^{-1} \1_{|x| \le t^{\frac{1}{2}}},
\\
&
\mbox{it suffices to ensure }
t^{1-\frac{3}{2}\gamma_1} \lambda_0 R
\lesssim t^{-\epsilon-\frac{\gamma_1}{2} } \ln t R^{-1},
	\end{aligned}
\end{equation*}
which can be achieved by taking
\begin{equation}
\frac{3}{2} + \beta - \gamma_1 < - \frac{\epsilon}{2}.
\end{equation}

When $|x| > t^{\frac{1}{2}}$,
notice that $ 
(\lambda_0 R)^{4}
|x|^{-3} \rme^{-\frac{|x|^2}{16 t}} \1_{ |x| > t^{\frac{1}{2}}}
=
\lambda_0 R (\lambda_0 R)^{3}
|x|^{-3} \rme^{-\frac{|x|^2}{16 t}} \1_{ |x| > t^{\frac{1}{2}}} 
$.
To make
\begin{align*}
&
(\lambda_0 R)^{3}
|x|^{-3} \rme^{-\frac{|x|^2}{16 t}} \1_{ |x| > t^{\frac{1}{2}}}
\lesssim t |x|^{-2} \1_{|x| > t^{\frac{1}{2}} }
\Leftrightarrow
(\lambda_0 R)^{3} t^{-1}
|x|^{-1} \rme^{-\frac{|x|^2}{16 t}} \1_{ |x| > t^{\frac{1}{2}}}
\lesssim \1_{|x| > t^{\frac{1}{2}} },
\\
&
\mbox{it suffices to make }
(\lambda_0 R)^{3} t^{-\frac{3}{2}} \lesssim 1 \Leftrightarrow 
\lambda_0
R
\lesssim t^{\frac{1}{2}},
\end{align*}
which is guaranteed by \eqref{qd25Dec23-2}. Then, the case $|x| > t^{\frac{1}{2}}$ goes back to the case $|x| \le t^{\frac{1}{2}}$ above.

{\textbf{Estimate for $\mathcal{G}_1$.}}
\begin{equation}\label{qd25Dec27-4}
	\begin{aligned}
&  |\mathcal{G}_1|
\lesssim  t^{\frac{3}{2}\gamma_1 -4}  \langle \lambda_{0}^{-1} |x| \rangle^{-3} \1_{\lambda_0 R/2 \le |x| \le 3 t^{\frac{1}{2}}}
+
t^{-\frac{3}{2} \gamma_1 + \frac{1}{2} }
\1_{ 2^{-1} t^{\frac{1}{2}} \le |x| \le 3 t^{\frac{1}{2}} }
+
\lambda_0^{-\frac{7}{2}}
\langle \lambda_{0}^{-1} |x| \rangle^{-7} \1_{ 1 \le |\tilde{y}| \le 2}
\\
\lesssim \ & t^{-\frac{3}{2} \gamma_1 + 2} |x|^{-3} \1_{\lambda_0 R/2 \le |x| \le 3 t^{\frac{1}{2}}}
+
t^{-\frac{3}{2} \gamma_1 + \frac{1}{2} }
\1_{ 2^{-1} t^{\frac{1}{2}} \le |x| \le 3 t^{\frac{1}{2}} }
+
\lambda_0^{-\frac{7}{2}}
\langle \lambda_{0}^{-1} |x| \rangle^{-7} \1_{ 2^{-1} t^{\frac{1}{2}} \le |x| \le 3 t^{\frac{1}{2}} }
\\
\sim \ & t^{-\frac{3}{2} \gamma_1 + 2} |x|^{-3} \1_{\lambda_0 R/2 \le |x| \le 3 t^{\frac{1}{2}}}
+
t^{-\frac{3}{2} \gamma_1 + \frac{1}{2} }
\1_{ 2^{-1} t^{\frac{1}{2}} \le |x| \le 3 t^{\frac{1}{2}} }
+
t^{-\frac{7}{2} \gamma_1 + \frac{7}{2}} \1_{ 2^{-1} t^{\frac{1}{2}} \le |x| \le 3 t^{\frac{1}{2}} }
\\
\stackrel{\gamma_1 \ge \frac{3}{2}}{\sim} \ & t^{-\frac{3}{2} \gamma_1 + 2} |x|^{-3} \1_{\lambda_0 R/2 \le |x| \le 3 t^{\frac{1}{2}}}
+
t^{-\frac{3}{2} \gamma_1 + \frac{1}{2} }
\1_{ 2^{-1} t^{\frac{1}{2}} \le |x| \le 3 t^{\frac{1}{2}} }
\sim  t^{-\frac{3}{2} \gamma_1 + 2} |x|^{-3} \1_{\lambda_0 R/2 \le |x| \le 3 t^{\frac{1}{2}}},
	\end{aligned}
\end{equation}
where we use
\begin{equation*}
\begin{aligned}
&
|\mathcal{E}_{U}^{\rm{cut}}| 
= \big| \lambda^{-\frac{3}{2}} U(y) 
( 2^{-1} t^{-1} \tilde{y}  + t^{-\frac{1}{2}} \dot{\xi} )\cdot ( \nabla \eta ) ( \tilde{y} )  
+
2\lambda^{-\frac{5}{2}} t^{-\frac{1}{2}} 
( \nabla U ) (y) \cdot 
( \nabla \eta ) ( \tilde{y} )
+
\lambda^{-\frac{3}{2}} t^{-1} U ( y ) ( \Delta \eta ) ( \tilde{y} ) \big|
\\
&
\mbox{since $ t^{-\frac{1}{2}} |\dot{\xi}| \lesssim t^{-1} $ is guaranteed by $ 3\gamma_1 + \gamma_2 \ge 6$,}
\\
\lesssim \ & \big( \lambda^{-\frac{3}{2}} \langle y \rangle^{-3} t^{-1}
+
\lambda^{-\frac{5}{2}} t^{-\frac{1}{2}}
\langle y \rangle^{-4}
+
\lambda^{-\frac{3}{2}} t^{-1} \langle y \rangle^{-3} \big)
\1_{ t^{\frac{1}{2}} \le |x-\xi| \le 2 t^{\frac{1}{2}} }
\\
\lesssim \ & \big(
\lambda^{-\frac{5}{2}} t^{-\frac{1}{2}}
\langle \lambda_{0}^{-1} |x| \rangle^{-4}
+
\lambda^{-\frac{3}{2}} t^{-1} \langle \lambda_{0}^{-1} |x| \rangle^{-3} \big)
\1_{ 2^{-1} t^{\frac{1}{2}} \le |x| \le 3 t^{\frac{1}{2}} }
\sim
t^{-\frac{3}{2} \gamma_1 + \frac{1}{2} }
\1_{ 2^{-1} t^{\frac{1}{2}} \le |x| \le 3 t^{\frac{1}{2}} }.
\end{aligned} 
\end{equation*}
The convolution estimate gives 
\begin{equation}\label{qd25Dec28-1}
	\begin{aligned}
		&
        |\mathcal{T}_{\rm o}[\mathcal{G}_1]|
        \lesssim
		\mathcal{T}_{\rm o} \big[ t^{-\frac{3}{2} \gamma_1 + 2} |x|^{-3} \1_{\lambda_0 R/2 \le |x| \le 3 t^{\frac{1}{2}}} \big]
		\lesssim 
		t^{-\frac 52}
		\rme^
		{-\frac{|x|^2}{16 t } } 
		\int_{t_0}^{\max\{ t_0, t/2 \} } s^{-\frac{3}{2} \gamma_1 + 3}
		\rmd s 
		\\
		&
		+ t^{-\frac{3}{2} \gamma_1 + 2}
		\big[ (\lambda_0 R)^{-1} \1_{|x| \le \lambda_0 R} + |x|^{-1} \1_{\lambda_0 R < |x| \le t^{\frac{1}{2}}}
		+
		t |x|^{-3} 
		\rme^{-\frac{|x|^2}{16 t}} \1_{|x| > t^{\frac{1}{2}}} \big]
\\
\stackrel{\gamma_1<\frac{8}{3}}{\lesssim} \ &  t^{-\frac{3}{2} \gamma_1 + \frac{3}{2}}
\rme^{-\frac{|x|^2}{16 t } }
+ t^{-\frac{3}{2} \gamma_1 + 2} (\lambda_0 R)^{-1}
\big[  \1_{|x| \le t^{\frac{1}{2}} } 
+
\lambda_0 R t |x|^{-3} 
\rme^{-\frac{|x|^2}{16 t}} \1_{|x| > t^{\frac{1}{2}}} \big]
\\
\stackrel{\eqref{qd25Dec23-2}}{\lesssim} \ &  t^{-\frac{\gamma_1}{2} } R^{-1}
( \1_{|x| \le t^{\frac{1}{2}} } 
+ t |x|^{-2} \1_{|x| > t^{\frac{1}{2}}} )
=
[ (\ln t)^2 R_0^6 ]^{-1} w_{\rm o}.
	\end{aligned}
\end{equation}

{\textbf{Estimate for $\mathcal{G}_2$.}}
For $\psi$ defined in $\overline{B_{T^{9}}}  \times [t_0, T] $ satisfying $\|\psi\|_{\rm o} \le 1$, by \eqref{qd25Dec25-4} and $t_0 \ge 1$,
\begin{align}
		|\Psi| + |\psi|
		\lesssim \ &
t^{-\frac{1}{2} \min\{ \gamma_1, \gamma_2 \}} 
\1_{|x| \le t^{\frac{1}{2}}} + |x|^{-\min\{ \gamma_1, \gamma_2 \}} \1_{|x| > t^{\frac{1}{2}}}
		+ t^{-\frac{\gamma_1}{2} } (\ln t)^2 R_0^6 R^{-1} 
		( \1_{|x| \le t^{\frac{1}{2}}} + t |x|^{-2} \1_{|x| > t^{\frac{1}{2}} } )
        \notag
		\\
		\stackrel{\gamma_1 \le 2}{\sim} \ &  t^{-\frac{1}{2} \min\{ \gamma_1, \gamma_2 \} } 
\1_{|x| \le t^{\frac{1}{2}}} + |x|^{-\min\{ \gamma_1, \gamma_2 \}} \1_{|x| > t^{\frac{1}{2}}}.
\label{qd25Oct25-2}
\end{align}
It follows that
\begin{equation}
\begin{aligned}
|\mathcal{G}_2|
\lesssim \ & \lambda_0^{-2}
\langle \lambda_{0}^{-1} |x| \rangle^{-4}
\1_{\lambda_0 R/2 \le |x| \le 3 t^{\frac{1}{2}}} \big( t^{-\frac{1}{2} \min\{ \gamma_1, \gamma_2 \} } 
\1_{|x| \le t^{\frac{1}{2}}} + |x|^{-\min\{ \gamma_1, \gamma_2 \}} \1_{|x| > t^{\frac{1}{2}}} \big)
\\
\sim \ &
t^{- 2 \gamma_1 -\frac{1}{2} \min\{ \gamma_1, \gamma_2 \}  +4} |x|^{-4}
\1_{\lambda_0 R/2 \le |x| \le 3 t^{\frac{1}{2}}}.
\end{aligned}
\end{equation}
By the convolution estimate,
\begin{align}
		& 
        |\mathcal{T}_{\rm o}[\mathcal{G}_2]|
        \lesssim
        \mathcal{T}_{\rm o} \big[ t^{-2 \gamma_1 -\frac{1}{2} \min\{ \gamma_1, \gamma_2 \} +4} |x|^{-4}
		\1_{\lambda_0 R/2 \le |x| \le 3 t^{\frac{1}{2}} } \big]
		\lesssim 
		t^{-\frac 52}
		\rme^{-\frac{|x|^2}{16 t } } 
		\int_{t_0}^{\max\{ t_0, t/2 \} }  s^{-2 \gamma_1 -\frac{1}{2} \min\{ \gamma_1, \gamma_2 \} + \frac{9}{2}}
		\rmd s 
        \notag
		\\
		&
		+ t^{-2 \gamma_1 -\frac{1}{2} \min\{ \gamma_1, \gamma_2 \} +4}
		\big[ (\lambda_0 R)^{-2} \1_{|x| \le \lambda_0 R} + |x|^{-2} \1_{\lambda_0 R < |x| \le t^{\frac{1}{2}}}
		+
		t^{\frac{1}{2}}
		|x|^{-3} 
		\rme^{-\frac{|x|^2}{16 t}} \1_{|x| > t^{\frac{1}{2}}} \big] 
        \notag
		\\
& \mbox{by $4\gamma_1 + \min\{ \gamma_1, \gamma_2 \} < 11$, then } 
        \\
		\lesssim \ &
		t^{3- 2 \gamma_1 -\frac{1}{2} \min\{ \gamma_1, \gamma_2 \}}
		\rme^{-\frac{|x|^2}{16 t } }
		+ t^{- 2 \gamma_1 -\frac{1}{2} \min\{ \gamma_1, \gamma_2 \} +4} (\lambda_0 R)^{-2}
		\big[ \1_{|x| \le t^{\frac{1}{2}}}
		+
		(\lambda_0 R)^{2}
		t^{\frac{1}{2}}
		|x|^{-3} 
		\rme^{-\frac{|x|^2}{16 t}} \1_{|x| > t^{\frac{1}{2}}} \big] 
        \notag
		\\
		\stackrel{\eqref{qd25Dec23-2}}{\lesssim} \ &
		t^{3-2 \gamma_1 -\frac{1}{2} \min\{ \gamma_1, \gamma_2 \}}
		\rme^{-\frac{|x|^2}{16 t } }
		+ t^{-\frac{1}{2} \min\{ \gamma_1, \gamma_2 \}} R^{-2}
		( \1_{|x| \le t^{\frac{1}{2}}}
		+ 
		\rme^{-\frac{|x|^2}{16 t}} \1_{|x| > t^{\frac{1}{2}}} ) 
		\lesssim t^{-\epsilon} w_{\rm o},
        \notag
\end{align}
where for the last step, we require the assumption
\begin{equation}
2 \epsilon + 6 + 2 \beta < 3 \gamma_1 + \min\{ \gamma_1, \gamma_2 \}, \quad 
2 \epsilon + \gamma_1 <\min\{ \gamma_1, \gamma_2 \} + 2 \beta.
\end{equation}

{\textbf{Estimate for $\mathcal{N}$.}} By \eqref{qd25Oct25-2} and \eqref{qd25Oct22-6} with $\langle y \rangle^{-1} \sim \langle \lambda_{0}^{-1} |x| \rangle^{-1} \sim \lambda_0
( \lambda_{0} + |x| )^{-1}$,
\begin{equation*}
|\Psi| + |\psi| + | \lambda^{-\frac{3}{2}} \phi(y,t) \eta_R(y) |
\lesssim
t^{-\frac{1}{2} \min\{ \gamma_1, \gamma_2 \} } 
\1_{|x| \le t^{\frac{1}{2}}} + |x|^{-\min\{ \gamma_1, \gamma_2 \}} \1_{|x| > t^{\frac{1}{2}}}
+
t^{-\frac{\gamma_1}{2}} \ln t R_0^6
\lambda_0
( \lambda_{0} + |x| )^{-1} \1_{|x| \le 3 \lambda_0 R}.
\end{equation*}
For $\mathcal{N}$ defined in \eqref{N-def} with $n=5$, using $\big| |A+B|^{\frac{4}{3}} (A+B) -|A|^{\frac{4}{3}}A -\frac 73 |A|^{\frac{4}{3}} B\big| \lesssim ( |A|^{\frac{1}{3}} + |B|^{\frac{1}{3}} ) B^2$ with $A= \lambda^{-\frac{3}{2}} U(y) \eta(\tilde{y})$ and $B=u-A$, we have
\begin{equation}\label{qd25Dec28-3}
	\begin{aligned}
		&
		| \mathcal{N}| 
		=  \big|
		|u|^{\frac{4}{3}} u 
		-
		\lambda^{-\frac{7}{2}} ( U(y) \eta(\tilde{y}) )^{\frac{7}{3}}
		-
		\frac{7}{3} \lambda^{-2} 
		(U(y) \eta(\tilde{y}) )^{\frac{4}{3}} 
		\big(\Psi + \psi + \lambda^{-\frac{3}{2}} \phi(y,t) \eta_R(y) \big) \big|
		\\
		\lesssim \ &
		\big(
		\big| \lambda^{-\frac{3}{2}} U(y) \eta(\tilde{y}) \big|^{\frac{1}{3}}
		+
		\big|
		\Psi + \psi + \lambda^{-\frac{3}{2}} \phi(y,t) \eta_R(y)
		\big|^{\frac{1}{3}}
		\big)
		\big|
		\Psi + \psi + \lambda^{-\frac{3}{2}} \phi(y,t) \eta_R(y)
		\big|^2
		\\
		\lesssim \ &
		\lambda_0^{\frac{1}{2}} (\lambda_0 + |x|)^{-1} \1_{ |x| \le 3 t^{1/2} }
		\big(
		|\Psi| + |\psi| + |\lambda^{-\frac{3}{2}} \phi(y,t) \eta_R(y)|
		\big)^2 
		+
		\big(
		|\Psi| + |\psi| + |\lambda^{-\frac{3}{2}} \phi(y,t) \eta_R(y)|
		\big)^{\frac{7}{3}}
		\\
		\lesssim \ &
		t^{1-\frac{\gamma_1}{2}} (\lambda_0 + |x|)^{-1} \1_{|x| \le 3 t^{1/2}}
		\big[ t^{-\min\{ \gamma_1, \gamma_2 \}} \1_{|x|\le t^{\frac{1}{2}}} + |x|^{-2 \min\{ \gamma_1, \gamma_2 \}} \1_{|x| > t^{\frac{1}{2}}}
        \\
        &
		+
		t^{-\gamma_1} (\ln t R_0^6)^{2} 
		\lambda_0^{2}
		( \lambda_{0} + |x| )^{-2} \1_{|x| \le 3 \lambda_0 R} \big]
		\\
		& +
		t^{-\frac{7}{6} \min\{ \gamma_1, \gamma_2 \}} \1_{|x|\le t^{\frac{1}{2}}} + |x|^{-\frac{7}{3} \min\{ \gamma_1, \gamma_2 \}} \1_{|x| > t^{\frac{1}{2}}}
		+
		t^{-\frac{7}{6} \gamma_1} (\ln t R_0^6)^{\frac{7}{3}}
		\lambda_0^{\frac{7}{3}}
		( \lambda_{0} + |x| )^{-\frac{7}{3}} \1_{|x| \le 3 \lambda_0 R}
		\\
		\sim \ & 
		t^{1-\frac{\gamma_1}{2} -\min\{ \gamma_1, \gamma_2 \} } (\lambda_0 + |x|)^{-1} \1_{|x| \le 3 t^{1/2}}
		+
		t^{1-\frac{3}{2} \gamma_1} (\ln t R_0^6)^{2}
		\lambda_0^{2}
		( \lambda_{0} + |x| )^{-3} \1_{|x| \le 3 \lambda_0 R}
		\\
		& +
		t^{-\frac{7}{6} \min\{ \gamma_1, \gamma_2 \}} \1_{|x|\le t^{\frac{1}{2}}} + |x|^{-\frac{7}{3} \min\{ \gamma_1, \gamma_2 \}} \1_{|x| > t^{\frac{1}{2}}}
		+
		t^{-\frac{7}{6} \gamma_1} (\ln t R_0^6)^{\frac{7}{3}}
		\lambda_0^{\frac{7}{3}}
		( \lambda_{0} + |x| )^{-\frac{7}{3}} \1_{|x| \le 3 \lambda_0 R}.
	\end{aligned}
\end{equation}
We will present the convolution estimate for the above term by term. For the 1st part,
\begin{equation}\label{qd25Oct23-2}
	\begin{aligned}
		&
		\mathcal{T}_{\rm o} \big[ t^{1-\frac{\gamma_1}{2} 
        -
        \min\{ \gamma_1, \gamma_2 \}
        } (\lambda_0 + |x|)^{-1} \1_{|x| \le 3 t^{1/2}} \big]
		\\
		\lesssim \ &
		t^{-\frac 52}
		\rme^
		{-\frac{|x|^2}{16 t } } 
		\int_{t_0}^{\max\{t_0, t/2\} }  s^{3-\frac{\gamma_1}{2}
        -
        \min\{ \gamma_1, \gamma_2 \}}
		\rmd s 
		+
		t^{1-\frac{\gamma_1}{2}
        -
        \min\{ \gamma_1, \gamma_2 \}
        }
		\big(
		t^{\frac{1}{2}} \1_{|x| \le t^{\frac{1}{2}}}
		+
		t^2
		|x|^{-3} 
		\rme^{-\frac{|x|^2}{16 t}} \1_{|x| > t^{\frac{1}{2}}}
		\big)
        \\
        &
\stackrel{\gamma_1 + 2\min\{ \gamma_1, \gamma_2 \} < 8}{\lesssim} 
t^{\frac{3}{2}-\frac{\gamma_1}{2} - \min\{ \gamma_1, \gamma_2 \}}
		\rme^{-\frac{|x|^2}{16 t } }
		\lesssim t^{-\epsilon} w_{\rm o},
	\end{aligned}
\end{equation}
where for the last step, we require 
\begin{equation}
\frac{3}{2} + \beta - \min\{ \gamma_1, \gamma_2 \} < -\epsilon.
\end{equation}

For the 2nd part,
\begin{equation*}
	\begin{aligned}
		&
		\mathcal{T}_{\rm o}\big[t^{1-\frac{3}{2} \gamma_1} (\ln t R_0^6)^{2}
		\lambda_0^{2}
		( \lambda_{0} + |x| )^{-3} \1_{|x| \le 3 \lambda_0 R} \big]
		\\
		\lesssim \ &
		t^{-\frac 52}
		\rme^
		{-\frac{|x|^2}{16 t } } 
		\int_{t_0}^{\max\{ t_0, t/2\} }   s^{1-\frac{3}{2} \gamma_1} (\ln s R_0^6)^2 \lambda_0^2(s) (\lambda_0 R)^2(s) \rmd s 
		\\
		&
		+
		t^{1-\frac{3}{2} \gamma_1} (\ln t R_0^6)^{2}
		\lambda_0^{2}
		\big[ \lambda_0^{-1} \1_{|x| \le \lambda_0} + |x|^{-1} \1_{\lambda_0 < |x| \le \lambda_0 R}
		+ (\lambda_0 R)^2
		|x|^{-3} 
		\rme^{-\frac{|x|^2}{16 t}} \1_{|x| > \lambda_0 R} \big]
\\
\sim \ &
t^{-\frac 52}
\rme^{-\frac{|x|^2}{16 t } } 
\int_{t_0}^{\max\{ t_0, t/2\} }   s^{9-\frac{11}{2} \gamma_1} (\ln s R_0^6)^2 R^2(s) \rmd s 
\\
&
+
t^{3-\frac{5}{2} \gamma_1} (\ln t R_0^6)^{2}
\big[ \1_{|x| \le \lambda_0} + \lambda_0 |x|^{-1} \1_{\lambda_0 < |x| \le \lambda_0 R}
+ \lambda_0 (\lambda_0 R)^2
|x|^{-3} 
\rme^{-\frac{|x|^2}{16 t}} \1_{|x| > \lambda_0 R} \big].
	\end{aligned}
\end{equation*}
We will impose parameter restrictions to make the above smaller than $t^{-\epsilon} w_{\rm o}$. On the one hand, to make
\begin{align}
		&
		t^{-\frac 52}
		\rme^{-\frac{|x|^2}{16 t } } 
		\int_{t_0}^{\max\{ t_0, t/2\} }   s^{9-\frac{11}{2} \gamma_1} (\ln s R_0^6)^2 R^2(s) \rmd s 
		\lesssim t^{-\epsilon} w_{\rm o},
        \notag
		\\
		&
		\mbox{it suffices to ensure }
		\int_{t_0}^{\max\{ t_0, t/2\} }   s^{9-\frac{11}{2} \gamma_1} (\ln s R_0^6)^2 R^2(s) \rmd s
		\lesssim 
		t^{-\epsilon-\frac{\gamma_1}{2} +\frac{5}{2}} (\ln t)^2 R_0^6 R^{-1}.
\end{align}
We take an almost necessary condition
\begin{equation}\label{qd25Oct26-1}
	\mathbf{P}_1[ t^{-\epsilon-\frac{\gamma_1}{2} +\frac{5}{2}} (\ln t)^2 R_0^6 R^{-1} ] >0 \Leftrightarrow 
	\epsilon < \frac{5}{2} -\beta -\frac{\gamma_1}{2}.
\end{equation}

If $\mathbf{P}_1[s^{9-\frac{11}{2} \gamma_1} (\ln s R_0^6)^2 R^2(s)] >-1$, it suffices to make
\begin{equation}\label{qd25Dec24-1}
	\begin{aligned}
		& 
		t^{10-\frac{11}{2} \gamma_1} (\ln t R_0^6)^2 R^2
		\lesssim 
		t^{-\epsilon-\frac{\gamma_1}{2} +\frac{5}{2}} (\ln t)^2 R_0^6 R^{-1}
		\Leftrightarrow 
		t^{\frac{15}{2}- 5 \gamma_1 + \epsilon} R_0^6 R^3
		\lesssim 1,
		\\
		&
		\mbox{which can be satisfied by }
		\epsilon < -\frac{15}{2} - 3 \beta + 5\gamma_1.
	\end{aligned}
\end{equation}

If $\mathbf{P}_1[s^{9-\frac{11}{2} \gamma_1} (\ln s R_0^6)^2 R^2(s)] =-1$, it suffices to make
$ (\ln t)^m \lesssim t^{-\epsilon-\frac{\gamma_1}{2} +\frac{5}{2}} (\ln t)^2 R_0^6 R^{-1}
$
with a large constant $m>0$, which is true by \eqref{qd25Oct26-1}. Notice that $1+ \mathbf{P}_1[s^{9-\frac{11}{2} \gamma_1} (\ln s R_0^6)^2 R^2(s)] < \mathbf{P}_1[t^{-\epsilon-\frac{\gamma_1}{2} +\frac{5}{2}} (\ln t)^2 R_0^6 R^{-1}]$ implies the requirement in $\eqref{qd25Dec24-1}$.

If $\mathbf{P}_1[s^{9-\frac{11}{2} \gamma_1} (\ln s R_0^6)^2 R^2(s)] <-1$, it suffices to make
\begin{equation*}
	\begin{aligned}
		&  t_0^{10-\frac{11}{2} \gamma_1} (\ln t_0 R_0^6)^2 R^2(t_0)
		\lesssim 
		t^{-\epsilon-\frac{\gamma_1}{2} +\frac{5}{2}} (\ln t)^2 R_0^6 R^{-1}
		\\
		&
		\mbox{By \eqref{qd25Oct26-1}, it suffices to ensure } t_0^{10-\frac{11}{2} \gamma_1} (\ln t_0 R_0^6)^2 R^2(t_0)
		\lesssim 
		t_0^{-\epsilon-\frac{\gamma_1}{2} +\frac{5}{2}} (\ln t_0)^2 (R_0^6 R^{-1})(t_0),
	\end{aligned}
\end{equation*}
which can be handled similarly to \eqref{qd25Dec24-1}.

On the other hand, to make
\begin{equation}
\begin{aligned}
&
t^{3-\frac{5}{2} \gamma_1} (\ln t R_0^6)^{2}
\big[ \1_{|x| \le \lambda_0} + \lambda_0 |x|^{-1} \1_{\lambda_0 < |x| \le \lambda_0 R}
+ \lambda_0 (\lambda_0 R)^2
|x|^{-3} 
\rme^{-\frac{|x|^2}{16 t}} \1_{|x| > \lambda_0 R} \big]
\lesssim t^{-\epsilon} w_{\rm o},
\\
&
\mbox{it suffices to make }
t^{\epsilon + 3-2\gamma_1} R_0^6 R
\lesssim 1,
\mbox{ \ which can be guaranteed by }
\epsilon < -3 - \beta + 2\gamma_1.
\end{aligned}
\end{equation}

For the 3rd part,
\begin{equation}\label{qd25Oct24-1}
	\begin{aligned}
		&
		\mathcal{T}_{\rm o} [ t^{-\frac{7}{6} \min\{ \gamma_1, \gamma_2 \}} \1_{|x| \le t^{\frac{1}{2}} } ]
		\lesssim
		t^{-\frac 52}
		\rme^
		{-\frac{|x|^2}{16 t } } 
		\int_{t_0}^{\max\{t_0, t/2 \}} 
		s^{-\frac{7}{6} \min\{ \gamma_1, \gamma_2 \} + \frac{5}{2}}
		\rmd s
        \\
        &
		+ t^{-\frac{7}{6} \min\{ \gamma_1, \gamma_2 \}}
		\big( t \1_{|x| \le t^{\frac{1}{2}}} + t^{\frac{5}{2}}
		|x|^{-3} 
		\rme^{-\frac{|x|^2}{16 t}} \1_{|x| > t^{\frac{1}{2}}} \big)
		\stackrel{\min\{ \gamma_1, \gamma_2 \}<3}{\lesssim} 
		t^{1-\frac{7}{6} \min\{ \gamma_1, \gamma_2 \}}
		\rme^{-\frac{|x|^2}{16 t } }
		\lesssim 
		t^{-\epsilon} w_{\rm o},
	\end{aligned}
\end{equation}
where for the last step, we require
\begin{equation}
\epsilon < -1 -\beta + \frac{7}{6} \min\{ \gamma_1, \gamma_2 \} - \frac{\gamma_1}{2}.
\end{equation}

For the 4th part,
\begin{equation}\label{qd25Dec28-6}
\mathcal{T}_{\rm o} [ |x|^{-\frac{7}{3} \min\{ \gamma_1, \gamma_2 \}} \1_{|x| > t^{\frac{1}{2}}} ]
		\stackrel{\min\{ \gamma_1, \gamma_2 \}< \frac{15}{7}}{\lesssim}
		t^{1-\frac{7}{6} \min\{ \gamma_1, \gamma_2 \}} \1_{|x|\le t^{\frac 12}} + t |x|^{-\frac{7}{3} \min\{ \gamma_1, \gamma_2 \}} \1_{|x| > t^{\frac 12}}
		\lesssim t^{-\epsilon} w_{\rm o},
\end{equation}
where for the last step, we require
\begin{equation}
\epsilon < -1 -\beta  + \frac{7}{6} \min\{ \gamma_1, \gamma_2 \} - \frac{\gamma_1}{2}, 
\quad
\min\{ \gamma_1, \gamma_2 \} \ge \frac{6}{7}.
\end{equation}

For the 5th part,
\begin{align*}
		&
		\mathcal{T}_{\rm o} \big[ t^{-\frac{7}{6} \gamma_1} (\ln t R_0^6)^{\frac{7}{3}}
		\lambda_0^{\frac{7}{3}}
		( \lambda_{0} + |x| )^{-\frac{7}{3}} \1_{|x| \le 3 \lambda_0 R} \big]
		\\
		\lesssim \ &
		t^{-\frac 52}
		\rme^{-\frac{|x|^2}{16 t } } 
		\int_{t_0}^{\max\{t_0, t/2\} }
		s^{-\frac{7}{6} \gamma_1}  (\ln s R_0^6 )^{\frac{7}{3}}
		\lambda_0^{\frac{7}{3}}(s)
		(\lambda_0 R)(s)^{\frac{8}{3}}
		\rmd s 
		\\
		&
		+
		t^{-\frac{7}{6} \gamma_1} (\ln t R_0^6)^{\frac{7}{3}}
		\lambda_0^{\frac{7}{3}}
		\big[
		\lambda_0^{-\frac{1}{3}} \1_{|x| \le \lambda_0}
		+
		|x|^{-\frac{1}{3}} \1_{\lambda_0 < |x| \le \lambda_0 R}
		+
		|x|^{-3} 
		\rme^{-\frac{|x|^2}{16 t}}
		(\lambda_0 R)^{\frac{8}{3}} \1_{|x| > \lambda_0 R}
	\big]
\\
\sim \ &
t^{-\frac 52}
\rme^{-\frac{|x|^2}{16 t } } 
\int_{t_0}^{\max\{t_0, t/2\} }
s^{10-\frac{37}{6} \gamma_1}  (\ln s R_0^6 )^{\frac{7}{3}}
R^{\frac{8}{3}}(s)
\rmd s 
\\
&
+
t^{4-\frac{19}{6} \gamma_1} (\ln t R_0^6)^{\frac{7}{3}}
\big[
\1_{|x| \le \lambda_0}
+
\lambda_0^{\frac{1}{3}}
|x|^{-\frac{1}{3}} \1_{\lambda_0 < |x| \le \lambda_0 R}
+
\lambda_0^{\frac{1}{3}}
|x|^{-3} 
\rme^{-\frac{|x|^2}{16 t}}
(\lambda_0 R)^{\frac{8}{3}} \1_{|x| > \lambda_0 R}
\big].
\end{align*}
We want to make the above upper bound smaller than $t^{-\epsilon} w_{\rm o}$. 
For the integral part, to make
\begin{equation}\label{qd25Dec28-7}
	\begin{aligned}
		&
		t^{-\frac 52}
		\rme^{-\frac{|x|^2}{16 t } } 
		\int_{t_0}^{\max\{t_0, t/2\} }
		s^{10-\frac{37}{6} \gamma_1}  (\ln s R_0^6 )^{\frac{7}{3}}
		R^{\frac{8}{3}}(s)
		\rmd s
		\lesssim t^{-\epsilon} w_{\rm o},
		\\
		&
		\mbox{it suffices to ensure }
		\int_{t_0}^{\max\{t_0, t/2\} }
		s^{10-\frac{37}{6} \gamma_1}  (\ln s R_0^6 )^{\frac{7}{3}}
		R^{\frac{8}{3}}(s)
		\rmd s
		\lesssim t^{-\epsilon-\frac{\gamma_1}{2} + \frac 52 } (\ln t)^2 R_0^6 R^{-1}.
	\end{aligned}
\end{equation}
We take an almost necessary condition for the above inequality:
\begin{equation}\label{qd25Oct28-2}
	\mathbf{P}_1[ t^{-\epsilon-\frac{\gamma_1}{2} + \frac 52 } (\ln t)^2 R_0^6 R^{-1}]>0 
	\Leftrightarrow 
	\epsilon < \frac 52 -\beta -\frac{\gamma_1}{2}.
\end{equation}

When $\mathbf{P}_1[ s^{10-\frac{37}{6} \gamma_1}  (\ln s R_0^6 )^{\frac{7}{3}}
R^{\frac{8}{3}}(s) ] >-1$, it suffices to make
\begin{equation}\label{qd25Oct28-3}
	\begin{aligned}
		&
		t^{11-\frac{37}{6} \gamma_1}  (\ln t R_0^6 )^{\frac{7}{3}}
		R^{\frac{8}{3}}
		\lesssim t^{-\epsilon-\frac{\gamma_1}{2} + \frac 52 } (\ln t)^2 R_0^6 R^{-1}
\Leftrightarrow 
t^{\epsilon + \frac{17}{2} - \frac{17}{3} \gamma_1 }
(\ln t)^{\frac{1}{3}}
(R_0^6)^{\frac{4}{3}}
R^{\frac{11}{3}}
\lesssim 1,
\\
		&
		\mbox{which can be satisfied by }
		\epsilon
		< - \frac{17}{2} - \frac{11}{3} \beta + \frac{17}{3} \gamma_1.
	\end{aligned}
\end{equation}

When $\mathbf{P}_1[s^{10-\frac{37}{6} \gamma_1}  (\ln s R_0^6 )^{\frac{7}{3}}
R^{\frac{8}{3}}(s)] =-1$, it suffices to make
$ (\ln t)^{m}
	\lesssim t^{-\epsilon-\frac{\gamma_1}{2} + \frac 52 } (\ln t)^2 R_0^6 R^{-1} $
with a large constant $m>0$, which true by \eqref{qd25Oct28-2}. It follows that
$ 1+
	\mathbf{P}_1[s^{10-\frac{37}{6} \gamma_1}  (\ln s R_0^6 )^{\frac{7}{3}}
	R^{\frac{8}{3}}(s)] < 
	\mathbf{P}_1[ t^{-\epsilon-\frac{\gamma_1}{2} + \frac 52 } (\ln t)^2 R_0^6 R^{-1} ]
$,
which implies the requirement in \eqref{qd25Oct28-3}.

When $\mathbf{P}_1[s^{10-\frac{37}{6} \gamma_1}  (\ln s R_0^6 )^{\frac{7}{3}}
R^{\frac{8}{3}}(s)] <-1$, it suffices to make
\begin{equation*}
	\begin{aligned}
		&
		t_0^{11-\frac{37}{6} \gamma_1}  (\ln t_0 R_0^6 )^{\frac{7}{3}}
		R^{\frac{8}{3}}(t_0)
		\lesssim t^{-\epsilon-\frac{\gamma_1}{2} + \frac 52 } (\ln t)^2 R_0^6 R^{-1}
		\\
		&
		\mbox{By \eqref{qd25Oct28-2}, it suffices to ensure }
		t_0^{11-\frac{37}{6} \gamma_1}  (\ln t_0 R_0^6 )^{\frac{7}{3}}
		R^{\frac{8}{3}}(t_0)
		\lesssim t_0^{-\epsilon-\frac{\gamma_1}{2} + \frac 52 } (\ln t_0)^2 (R_0^6 R^{-1})(t_0),
	\end{aligned}
\end{equation*}
which can be handled similarly to \eqref{qd25Oct28-3}.

For the other part,
\begin{equation}
	\begin{aligned}
		&
		t^{4-\frac{19}{6} \gamma_1} (\ln t R_0^6)^{\frac{7}{3}}
		\big[
		\1_{|x| \le \lambda_0}
		+
		\lambda_0^{\frac{1}{3}}
		|x|^{-\frac{1}{3}} \1_{\lambda_0 < |x| \le \lambda_0 R}
		+
		\lambda_0^{\frac{1}{3}}
		|x|^{-3} 
		\rme^{-\frac{|x|^2}{16 t}}
		(\lambda_0 R)^{\frac{8}{3}} \1_{|x| > \lambda_0 R}
		\big]
		\lesssim t^{-\epsilon} w_{\rm o},
\\
&
\mbox{it suffices to make }
t^{\epsilon + 4 -\frac{8}{3} \gamma_1} 
(\ln t)^{\frac{1}{3}}
( R_0^6)^{\frac{4}{3}}
R
\lesssim 1,
\mbox{ \ which can be satisfied by }
\epsilon< -4 - \beta + \frac{8}{3} \gamma_1.
	\end{aligned}
\end{equation}

Finally, we obtain $|\mathcal{T}_{\rm o}[\mathcal{N}]| \lesssim t^{-\epsilon} w_{\rm o}$.

In sum, under the parameter assumption \eqref{qd25Dec24-2}, by taking $\epsilon>0$ sufficiently small and then $t_0 \gg 1$, all parameter restrictions imposed thus far in this proof are satisfied. Then, we get \eqref{psi-est}.

Summarizing all above upper bounds of terms in $\mathcal{G}$, we have
\begin{align}
		&
		| \mathcal{G} |
		\lesssim 
		t^{\frac{3}{2} \gamma_1 - 4} \ln t R_0^{6} R^{-3} \1_{|x| \le 3 \lambda_0 R}
		+
		t^{1-\frac{3}{2} \gamma_1} \ln t R_0^6
		(\lambda_{0} + |x|)^{-1} \1_{|x| \le 3 \lambda_0 R}
        \notag
        \\
&
		+
		t^{-\frac{3}{2} \gamma_1 + 2} |x|^{-3} \1_{\lambda_0 R/2 \le |x| \le 3 t^{\frac{1}{2}}}
		+ t^{- 2 \gamma_1 -\frac{1}{2} \min\{ \gamma_1, \gamma_2 \}  +4} |x|^{-4}
\1_{\lambda_0 R/2 \le |x| \le 3 t^{\frac{1}{2}}}
\notag
\\
& + 
		t^{1-\frac{\gamma_1}{2} -\min\{ \gamma_1, \gamma_2 \} } (\lambda_0 + |x|)^{-1} \1_{|x| \le 3 t^{1/2}}
		+
		t^{1-\frac{3}{2} \gamma_1} (\ln t R_0^6)^{2}
		\lambda_0^{2}
		( \lambda_{0} + |x| )^{-3} \1_{|x| \le 3 \lambda_0 R}
        \notag
		\\
		& +
		t^{-\frac{7}{6} \min\{ \gamma_1, \gamma_2 \}} \1_{|x|\le t^{\frac{1}{2}}} + |x|^{-\frac{7}{3} \min\{ \gamma_1, \gamma_2 \}} \1_{|x| > t^{\frac{1}{2}}}
		+
		t^{-\frac{7}{6} \gamma_1} (\ln t R_0^6)^{\frac{7}{3}}
		\lambda_0^{\frac{7}{3}}
		( \lambda_{0} + |x| )^{-\frac{7}{3}} \1_{|x| \le 3 \lambda_0 R}
        \label{move-26Aug16-6}
\\
\lesssim \ &
		t^{\frac{3}{2} \gamma_1 - 4} \ln t R_0^{6} R^{-3}
		+
		t^{-\frac{1}{2} \gamma_1 -1} \ln t R_0^6
        +
		t^{\frac{3}{2} \gamma_1 -4} R^{-3}
		+ t^{2 \gamma_1 -\frac{1}{2} \min\{ \gamma_1, \gamma_2 \}  - 4} R^{-4}
        \notag
\\
& + 
		t^{\frac{\gamma_1}{2} -\min\{ \gamma_1, \gamma_2 \} -1}
		+
		t^{-\frac{1}{2} \gamma_1 - 1} (\ln t R_0^6)^{2}
		+
		t^{-\frac{7}{6} \min\{ \gamma_1, \gamma_2 \}}  +
		t^{-\frac{7}{6} \gamma_1} (\ln t R_0^6)^{\frac{7}{3}}
        \notag
\\
\sim \ &
		t^{\frac{3}{2} \gamma_1 - 4} \ln t R_0^{6} R^{-3}
		+ t^{2 \gamma_1 -\frac{1}{2} \min\{ \gamma_1, \gamma_2 \}  - 4} R^{-4}
        \notag
\\
& + 
		t^{\frac{\gamma_1}{2} -\min\{ \gamma_1, \gamma_2 \} -1}
		+
		t^{-\frac{1}{2} \gamma_1 - 1} (\ln t R_0^6)^{2}
		+
		t^{-\frac{7}{6} \min\{ \gamma_1, \gamma_2 \}}  +
		t^{-\frac{7}{6} \gamma_1} (\ln t R_0^6)^{\frac{7}{3}}
\sim 
		t^{\frac{3}{2} \gamma_1 - 4} \ln t R_0^{6} R^{-3},
        \notag
\end{align}
where for the last step, we use the following restrictions in order: $\gamma_1 <\min\{ \gamma_1, \gamma_2 \} + 2 \beta$ in \eqref{qd25Dec24-2};
and
$1+\beta < \frac{1}{3} (\gamma_1 + \min\{ \gamma_1, \gamma_2 \})$, $-1-\beta + \frac{2}{3} \gamma_1 >0$, $\frac{3}{2} \gamma_1 + \frac{7}{6} \min\{ \gamma_1, \gamma_2 \} - 4 - 3 \beta > 0$, where all these three inequalities follow from $0 < -1 -\beta + \frac{7}{6} \min\{ \gamma_1, \gamma_2 \} - \frac{\gamma_1}{2}$ and $\gamma_1 > \frac{3}{2}$ in \eqref{qd25Dec24-2}.

Using \eqref{psi-est} and \eqref{move-26Aug16-6}, by \cite[Lemma 4.1]{decay5d} (with $\rho = t^{2-\gamma_1} (\ln t)^{\frac{1}{2}} R$, which is $\ll t^{\frac{1}{2}}$ by $\frac{3}{2} + \beta - \min\{ \gamma_1, \gamma_2 \} < 0$ in \eqref{qd25Dec24-2}),
we get \eqref{qd25Oct30-3}. Similar to the H\"older estimate in \cite[Lemma 2.15]{WYZZ2024} (with $\rho =1$), we have \eqref{out-holder-est}. The parabolic regularity theory gives the qualitative H\"older continuity of $\mathcal{T}_{\rm{o}} [ \mathcal{G} \1_{|x| \le T^{9}, t \le T} ]$, $\nabla \mathcal{T}_{\rm{o}} [ \mathcal{G} \1_{|x| \le T^{9}, t \le T} ]$.
\end{proof}

\section{Mappings for solving reduced orthogonality equations and the inner problem in $[t_0, T]$}\label{inner-orth-subsec}

In this section, we will prove that the solution mapping system for the inner problem and $\dot{\lambda}_1, \dot{\xi}^{[1]}$ is from $\mathcal{X}$ into a (projected) subspace of $\mathcal{X}$ and is compact, with $\mathcal{X}$ given in \eqref{Xspace-def} later.
Set the spaces
\begin{equation}
	B_{\dot{\lambda}_1} := \{ 
	f\in C([t_0, T]) \mid
	\|f\|_{\dot{\lambda}_1} \le 1
	\},
	\quad 
	B_{\dot{\xi}^{[1]}} := \{ \vec{f}= ( f_1, 0,0,0,0) \in C([t_0, T]) \mid \| \vec{f} \|_{\dot{\xi}^{[1]}} \le 1 \}
\end{equation}
equipped with the norm 
\begin{equation}\label{mu-xi-norm}
	\|f\|_{\dot{\lambda}_1}:= \sup_{t \in [t_0, T] } ( R_0^{-\frac{\zeta}{2}} t^{1-\gamma_1} )^{-1}  | f(t) |,
	\quad
	\| \vec{f} \|_{\dot{\xi}^{[1]}}
	:=
	\sup\limits_{t \in [t_0, T]} (  R_0^{-\frac{\zeta}{4}} t^{\frac{5}{2}-\frac{3\gamma_1}{2} - \frac{\gamma_2}{2} } )^{-1} | \vec{f}(t) |
\mbox{ \ with a constant \ } \zeta \in (0,1),
\end{equation} 
where the topology is inspired by the estimate \eqref{qd26Jan3-6} and
\eqref{vec-S-est} later. Recall $\mathcal{T}^{\rm{in}}[\mathcal{J}]$ given in \eqref{Tin-def} and $\mathcal{S}_i$, $i\in \overline{1,6}$, $\vec{\mathcal{S}}$ given in \eqref{mu1-xi-sys}.

\begin{lemma}\label{mu1-xi-lem}

$(1).$ Suppose $\gamma_1 <2$,
for any $(\dot{\lambda}_1, \dot{\xi}^{[1]} ) \in B_{\dot{\lambda}_1} \times B_{\dot{\xi}^{[1]}}$ and $\lambda_1 = \lambda_1[\dot{\lambda}_1](t)$, $\xi^{[1]} = \xi^{[1]}[\dot{\xi}^{[1]}](t)$ given in \eqref{mu1-xi-sys}, then for $t_0$ sufficiently large, $\lambda_1, \dot{\lambda}_1,  \xi^{[1]}, \dot{\xi}^{[1]}$ satisfy the assumption
\eqref{mu1-ansatz}.

$(2).$ 
Suppose
\begin{equation}\label{qd26Jan3-7}
\begin{aligned}
&
\zeta \in (0,1),
\quad
\gamma_1 \in (\frac{3}{2}, 2),
\quad
\gamma_2 \in [0,5),
\quad
2 \beta < 2\gamma_1 -3,
\quad
6 < 3\gamma_1 + \min\{ \gamma_1, \gamma_2\},
\\
&
3 < \gamma_1 + \gamma_2,
\quad
0 < \frac{3}{2} + 2\beta - \frac{1}{2} (\gamma_1 + \gamma_2),
\quad
\frac{3}{2} + \beta -\gamma_1 + \frac{1}{2} \gamma_2 - \frac{1}{2} \min\{ \gamma_1, \gamma_2\} < 0,
\end{aligned}
\end{equation}
then for $t_0$ sufficiently large, with $\lambda_1 = \lambda_1[\dot{\lambda}_1](t)$ and $\xi^{[1]} = \xi^{[1]}[\dot{\xi}^{[1]}](t)$ given in \eqref{mu1-xi-sys}, 
\begin{equation*}
\big( 
\mathcal{T}^{\rm{in}}\big[ \mathcal{J}[\psi, \lambda_{0} + \lambda_{1}, \xi^{[0]} + \xi^{[1]} ] \big](y,\tau[\lambda_0 + \lambda_1](t)), \, 
\mathcal{S}_6[\psi, \lambda_1,\xi^{[1]}](t), \, 
\vec{\mathcal{S}}[\psi,\lambda_1,\xi^{[1]}](t)
\big)
\end{equation*}
is a compact mapping from $(\psi, \dot{\lambda}_1, \dot{\xi}^{[1]}) \in B_{\rm o} \times B_{\dot{\lambda}_1} \times B_{\dot{\xi}^{[1]}}$ into $ B_{\rm{in}} \times B_{\dot{\lambda}_1} \times B_{\dot{\xi}^{[1]}}$.

Moreover, for any compact set $K \subset [t_0,\infty)$, there exists $C_{K}$ sufficiently large such that for all $T > C_{K}$, then $K\subset [t_0, T]$, $\mathcal{S}_6[\psi, \lambda_1,\xi^{[1]}](t)$, 
$\vec{\mathcal{S}}[\psi,\lambda_1,\xi^{[1]}](t)$ have uniform H\"older continuity in $K$,
and $\mathcal{T}^{\rm{in}}\big[ \mathcal{J}[\psi, \lambda_{0} + \lambda_{1}, \xi^{[0]} + \xi^{[1]} ] \big](y,\tau[\lambda_0 + \lambda_1](t))$, $\nabla_{y} \mathcal{T}^{\rm{in}}\big[ \mathcal{J}[\psi, \lambda_{0} + \lambda_{1}, \xi^{[0]} + \xi^{[1]} ] \big](y,\tau[\lambda_0 + \lambda_1](t))$ have uniform H\"older continuity in $\{ (y,t) \mid t\in K, y \in \overline{B_{4 R}} \}$, independent of the choice of $(\psi, \dot{\lambda}_1, \dot{\xi}^{[1]}) \in B_{\rm o} \times B_{\dot{\lambda}_1} \times B_{\dot{\xi}^{[1]}}$.
\end{lemma}

\begin{proof}
In this proof, we always assume $(\psi, \dot{\lambda}_1, \dot{\xi}^{[1]} ) \in B_{\rm o} \times B_{\dot{\lambda}_1} \times B_{\dot{\xi}^{[1]}}$. We agree that $t\in [t_0, T]$ and $\tau \in [\tau_0, \tau(T)]$ except when handling $\mathcal{J}$ defined in \eqref{calJ-def} as the extension of $\mathcal{H}(y,t(\tau))$, the relevant $\tilde{\mathcal{J}}_{i,j}$ defined in \eqref{move-26Aug21-5}, and the mappings for the inner problem in different modes.

{\textbf{Step $1$.}}
For $\lambda_1 = \lambda_1[\dot{\lambda}_1](t)$, $\xi^{[1]} = \xi^{[1]}[\dot{\xi}^{[1]}](t)$ given in \eqref{mu1-xi-sys}, direct calculations give
\begin{equation}\label{qd23Dec28-1}
	|\lambda_1|  
	\stackrel{\gamma_1<2}{\lesssim} R_0^{-\frac{\zeta}{2}} t^{2-\gamma_1},
	\quad
\xi^{[1]} = (\xi_1^{[1]},0,0,0,0),
\quad    
	|\xi^{[1]}| \lesssim R_0^{-\frac{\zeta}{4}}
	\begin{cases} 
t^{\frac{7}{2}-\frac{3\gamma_1}{2} - \frac{\gamma_2}{2}}  
			& \mbox{ \ if \ } 3 \gamma_1 + \gamma_2 \ne 7
			\\
            \ln t 
			& \mbox{ \ if \ } 3 \gamma_1 + \gamma_2 = 7
		\end{cases}
	\sim
	R_0^{-\frac{\zeta}{4}}
	|\xi^{[0]}|.
\end{equation}
Thus, for $t_0\gg 1$, $\lambda_1, \dot{\lambda}_1,  \xi^{[1]}, \dot{\xi}^{[1]}$ satisfy the ansatz
\eqref{mu1-ansatz}. That is, Lemma \ref{mu1-xi-lem} $(1)$ holds. The remaining part is left to the proof of Lemma \ref{mu1-xi-lem} $(2)$.

{\textbf{Step $2$ - for $\mathcal{S}_6$ and mode $0$.}}
Recall $\mathcal{H}$ in \eqref{H-def} and the cut-off function in $\mathcal{J}$. For $|y| \le 4R$, $t_0 \gg 1$, we have $|\lambda y+\xi| \le |\lambda| 4R +|\xi| \lesssim t^{2-\gamma_1 +\beta} \ll T^9$ for $\beta- \gamma_1 <7$. Thus, $\mathcal{J}$ is well-defined in $\mathbb{R}^5 \times (\tau_0, \infty)$ for $(\psi, \dot{\lambda}_1, \dot{\xi}^{[1]} ) \in B_{\rm o} \times B_{\dot{\lambda}_1} \times B_{\dot{\xi}^{[1]}}$.

Recall $\mathcal{F}(t)$ in \eqref{qd26Mar3-1}. We first establish the following estimates.
\begin{equation}\label{qd26Apr28-2}  
	\begin{aligned} 
		&
		\Big| \lambda^{\frac{1}{2}}
		\int_{B_{R_0}} U(y)^{\frac{4}{3}} Z_{6}(y)
		\psi(\lambda y+\xi, t) \rmd y \Big| \stackrel{\|\psi\|_{\rm o} \le 1}{\lesssim} 
		t^{1-\gamma_1} (\ln t)^2 R_0^6 R^{-1},
		\\
		& 
		\Big| (\lambda^{\frac{1}{2}} - \lambda_0^{\frac{1}{2}} - 2^{-1} \lambda_0^{-\frac{1}{2}} \lambda_1 ) \Psi(0, t)
		\int_{B_{R_0}} U(y)^{\frac{4}{3}}  Z_{6}(y) \rmd y \Big| 
		\lesssim 
		\lambda_0^{-\frac{3}{2}} \lambda_1^2 | \Psi(0,t) |
		\stackrel{\eqref{qd23Dec28-1} \eqref{Psi0t}}{\lesssim}
        R_0^{-\zeta} t^{1-\gamma_1},
\\
&
		\Big| \lambda^{\frac{1}{2}}
		\int_{B_{R_0}} U(y)^{\frac{4}{3}} Z_{6}(y)
		\big( \Psi(\lambda y+\xi,t) - \Psi(0,t) \big) \rmd y \Big|
        \\
= \	& \Big| \lambda^{\frac{1}{2}}
		\int_{B_{R_0}} U(y)^{\frac{4}{3}} Z_{6}(y)
		\big[ \Psi(\lambda y+\xi,t) - \Psi(0,t) - (\nabla \Psi)(0,t) \cdot (\lambda y+\xi) + (\nabla \Psi)(0,t) \cdot (\lambda y+\xi) \big] \rmd y \Big|
        \\
\stackrel{\eqref{partial-Psi0t}}{=} \ & \Big| \lambda^{\frac{1}{2}}
		\int_{B_{R_0}} U(y)^{\frac{4}{3}} Z_{6}(y)
		\big[ \Psi(\lambda y+\xi,t) - \Psi(0,t) - (\nabla \Psi)(0,t) \cdot (\lambda y+\xi) + (\partial_{x_1} \Psi)(0,t) \xi_{1} \big] \rmd y \Big|
        \\
\stackrel{\eqref{qd25Dec30-1}}{\lesssim} \ & \lambda^{\frac{1}{2}}
		\int_{B_{R_0}} \langle y \rangle^{-7}
		\Big[ (\lambda^2 |y|^2+ |\xi|^2) t^{-1-\frac{1}{2}\min\{ \gamma_1, \gamma_2\} } 
        +  
        \begin{cases}
t^{3-\frac{3}{2}\gamma_1 - \gamma_2}
	&
	\mbox{ \ if \ } 3 \gamma_1 + \gamma_2 \ne 7
	\\
t^{-\frac{\gamma_2 + 1}{2}} \ln t
	&
	\mbox{ \ if \ } 3 \gamma_1 + \gamma_2 = 7
\end{cases} \Big] \rmd y
		\\
\lesssim \ & \lambda^{\frac{1}{2}}
		\Big[ (\ln R_0) \lambda^2 t^{-1-\frac{1}{2}\min\{ \gamma_1, \gamma_2\} } 
        +  
        \begin{cases}
t^{3-\frac{3}{2}\gamma_1 - \gamma_2}
	&
	\mbox{ \ if \ } 3 \gamma_1 + \gamma_2 \ne 7
	\\
t^{-\frac{\gamma_2 + 1}{2}} \ln t
	&
	\mbox{ \ if \ } 3 \gamma_1 + \gamma_2 = 7
\end{cases} \Big]
		\\
\sim \ & 
(\ln R_0) t^{4 -\frac{5}{2} \gamma_1 - \frac{1}{2}\min\{ \gamma_1, \gamma_2\} } 
        +  
        \begin{cases}
t^{4- 2 \gamma_1 - \gamma_2}
	&
	\mbox{ \ if \ } 3 \gamma_1 + \gamma_2 \ne 7
	\\
t^{\frac{1}{2} - \frac{\gamma_1 + \gamma_2}{2}} \ln t
	&
	\mbox{ \ if \ } 3 \gamma_1 + \gamma_2 = 7,
\end{cases}
	\end{aligned} 
\end{equation}
where we use
\begin{equation}\label{qd25Dec30-1}
		\big|
		\Psi(\lambda y+\xi,t)
		-
		\Psi(0,t) -
		(\nabla \Psi)(0,t) \cdot (\lambda y+\xi) \big|
		\stackrel{\eqref{qd25Dec25-4}, \min\{ \gamma_1, \gamma_2\} \ge 0}{\lesssim} (\lambda^2 |y|^2+ |\xi|^2) t^{-1-\frac{1}{2}\min\{ \gamma_1, \gamma_2\} }.
\end{equation}

We require the inner linear theory to handle $\varrho_{0,1}^{*}(\tau(t))$. Recall \eqref{qd26Jan3-1} and $\Upsilon_{0,1}$ in \eqref{move-26Sep5-2}.
\begin{equation*}
	\mathcal{H}_{0,1}^{\sharp}(r,t)
		= \int_{ S^{4} } \mathcal{H}(r\theta, t) \Upsilon_{0,1}(\theta) \rmd \theta
	= \int_{ S^{4} } 
	\Big[
	\dot{\lambda} \lambda Z_{6}(y) 
	+\frac{7}{3}\lambda^{\frac{3}{2}} U(y)^{\frac{4}{3}}
	\big(
	\Psi(\lambda y+\xi,t)
	+
	\psi(\lambda y+\xi, t)
	\big)
	\Big] \Big|_{y = r\theta} \Upsilon_{0,1}(\theta) \rmd \theta,
\end{equation*}
\begin{equation}\label{qd25Jan3-3}
\tilde{\mathcal{H}}_{0,1}(y,t) = \mathcal{H}_{0,1}^{\sharp}(|y|,t) \Upsilon_{0,1}(y/|y|),
\quad
	|\tilde{\mathcal{H}}_{0,1}(y,t)| \lesssim 
	|\mathcal{H}_{0,1}^{\sharp}(|y|,t)|
	\lesssim 
	t^{3-2\gamma_1} \langle y \rangle^{-3} 
    \stackrel{\eqref{tau-est}}{\sim} (\tau(t))^{-1} \langle y \rangle^{-3},
\end{equation}
where we use
\begin{align*}
&
\Big|
\int_{ S^{4} } 
	\Big[
	\dot{\lambda} \lambda Z_{6}(y) 
	\Big] \Big|_{y = r\theta} \Upsilon_{0,1}(\theta) \rmd \theta
\Big|
\lesssim t^{3-2\gamma_1} \langle r \rangle^{-3},
\\
&
\Big| 
\int_{ S^{4} } 
	\Big[ \lambda^{\frac{3}{2}} U(y)^{\frac{4}{3}}
	\psi(\lambda y+\xi, t)
	\Big] \Big|_{y = r\theta} \Upsilon_{0,1}(\theta) \rmd \theta \Big|
\stackrel{\|\psi\|_{\rm o} \le 1}{\lesssim}
t^{3-2\gamma_1} (\ln t)^2 R_0^6 R^{-1} \langle r \rangle^{-4},
\\
&
\Big|
\int_{ S^{4} } 
	\Big[ \lambda^{\frac{3}{2}} U(y)^{\frac{4}{3}}
\Psi(\lambda y+\xi,t)
	\Big] \Big|_{y = r\theta} \Upsilon_{0,1}(\theta) \rmd \theta
\Big|
\stackrel{\gamma_1, \gamma_2 \in [0,5), \, \eqref{move-26Aug19-1}}{\lesssim} 
\lambda^{\frac{3}{2}}
\langle r\rangle^{-4}
\big[ 
t^{-\frac{\gamma_1}{2}} + | c_{\sharp} | t^{-\frac{\gamma_2}{2} - \frac{1}{2}} (\lambda r+ |\xi|)
\big]
\\ 
& \quad
\lesssim \lambda^{\frac{3}{2}}
\langle r\rangle^{-4}
\big[ 
t^{-\frac{\gamma_1}{2}} + t^{-\gamma_1 - \frac{\gamma_2}{2} +\frac{3}{2}} \langle r\rangle
\big]
\stackrel{\gamma_1 + \gamma_2 \ge 3}{\lesssim}
t^{-\frac{\gamma_1}{2}} \lambda^{\frac{3}{2}}
\langle r\rangle^{-3} \sim t^{3-2\gamma_1} \langle r \rangle^{-3}.
\end{align*}

For $\tau \in (\tau_0, C_{tm} T^{2\gamma_1 -3}]$, $y\in \mathbb{R}^5$, since $	C_{tm}^{-1} T^{2\gamma_1 -3} \le \tau(T)
\le C_{tm} T^{2\gamma_1 -3}$ due to \eqref{tau-est}, then by \eqref{move-26Aug21-5},
\begin{equation*}
	|\tilde{\mathcal{J}}_{0,1}(y,\tau)| \lesssim \tau^{-1} \langle y \rangle^{-3} \1_{\tau_0 < \tau \le \tau(T)}
	+
	(\tau(T))^{-1} \langle y \rangle^{-3} \1_{ \tau(T) < \tau \le C_{tm} T^{2\gamma_1 -3}}
\sim   
\tau^{-1} \langle y \rangle^{-3} \1_{ \tau_0 < \tau \le C_{tm} T^{2\gamma_1 -3}}.
\end{equation*}

By Proposition \ref{mode0-regluing-prop} (with $v=\tau^{-1}$, $a=1$, $R_0=\ln t_0$, $\epsilon_0 = \zeta$) and Remark \ref{qd26Apr28-1-rek}, under the assumption
\begin{equation}\label{move-26Aug21-6}
\zeta \in (0,1),
\quad
\frac{2 \beta}{2\gamma_1 -3} <1,
\quad
t_0 \gg 1,
\end{equation}
\eqref{qd26Apr12-3} and \eqref{qd26Apr12-4} are well-defined, and it holds that
\begin{align}
&
\mbox{$\mathcal{T}_{0,1}^{\rm{in}}(y,\tau) \in C^{1+\sigma,\frac{1+\sigma}{2}} \big( \{ (y,\tau) \mid \tau \in (\tau_0, C_{tm} T^{2\gamma_1 -3}], y\in B_{R_{*}(\tau)} \} \big)$ for some $\sigma \in (0,1)$};
\notag
\\
&
\mbox{$\mathcal{T}_{0,1}^{\rm{in}}(y,\tau)$
is radially symmetric in the spatial variable $y$};
\label{qd26Apr3-5}
\\
&
	\langle y \rangle |\nabla \mathcal{T}_{0,1}^{\rm{in}} | + | \mathcal{T}_{0,1}^{\rm{in}} | \lesssim R_0^{6} \tau^{-1} \langle y \rangle^{-1};
	\quad
    \mathcal{T}_{0,1}^{\rm{in}}(y, \tau_0) = g_0 \eta(2 y/R_0) Z_0(y)
    \mbox{ \ with \ }
	|g_0| \lesssim R_0 \tau_0^{-1};
    \notag
\end{align}
\begin{equation}\label{qd26Apr28-3}
|\lambda^{-1} \varrho_{0,1}^{*}(\tau(t))| \lesssim \lambda^{-1} (\tau(t))^{-1} R_0^{-\zeta} \sim R_0^{-\zeta} t^{1-\gamma_1};
\quad
\mbox{$\varrho_{0,1}^{*}(\tau(t))$ is H\"older continuous for $t \in (t_0, T]$.}
\end{equation}

By \eqref{qd26Apr28-2} and \eqref{qd26Apr28-3}, under the assumption
\begin{equation}
6 < 3\gamma_1 + \min\{ \gamma_1, \gamma_2\},
\quad
3 < \gamma_1 + \gamma_2,
\end{equation}
we have
\begin{equation}\label{F-upp}
|\mathcal{F}(t) | \lesssim R_0^{-\zeta} t^{1-\gamma_1}. 
\end{equation}

For $\gamma_1<2$, we can take $0<\delta < 1$ to make $  - \frac{2-\gamma_1}{2} (1+ \delta)  + (1-\gamma_1) >-1$. Then
\begin{equation*}
	\begin{aligned}
		&
		\Big| \omega(t) \int_{t_0}^t
		\rme^{\int_t^s \omega(a) \rmd a}  \mathcal{F}(s) \rmd s \Big|
 \stackrel{\eqref{qd25Dec26-1}}{\lesssim}
		 t^{-1} \int_{t_0}^t
		\rme^{ - \int_s^t \omega(a) \rmd a}  |\mathcal{F}(s)| \rmd s
    \stackrel{\eqref{qd25Dec26-1}, t_0\gg 1}{\lesssim}  t^{-1} \int_{t_0}^t
		\rme^{ \int_s^t \frac{2-\gamma_1}{2} (1+ \delta) a^{-1} \rmd a} R_0^{-\zeta} s^{1-\gamma_1} \rmd s
		\\
		= \ & R_0^{-\zeta} t^{-1} \int_{t_0}^t
		\rme^{ \frac{2-\gamma_1}{2} (1+ \delta) \ln(\frac{t}{s})}  s^{1-\gamma_1} \rmd s
		= R_0^{-\zeta} t^{-1} \int_{t_0}^t
		\big( \frac{t}{s} \big)^{ \frac{2-\gamma_1}{2} (1+ \delta) }  s^{1-\gamma_1} \rmd s
		\\
		= \ & R_0^{-\zeta} t^{-1+ \frac{2-\gamma_1}{2} (1+ \delta)} \int_{t_0}^t
		s^{ - \frac{2-\gamma_1}{2} (1+ \delta) }  s^{1-\gamma_1} \rmd s
		\lesssim R_0^{-\zeta} t^{-1+ \frac{2-\gamma_1}{2} (1+ \delta)}
		t^{1 - \frac{2-\gamma_1}{2} (1+ \delta) } 
		 t^{1-\gamma_1}
		= R_0^{-\zeta} t^{1-\gamma_1}.
	\end{aligned}
\end{equation*}
In sum, we have
\begin{equation}\label{qd26Jan3-6}
| \mathcal{S}_6
| \lesssim R_0^{-\zeta} t^{1-\gamma_1}.
\end{equation}

{\textbf{Step $3$ - for $\mathcal{S}_i$, $i\in \overline{2,5}$ and mode $1$.}} In {\textbf{Step $3$}}, we always assume $i\in \overline{2,5}$. We will cope with $\mathcal{S}_i$  given in \eqref{qd25Dec25-3} term by term.

$(\partial_{x_i} \Psi)(0,t) 
\stackrel{\eqref{partial-Psi0t}}{=} 0$.
\begin{equation*}
\int_{B_{R_0}} U(y)^{\frac{4}{3}} Z_{i}(y) (\Psi + \psi)(\lambda y+\xi,t) \rmd y
=
\int_{B_{R_0}} U(y)^{\frac{4}{3}} Z_{i}(y) (\Psi + \psi)\big( (\lambda y_1 +\xi_1, \lambda y_2 , \lambda y_3,\lambda y_4,\lambda y_5),t \big) \rmd y
= 0
\end{equation*}
since $U$ is radially symmetric; $\Psi(x,t)$, $\psi(x,t)$ are even with respect to the $i$-th component of $x$; $Z_i$ is odd about $y_i$.
\begin{equation*}
\int_{B_{R_0}} U(y)^{\frac{4}{3}} Z_{i}(y) ( - \Psi(0,t) - \psi(0, t) ) \rmd y = 0.
\end{equation*}
\begin{equation*}
\begin{aligned}
&
\int_{B_{R_0}} U(y)^{\frac{4}{3}} Z_{i}(y) \big[ -
		(\nabla \Psi)(0,t) \cdot (\lambda y+\xi)
		\big] \rmd y
=
\int_{B_{R_0}} U(y)^{\frac{4}{3}} Z_{i}(y) \big[ -
		(\nabla \Psi)(0,t) \cdot (\lambda y)
		\big] \rmd y
\\
= \ & \int_{B_{R_0}} U(y)^{\frac{4}{3}} Z_{i}(y) \big[ -
		(\partial_{x_i} \Psi)(0,t) \cdot (\lambda y_i)
		\big] \rmd y \stackrel{\eqref{partial-Psi0t}}{=} 0.
\end{aligned}
\end{equation*}
Note that
\begin{equation}\label{move-26Aug22-2}
\begin{aligned}
&
	\mathcal{H}
= \dot{\lambda} \lambda Z_{6}(y)
	+\lambda \dot{\xi} \cdot (\nabla U )(y) 
	+\frac{7}{3}\lambda^{\frac{3}{2}} U(y)^{\frac{4}{3}}
	(\Psi + \psi)(\lambda y+\xi, t)
\\
= \ & \dot{\lambda} \lambda Z_{6}(y)
	+\lambda \dot{\xi}_1 (\partial_{x_1} U )(y) 
	+
    \frac{7}{3}\lambda^{\frac{3}{2}} U(y)^{\frac{4}{3}}
	(\Psi
	+
	\psi)\big( (\lambda y_1+\xi_1, \lambda y_2, \lambda y_3, \lambda y_4, \lambda y_5), t \big)
\end{aligned}
\end{equation}
is even for the $i$-th component of $y$, $i \in \overline{2,5}$. Recalling $\Upsilon_{1,j}$ in \eqref{move-26Sep5-2} with $n=5$, $\mathcal{H}_{1,i}^{\sharp}(|y|,t)$ and $\tilde{\mathcal{H}}_{1,i}(y,t)$ defined in \eqref{qd26Jan3-1}, and $\tilde{\mathcal{J}}_{1,i}$ defined in \eqref{move-26Aug21-5}, we have
\begin{equation}\label{move-26Aug22-1}
\mathcal{H}_{1,i}^{\sharp}(|y|,t) = \int_{ S^{4} } \mathcal{H}(|y| \theta, t) \Upsilon_{1,i}(\theta) \rmd \theta = 0,
\quad 
\tilde{\mathcal{H}}_{1,i}(y,t) = 0,
\quad
\tilde{\mathcal{J}}_{1,i}(y,\tau) = 0
\mbox{ \ for \ } i\in \overline{2,5}.
\end{equation}
Under the assumption \eqref{move-26Aug21-6}, by Proposition \ref{regluing-mode1-prop}, $(\mathcal{T}_{1,i}^{\rm{in}},  \varrho_{1,i}(\tau) ) = (\mathcal{T}_{1,i}^{\rm{in}}(y,\tau), \varrho_{1,i}(\tau) ) [ \tilde{\mathcal{J}}_{1,i}]$ is well-defined in \eqref{qd26Apr12-5}
linearly depends on $ \tilde{\mathcal{J}}_{1,i}$, and $\varrho_{1,i}^{*}(\tau) = \varrho_{1,i}^{*}[\tilde{\mathcal{J}}_{1,i}](\tau)$ linearly depends on $\tilde{\mathcal{J}}_{1,i}$ in the line below \eqref{move-26Aug20-4}. It follows that
\begin{equation}\label{move-26Aug22-5}
\mathcal{T}_{1,i}^{\rm{in}}(y,\tau) = 0,
\quad
\varrho_{1,i}(\tau) = 0,
\quad
\varrho_{1,i}^{*}(\tau) = 0
\mbox{ \ for \ } i\in \overline{2,5}.
\end{equation}

In sum,
\begin{equation}\label{move-26Aug21-7}
\mathcal{S}_i= 0 \mbox{ \ for \ } i \in \overline{2,5}.
\end{equation}

{\textbf{Step $4$ - for $\mathcal{S}_1$ and mode $1$.}} We will handle $\mathcal{S}_1$ given in \eqref{qd25Dec25-3} term by term.
\begin{equation}\label{qd26Apr29-1}
	\Big| ( \lambda^{\frac{3}{2}} - \lambda_0^{\frac{3}{2}} ) (\partial_{x_1} \Psi)(0,t)
	\int_{B_{R_0}} U(y)^{\frac{4}{3}} y_1 Z_{1}(y) \rmd y \Big|
	\stackrel{\eqref{partial-Psi0t}}{\lesssim}
	\lambda_0^{\frac{1}{2}} \lambda_1 t^{-\frac{\gamma_2 +1}{2}} 
	\stackrel{\eqref{qd23Dec28-1}}{\lesssim} R_0^{-\frac{\zeta}{2}} t^{\frac{5}{2}- \frac{3 \gamma_1}{2} - \frac{\gamma_2}{2}}.
\end{equation}
\begin{equation}\label{qd26Apr29-2}
	\begin{aligned}
		&
		\Big|
		\lambda^{\frac{1}{2}}
		\int_{B_{R_0}} U(y)^{\frac{4}{3}} Z_{1}(y)
		\big[
		\Psi(\lambda y+\xi,t) - \Psi(0,t)
		-
		(\nabla \Psi)(0,t) \cdot (\lambda y+\xi)
		\big] \rmd y
		\Big|
		\\
        \stackrel{\eqref{qd25Dec30-1}}{\lesssim} \ & \lambda^{\frac{1}{2}}
		\int_{B_{R_0}} \langle y \rangle^{-8}
		(\lambda^2 |y|^2+ |\xi|^2) t^{-1-\frac{1}{2}\min\{ \gamma_1, \gamma_2\} } \rmd y
        \\
		\lesssim \ &
		\lambda_0^{\frac{5}{2}} t^{-1-\frac{1}{2}\min\{ \gamma_1, \gamma_2\} }
		\sim t^{4-\frac{5}{2} \gamma_1 - \frac{1}{2}\min\{ \gamma_1, \gamma_2\}}
    \lesssim
		t^{-\epsilon} t^{\frac{5}{2}- \frac{3 \gamma_1}{2} - \frac{\gamma_2}{2}}
	\end{aligned}
\end{equation}
with a constant $0<\epsilon\ll 1$, where for the last step, we require
\begin{equation}
\frac{3}{2} - \gamma_1 + \frac{\gamma_2}{2} - \frac{1}{2} \min\{ \gamma_1, \gamma_2\} < -\epsilon.
\end{equation}

Using the gradient estimate of $\psi$ in \eqref{Bo-def}, we have
\begin{align}
		&
		\Big|
		\lambda^{\frac{1}{2}}
		\int_{B_{R_0}} U(y)^{\frac{4}{3}} Z_{1}(y)
		\big[
		\psi(\lambda y+\xi, t) - \psi(0, t)
		\big] \rmd y
		\Big|
		\lesssim 
		\lambda^{\frac{1}{2}}
		\int_{B_{R_0}} \langle y \rangle^{-8}
		t^{-2+\frac{\gamma_1}{2}} (\ln t)^{\frac{3}{2} + \frac{1}{100}} R_0^6 R^{-2} (\lambda |y| + |\xi|) \rmd y
        \notag
		\\
		\lesssim \ &
		t^{1-\gamma_1} (\ln t)^{\frac{3}{2}+ \frac{1}{100}} R_0^6 R^{-2}
		\lesssim 
		t^{-\epsilon} t^{\frac{5}{2}- \frac{3 \gamma_1}{2} - \frac{\gamma_2}{2}},
        \label{qd26Apr12-7}
\end{align}
where for the last step, we require
\begin{equation}
\epsilon < \frac{3}{2} + 2\beta - \frac{1}{2} (\gamma_1 + \gamma_2).
\end{equation}

Recall $\mathcal{H}_{1,1}^{\sharp}$ defined in \eqref{qd26Jan3-1}. Using the formulae of $U, Z_{6}, \Upsilon_{1,1}$, we have
\begin{align*}
		&
		\mathcal{H}_{1,1}^{\sharp}(r,t) 
        = \int_{ S^{4} } \mathcal{H}(r \theta, t) \Upsilon_{1,1}(\theta) \rmd \theta 
        \\
        = \ & \int_{ S^{4} } \Big[
		\dot{\lambda} \lambda Z_{6}(y)
		+\lambda \dot{\xi} \cdot (\nabla U )(y) 
		+\frac{7}{3}\lambda^{\frac{3}{2}} U(y)^{\frac{4}{3}}
		\big(
		\Psi(\lambda y+\xi,t)
		+
		\psi(\lambda y+\xi, t)
		\big)
		\Big] \Big|_{y = r\theta} \Upsilon_{1,1}(\theta) \rmd \theta
		\\
		= \ & \int_{ S^{4} } \Big\{
		\lambda \dot{\xi}_1 \partial_{y_1} U(y)
		+ \frac{7}{3}\lambda^{\frac{3}{2}} U(y)^{\frac{4}{3}}
		\big[
		\Psi(\lambda y+\xi,t)
		-
		\Psi(0,t)
		-
		(\lambda y+\xi) \cdot \nabla \Psi(0,t)
		+
		(\lambda y+\xi) \cdot \nabla \Psi(0,t)
		\\
		&
		+
		\psi(\lambda y+\xi, t)
		-
		\psi(0, t)
		\big]
		\Big\} \Big|_{y = r\theta} \Upsilon_{1,1}(\theta) \rmd \theta
		\\
		= \ & \int_{ S^{4} } \Big\{
		\lambda \dot{\xi}_1 \partial_{y_1} U(y)
		+ \frac{7}{3}\lambda^{\frac{3}{2}} U(y)^{\frac{4}{3}}
		\big[
		\Psi(\lambda y+\xi,t)
		-
		\Psi(0,t)
		-
		(\lambda y+\xi) \cdot \nabla \Psi(0,t)
		+
		\lambda y_1 \partial_{x_1} \Psi(0,t)
		\\
		&
		+
		\psi(\lambda y+\xi, t)
		-
		\psi(0, t)
		\big]
		\Big\} \Big|_{y = r\theta} \Upsilon_{1,1}(\theta) \rmd \theta.
\end{align*}
Therein,
\begin{align}
    &
    \Big|
	\int_{ S^{4} } \Big(
	\lambda \dot{\xi}_1 \partial_{y_1} U(y)
	\Big) \Big|_{y = r\theta} \Upsilon_{1,1}(\theta) \rmd \theta \Big|
	\lesssim
t^{\frac{9}{2} - \frac{5}{2} \gamma_1 - \frac{1}{2} \gamma_2 } \langle r \rangle^{-4},
\notag
\\
		& 
		\Big|
		\int_{ S^{4} } \Big\{ \lambda^{\frac{3}{2}} U(y)^{\frac{4}{3}}
		\big[
		\Psi(\lambda y+\xi,t)
		-
		\Psi(0,t)
		-
		(\lambda y+\xi) \cdot \nabla \Psi(0,t)
		\big]
		\Big\} \Big|_{y = r\theta} \Upsilon_{1,1}(\theta) \rmd \theta
		\Big|
        \notag
		\\
		\stackrel{\eqref{qd25Dec30-1}}{\lesssim} \ &  \lambda^{\frac{3}{2}} \langle r \rangle^{-4}
        (\lambda^2 r^2+ |\xi|^2) t^{-1-\frac{1}{2}\min\{ \gamma_1, \gamma_2\} }
		\lesssim
		\lambda^{\frac{7}{2}} \langle r \rangle^{-2} t^{-1-\frac{1}{2}\min\{ \gamma_1, \gamma_2\} }
		\sim
		t^{6-\frac{7}{2}\gamma_1 -\frac{1}{2}\min\{ \gamma_1, \gamma_2\} } \langle r \rangle^{-2},
\label{move-26Sep5-4}
\\
& \Big| \int_{ S^{4} } \Big( \lambda^{\frac{3}{2}} U(y)^{\frac{4}{3}}
		\lambda y_1 \partial_{x_1} \Psi(0,t)
		\Big) \Big|_{y = r\theta} \Upsilon_{1,1}(\theta) \rmd \theta \Big|
		\stackrel{\eqref{partial-Psi0t}}{\lesssim} t^{\frac{9}{2} - \frac{5}{2} \gamma_1 - \frac{1}{2} \gamma_2} \langle r \rangle^{-3},
\notag
\\
& 
\Big|
\int_{ S^{4} } \Big\{ \lambda^{\frac{3}{2}} U(y)^{\frac{4}{3}}
\big[
\psi(\lambda y+\xi, t)
-
\psi(0, t)
\big]
\Big\} \Big|_{y = r\theta} \Upsilon_{1,1}(\theta) \rmd \theta \Big|
\notag
\\
\stackrel{\eqref{Bo-def}}{\lesssim} \ & \lambda^{\frac{3}{2}} \langle r \rangle^{-4}
t^{-2+\frac{\gamma_1}{2}} (\ln t)^{\frac{3}{2} + \frac{1}{100}} R_0^6 R^{-2} (\lambda r + |\xi|)
\lesssim
t^{3-2\gamma_1} (\ln t)^{\frac{3}{2} + \frac{1}{100}} R_0^6 R^{-2} \langle r \rangle^{-3}.
\notag
\end{align}
Notice that for $r < R_{*}(\tau(t)) \sim t^{\beta}$, to make
\begin{equation}
t^{6-\frac{7}{2}\gamma_1 -\frac{1}{2}\min\{ \gamma_1, \gamma_2\} } \langle r \rangle^{-2} \lesssim t^{\frac{9}{2} - \frac{5}{2} \gamma_1 - \frac{1}{2} \gamma_2} \langle r \rangle^{-3},
	\mbox{ it suffices to make }
	\frac{3}{2} + \beta -\gamma_1 + \frac{1}{2} \gamma_2 - \frac{1}{2} \min\{ \gamma_1, \gamma_2\} \le 0.
\end{equation}
To make
\begin{equation}
t^{3-2\gamma_1} (\ln t)^{\frac{3}{2} + \frac{1}{100}} R_0^6 R^{-2} \langle r \rangle^{-3} \lesssim 
t^{\frac{9}{2} - \frac{5}{2} \gamma_1 - \frac{1}{2} \gamma_2} \langle r \rangle^{-3},
\mbox{ it suffices to make }
0 < \frac{3}{2} + 2\beta -\frac{1}{2}(\gamma_1 + \gamma_2).
\end{equation}
In sum,
\begin{equation}\label{qd25Jan3-2}
\begin{aligned}
&
\tilde{\mathcal{H}}_{1,1}(y,t) 
\stackrel{\eqref{qd26Jan3-1}}{=}
\mathcal{H}_{1,1}^{\sharp}(|y|,t) \Upsilon_{1,1}(y/|y|),
\\
& 
	|\tilde{\mathcal{H}}_{1,1}(y,t)|
    \lesssim
     |\mathcal{H}_{1,1}^{\sharp}(|y|,t)|
	\lesssim t^{\frac{9}{2} - \frac{5}{2} \gamma_1 - \frac{1}{2} \gamma_2} \langle y \rangle^{-3}
	\stackrel{\eqref{tau-est}}{\sim}
    (\tau(t))^{-\frac{5\gamma_1 + \gamma_2 - 9}{4\gamma_1 - 6}}
    \langle y \rangle^{-3}.
\end{aligned}
\end{equation}
For $\tau \in (\tau_0, C_{tm} T^{2\gamma_1 -3}]$, $y\in \mathbb{R}^5$, by \eqref{move-26Aug21-5},
\begin{equation*}
\begin{aligned}
&
\tilde{\mathcal{J}}_{1,1}(y,\tau) = 
\mathcal{J}_{1,1}^{\sharp}(|y|,\tau) \Upsilon_{1,1}(y/|y|),
\\
&
	|\tilde{\mathcal{J}}_{1,1}(y,\tau)| \lesssim \tau^{-\frac{5\gamma_1 + \gamma_2 - 9}{4\gamma_1 - 6}} \langle y \rangle^{-3} \1_{\tau_0 < \tau \le \tau(T)}
	+
	(\tau(T))^{-\frac{5\gamma_1 + \gamma_2 - 9}{4\gamma_1 - 6}} \langle y \rangle^{-3} \1_{ \tau(T) < \tau \le C_{tm} T^{2\gamma_1 -3} }
    \\
\stackrel{\eqref{tau-est}}{\sim} \ &
\tau^{-\frac{5\gamma_1 + \gamma_2 - 9}{4\gamma_1 - 6}} \langle y \rangle^{-3} \1_{\tau_0 < \tau \le C_{tm} T^{2\gamma_1 -3}}.
\end{aligned}
\end{equation*}
By Proposition \ref{regluing-mode1-prop} (with $v= \tau^{-\frac{5\gamma_1 + \gamma_2 - 9}{4\gamma_1 - 6}}$, $a=1$, $R_0=\ln t_0$, $\epsilon_0 = \zeta$) and Remark \ref{qd26Apr30-1-rmk}, under the assumption
\begin{equation}
\zeta \in (0,1),
\quad
\frac{2 \beta}{2\gamma_1 -3} <1,
\quad
t_0 \gg 1,
\end{equation}
\eqref{qd26Apr12-5} and \eqref{qd26Apr12-6} with $i=1$ are well-defined, and it holds that
\begin{equation}\label{qd26Apr3-6}
\begin{aligned}
&
\mbox{$\mathcal{T}_{1,1}^{\rm{in}}(y,\tau) \in C^{1+\sigma,\frac{1+\sigma}{2}} \big( \{ (y,\tau) \mid \tau \in (\tau_0, C_{tm} T^{2\gamma_1 -3}], y\in B_{R_{*}(\tau)} \} \big)$ for some $\sigma \in (0,1)$};
\\
&
\mbox{$\mathcal{T}_{1,1}^{\rm{in}}(y,\tau)$ is even with respect to the $j$-th component of $y$, $j \in \overline{2,5}$};
\\
&
	\langle y \rangle |\nabla \mathcal{T}_{1, 1}^{\rm{in}} | + | \mathcal{T}_{1, 1}^{\rm{in}} | \lesssim R_0^{6} \tau^{-\frac{5\gamma_1 + \gamma_2 - 9}{4\gamma_1 - 6}} \langle y \rangle^{-1};
	\quad
	\mathcal{T}_{1, 1}^{\rm{in}}(y, \tau_0) =0;
\end{aligned}
\end{equation}
\begin{equation}\label{qd26Apr29-3}
| \lambda^{-1} \varrho_{1,1}^{*}(\tau(t)) |
\lesssim
\lambda^{-1} (\tau(t))^{-\frac{5\gamma_1 + \gamma_2 - 9}{4\gamma_1 - 6}} R_0^{-\zeta}
\sim R_0^{-\zeta} t^{\frac{5}{2} -\frac{3 \gamma_1}{2}  -\frac{\gamma_2}{2}};
\quad
\mbox{$\varrho_{1,1}^{*}(\tau(t))$ is H\"older continuous for $t \in (t_0, T]$.}
\end{equation}

Combining \eqref{qd26Apr29-1}, \eqref{qd26Apr29-2}, \eqref{qd26Apr12-7}, and \eqref{qd26Apr29-3}, we get 
\begin{equation}\label{vec-S-est}
	|\mathcal{S}_1| \lesssim R_0^{-\frac{\zeta}{2}} t^{\frac{5}{2}- \frac{3 \gamma_1}{2} - \frac{\gamma_2}{2}}.
\end{equation}

{\textbf{Step $5$.}}
By \eqref{qd26Jan3-6}, \eqref{move-26Aug21-7}, \eqref{vec-S-est}; the H\"older continuity of $\varrho_{0,1}^{*}(\tau(t))$ in \eqref{qd26Apr28-3}, $\varrho_{1,1}^{*}(\tau(t))$ in \eqref{qd26Apr29-3}, and the gradient and H\"older estimates of $\psi\in B_{\rm o}$; the definition of $( \mathcal{S}_{6} ,
\vec{\mathcal{S} } )$ in \eqref{mu1-xi-sys}, it follows that
$ \mathcal{S}_{6} \times
\vec{\mathcal{S} } $ is a compact mapping from $B_{\rm o} \times B_{\dot{\lambda}_1} \times B_{\dot{\xi}^{[1]}}$ into $B_{\dot{\lambda}_1} \times B_{\dot{\xi}^{[1]}}$, and we get the qualitative H\"older continuity of $( \mathcal{S}_{6} ,
\vec{\mathcal{S} } )$ in Lemma \ref{mu1-xi-lem}.

{\textbf{Step $6$ - for higher modes.}} By \eqref{move-26Aug22-2}, \eqref{qd25Jan3-3}, \eqref{move-26Aug22-1}, and \eqref{qd25Jan3-2},
\begin{equation}
\tilde{\mathcal{H}}_{\perp}(y,t) 
\stackrel{\eqref{qd26Jan3-1}}{=} \mathcal{H} - \tilde{\mathcal{H}}_{0,1} - \sum_{i=1}^{5} \tilde{\mathcal{H}}_{1,i}
= \mathcal{H} - \tilde{\mathcal{H}}_{0,1} -\tilde{\mathcal{H}}_{1,1}
\end{equation}
is even for the $j$-th component of $y$, $j \in \overline{2,5}$. So is $\tilde{\mathcal{J}}_{\perp}(y,\tau)$ given in \eqref{move-26Aug21-5}.

By $\gamma_1, \gamma_2 \in [0,5)$, \eqref{move-26Aug19-1}, and $\|\psi\|_{\rm o} \le 1$, we have
\begin{align}
&
|\mathcal{H}|
= \big| \dot{\lambda} \lambda Z_{6}(y)
	+\lambda \dot{\xi} \cdot (\nabla U )(y) 
	+\frac{7}{3}\lambda^{\frac{3}{2}} U(y)^{\frac{4}{3}}
	(\Psi + \psi)(\lambda y+\xi, t) \big|
    \notag
\\
\lesssim \ & t^{3-2\gamma_1} \langle y \rangle^{-3} + t^{3-\frac{3}{2} \gamma_1} \langle y \rangle^{-4} \big[ t^{-\frac{\gamma_1}{2}} + t^{-\frac{\gamma_2}{2} - \frac{1}{2}} |\lambda y+\xi| + t^{-\frac{\gamma_1}{2} } (\ln t)^2 R_0^6 R^{-1} \big]
\label{move-26Aug21-3}
\\
\lesssim \ & t^{3-2\gamma_1} \langle y \rangle^{-3} + t^{\frac{9}{2}-\frac{5}{2} \gamma_1 - \frac{1}{2}\gamma_2 } \langle y \rangle^{-3}
\stackrel{\gamma_1 + \gamma_2 \ge 3}{\sim}
t^{3-2\gamma_1} \langle y \rangle^{-3}.
\notag
\end{align}
Combining \eqref{move-26Aug21-3}, \eqref{qd25Jan3-3}, and \eqref{qd25Jan3-2}, we have
\begin{equation}
|\tilde{\mathcal{H}}_{\perp}|
\stackrel{\gamma_1 + \gamma_2 \ge 3}{\lesssim} t^{3-2\gamma_1} \langle y \rangle^{-3}
\stackrel{\eqref{tau-est}}{\sim} (\tau(t))^{-1} \langle y \rangle^{-3}.
\end{equation}
For $\tau \in (\tau_0, C_{tm} T^{2\gamma_1 -3}]$, $y\in \mathbb{R}^5$,
\begin{equation*}	|\tilde{\mathcal{J}}_{\perp}(y,\tau)| \lesssim \tau^{-1} \langle y \rangle^{-3} \1_{\tau_0 < \tau \le \tau(T)}
	+
	(\tau(T))^{-1} \langle y \rangle^{-3} \1_{\tau(T) < \tau \le C_{tm} T^{2\gamma_1 -3}}
\sim 
\tau^{-1} \langle y \rangle^{-3} \1_{\tau_0 < \tau \le C_{tm} T^{2\gamma_1 -3}}.
\end{equation*}
By Proposition \ref{regluing-higher-mode-prop} (with $v=\tau^{-1}$, $a=1$, a constant $R_0 = C_{0} \gg 1$), under the assumption 
\begin{equation}
\frac{2 \beta}{2\gamma_1 -3} <1,
\quad
t_0\gg 1,
\end{equation}
\eqref{move-26Sep6-1} is well-defined and we have
\begin{equation}\label{qd26Apr3-7}
\begin{aligned}
&
\mbox{$\mathcal{T}_{\perp}^{\rm{in}}(y,\tau) \in C^{1+\sigma,\frac{1+\sigma}{2}} \big( \{ (y,\tau) \mid \tau \in (\tau_0, C_{tm} T^{2\gamma_1 -3}], y\in B_{R_{*}(\tau)} \} \big)$ for some $\sigma \in (0,1)$};
\\
&
\mbox{$\mathcal{T}_{\perp}^{\rm{in}}(y,\tau)$ is even with respect to the $j$-th component of $y$, $j \in \overline{2,5}$};
\\
&
	\langle y \rangle |\nabla  \mathcal{T}_{\perp}^{\rm{in}} | + | \mathcal{T}_{\perp}^{\rm{in}} | \lesssim 
	\tau^{-1} \langle y \rangle^{-1};
	\quad \mathcal{T}_{\perp}^{\rm{in}}(y,\tau_0) = 0.
\end{aligned}
\end{equation}

{\textbf{Step $7$.}}
For $\mathcal{T}^{\rm{in}}$ defined in \eqref{Tin-def},
combining \eqref{qd26Apr3-5}, \eqref{move-26Aug22-5}, \eqref{qd26Apr3-6}, and \eqref{qd26Apr3-7}, we have
\begin{equation}\label{T*-est}
\begin{aligned}
&
\mbox{$\mathcal{T}^{\rm{in}}(y,\tau) \in C^{1+\sigma,\frac{1+\sigma}{2}} \big( \{ (y,\tau) \mid \tau \in (\tau_0, C_{tm} T^{2\gamma_1 -3}], y\in B_{R_{*}(\tau)} \} \big)$
for some $\sigma \in (0,1)$};
\\
&
\mbox{$\mathcal{T}^{\rm{in}}(y,\tau)$ is even with respect to the $j$-th component of $y$, $j \in \overline{2,5}$;}
\\
&
	\langle y \rangle |\nabla \mathcal{T}^{\rm{in}}| + |\mathcal{T}^{\rm{in}}| 
    \stackrel{\gamma_1 > \frac{3}{2}, \gamma_1 + \gamma_2 \ge 3}{\lesssim} R_0^{6} \tau^{-1} \langle y \rangle^{-1}
    \stackrel{\eqref{tau-est}}{\sim} (t(\tau))^{3-2\gamma_1} R_0^{6} \langle y \rangle^{-1};
\\
&
\mathcal{T}^{\rm{in}}(y, \tau_0) = g_0 \eta(2 y/R_0) Z_0(y)
    \mbox{ \ with \ }
	|g_0| \lesssim R_0 \tau_0^{-1}.
\end{aligned}
\end{equation}
Thus, for $t_0 \gg 1$, $\mathcal{T}^{\rm{in}}[\mathcal{J}](y,\tau(t))$ is a compact mapping from $B_{\rm o} \times B_{\dot{\lambda}_1} \times B_{\dot{\xi}^{[1]}}$ into $B_{\rm{in}}$, and we get the qualitative H\"older continuity of $\mathcal{T}^{\rm{in}}[\mathcal{J}](y,\tau(t))$, $\nabla_y \mathcal{T}^{\rm{in}}[\mathcal{J}](y,\tau(t))$ in Lemma \ref{mu1-xi-lem}. Under the assumption \eqref{qd26Jan3-7}, we can take $\epsilon \ll 1$ to meet all parameter restrictions in this proof.
\end{proof}

\section{Continuity argument for the mappings}\label{conti-Sec}

We will solve $(\psi,\phi, \dot{\lambda}_1, \dot{\xi}^{[1]})$ in the space
\begin{equation}\label{Xspace-def}
	\mathcal{X} := B_{\rm o} \times B_{\rm{in}} \times B_{\dot{\lambda}_1} \times B_{\dot{\xi}^{[1]}}
\end{equation}
equipped with the norm 
\begin{equation}
\| (\psi,\phi, \dot{\lambda}_1, \dot{\xi}^{[1]} ) \|_{\mathcal{X}} := \max\{ \| \psi \|_{B_{\rm o}}, \|\phi\|_{\rm{in}}, \|\dot{\lambda}_1\|_{\dot{\lambda}_1}, \| \dot{\xi}^{[1]} \|_{\dot{\xi}^{[1]}} \}.
\end{equation}

Using the relationship $\lambda_1 = \lambda_1[\dot{\lambda}_1](t), 
\xi^{[1]} = \xi^{[1]}[\dot{\xi}^{[1]}](t)$ in \eqref{mu1-xi-sys}, we rewrite the dependence of $\mathcal{G} = \mathcal{G}[\psi,\phi,\dot{\lambda}_{1},\dot{\xi}_{1}]$ in \eqref{g-n=5}. The dependence of other quantities is rewritten similarly. 

We will give the continuity argument for mappings in the next lemma. The main difficulty is the continuity argument for the mappings for the inner problem and $\varrho_{0,1}^{*}(\tau(t))$, $\varrho_{1,1}^{*}(\tau(t))$ due to the $t, \tau$ transformation in \eqref{tau-def}. This transformation makes the domain of the original inner problem \eqref{inner-tau-eq} in the $\tau$ variable vary with the choice of $\dot{\lambda}_1$, so we cannot use the inner linear theory directly to obtain continuity of the mappings. This is why we extend the inner problem in Section \ref{inner-map-sec}. See the related argument in {\textbf{Step 2}} in the following proof.

\begin{lemma}\label{mapping-conti-lem}

Under the parameter restrictions \eqref{mov-26Aug11-1} and $\beta-\gamma_1<7$, then for $t_0$ sufficiently large, suppose a sequence $(\psi_i, \phi_i, \dot{\lambda}_{1,i}, \dot{\xi}^{[1,i]} )_{i\ge 1}$ converges to $(\psi_{\infty}, \phi_{\infty}, \dot{\lambda}_{1,\infty}, \dot{\xi}^{[1,\infty]} )$ in $\mathcal{X}$, that is,
\begin{equation}\label{move-26Sep6-2}
\lim_{i\to \infty} \| (\psi_i, \phi_i, \dot{\lambda}_{1,i}, \dot{\xi}^{[1,i]} ) - (\psi_{\infty}, \phi_{\infty}, \dot{\lambda}_{1,\infty}, \dot{\xi}^{[1,\infty]} ) \|_{\mathcal{X}} = 0,
\end{equation}
it holds that
\begin{equation}\label{out-map-converge}
\lim_{i\to \infty}
\big\|
\mathcal{T}_{\rm{o}} [ \mathcal{G} [\psi_{i},\phi_{i}, \dot{\lambda}_{1,i}, \dot{\xi}^{[1,i]}] \1_{|x|\le T^{9}, t\le T } ](x,t)
-
\mathcal{T}_{\rm{o}} [ \mathcal{G} [\psi_{\infty},\phi_{\infty}, \dot{\lambda}_{1,\infty}, \dot{\xi}^{[1,\infty]}] \1_{|x|\le T^{9}, t\le T} ](x,t) \big\|_{B_{\rm o}} = 0,
\end{equation}
\begin{equation}\label{S-conti} 
\lim_{i\to \infty}
\big(
\big\|
\mathcal{S}_6[\psi_{i}, \dot{\lambda}_{1,i}, \dot{\xi}^{[1,i]}](t)
-
\mathcal{S}_6[\psi_{\infty}, \dot{\lambda}_{1,\infty}, \dot{\xi}^{[1,\infty]}](t) \big\|_{\dot{\lambda}_1} +
\big\|
\vec{\mathcal{S}}[\psi_{i},\dot{\lambda}_{1,i},\dot{\xi}^{[1,i]}](t) 
-
\vec{\mathcal{S}}[\psi_{\infty},\dot{\lambda}_{1,\infty},\dot{\xi}^{[1,\infty]}](t) \big\|_{\dot{\xi}^{[1]}}
\big)
= 0, 
\end{equation}
\begin{equation}\label{calT-in-conti}
\lim_{i\to \infty}
\big\|
\mathcal{T}^{\rm{in}} [\mathcal{J}[\psi_{i}, \dot{\lambda}_{1,i}, \dot{\xi}^{[1,i]}] ](y,\tau[\dot{\lambda}_{1,i}](t)) 
-
\mathcal{T}^{\rm{in}} [ \mathcal{J} [\psi_{\infty}, \dot{\lambda}_{1,\infty}, \dot{\xi}^{[1,\infty]}] ](y,\tau[\dot{\lambda}_{1,\infty}](t)) \big\|_{\rm{in}} = 0.
\end{equation}
 
\end{lemma}

\begin{proof}

For brevity, 
\begin{equation}
\mbox{$f_{i} \rightsquigarrow f_{\infty}$ in $\Omega$ denotes $\lim\limits_{i\to \infty} \sup\limits_{z \in \Omega} |f_{i}(z) - f_{\infty}(z)| = 0$.}
\end{equation}
Similar to \eqref{qd23Dec28-1},
\begin{equation}\label{qd26Apr5-1}
\begin{aligned}
&
	| \lambda_{1,i} - \lambda_{1,\infty} |  
	\stackrel{\gamma_1<2}{\lesssim} R_0^{-\frac{\zeta}{2}} t^{2-\gamma_1} \| \dot{\lambda}_{1,i} - \dot{\lambda}_{1,\infty} \|_{\dot{\lambda}_1},
\\
& | \xi^{[1,i]} - \xi^{[1,\infty]} | \lesssim \| \dot{\xi}^{[1,i]} - \dot{\xi}^{[1,\infty]} \|_{\dot{\xi}^{[1]}}
    R_0^{-\frac{\zeta}{4}}
	\begin{cases} 
t^{\frac{7}{2}-\frac{3\gamma_1}{2} - \frac{\gamma_2}{2}}  
			& \mbox{ \ if \ } 3 \gamma_1 + \gamma_2 \ne 7
			\\
            \ln t 
			& \mbox{ \ if \ } 3 \gamma_1 + \gamma_2 = 7.
		\end{cases}
\end{aligned}
\end{equation}

Since $T<\infty$ and we restrict $B_{\rm{o}}$ in $\mathcal{D}_{\rm{o}}$, $B_{\rm{in}}$ in $\mathcal{D}_{\rm{in}}$, $B_{\dot{\lambda}_1}$ and  $B_{\dot{\xi}^{[1]}}$ in $[t_0, T]$, 
where for brevity, we denote
\begin{equation}
\mathcal{D}_{\rm{o}} := \overline{B_{T^9}} \times [t_0, T],
\quad
\mathcal{D}_{\rm{in}} := \{ (y,t)  \mid t\in [t_0,T], y\in \overline{B_{4 R}} \},
\end{equation}
then \eqref{move-26Sep6-2} implies
\begin{equation}\label{qd26May1-1}
\psi_{i} \rightsquigarrow \psi_{\infty} \mbox{ \ in \ } \mathcal{D}_{\rm{o}};
\quad
\phi_{i} \rightsquigarrow \phi_{\infty}, \ 
\nabla_{y}\phi_{i} \rightsquigarrow \nabla_{y}  \phi_{\infty} \mbox{ \ in \ } \mathcal{D}_{\rm{in}};
\quad
\dot{\lambda}_{1,i} \rightsquigarrow \dot{\lambda}_{1,\infty}, 
\ 
\dot{\xi}^{[1,i]} \rightsquigarrow \dot{\xi}^{[1,\infty]} \mbox{ \ in \ } [t_0, T].
\end{equation}

The {\textbf{key point}} is that in bounded domains, $\mathcal{D}_{\rm{o}}$, $\mathcal{D}_{\rm{in}}$, $[t_0, T]$, the weights in the weighted norms $ \| \cdot \|_{B_{\rm o}}, \|\cdot\|_{\rm{in}}, \|\cdot\|_{\dot{\lambda}_1}, \| \cdot \|_{\dot{\xi}^{[1]}}$ have positive lower and upper bounds, which could depend on $t_0, T$. Then the convergence in these weighted norms is equivalent to uniform convergence in the corresponding bounded domains.

The following basic lemma will be used frequently in the continuity argument but is not stated.
\begin{lemma}
	
	\begin{enumerate}
		
		\item 
		Given a set $\Omega \subset \mathbb{R}^n$, two sequences of functions  $(f_{i})_{i\ge 1}, (g_{i})_{i\ge 1}$ defined in $\Omega$ satisfying $\lim\limits_{i\to\infty} \sup\limits_{x\in \Omega} | f_{i}(x)  - f_{\infty}(x) | = 0$ and $\lim\limits_{i\to\infty} \sup\limits_{x\in \Omega} | g_{i}(x)  - g_{\infty}(x) | = 0$ for some functions $f_{\infty}, g_{\infty}$ defined in $\Omega$,  and $(f_{i})_{i\ge 1}, (g_{i})_{i\ge 1}$ are uniformly bounded in $\Omega$. Then $\lim\limits_{i\to\infty} \sup\limits_{x\in \Omega} | (f_{i} g_{i})(x)  - (f_{\infty} g_{\infty})(x) | = 0$.

		Under the additional assumption that $(|g_i|)_{i \ge 1}$ have a uniform positive lower bound in $\Omega$, then $\lim\limits_{i\to\infty} \sup\limits_{x\in \Omega} | ( \frac{f_{i}}{g_{i}} )(x)  - ( \frac{f_{\infty}}{g_{\infty}} )(x) | = 0$.

		\item 
		Suppose that $\Omega_{2} \subset \mathbb{R}^{n_2}$ is a compact set and $\Omega_{1} \subset \mathbb{R}^{n_1}$, a sequence of functions $g_{i} : \Omega_1 \to \Omega_2$, $i\ge 1$, and $f$ is a continuous function in $\Omega_2$. If there exists a function $g_{\infty} : \Omega_1 \to \Omega_2$ such that $\lim\limits_{i \to \infty} \sup\limits_{x \in \Omega_1} |g_i(x) - g_{\infty}(x)| = 0$, then $\lim\limits_{i \to \infty} \sup\limits_{x \in \Omega_1} |f(g_i(x)) - f(g_{\infty}(x))| = 0$.

		If we only change the above assumption about $\Omega_2$ to that $ \Omega_2 \subset \mathbb{R}^{n_2} $ is an unbounded closed set, under the additional assumption $\lim\limits_{x\in \Omega_2, |x|\to \infty} f(x) = 0$, then the conclusion still holds. 
		
	\end{enumerate}
	
\end{lemma}

\begin{proof}
	
	(1). Obviously, $f_{\infty}, g_{\infty}$ are bounded functions. Since $
	|f_{i} g_{i} - f_{\infty} g_{\infty}| \le 
	|f_{i} - f_{\infty}| |g_{i}| + |f_{\infty}| |g_i - g_{\infty}| $, we get the first result. Under the additional assumption, $|g_{\infty}|$ has a positive lower bound in $\Omega$. Since $\frac{f_i}{g_i} - \frac{f_{\infty}}{g_{\infty}} = \frac{f_i - f_{\infty}}{g_i} + \frac{f_{\infty} (g_{\infty} - g_{i})}{g_i g_{\infty}}$, we get the second result.

	(2). Both results are deduced from the fact that $f$ is uniformly continuous in $\Omega_2$.
\end{proof}

Hereafter in this proof, we always assume $t_0 \gg 1$ and $i\gg 1$.
	
\textbf{Step 1 - Proof of \eqref{out-map-converge}.}
Claim:
\begin{equation}\label{cal-G-converge}
\mathcal{G} [\psi_{i},\phi_{i}, \dot{\lambda}_{1,i}, \dot{\xi}^{[1,i]}]
\rightsquigarrow
\mathcal{G} [\psi_{\infty},\phi_{\infty}, \dot{\lambda}_{1,\infty}, \dot{\xi}^{[1,\infty]}]
\mbox{ \ in \ } \mathcal{D}_{\rm{o}}.
\end{equation}

Indeed, using $T<\infty$ and the similar scaling argument in Lemma \ref{outer-exist}, we have that \eqref{cal-G-converge} implies \eqref{out-map-converge}.

We will check $\mathcal{G}$ in \eqref{g-n=5} term by term to conclude \eqref{cal-G-converge}. Readers may focus on the typical analysis to deduce \eqref{qd26May1-2}.

Same as \eqref{move-26Sep6-5},
$\lambda^{-2}
U(y)^{\frac{4}{3}}
\big(
\eta(\tilde{y})^{\frac{4}{3}} - 1 \big)
\lambda^{-\frac{3}{2}}\phi(y,t) \eta_{R}(y) = 0$ for $(\psi,\phi, \dot{\lambda}_1, \dot{\xi}^{[1]} ) \in \mathcal{X}$.

{\textbf{For $\dot{\lambda}\lambda^{-\frac{5}{2}} Z_{6}(y) = ( \dot{\lambda}_0 + \dot{\lambda}_1 ) (\lambda_0 + \lambda_1)^{-\frac{5}{2}}  Z_{6}(\frac{x-(\xi^{[0]} + \xi^{[1]})}{\lambda_0 + \lambda_1})$.}}
In $[t_0, T]$,
$\dot{\lambda}_0 + \dot{\lambda}_{1,i} \rightsquigarrow  \dot{\lambda}_0 + \dot{\lambda}_{1,\infty} $,
$\lambda_0 + \lambda_{1,i} \rightsquigarrow \lambda_0 + \lambda_{1,\infty}$, $(|\lambda_0 + \lambda_{1,i}|)_{i\ge 1}$ have a uniform positive lower bound.
Thus, in $\mathcal{D}_{\rm{o}}$, $ \frac{x-(\xi^{[0]} + \xi^{[1,i]})}{\lambda_0 + \lambda_{1,i}} \rightsquigarrow \frac{x-(\xi^{[0]} + \xi^{[1,\infty]})}{\lambda_0 + \lambda_{1,\infty}} $, and 
$\big| \frac{x-(\xi^{[0]} + \xi^{[1,i]})}{\lambda_0 + \lambda_{1,i}} \big|$ has a uniform upper bound, which depends on $T$ and independent of $i$. Then in $\mathcal{D}_{\rm{o}}$,
\begin{equation}\label{move-26Sep6-4}
( \dot{\lambda}_0 + \dot{\lambda}_{1,i} ) (\lambda_0 + \lambda_{1,i})^{-\frac{5}{2}}  Z_{6}\Big( \frac{x-(\xi^{[0]} + \xi^{[1,i]})}{\lambda_0 + \lambda_{1,i}} \Big) 
\rightsquigarrow 
( \dot{\lambda}_0 + \dot{\lambda}_{1,\infty} ) (\lambda_0 + \lambda_{1,\infty})^{-\frac{5}{2}}  Z_{6}\Big( \frac{x-(\xi^{[0]} + \xi^{[1,\infty]})}{\lambda_0 + \lambda_{1,\infty}} \Big).
\end{equation}

The continuity argument about the following terms in $\mathcal{D}_{\rm{o}}$ are similar to \eqref{move-26Sep6-4}:
\begin{equation*}
\begin{aligned}
&
(\lambda_0 + \lambda_{1})^{-\frac{5}{2}} (\dot{\xi}^{[0]} + \dot{\xi}^{[1]}) \cdot 
(\nabla U)(\frac{x-(\xi^{[0]} + \xi^{[1]})}{\lambda_0 + \lambda_{1}});
\quad 
\eta(\frac{x-(\xi^{[0]} + \xi^{[1]})}{\sqrt{t}} ) - \eta(\frac{x-(\xi^{[0]} + \xi^{[1]})}{(\lambda_0 + \lambda_1) R});
\\
&
\mathcal{E}_{U}^{\rm{cut}}[\dot{\lambda}_{1}, \dot{\xi}^{[1]}]
	= (\lambda_0 + \lambda_1)^{-\frac{3}{2}} U\Big(\frac{x-(\xi^{[0]} + \xi^{[1]})}{\lambda_0 + \lambda_1} \Big) 
	\Big[ 2^{-1} t^{-1} \frac{x-(\xi^{[0]} + \xi^{[1]})}{\sqrt{t}}  + t^{-\frac{1}{2}} (\dot{\xi}^{[0]} + \dot{\xi}^{[1]}) \Big] 
    \\
    &
    \quad
    \cdot ( \nabla \eta ) \Big( \frac{x-(\xi^{[0]} + \xi^{[1]})}{\sqrt{t}} \Big)
	+
	2 (\lambda_0 + \lambda_1)^{-\frac{5}{2}} t^{-\frac{1}{2}} 
	( \nabla U )\Big(\frac{x-(\xi^{[0]} + \xi^{[1]})}{\lambda_0 + \lambda_1} \Big) \cdot 
	( \nabla \eta )\Big( \frac{x-(\xi^{[0]} + \xi^{[1]})}{\sqrt{t}} \Big)
	\\
	& \quad
	+
	(\lambda_0 + \lambda_1)^{-\frac{3}{2}} t^{-1} U\Big(\frac{x-(\xi^{[0]} + \xi^{[1]})}{\lambda_0 + \lambda_1} \Big) ( \Delta \eta )\Big( \frac{x-(\xi^{[0]} + \xi^{[1]})}{\sqrt{t}} \Big);
\\
&
(\lambda_0 + \lambda_{1})^{-\frac{7}{2}} U^{\frac{7}{3}}\Big(\frac{x-(\xi^{[0]} + \xi^{[1]})}{\lambda_0 + \lambda_1} \Big)
\Big[ \eta^{\frac{7}{3}}\Big( \frac{x-(\xi^{[0]} + \xi^{[1]})}{\sqrt{t}} \Big) 
-
\eta\Big( \frac{x-(\xi^{[0]} + \xi^{[1]})}{\sqrt{t}} \Big)
\Big];
\\
&
(\lambda_{0} + \lambda_{1})^{-2}
U^{\frac{4}{3}}
\Big( \frac{x-( \xi^{[0]} + \xi^{[1]} ) }{\lambda_{0} + \lambda_{1} } \Big)
\Big[
\eta^{\frac{4}{3}}\Big( \frac{x-( \xi^{[0]} + \xi^{[1]} ) }{\sqrt{t}} \Big) - \eta\Big( \frac{x-( \xi^{[0]} + \xi^{[1]} ) }{ (\lambda_{0} + \lambda_{1})R } \Big) \Big] ( \Psi + \psi).
\end{aligned}
\end{equation*}

{\textbf{For $\Lambda_1$.}}
\begin{equation*}
	\begin{aligned}  
		& \Lambda_1[\phi, \dot{\lambda}_{1}, \dot{\xi}^{[1]} ]
= (\lambda_0 + \lambda_{1} )^{-\frac{7}{2}} R^{-2} \phi\Big( \frac{x- (\xi^{[0]} + \xi^{[1]})}{\lambda_0 + \lambda_1}, t \Big)  ( \Delta \eta )\Big( \frac{x- (\xi^{[0]} + \xi^{[1]})}{(\lambda_0 + \lambda_1) R} \Big)
\\
&
		+ 
		2 (\lambda_0 + \lambda_{1} )^{-\frac{7}{2}}  R^{-1}  (\nabla_y \phi)\Big( \frac{x- (\xi^{[0]} + \xi^{[1]})}{\lambda_0 + \lambda_1}, t \Big) \cdot (\nabla \eta )\Big( \frac{x- (\xi^{[0]} + \xi^{[1]})}{(\lambda_0 + \lambda_1) R} \Big) 
		\\
		&
		+
		(\lambda_0 + \lambda_{1} )^{-\frac{3}{2}} \phi\Big( \frac{x- (\xi^{[0]} + \xi^{[1]})}{\lambda_0 + \lambda_1}, t \Big) 
		( \nabla \eta )\Big( \frac{x- (\xi^{[0]} + \xi^{[1]})}{(\lambda_0 + \lambda_1) R} \Big) 
        \\
        &
        \cdot \Big[ \frac{\dot{\xi}^{[0]} + \dot{\xi}^{[1]}}{(\lambda_0 + \lambda_{1} ) R}
		+ \frac{x- (\xi^{[0]} + \xi^{[1]})}{(\lambda_0 + \lambda_1) R} \Big( \frac{ \dot{\lambda}_0 + \dot{\lambda}_1 }{\lambda_0 + \lambda_1} + \frac{\dot{R}}{R} \Big) \Big].
	\end{aligned}
\end{equation*}
Therein, for the second part,
{\small
\begin{align}
& 
(\lambda_0 + \lambda_{1,i} )^{-\frac{7}{2}}  R^{-1}  (\nabla_y \phi_{i})\Big( \frac{x- (\xi^{[0]} + \xi^{[1,i]})}{\lambda_0 + \lambda_{1,i}}, t \Big) \cdot (\nabla \eta )\Big( \frac{x- (\xi^{[0]} + \xi^{[1,i]})}{(\lambda_0 + \lambda_{1,i}) R} \Big) 
\notag
\\
& - (\lambda_0 + \lambda_{1,\infty} )^{-\frac{7}{2}}  R^{-1}  (\nabla_y \phi_{\infty})\Big( \frac{x- (\xi^{[0]} + \xi^{[1,\infty]})}{\lambda_0 + \lambda_{1,\infty}}, t \Big) \cdot (\nabla \eta )\Big( \frac{x- (\xi^{[0]} + \xi^{[1,\infty]})}{(\lambda_0 + \lambda_{1,\infty}) R} \Big)
\notag
\\
= \ &
\big[
(\lambda_0 + \lambda_{1,i} )^{-\frac{7}{2}}
- (\lambda_0 + \lambda_{1,\infty} )^{-\frac{7}{2}}
\big]  R^{-1}  (\nabla_y \phi_{i})\Big( \frac{x- (\xi^{[0]} + \xi^{[1,i]})}{\lambda_0 + \lambda_{1,i}}, t \Big) \cdot (\nabla \eta )\Big( \frac{x- (\xi^{[0]} + \xi^{[1,i]})}{(\lambda_0 + \lambda_{1,i}) R} \Big)
\label{qd26May1-4}
\\
& + (\lambda_0 + \lambda_{1,\infty} )^{-\frac{7}{2}}  R^{-1}  
\Big[
(\nabla_y \phi_{i})\Big( \frac{x- (\xi^{[0]} + \xi^{[1,i]})}{\lambda_0 + \lambda_{1,i}}, t \Big) - (\nabla_y \phi_{\infty})\Big( \frac{x- (\xi^{[0]} + \xi^{[1,i]})}{\lambda_0 + \lambda_{1,i}}, t \Big) \Big] \cdot (\nabla \eta )\Big( \frac{x- (\xi^{[0]} + \xi^{[1,i]})}{(\lambda_0 + \lambda_{1,i}) R} \Big)
\notag
\\
& + (\lambda_0 + \lambda_{1,\infty} )^{-\frac{7}{2}}  R^{-1} \Big[ (\nabla_y \phi_{\infty})\Big( \frac{x- (\xi^{[0]} + \xi^{[1,i]})}{\lambda_0 + \lambda_{1,i}}, t \Big) - (\nabla_y \phi_{\infty})\Big( \frac{x- (\xi^{[0]} + \xi^{[1,\infty]})}{\lambda_0 + \lambda_{1,\infty}}, t \Big) \Big] \cdot (\nabla \eta )\Big( \frac{x- (\xi^{[0]} + \xi^{[1,i]})}{(\lambda_0 + \lambda_{1,i}) R} \Big) 
\notag
\\
& + (\lambda_0 + \lambda_{1,\infty} )^{-\frac{7}{2}}  R^{-1}  (\nabla_y \phi_{\infty})\Big( \frac{x- (\xi^{[0]} + \xi^{[1,\infty]})}{\lambda_0 + \lambda_{1,\infty}}, t \Big) \cdot 
\Big[
(\nabla \eta )\Big( \frac{x- (\xi^{[0]} + \xi^{[1,i]})}{(\lambda_0 + \lambda_{1,i}) R} \Big)
-
(\nabla \eta )\Big( \frac{x- (\xi^{[0]} + \xi^{[1,\infty]})}{(\lambda_0 + \lambda_{1,\infty}) R} \Big)
\Big].
\notag
\end{align}
}
The convergence of the 1st and 4th lines in $\mathcal{D}_{\rm{o}}$ is straightforward. $\nabla_{y}\phi_{i} \rightsquigarrow \nabla_{y}  \phi_{\infty}$ in $\mathcal{D}_{\rm{in}}$ in \eqref{qd26May1-1} deduces the convergence of the 2nd line.
For the 3rd line,
since $(\nabla \eta)(\cdot)$ imposes the restriction $ | \frac{x- (\xi^{[0]} + \xi^{[1,i]})}{(\lambda_0 + \lambda_{1,i}) R} | \le 2$, then
\begin{align}
		&
		\Big|
		\frac{x- (\xi^{[0]} + \xi^{[1,\infty]})}{\lambda_0 + \lambda_{1,\infty}} - \frac{x- (\xi^{[0]} + \xi^{[1,i]})}{\lambda_0 + \lambda_{1,i} } \Big|
		=  \Big|
		\frac{ \xi^{[1,i]} - \xi^{[1,\infty]} }{\lambda_0 + \lambda_{1,\infty}}
		+
		[x- (\xi^{[0]} + \xi^{[1,i]})]
		\frac{\lambda_{1,i} - \lambda_{1,\infty}}{(\lambda_0 + \lambda_{1,\infty} ) (\lambda_0 + \lambda_{1,i} )} \Big|
        \notag
		\\
		\le \ &
		\frac{ | \xi^{[1,i]} - \xi^{[1,\infty]} | }{ | \lambda_0 + \lambda_{1,\infty} | }
		+
		2 R \Big|
		\frac{\lambda_{1,i} - \lambda_{1,\infty}}{ \lambda_0 + \lambda_{1,\infty} } \Big| 
		\stackrel{\eqref{qd26Apr5-1}, \eqref{mov-26Aug11-1}, i \gg 1}{\le} R.
        \label{qd26Mar2-1}
\end{align}
It follows that $
	\big| \frac{x- (\xi^{[0]} + \xi^{[1,\infty]})}{\lambda_0 + \lambda_{1,\infty}} \big| \le 3R $. Due to the well-designed space $B_{\rm{in}}$ in \eqref{Bin-def}, 
$
(\nabla_y \phi_{\infty})\big( \frac{x- (\xi^{[0]} + \xi^{[1,\infty]})}{\lambda_0 + \lambda_{1,\infty}}, t \big) \cdot (\nabla \eta )\big( \frac{x- (\xi^{[0]} + \xi^{[1,i]})}{(\lambda_0 + \lambda_{1,i}) R} \big)
$
is well-defined. Since $\phi_{\infty} \in C\big( [t_0,T]; C^1 ( \overline{B_{4 R}} ) \big)$ and $  \frac{x- (\xi^{[0]} + \xi^{[1,i]})}{\lambda_0 + \lambda_{1,i}} \rightsquigarrow  \frac{x- (\xi^{[0]} + \xi^{[1,\infty]})}{\lambda_0 + \lambda_{1,\infty}}$ in $\mathcal{D}_{\rm{o}}$, we have the convergence of the 3rd line. Thus,
\begin{equation*}
\begin{aligned}
& 
(\lambda_0 + \lambda_{1,i} )^{-\frac{7}{2}}  R^{-1}  (\nabla_y \phi_{i})\Big( \frac{x- (\xi^{[0]} + \xi^{[1,i]})}{\lambda_0 + \lambda_{1,i}}, t \Big) \cdot (\nabla \eta )\Big( \frac{x- (\xi^{[0]} + \xi^{[1,i]})}{(\lambda_0 + \lambda_{1,i}) R} \Big) 
\\
\rightsquigarrow \ & (\lambda_0 + \lambda_{1,\infty} )^{-\frac{7}{2}}  R^{-1}  (\nabla_y \phi_{\infty})\Big( \frac{x- (\xi^{[0]} + \xi^{[1,\infty]})}{\lambda_0 + \lambda_{1,\infty}}, t \Big) \cdot (\nabla \eta )\Big( \frac{x- (\xi^{[0]} + \xi^{[1,\infty]})}{(\lambda_0 + \lambda_{1,\infty}) R} \Big) \mbox{ \ in \ } \mathcal{D}_{\rm{o}}.
\end{aligned}
\end{equation*}
The first and third parts in $\Lambda_1[\phi, \dot{\lambda}_{1}, \dot{\xi}^{[1]} ]$ can be handled similarly. Thus,
\begin{equation}\label{qd26May1-2}
\Lambda_1[\phi_{i}, \dot{\lambda}_{1,i}, \dot{\xi}^{[1,i]} ] \rightsquigarrow  \Lambda_1[\phi_{\infty}, \dot{\lambda}_{1,\infty}, \dot{\xi}^{[1,\infty]} ] 
\mbox{ \ in \ }
\mathcal{D}_{\rm{o}}.
\end{equation}

{\textbf{For $\Lambda_2$.}}
\begin{equation*}
\begin{aligned}
&
	\Lambda_2[\phi,\dot{\lambda}_{1}, \dot{\xi}^{[1]}]
= (\dot{\lambda}_0 + \dot{\lambda}_{1}) (\lambda_0 + \lambda_{1})^{-\frac{5}{2}}   \Big[ \frac{3}{2} \phi\Big(\frac{x-(\xi^{[0]} + \xi^{[1]})}{\lambda_0 + \lambda_{1}},t \Big)
+ \frac{x-(\xi^{[0]} + \xi^{[1]})}{\lambda_0 + \lambda_{1}} \cdot (\nabla_y \phi) \Big(\frac{x-(\xi^{[0]} + \xi^{[1]})}{\lambda_0 + \lambda_{1}},t \Big) \Big]
\\
& \times \eta\Big( \frac{x-(\xi^{[0]} + \xi^{[1]})}{(\lambda_0 + \lambda_{1})R} \Big)
+ 
(\lambda_0 + \lambda_{1})^{-\frac{5}{2}} (\dot{\xi}^{[0]} + \dot{\xi}^{[1]}) \cdot (\nabla_y \phi)\Big(\frac{x-(\xi^{[0]} + \xi^{[1]})}{\lambda_0 + \lambda_{1}}, t \Big) \eta\Big( \frac{x-(\xi^{[0]} + \xi^{[1]})}{(\lambda_0 + \lambda_{1})R} \Big).
\end{aligned}
\end{equation*}
For the first part, using similar analysis especially  \eqref{qd26Mar2-1} for deducing \eqref{qd26May1-2}, we have
\begin{align*}
		& (\dot{\lambda}_0 + \dot{\lambda}_{1,i}) (\lambda_0 + \lambda_{1,i})^{-\frac{5}{2}}
		\\
		& \times   
		\Big[ \frac{3}{2} \phi_{i} \Big(\frac{x-(\xi^{[0]} + \xi^{[1,i]})}{\lambda_0 + \lambda_{1,i}},t \Big)
		+ \frac{x-(\xi^{[0]} + \xi^{[1,i]})}{\lambda_0 + \lambda_{1,i}} \cdot (\nabla_y \phi_{i}) \Big(\frac{x-(\xi^{[0]} + \xi^{[1,i]})}{\lambda_0 + \lambda_{1,i}},t \Big) \Big] \eta\Big( \frac{x-(\xi^{[0]} + \xi^{[1,i]})}{(\lambda_0 + \lambda_{1,i})R} \Big)
		\\
		& - (\dot{\lambda}_0 + \dot{\lambda}_{1,\infty}) (\lambda_0 + \lambda_{1,\infty})^{-\frac{5}{2}}
		\\
		& \times   
		\Big[ \frac{3}{2} \phi_{\infty} \Big(\frac{x-(\xi^{[0]} + \xi^{[1,\infty]})}{\lambda_0 + \lambda_{1,\infty}},t \Big)
		+ \frac{x-(\xi^{[0]} + \xi^{[1,\infty]})}{\lambda_0 + \lambda_{1,\infty}}  \cdot (\nabla_y \phi_{\infty}) \Big(\frac{x-(\xi^{[0]} + \xi^{[1,\infty]})}{\lambda_0 + \lambda_{1,\infty}},t \Big) \Big] \eta\Big( \frac{x-(\xi^{[0]} + \xi^{[1,\infty]})}{(\lambda_0 + \lambda_{1,\infty})R} \Big)
		\\
		= \ & (\dot{\lambda}_0 + \dot{\lambda}_{1,i}) (\lambda_0 + \lambda_{1,i})^{-\frac{5}{2}} \eta\Big( \frac{x-(\xi^{[0]} + \xi^{[1,i]})}{(\lambda_0 + \lambda_{1,i})R} \Big)  
		\Big\{ \frac{3}{2} 
        \Big[ \phi_{i} \Big(\frac{x-(\xi^{[0]} + \xi^{[1,i]})}{\lambda_0 + \lambda_{1,i}},t \Big)
        -
        \phi_{\infty} \Big(\frac{x-(\xi^{[0]} + \xi^{[1,i]})}{\lambda_0 + \lambda_{1,i}},t \Big)
        \\
        &
        +
        \phi_{\infty} \Big(\frac{x-(\xi^{[0]} + \xi^{[1,i]})}{\lambda_0 + \lambda_{1,i}},t \Big)
        -
        \phi_{\infty} \Big(\frac{x-(\xi^{[0]} + \xi^{[1,\infty]})}{\lambda_0 + \lambda_{1,\infty}},t \Big)
        \Big]
        \\
        &
		+
        \Big[ \frac{x-(\xi^{[0]} + \xi^{[1,i]})}{\lambda_0 + \lambda_{1,i}} 
        -
        \frac{x-(\xi^{[0]} + \xi^{[1,\infty]})}{\lambda_0 + \lambda_{1,\infty}} \Big] \cdot (\nabla_y \phi_{i}) \Big(\frac{x-(\xi^{[0]} + \xi^{[1,i]})}{\lambda_0 + \lambda_{1,i}},t \Big)
        \\
        &
		+
        \frac{x-(\xi^{[0]} + \xi^{[1,\infty]})}{\lambda_0 + \lambda_{1,\infty}} \cdot 
        \Big[ 
        (\nabla_y \phi_{i}) \Big(\frac{x-(\xi^{[0]} + \xi^{[1,i]})}{\lambda_0 + \lambda_{1,i}},t \Big)
        -
        (\nabla_y \phi_{\infty}) \Big(\frac{x-(\xi^{[0]} + \xi^{[1,i]})}{\lambda_0 + \lambda_{1,i}},t \Big)
        \\
        &
        +
        (\nabla_y \phi_{\infty}) \Big(\frac{x-(\xi^{[0]} + \xi^{[1,i]})}{\lambda_0 + \lambda_{1,i}},t \Big)
        -
        (\nabla_y \phi_{\infty}) \Big(\frac{x-(\xi^{[0]} + \xi^{[1,\infty]})}{\lambda_0 + \lambda_{1,\infty}},t \Big)
        \Big] \Big\} 
		\\
		&  + \Big[ (\dot{\lambda}_0 + \dot{\lambda}_{1,i}) (\lambda_0 + \lambda_{1,i})^{-\frac{5}{2}} \eta\Big( \frac{x-(\xi^{[0]} + \xi^{[1,i]})}{(\lambda_0 + \lambda_{1,i})R} \Big)
        - (\dot{\lambda}_0 + \dot{\lambda}_{1,\infty}) (\lambda_0 + \lambda_{1,\infty})^{-\frac{5}{2}} \eta\Big( \frac{x-(\xi^{[0]} + \xi^{[1,\infty]})}{(\lambda_0 + \lambda_{1,\infty})R} \Big) \Big]
		\\
		& \times   
		\Big[ \frac{3}{2} \phi_{\infty} \Big(\frac{x-(\xi^{[0]} + \xi^{[1,\infty]})}{\lambda_0 + \lambda_{1,\infty}},t \Big)
		+ \frac{x-(\xi^{[0]} + \xi^{[1,\infty]})}{\lambda_0 + \lambda_{1,\infty}}  \cdot (\nabla_y \phi_{\infty}) \Big(\frac{x-(\xi^{[0]} + \xi^{[1,\infty]})}{\lambda_0 + \lambda_{1,\infty}},t \Big) \Big] \rightsquigarrow 0 
		\mbox{ \ in \ } \mathcal{D}_{\rm{o}}.
\end{align*}
The second part can be handled similarly. Thus, $\Lambda_2[\phi_{i},\dot{\lambda}_{1,i}, \dot{\xi}^{[1,i]}] \rightsquigarrow \Lambda_2[\phi_{\infty},\dot{\lambda}_{1,\infty}, \dot{\xi}^{[1,\infty]}]$ in $\mathcal{D}_{\rm{o}}$.

{\textbf{For $\mathcal{N}$.}}
\begin{equation*}
\begin{aligned}
&
	\mathcal{N} [\psi,\phi,\dot{\lambda}_{1}, \dot{\xi}^{[1]} ]
= 
|u|^{\frac{4}{3}} u \big|_{u= (\lambda_0 + \lambda_{1})^{-\frac{3}{2}} U \big( \frac{x-(\xi^{[0]} + \xi^{[1]})}{\lambda_0 + \lambda_{1}} \big)
	\eta\big( \frac{x-(\xi^{[0]} + \xi^{[1]})}{\sqrt{t}} \big)
	+ \Psi + \psi
    + (\lambda_0 + \lambda_{1})^{-\frac{3}{2}} \phi\big(\frac{x-( \xi^{[0]} + \xi^{[1]} ) }{\lambda_{0} + \lambda_{1} },t \big) \eta\big( \frac{x-(\xi^{[0]} + \xi^{[1]})}{(\lambda_0 + \lambda_1) R } \big) }
\\
&
	-
	(\lambda_0 + \lambda_{1})^{-\frac{7}{2}} \Big[ U\Big( \frac{x-(\xi^{[0]} + \xi^{[1]})}{\lambda_0 + \lambda_{1}} \Big) \eta\Big( \frac{x-(\xi^{[0]} + \xi^{[1]})}{\sqrt{t}} \Big) \Big]^{\frac{7}{3}}
    \\
    &
	-
	\frac{7}{3} (\lambda_0 + \lambda_{1})^{-2} 
	\Big[ U\Big( \frac{x-(\xi^{[0]} + \xi^{[1]})}{\lambda_0 + \lambda_{1}} \Big) \eta\Big( \frac{x-(\xi^{[0]} + \xi^{[1]})}{\sqrt{t}} \Big) \Big]^{\frac{4}{3}}
    \\
    & \times
	\Big[ \Psi + \psi + (\lambda_0 + \lambda_{1})^{-\frac{3}{2}} \phi\Big( \frac{x-(\xi^{[0]} + \xi^{[1]})}{\lambda_0 + \lambda_{1}} ,t \Big) \eta\Big( \frac{x-(\xi^{[0]} + \xi^{[1]})}{(\lambda_0 + \lambda_1) R } \Big) \Big].
\end{aligned}
\end{equation*}
Therein, similar to \eqref{qd26May1-4}, we have $\phi_{i} \big(\frac{x-( \xi^{[0]} + \xi^{[1,i]} ) }{\lambda_{0} + \lambda_{1,i} },t \big) \eta\big( \frac{x-(\xi^{[0]} + \xi^{[1,i]})}{(\lambda_0 + \lambda_{1,i}) R } \big)
		\rightsquigarrow
		\phi_{\infty} \big(\frac{x-( \xi^{[0]} + \xi^{[1,\infty]} ) }{\lambda_{0} + \lambda_{1,\infty} },t \big) \eta\big( \frac{x-(\xi^{[0]} + \xi^{[1,\infty]})}{(\lambda_0 + \lambda_{1,\infty}) R } \big)$ in $\mathcal{D}_{\rm{o}}$. It is ready to get the convergence of the other terms. Thus, 
$\mathcal{N} [\psi_{i},\phi_{i}, \dot{\lambda}_{1,i},\dot{\xi}^{[1,i]} ] \rightsquigarrow \mathcal{N} [\psi_{\infty},\phi_{\infty}, \dot{\lambda}_{1,\infty}, \dot{\xi}^{[1,\infty]} ]$ in $\mathcal{D}_{\rm{o}}$. We complete the proof of Claim \eqref{cal-G-converge}.

{\textbf{Step 2 - Proof of \eqref{S-conti} and \eqref{calT-in-conti}.}}
In {\textbf{Step 2}}, $y$ is regarded as a pure spatial variable. Recall $\mathcal{F}$ in \eqref{qd26Mar3-1}.
\begin{align}
&
\mathcal{F}[\psi, \dot{\lambda}_1, \dot{\xi}^{[1]}](t) 
= - \frac{7}{3} \Big(
\int_{B_{R_0}} Z_{6}(y)^2 \rmd y
\Big)^{-1}
\Big[
(\lambda_0 + \lambda_{1} )^{\frac{1}{2}}
\int_{B_{R_0}} U(y)^{\frac{4}{3}} Z_{6}(y)
\psi( (\lambda_0 + \lambda_{1} ) y+ \xi^{[0]} + \xi^{[1]}, t) \rmd y
\notag
\\
& +
(\lambda_0 + \lambda_{1} )^{\frac{1}{2}}
\int_{B_{R_0}} U(y)^{\frac{4}{3}} Z_{6}(y)
\big[ \Psi( ( \lambda_0 + \lambda_{1} ) y+ \xi^{[0]} + \xi^{[1]},t) - \Psi(0,t) \big] \rmd y
\label{move-26Sep6-7}
\\
& + \big[ (\lambda_0 + \lambda_{1} )^{\frac{1}{2}} - \lambda_0^{\frac{1}{2}} - 2^{-1} \lambda_0^{-\frac{1}{2}} \lambda_1 \big] \Psi(0, t)
\int_{B_{R_0}} U(y)^{\frac{4}{3}}  Z_{6}(y) \rmd y \Big]
\notag
\\
&
- 
\Big( \int_{B_{R_0}} Z_{6}(y)^2 \rmd y \Big)^{-1}
( \lambda_0 + \lambda_{1} )^{-1} \varrho_{0,1}^{*}\big[ \tilde{\mathcal{J}}_{0,1} [\psi, \dot{\lambda}_1, \dot{\xi}^{[1]}] \big](\tau[\dot{\lambda}_{1}](t)).
\notag
\end{align}
For $|y|\le 4R$,
$ | (\lambda_0 + \lambda_{1} ) y+ \xi^{[0]} + \xi^{[1]} |
	\le | \lambda_0 + \lambda_{1} | 4R + | \xi^{[0]} + \xi^{[1]} |
	\lesssim t^{2-\gamma_1 + \beta} \ll T^9$ by the assumption $\beta-\gamma_1<7$. Thus, the domain $B_{T^9}$ of $\psi$ is sufficient for the integral. In $[t_0, T]$,
\begin{align}
		&
		\int_{B_{R_0}} U(y)^{\frac{4}{3}} Z_{6}(y)
		\psi_{i} ( (\lambda_0 + \lambda_{1,i} ) y+ \xi^{[0]} + \xi^{[1,i]}, t) \rmd y
        \notag
		\\
		& -
		\int_{B_{R_0}} U(y)^{\frac{4}{3}} Z_{6}(y)
		\psi_{\infty} ( (\lambda_0 + \lambda_{1,\infty} ) y+ \xi^{[0]} + \xi^{[1,\infty]}, t) \rmd y
        \label{move-26Sep6-8}
		\\
		= \ & \int_{B_{R_0}} U(y)^{\frac{4}{3}} Z_{6}(y)
		\big[
		\psi_{i} ( (\lambda_0 + \lambda_{1,i} ) y+ \xi^{[0]} + \xi^{[1,i]}, t) 
		-
		\psi_{\infty} ( (\lambda_0 + \lambda_{1,i} ) y+ \xi^{[0]} + \xi^{[1,i]}, t)
		\big] \rmd y
        \notag
		\\
		& +
		\int_{B_{R_0}} U(y)^{\frac{4}{3}} Z_{6}(y)
		\big[ \psi_{\infty} ( (\lambda_0 + \lambda_{1,i} ) y+ \xi^{[0]} + \xi^{[1,i]}, t) - 
		\psi_{\infty} ( (\lambda_0 + \lambda_{1,\infty} ) y+ \xi^{[0]} + \xi^{[1,\infty]}, t) \big] \rmd y
		\rightsquigarrow 0,
        \notag
\end{align}
where for the last step, we use $\psi_{i} \rightsquigarrow \psi_{\infty}$ in $\mathcal{D}_{\rm{o}}$, and $\psi_{\infty} \in B_{\rm o}$. The continuity argument for other terms in \eqref{move-26Sep6-7} is similar to \eqref{move-26Sep6-8} except for $\varrho_{0,1}^{*}$.

We will give the continuity argument for $\varrho_{0,1}^{*}\big[ \tilde{\mathcal{J}}_{0,1} [\psi, \dot{\lambda}_1, \dot{\xi}^{[1]}] \big](\tau[\dot{\lambda}_{1}](t))$, which is much more complicated.
We consider the continuous dependence of $\mathcal{H}[ \psi,\dot{\lambda}_{1}, \dot{\xi}^{[1]} ](y,t) \eta(\frac{y}{2R})$, where
\begin{equation*}
\begin{aligned}
&
\mathcal{H}[ \psi,\dot{\lambda}_{1}, \dot{\xi}^{[1]} ](y,t) 
=  (\dot{\lambda}_{0} + \dot{\lambda}_{1}) (\lambda_{0} + \lambda_{1}) Z_{6}(y)
+ (\lambda_{0} + \lambda_{1}) (\dot{\xi}^{[0]} + \dot{\xi}^{[1]}) \cdot (\nabla U )(y) 
\\
&
+\frac{7}{3} (\lambda_{0} + \lambda_{1})^{\frac{3}{2}} U(y)^{\frac{4}{3}}
\big[
\Psi\big( (\lambda_{0} + \lambda_{1}) y+ \xi^{[0]} + \xi^{[1]},t \big)
+
\psi\big( (\lambda_{0} + \lambda_{1}) y+ \xi^{[0]} + \xi^{[1]}, t \big)
\big].
\end{aligned}
\end{equation*}
For $t \in [t_0,T]$, $|y| \le 4R$, we have
$ (\lambda_{0} + \lambda_{1,i}) y+ \xi^{[0]} + \xi^{[1,i]} 
		\rightsquigarrow
		(\lambda_{0} + \lambda_{1,\infty}) y+ \xi^{[0]} + \xi^{[1,\infty]} $.
Similar to {\textbf{Step 1}}, we have
\begin{equation}\label{qd26May2-5}
	\mathcal{H}[\psi_{i}, \dot{\lambda}_{1,i}, \dot{\xi}^{[1,i]}](y,t) \eta\Big(\frac{y}{2R}\Big) \rightsquigarrow \mathcal{H}[\psi_{\infty}, \dot{\lambda}_{1,\infty}, \dot{\xi}^{[1,\infty]}](y,t) \eta\Big(\frac{y}{2R}\Big)
	\mbox{ \ in \ } \mathbb{R}^5 \times [t_0, T].
\end{equation}

Recall \eqref{qd26May1-7}. 
\begin{equation}\label{qd26May2-1}
\begin{aligned}
\tau[\dot{\lambda}_{1}](t) = \ &
\int_{t_0}^t \Big\{ (\lambda_{0}(s) + \lambda_{1}(s) ) \1_{s\le T} + \big[ (T+1-s) (\lambda_{0}(T) + \lambda_{1}(T) ) + (s-T) s^{2-\gamma_1} \big] \1_{T < s\le T+1} 
\\
&
\qquad
+ s^{2-\gamma_1} \1_{s > T+1} \Big\}^{-2} \rmd s + \tau_0
\mbox{ \ for \ } t \in [t_0, \infty),
\end{aligned}
\end{equation}
whose inverse function is denoted by $t[\dot{\lambda}_{1}](\tau)$ defined in $[\tau_0,\infty)$.

Claim:
Given $\gamma_1 \in (\frac{3}{2}, 2)$, for $t_0\gg 1$, then for any $\dot{\lambda}_{1,1}, \dot{\lambda}_{1,2} \in B_{\dot{\lambda}_1}$, we have
\begin{equation}\label{qd26Apr17-3}
| \tau[\dot{\lambda}_{1,1}](t) - \tau[\dot{\lambda}_{1,2}](t) | 
\lesssim
R_0^{-\frac{\zeta}{2}} \| \dot{\lambda}_{1,1} - \dot{\lambda}_{1,2} \|_{\dot{\lambda}_1}
\big[ t^{2\gamma_1-3} \1_{t\le T+1} + (T+1)^{2\gamma_1-3} \1_{t > T+1} \big]
\mbox{ \ for \ } t \in [t_0, \infty),
\end{equation}
\begin{equation}\label{qd26Apr16-1}
	\begin{aligned}
		\big| t[\dot{\lambda}_{1,1}](\tau) - t[\dot{\lambda}_{1,2}](\tau) \big| \lesssim \ & 
		 \| \dot{\lambda}_{1,1} - \dot{\lambda}_{1,2} \|_{\dot{\lambda}_1}
         R_0^{-\frac{\zeta}{2}}
\big[ \tau^{\frac{1}{2\gamma_1-3} }  \1_{\tau \le C_{tm} (T+1)^{2\gamma_1 -3}} 
\\
&
+ (T+1)^{2\gamma_1 - 3} \tau^{-\frac{2\gamma_1 - 4}{2\gamma_1-3}} \1_{\tau > C_{tm}^{-1} (T+1)^{2\gamma_1 -3} } \big]
\mbox{ \ for \ } \tau \in [\tau_0,  \infty).
	\end{aligned}
\end{equation}

\begin{proof}[Proof of \eqref{qd26Apr17-3} and \eqref{qd26Apr16-1}]

We first give a useful estimate.
Denote
\begin{equation*}
\begin{aligned}
	f(\theta) := \ & \Big( \big[ \lambda_0(s) + \theta \lambda_{1,1}(s) + (1-\theta) \lambda_{1,2}(s) \big] \1_{s\le T}
    \\
    &
    + \Big\{ (T+1-s) \big[ \lambda_{0}(T) + \theta \lambda_{1,1}(T) + (1-\theta) \lambda_{1,2}(T) \big] + (s-T) s^{2-\gamma_1} \Big\} \1_{T < s\le T+1} 
    + s^{2-\gamma_1} \1_{s > T+1} \Big)^{-2}.
\end{aligned}
\end{equation*}
\begin{equation}\label{qd26May1-8}
\begin{aligned}
&
f(1) - f(0) = f'(\theta)
=
-2 \Big( \big[ \lambda_0(s) + \theta \lambda_{1,1}(s) + (1-\theta) \lambda_{1,2}(s) \big] \1_{s\le T}
    \\
    &
    + \Big\{ (T+1-s) \big[ \lambda_{0}(T) + \theta \lambda_{1,1}(T) + (1-\theta) \lambda_{1,2}(T) \big] + (s-T) s^{2-\gamma_1} \Big\} \1_{T < s\le T+1} 
    \\
    &
    + s^{2-\gamma_1} \1_{s > T+1} \Big)^{-3} \Big( \big[ \lambda_{1,1}(s) - \lambda_{1,2}(s) \big] \1_{s\le T}
    + (T+1-s) \big[ \lambda_{1,1}(T) - \lambda_{1,2}(T) \big]  \1_{T < s\le T+1} 
    \Big)
\end{aligned}
\end{equation}
for some $\theta \in [0,1]$.
Then
\begin{align}
&
|f(1) - f(0)|
\stackrel{\eqref{qd26Apr5-1}, \gamma_1<2}{\lesssim} ( s^{2-\gamma_1} )^{-3}
\big( R_0^{-\frac{\zeta}{2}} s^{2-\gamma_1} \| \dot{\lambda}_{1,1} - \dot{\lambda}_{1,2} \|_{\dot{\lambda}_1} \1_{s\le T}
    + R_0^{-\frac{\zeta}{2}} T^{2-\gamma_1} \| \dot{\lambda}_{1,1} - \dot{\lambda}_{1,2} \|_{\dot{\lambda}_1} \1_{T < s\le T+1} 
    \big)
    \notag
\\
\sim \ &  
R_0^{-\frac{\zeta}{2}} s^{2\gamma_1 - 4} \| \dot{\lambda}_{1,1} - \dot{\lambda}_{1,2} \|_{\dot{\lambda}_1} \1_{s\le T+1}.
\label{qd26May1-11}
\end{align}

Now we will prove the claim. For $t \in [t_0, \infty)$,
{\small
\begin{equation}
	\begin{aligned}
		&
		| \tau[\dot{\lambda}_{1,1}](t) - \tau[\dot{\lambda}_{1,2}](t) 
		|
        \\
        = \ & \Big|
        \int_{t_0}^t \Big( \Big\{ \big( \lambda_0(s) + \lambda_{1,1}(s) \big) \1_{s\le T} + \big[ (T+1-s) \big( \lambda_{0}(T) + \lambda_{1,1}(T) \big) + (s-T) s^{2-\gamma_1} \big] \1_{T < s\le T+1} + s^{2-\gamma_1} \1_{s > T+1} \Big\}^{-2} 
        \\
        & - \Big\{ \big( \lambda_0(s) + \lambda_{1,2}(s) \big) \1_{s\le T} + \big[ (T+1-s) \big( \lambda_{0}(T) + \lambda_{1,2}(T) \big) + (s-T) s^{2-\gamma_1} \big] \1_{T < s\le T+1} + s^{2-\gamma_1} \1_{s > T+1} \Big\}^{-2}
        \Big) \rmd s \Big|
        \\
        \stackrel{\eqref{qd26May1-11}}{\lesssim}  \ & R_0^{-\frac{\zeta}{2}} \| \dot{\lambda}_{1,1} - \dot{\lambda}_{1,2} \|_{\dot{\lambda}_1}
        \int_{t_0}^t  s^{2\gamma_1 - 4}  \1_{s\le T + 1} \rmd s
\stackrel{\gamma_1>\frac{3}{2}}{\lesssim} 
R_0^{-\frac{\zeta}{2}} \| \dot{\lambda}_{1,1} - \dot{\lambda}_{1,2} \|_{\dot{\lambda}_1}
\big[ t^{2\gamma_1-3} \1_{t\le T+1} + (T+1)^{2\gamma_1-3} \1_{t > T+1} \big].
	\end{aligned}
\end{equation}
}

By the inverse function of \eqref{qd26May2-1}. For $\tau \in [\tau_0, \infty)$,
\begin{align*}
&
\tau = \int_{t_0}^{t[\dot{\lambda}_{1,1}](\tau)} \Big\{ (\lambda_{0}(s) + \lambda_{1,1}(s) ) \1_{s\le T} + \big[ (T+1-s) (\lambda_{0}(T) + \lambda_{1,1}(T) ) + (s-T) s^{2-\gamma_1} \big] \1_{T < s\le T+1} 
\\
&
+ s^{2-\gamma_1} \1_{s > T+1} \Big\}^{-2} \rmd s + \tau_0
\\
= \ & \Big( \int_{t_0}^{t[\dot{\lambda}_{1,1}](\tau)}
		+
		\int_{t[\dot{\lambda}_{1,1}](\tau)}^{t[\dot{\lambda}_{1,2}](\tau)}
		\Big) \Big\{ (\lambda_{0}(s) + \lambda_{1,2}(s) ) \1_{s\le T} + \big[ (T+1-s) (\lambda_{0}(T) + \lambda_{1,2}(T) ) + (s-T) s^{2-\gamma_1} \big] \1_{T < s\le T+1} 
\\
&
+ s^{2-\gamma_1} \1_{s > T+1} \Big\}^{-2} \rmd s + \tau_0
\end{align*}
$\Rightarrow$
\begin{align*}
& 
\mathbf{(LP)} :=
\int_{t_0}^{t[\dot{\lambda}_{1,1}](\tau)} \Big( \Big\{ (\lambda_{0}(s) + \lambda_{1,1}(s) ) \1_{s\le T} + \big[ (T+1-s) (\lambda_{0}(T) + \lambda_{1,1}(T) ) + (s-T) s^{2-\gamma_1} \big] \1_{T < s\le T+1} 
\\
&
+ s^{2-\gamma_1} \1_{s > T+1} \Big\}^{-2} - \Big\{ (\lambda_{0}(s) + \lambda_{1,2}(s) ) \1_{s\le T} + \big[ (T+1-s) (\lambda_{0}(T) + \lambda_{1,2}(T) ) + (s-T) s^{2-\gamma_1} \big] \1_{T < s\le T+1} 
\\
&
+ s^{2-\gamma_1} \1_{s > T+1} \Big\}^{-2} \Big) \rmd s
\\
= \ & \int_{t[\dot{\lambda}_{1,1}](\tau)}^{t[\dot{\lambda}_{1,2}](\tau)} \Big\{ (\lambda_{0}(s) + \lambda_{1,2}(s) ) \1_{s\le T} + \big[ (T+1-s) (\lambda_{0}(T) + \lambda_{1,2}(T) ) + (s-T) s^{2-\gamma_1} \big] \1_{T < s\le T+1} 
\\
&
+ s^{2-\gamma_1} \1_{s > T+1} \Big\}^{-2} \rmd s =: \mathbf{(RP)}.
\end{align*}
Therein, for the left part,
\begin{equation*}
\begin{aligned}
& | \mathbf{(LP)} |
\stackrel{\eqref{qd26May1-11}}{\lesssim}   R_0^{-\frac{\zeta}{2}} \| \dot{\lambda}_{1,1} - \dot{\lambda}_{1,2} \|_{\dot{\lambda}_1} \int_{t_0}^{t[\dot{\lambda}_{1,1}](\tau)}  s^{2\gamma_1 - 4}  \1_{s\le T+1} \rmd s
\\
\stackrel{\gamma_1>\frac{3}{2}}{\lesssim} \ & R_0^{-\frac{\zeta}{2}} \| \dot{\lambda}_{1,1} - \dot{\lambda}_{1,2} \|_{\dot{\lambda}_1} 
\big[ (t[\dot{\lambda}_{1,1}](\tau))^{2\gamma_1 - 3}  \1_{t[\dot{\lambda}_{1,1}](\tau) \le T+1} + (T+1)^{2\gamma_1 - 3} \1_{t[\dot{\lambda}_{1,1}](\tau) > T+1} \big]
\\
\stackrel{\eqref{tau-est}}{\lesssim} \ & R_0^{-\frac{\zeta}{2}} \| \dot{\lambda}_{1,1} - \dot{\lambda}_{1,2} \|_{\dot{\lambda}_1} 
\big[ \tau  \1_{\tau \le C_{tm} (T+1)^{2\gamma_1 -3}} + (T+1)^{2\gamma_1 - 3} \1_{\tau > C_{tm}^{-1} (T+1)^{2\gamma_1 -3} } \big].
\end{aligned}
\end{equation*}
For the right part, 
\begin{equation*}
| \mathbf{(RP)} |
\sim \Big| \int_{t[\dot{\lambda}_{1,1}](\tau)}^{t[\dot{\lambda}_{1,2}](\tau)} s^{2\gamma_1 - 4} \rmd s \Big|
\stackrel{\eqref{tau-est}}{\sim}
\tau^{\frac{2\gamma_1 - 4}{2\gamma_1 -3}}
\Big| \int_{t[\dot{\lambda}_{1,1}](\tau)}^{t[\dot{\lambda}_{1,2}](\tau)} \rmd s \Big|
=
\tau^{\frac{2\gamma_1 - 4}{2\gamma_1 -3}}
\big| t[\dot{\lambda}_{1,2}](\tau)  - t[\dot{\lambda}_{1,1}](\tau) \big|.
\end{equation*}
Thus, we conclude \eqref{qd26Apr16-1}.
\end{proof}

The following lemma will be used to extend the right-hand side of the inner problem.
\begin{lemma}\label{qd26Apr17-5-lem}
	Given $1 \ll t_0<T <\infty$, $\Omega \subset \mathbb{R}^n$, a sequence 
$(g_{i}(y,t))_{i\ge 1}$ defined in $t\in [t_0, T], y\in \Omega$, suppose that for any fixed $i$, $g_{i}$ is uniformly continuous in $t\in [t_0, T]$, $y\in \Omega$; $\lim_{i\to \infty} \sup_{t\in [t_0, T], y\in \Omega} |g_{i}(y,t) - g_{\infty}(y,t)| = 0$, then 
\begin{equation}\label{qd26May2-3}
\mbox{$g_{\infty}$ is uniformly continuous in $t\in [t_0, T]$, $y\in \Omega$.}
\end{equation}
For $\gamma_1 \in (\frac{3}{2}, 2)$, 
under the additional assumption $\lim_{i\to \infty} \|\dot{\lambda}_{1,i} - \dot{\lambda}_{1,\infty}\|_{\dot{\lambda}_1} = 0$, then
    \begin{equation}
    \begin{aligned}
    \lim_{i\to \infty} \sup_{\tau\in [\tau_0, \infty), y\in \Omega}
    &
    \big|
    \big[ 
		g_{i}(y, t[\dot{\lambda}_{1,i}](\tau) ) \1_{\tau_0 \le \tau \le  \tau[\dot{\lambda}_{1,i}](T)}
		+
		g_{i}(y,T) \1_{\tau > \tau[\dot{\lambda}_{1,i}](T)}
        \big]
        \\
        &
		-
        \big[
		g_{\infty}(y,t[\dot{\lambda}_{1,\infty}](\tau)) \1_{\tau_0 \le \tau \le  \tau[\dot{\lambda}_{1,\infty}](T)}
		+ g_{\infty}(y,T) \1_{\tau > \tau[\dot{\lambda}_{1,\infty}](T)}\big]\big| = 0.
       \end{aligned}
	\end{equation}

\end{lemma}

\begin{proof}
For coherence of the continuity argument, we postpone the proof of Lemma \ref{qd26Apr17-5-lem} to Section \ref{ext-lem-proof}.
\end{proof}

$
\mathcal{H}[\psi_{i}, \dot{\lambda}_{1,i}, \dot{\xi}^{[1,i]}](y,t[\dot{\lambda}_{1,i}](\tau)) \eta\big( \frac{y}{2R(t[\dot{\lambda}_{1,i}](\tau))} \big)
$
is well-defined for $y\in \mathbb{R}^5$, $\tau \in [\tau_0, \tau[\dot{\lambda}_{1,i}](T)]$. Similar to \eqref{calJ-def},
for $l=i$ or $\infty$, we make the extension
\begin{equation*}
\begin{aligned}
\mathcal{J}^{[l]}(y,\tau) := \ & \mathcal{H}[\psi_{l}, \dot{\lambda}_{1,l}, \dot{\xi}^{[1,l]}](y,t[\dot{\lambda}_{1,l}](\tau)) \eta\Big(\frac{y}{2R(t[\dot{\lambda}_{1,l}](\tau))}\Big) \1_{\tau_0 \le \tau \le \tau[\dot{\lambda}_{1,l}](T)}
\\
&
+
\mathcal{H}[\psi_{l}, \dot{\lambda}_{1,l}, \dot{\xi}^{[1,l]}](y,T) \eta\Big(\frac{y}{2R(T)}\Big) \1_{ \tau > \tau[\dot{\lambda}_{1,l}](T)}.
\end{aligned}
\end{equation*}
By \eqref{qd26May2-5} and Lemma \ref{qd26Apr17-5-lem}, we have
\begin{equation*}
\mathcal{J}^{[i]}(y,\tau) \rightsquigarrow \mathcal{J}^{[\infty]}(y,\tau)
\mbox{ \ for \ } (y, \tau) \in \mathbb{R}^5 \times [\tau_0, \infty).
\end{equation*}
Similar to \eqref{qd26Jan3-1} and \eqref{move-26Aug21-5}, for $(j,k) \in \mathbf{M_{0,1}} \cup \{ \perp \}$, we can define $\tilde{\mathcal{J}}_{j,k}^{[l]}(y,\tau)$. Then
\begin{equation*}
		\tilde{\mathcal{J}}_{j,k}^{[i]}(y,\tau) \rightsquigarrow \tilde{\mathcal{J}}_{j,k}^{[\infty]}(y,\tau)
		\mbox{ \ for \ } (y,\tau) \in \mathbb{R}^5 \times [\tau_0, \infty),
		\quad (j,k) \in \mathbf{M_{0,1}} \cup \{ \perp \},
\end{equation*}
which implies
\begin{equation*}
	\sup_{\tau \in [\tau_0, C_{tm} T^{2\gamma_1 -3}], y\in B_{R_*(\tau)}}
	(v(\tau) \langle y \rangle^{-2-a} )^{-1}
	\big| \tilde{\mathcal{J}}_{j,k}^{[i]}(y,\tau) - \tilde{\mathcal{J}}_{j,k}^{[\infty]}(y,\tau) \big| \to 0 
	\mbox{ \ as \ } i \to \infty
\end{equation*}
with $v(\tau) = \tau^{-1}$ or $\tau^{-\frac{5\gamma_1 + \gamma_2 - 9}{4\gamma_1 - 6}}$, $a=1$. We {\textbf{emphasize}} that the mappings in \eqref{qd26May2-6} are independent of the choice of $(\psi_{i},\phi_{i}, \dot{\lambda}_{1,i}, \dot{\xi}^{[1,i]})$. Thus, by Propositions \ref{mode0-regluing-prop}, \ref{regluing-mode1-prop}, \ref{regluing-higher-mode-prop}, for $(j,k) \in \mathbf{M_{0,1}} \cup \{ \perp \}$,
\begin{equation*}
\begin{aligned}
&
		\big|
		\mathcal{T}_{j,k}^{\rm{in}} [ \tilde{\mathcal{J}}_{j,k}^{[i]} ]  (y,\tau)
		- \mathcal{T}_{j,k}^{\rm{in}} [ \tilde{\mathcal{J}}_{j,k}^{[\infty]} ](y,\tau) \big|
	+ 
		\big| \nabla_{y}
		\big(
		\mathcal{T}_{j,k}^{\rm{in}} [ \tilde{\mathcal{J}}_{j,k}^{[i]} ] (y,\tau)
		- \mathcal{T}_{j,k}^{\rm{in}} [ \tilde{\mathcal{J}}_{j,k}^{[\infty]} ](y,\tau) \big) \big|
\\
		\rightsquigarrow \ & 0
		\mbox{ \ in \ } \tau \in [\tau_0, C_{tm} T^{2\gamma_1 -3}], y\in B_{R_*(\tau)}.
\end{aligned}
\end{equation*}
It follows that
\begin{equation*}
	\begin{aligned}
		&
		\big|
		\mathcal{T}_{j,k}^{\rm{in}} [\tilde{\mathcal{J}}_{j,k}^{[i]} ]  (y,\tau[\dot{\lambda}_{1,i}](t))
		- \mathcal{T}_{j,k}^{\rm{in}} [\tilde{\mathcal{J}}_{j,k}^{[\infty]} ](y,\tau[\dot{\lambda}_{1,\infty}](t)) \big|
		\\
= \ & \big|
\mathcal{T}_{j,k}^{\rm{in}} [\tilde{\mathcal{J}}_{j,k}^{[i]} ]  (y,\tau[\dot{\lambda}_{1,i}](t))
- \mathcal{T}_{j,k}^{\rm{in}} [\tilde{\mathcal{J}}_{j,k}^{[\infty]} ](y,\tau[\dot{\lambda}_{1,i}](t))
+ \mathcal{T}_{j,k}^{\rm{in}} [\tilde{\mathcal{J}}_{j,k}^{[\infty]} ](y,\tau[\dot{\lambda}_{1,i}](t))
- \mathcal{T}_{j,k}^{\rm{in}} [\tilde{\mathcal{J}}_{j,k}^{[\infty]} ](y,\tau[\dot{\lambda}_{1,\infty}](t)) \big|
\\
\stackrel{\eqref{qd26Apr17-3}}{\rightsquigarrow}  \ & 0
		\mbox{ \ in \ } t \in [t_0, T], y\in \overline{B_{4 R(t)}}.
	\end{aligned}
\end{equation*}
And the gradient estimate is similar. Thus, \eqref{calT-in-conti} holds. Moreover,
\begin{equation*}
\varrho_{j,k}^{*} [\tilde{\mathcal{J}}_{j,k}^{[i]}](\tau)
\rightsquigarrow \varrho_{j,k}^{*} [\tilde{\mathcal{J}}_{j,k}^{[\infty]} ](\tau)
\mbox{ \ in \ } [\tau_0, C_{tm} T^{2\gamma_1 -3}]
\mbox{ \ for \ } (j,k) \in \mathbf{M_{0,1}}.
\end{equation*}
It follows that
\begin{equation*}
	\begin{aligned}
		&
		\varrho_{j,k}^{*} [\tilde{\mathcal{J}}_{j,k}^{[i]} ]  (\tau[\dot{\lambda}_{1,i}](t))
		- \varrho_{j,k}^{*} [\tilde{\mathcal{J}}_{j,k}^{[\infty]} ](\tau[\dot{\lambda}_{1,\infty}](t))
\\
= \ & \varrho_{j,k}^{*} [\tilde{\mathcal{J}}_{j,k}^{[i]} ]  (\tau[\dot{\lambda}_{1,i}](t))
- \varrho_{j,k}^{*} [\tilde{\mathcal{J}}_{j,k}^{[\infty]} ](\tau[\dot{\lambda}_{1,i}](t))
+ \varrho_{j,k}^{*} [\tilde{\mathcal{J}}_{j,k}^{[\infty]} ](\tau[\dot{\lambda}_{1,i}](t))
- \varrho_{j,k}^{*} [\tilde{\mathcal{J}}_{j,k}^{[\infty]}](\tau[\dot{\lambda}_{1,\infty}](t))
\stackrel{\eqref{qd26Apr17-3}}{\rightsquigarrow} 0
	\end{aligned}
\end{equation*}
in $t \in [t_0, T]$.
Thus, we arrive that
\begin{equation}\label{qd26May2-7}
\mathcal{F}[\psi_{i}, \dot{\lambda}_{1,i}, \dot{\xi}^{[1,i]}](t) \rightsquigarrow \mathcal{F}[\psi_{\infty}, \dot{\lambda}_{1,\infty}, \dot{\xi}^{[1,\infty]}](t),
\quad
\mathcal{S}_6[\psi_{i}, \dot{\lambda}_{1,i}, \dot{\xi}^{[1,i]}](t) 
\stackrel{\eqref{mu1-xi-sys}}{\rightsquigarrow} \mathcal{S}_6[\psi_{\infty}, \dot{\lambda}_{1,\infty}, \dot{\xi}^{[1,\infty]}](t)
\mbox{ \ in \ } [t_0, T].
\end{equation}

By \eqref{qd25Dec25-3}, for $j\in \overline{1,5}$,
\begin{equation*}
\begin{aligned}
&\mathcal{S}_{j}[\psi, \dot{\lambda}_1, \dot{\xi}^{[1]}](t)
= - \frac{7}{3}
\Big( \int_{B_{R_0}} Z_1(y)^2 \rmd y \Big)^{-1}
\Big\{ \big[ (\lambda_0 + \lambda_{1} )^{\frac{3}{2}} - \lambda_0^{\frac{3}{2}} \big] (\partial_{x_j} \Psi)(0,t)
\int_{B_{R_0}} U(y)^{\frac{4}{3}} y_1 Z_{1}(y) \rmd y
\\
& + (\lambda_0 + \lambda_{1} )^{\frac{1}{2}}
\int_{B_{R_0}} U(y)^{\frac{4}{3}} Z_{j}(y)
\big[
\Psi\big( (\lambda_0 + \lambda_{1} ) y+ \xi^{[0]} + \xi^{[1]},t \big) - \Psi(0,t)
\\
&
-
(\nabla \Psi)(0,t) \cdot \big[ ( \lambda_0 + \lambda_{1} ) y+ \xi^{[0]} + \xi^{[1]} \big]
+
\psi\big( ( \lambda_0 + \lambda_{1} ) y+ \xi^{[0]} + \xi^{[1]}, t \big) - \psi(0, t)
\big] \rmd y
\Big\}
\\
& - \Big( \int_{B_{R_0}} Z_1(y)^2 \rmd y \Big)^{-1} (\lambda_0 + \lambda_{1} )^{-1} \varrho_{1,j}^{*}\big[ \tilde{\mathcal{J}}_{1,j} [\psi, \dot{\lambda}_1, \dot{\xi}^{[1]}] \big](\tau[\dot{\lambda}_1](t)).
\end{aligned}
\end{equation*}
Similar to deduce \eqref{qd26May2-7}, we have $\vec{\mathcal{S}}[\psi_{i}, \dot{\lambda}_{1,i}, \dot{\xi}^{[1,i]}](t) \rightsquigarrow \vec{\mathcal{S}}[\psi_{\infty}, \dot{\lambda}_{1,\infty}, \dot{\xi}^{[1,\infty]}](t)$ in $[t_0, T]$. Thus, \eqref{S-conti} holds.
\end{proof}

\section{Completion of construction}\label{Final-u-Sec}

\subsection{Proof of Theorem \ref{5d-main-th}}
\begin{proof}[Proof of Theorem \ref{5d-main-th}]

Collecting all parameter restrictions $\gamma_1, \gamma_2 \in (0,5)$, $\beta>0$, \eqref{mov-26Aug11-1}, \eqref{qd25Dec24-2}, \eqref{qd26Jan3-7}, and $\beta-\gamma_1<7$, we use the software Mathematica (\href{https://raw.githubusercontent.com/qdz-wonderful/Mathematica-codes-for-One-bubble-solution-with-a-moving-maximum-point/refs/heads/main/Moving%20maximum%20point-parameters%20in%20Th%201.1.wl}{code link}) to find parameters
\begin{align}
&
\zeta \in (0,1), \quad \frac{3}{2} < \gamma_1 < 2, 
\notag
\\
&
\begin{aligned}
\mbox{$\gamma_2$ and $\beta$ satisfy \ } &
\frac{1}{11} (3+9\gamma_1) < \gamma_2 \le \gamma_1, \quad 
\frac{1}{4}(-3+\gamma_1 + \gamma_2) < \beta < \frac{1}{6} (-6-3\gamma_1 + 7 \gamma_2),
\\
\mbox{or \ } &
\gamma_1 < \gamma_2 < \frac{1}{3}(-3+5\gamma_1),
\quad
\frac{1}{4} (-3+\gamma_1+\gamma_2) < \beta < \frac{1}{3}(-3+2\gamma_1),
\end{aligned}
\end{align}
which lead to the assumption $\zeta \in (0, 1)$ and \eqref{move-26Aug23-1}.

One integrating Lemmas \ref{outer-exist}, \ref{mu1-xi-lem}, \ref{mapping-conti-lem}, 
for $(\psi,\phi, \dot{\lambda}_{1}, \dot{\xi}^{[1]}) \in \mathcal{X}$ and $\lambda_1[\dot{\lambda}_1], \xi^{[1]}[\dot{\xi}^{[1]}]$ given in \eqref{mu1-xi-sys},
\begin{equation}
\mathcal{T}_{\rm{o}} [ \mathcal{G} [\psi,\phi, \dot{\lambda}_{1}, \dot{\xi}^{[1]}] \1_{|x|\le T^{9}, t\le T } ](x,t)
\times
\mathcal{T}^{\rm{in}} [\mathcal{J}[\psi, \dot{\lambda}_{1}, \dot{\xi}^{[1]}] ](y,\tau[\dot{\lambda}_{1}](t))
\times
\mathcal{S}_6[\psi, \dot{\lambda}_{1}, \dot{\xi}^{[1]}](t)
\times
\vec{\mathcal{S}}[\psi,\dot{\lambda}_{1},\dot{\xi}^{[1]}](t)
\end{equation}
is a compact and continuous mapping system from $\mathcal{X}$ into itself. By the Schauder fixed-point theorem, we find a solution $(\psi,\phi, \dot{\lambda}_{1}, \dot{\xi}^{[1]}) \in \mathcal{X}$ and then a solution $u_{T}$ for
\begin{equation}\label{qd26May3-1}
\partial_t u_{T} =\Delta u_{T} + |u_{T}|^{\frac{4}{3}} u_{T}
		\mbox{ \ in \ } 
		  B_{T^9} \times (t_0,T]
	\end{equation}
with the form
{\small
\begin{equation}\label{move26Sep6-9} 
    u_{T}(x,t) = \lambda_{T}^{-\frac{3}{2}} U\Big( \frac{x-\xi_{T}}{\lambda_{T}} \Big)
	\eta\Big( \frac{x-\xi_{T}}{\sqrt{t}} \Big)
	+ \Psi(x,t) + \psi_{T}(x,t) + \lambda_{T}^{-\frac{3}{2}}\phi_{T}\Big(\frac{x-\xi_{T}}{\lambda_{T}},t \Big) \eta\Big( \frac{x-\xi_{T}}{\lambda_{T} R} \Big),
\end{equation}
\begin{equation}\label{move-26Aug24-1}
    u_{T}(x,t_0) 
\stackrel{\eqref{T*-est}}{=}  \lambda_{T}(t_0)^{-\frac{3}{2}} U\Big( \frac{x-\xi_{T}(t_0)}{\lambda_{T}(t_0)} \Big)
	\eta\Big( \frac{x-\xi_{T}(t_0)}{\sqrt{t_0}} \Big)
	+ \Psi(x,t_0) 
    + \lambda_{T}(t_0)^{-\frac{3}{2}}
    g_{0,T} \eta\Big( \frac{2}{R_0} \frac{x-\xi_{T}(t_0)}{\lambda_{T}(t_0)} \Big)
    Z_{0}\Big( \frac{x-\xi_{T}(t_0)}{\lambda_{T}(t_0)} \Big),
\end{equation}
}
where $\psi_{T} \in B_{\rm o}$, $\phi_{T} \in B_{\rm{in}}$, $\lambda_{T} = \lambda_{0} + \lambda_{1, T}$, $\xi_{T} = \xi^{[0]} + \xi^{[1,T]}$, $\lambda_{1, T} = \lambda_{1, T}[\dot{\lambda}_{1, T}]$, $\xi^{[1, T]} = \xi^{[1, T]}[\dot{\xi}^{[1, T]}]$ defined in \eqref{mu1-xi-sys} with $\dot{\lambda}_{1, T} \in B_{\dot{\lambda}_1}$, $\dot{\xi}^{[1, T]} \in B_{\dot{\xi}^{[1]}}$, $g_{0,T}$ is a constant possibly depending on $t_0$ and satisfying $|g_{0,T}| \stackrel{\eqref{qd26Apr3-5}}{\le} C R_0 t_0^{3-2\gamma_1}$ with a constant $C>0$ independent of $t_0, T$.

The following convergence argument is similar to \cite[Subsection 7.4]{WYZZ2024}.
We will always take a subsequence when $T\to \infty$ but will not state it. Obviously, $g_{0,T} \to g_{0,*}$ for some $|g_{0,*}| \le C R_0 t_0^{3-2\gamma_1}$. By Lemma \ref{outer-exist}, there exists a function $\psi_{*}$ such that $\psi_{T} \to \psi_{*}$, $\nabla \psi_{T} \to \nabla \psi_{*}$ in $L_{\rm{loc}}^{\infty}(\mathbb{R}^5 \times [t_0,\infty))$ and then 
\begin{equation}\label{move-26Sep6-11}
\begin{aligned}
&
|\psi_{*}| \le w_{\rm o}(x,t)
	=
	t^{-\frac{\gamma_1}{2} } (\ln t)^2 R_0^6 R^{-1} 
	( \1_{|x| \le t^{\frac{1}{2}}} + t |x|^{-2} \1_{|x| > t^{\frac{1}{2}} } )
\mbox{ \ in \ }
\mathbb{R}^5 \times [t_0,\infty),
\\
&
\mbox{$\psi_{*}(x,t)$ is even with respect to the $i$-th component of $x$, $i =2,3,4,5$.}
\end{aligned}
\end{equation}
By Lemma \ref{mu1-xi-lem}, $\dot{\lambda}_{1, T} \to \dot{\lambda}_{1, *}$, $\dot{\xi}_{1, T} \to \dot{\xi}^{[1, *]} = (\dot{\xi}_1^{[1, *]}, 0,0,0,0)$ in $C_{\rm{loc}}([t_0,\infty))$, and we define $\lambda_{1, *} = \lambda_{1, *}[\dot{\lambda}_{1, *}]$, $\xi^{[1, *]} = \xi^{[1, *]}[\dot{\xi}^{[1, *]}]$ as \eqref{mu1-xi-sys}. Then, similar to \eqref{qd23Dec28-1}, for $t\in [t_0,\infty)$,
\begin{equation*}
\begin{aligned}
&
|\dot{\lambda}_{1, *}| \le  R_0^{-\frac{\zeta}{2}} t^{1-\gamma_1},
\quad
|\dot{\xi}^{[1, *]}| \le   R_0^{-\frac{\zeta}{4}} t^{\frac{5}{2}-\frac{3\gamma_1}{2} - \frac{\gamma_2}{2} },
\\
&
|\lambda_{1, *}| \lesssim R_0^{-\frac{\zeta}{2}} t^{2-\gamma_1},
    \quad
    \xi^{[1, *]} = (\xi_1^{[1, *]},0,0,0,0),
	\quad
	|\xi^{[1, *]}| \lesssim R_0^{-\frac{\zeta}{4}}
	\begin{cases} 
t^{\frac{7}{2}-\frac{3\gamma_1}{2} - \frac{\gamma_2}{2}}  
			& \mbox{ \ if \ } 3 \gamma_1 + \gamma_2 \ne 7
			\\
            \ln t 
			& \mbox{ \ if \ } 3 \gamma_1 + \gamma_2 = 7.
		\end{cases}
\end{aligned}
\end{equation*}
For any compact set $K\subset [t_0, \infty)$, $\phi_{T} \to \phi_{*}$, $\nabla_{y}\phi_{T} \to \nabla_{y} \phi_{*}$ uniformly in $\{ (y,t) \mid t\in K, y\in \overline{B_{4R(t)}} \}$. Then
\begin{equation}\label{move-26Sep6-12}
\begin{aligned}
&
\langle y\rangle |\nabla \phi_{*}(y,t)| + |\phi_{*}(y,t)| \le  R_0^{6} t^{3-2\gamma_1} \ln t \langle y \rangle^{-1}
\mbox{ \ for \ } t \in [t_0,\infty), y\in \overline{B_{4 R(t)}},
\\
&
\mbox{$\phi_{*}(y,t)$ is even with respect to the $i$-th component of $y$, $i =2,3,4,5$.}
\end{aligned}
\end{equation}

Set
\begin{equation}\label{move-26Sep6-10}
\begin{aligned}
&
g_0(t_0) := g_{0,*},
\quad
\lambda := \lambda_{0} + \lambda_{1, *},
\quad
\xi := \xi^{[0]} + \xi^{[1,*]}
\mbox{ \ with the form  \ } \xi = (\xi_1, 0,0,0,0),
\\
&
u := \lambda^{-\frac{3}{2}} U\Big( \frac{x-\xi}{\lambda} \Big)
	\eta\Big( \frac{x-\xi}{\sqrt{t}} \Big)
	+ \Psi + h_{\rm{error}}, 
\quad
h_{\rm{error}} :=
\psi_{*} + \lambda^{-\frac{3}{2}}\phi_{*}\Big(\frac{x-\xi}{\lambda},t \Big) \eta\Big( \frac{x-\xi}{\lambda R} \Big).
\end{aligned}
\end{equation}
Then \eqref{u-behavior}, \eqref{move-26Aug23-6} hold. Similar to continuity argument in \eqref{qd26May1-4}, we have $u_{T} \to u$ in $L_{\rm{loc}}^{\infty}(\mathbb{R}^5 \times [t_0, \infty) )$. Given any $ f\in C_c^{\infty}(\mathbb{R}^5 \times [t_0,\infty) )$, for $T \gg 1$, we test \eqref{qd26May3-1} with $f$, and then take $T\to \infty$. It deduces that $u$ is a weak solution of  
\eqref{u-eq-5d}. By \eqref{move-26Aug24-1}, we get $u(x,t_0)$ \eqref{move-26Sep7-3}. By \eqref{move-26Sep6-11}, \eqref{move-26Sep6-12}, and $\xi = (\xi_1, 0,0,0,0)$, $u(x,t)$ is even with respect to the $i$-th component of $x$, $i =2,3,4,5$. Since $\eta_{R}(y) \le \1_{|x| \le t^{\frac{1}{2}}}$, we get $|h_{\rm{error}}| \lesssim t^{-\frac{\gamma_1}{2}} \ln t (\ln t_0)^{6} \langle t^{-1} |x|^2 \rangle^{-1}$.
\end{proof}

\subsection{Proof of Corollary \ref{move26Sep8-1-cor}}
\begin{proof}[Proof of Corollary \ref{move26Sep8-1-cor}]

{\textbf{For (i).}} 
Denote $f(t) :=  t^{-\frac{1}{2} \min\{\gamma_1, \gamma_2 \} } + t^{-\frac{\gamma_1}{2}} \ln t (\ln t_0)^{6} $ for brevity. By \eqref{u-behavior}, \eqref{qd25Dec25-4},
\begin{equation}\label{move-26Aug23-9}
		u(\xi , t) = 
		15^{\frac{3}{4}}
		\lambda^{-\frac{3}{2}} 
		+
		O( f(t) )
= 
		15^{\frac{3}{4}}
		\lambda^{-\frac{3}{2}} \big[ 1
		+
		O\big( t^{-\frac{1}{2} \min\{\gamma_1, \gamma_2 \} + 3 - \frac{3}{2} \gamma_1 } + t^{3-2 \gamma_1} \ln t (\ln t_0)^{6} \big) \big] \sim \lambda^{-\frac{3}{2}} > 0,
	\end{equation}
where $-\frac{1}{2} \min\{\gamma_1, \gamma_2 \} + 3 - \frac{3}{2} \gamma_1 < 0$ and $3-2 \gamma_1 < 0$ by the assumption \eqref{move-26Aug23-1}. In other word, for $t_0 \gg 1$,
\begin{equation}\label{move-26Aug23-10}
f(t) / (\lambda^{-\frac{3}{2}})
\ll 1.
\end{equation}

For $t>t_0$ and $\mathcal{M}_{t} = \{ w \mid u(w, t) = \sup_{x \in \mathbb{R}^5} u(x, t) \} $ as the set of maximum points of $u(\cdot,t)$, since $\lim_{|x| \to \infty} u(x,t) = 0$, then $\mathcal{M}_{t} \ne \emptyset$.
Obviously, $\max_{x\in\mathbb{R}^5} u(x,t) \ge u(\xi , t)$. On the other hand, by \eqref{u-behavior}, there exists a constant $C_1 >0$ such that
$ - C_1 f(t) \le u \le 15^{\frac{3}{4}} \lambda^{-\frac{3}{2}} + C_1 f(t)$. By \eqref{move-26Aug23-9} and \eqref{move-26Aug23-10}, $\| u(\cdot,t)\|_{L^\infty(\mathbb{R}^5)}$ is not attained by the non-positive value of $u(\cdot,t)$ and
\begin{equation}\label{move-26Aug23-11}
\max_{x\in\mathbb{R}^5} u(x,t)
=
\| u(\cdot,t)\|_{L^\infty(\mathbb{R}^5)} = 
		15^{\frac{3}{4}}
		\lambda^{-\frac{3}{2}} \big[ 1
		+
		O\big( t^{-\frac{1}{2} \min\{\gamma_1, \gamma_2 \} + 3 - \frac{3}{2} \gamma_1 } + t^{3-2 \gamma_1} \ln t (\ln t_0)^{6} \big) \big].
\end{equation}
We will estimate the location of $\mathcal{M}_{t}$. The basic idea is that for $x$ such that $u(x,t)$ is smaller than $u(\xi,t)$, then $x \notin \mathcal{M}_{t}$.
When $| \frac{x-\xi}{\lambda} | > 2^{-1}$,
\begin{equation*}
u \le 15^{\frac{3}{4}}
		\lambda^{-\frac{3}{2}} 
		\Big(\frac{5}{4} \Big)^{-\frac{3}{2}}
		+
		O(f(t)) < u(\xi, t)
\end{equation*}
by \eqref{move-26Aug23-9} and \eqref{move-26Aug23-10}, which implies that $x \notin \mathcal{M}_{t}$. When $| \frac{x-\xi}{\lambda} | \le 2^{-1}$, then $|\frac{x-\xi}{\sqrt{t}}| \lesssim t^{\frac{3}{2} - \gamma_1}$ and $\eta( \frac{x-\xi}{\sqrt{t}} ) =1$.
\begin{align*}
&
		u = 
		15^{\frac{3}{4}}
		\lambda^{-\frac{3}{2}} 
        \Big[ 1 + 
		\Big( 1+
		\Big|  \frac{x-\xi}{\lambda} \Big|^2 \Big)^{-\frac{3}{2}}
        - 1
        \Big]
		+
		O(f(t))
\\
= \ & 15^{\frac{3}{4}}
		\lambda^{-\frac{3}{2}} 
        \Big[ 1 + (-\frac{3}{2})
		\Big( 1+ \theta
		\Big|  \frac{x-\xi}{\lambda} \Big|^2 \Big)^{-\frac{5}{2}}
        \Big|  \frac{x-\xi}{\lambda} \Big|^2
        \Big]
		+
		O(f(t))
\\
\le \ & 15^{\frac{3}{4}}
		\lambda^{-\frac{3}{2}}
        -
        15^{\frac{3}{4}}
		\lambda^{-\frac{3}{2}} \frac{3}{2}
		\Big( \frac{5}{4} \Big)^{-\frac{5}{2}}
        \Big|  \frac{x-\xi}{\lambda} \Big|^2
		+
		O(f(t))
\end{align*}
for some $\theta\in [0,1]$. If
\begin{equation}\label{move-26Aug23-3}
\lambda^{-\frac{3}{2}} \Big|  \frac{x-\xi}{\lambda} \Big|^2 \ge (\ln \ln t)^{\frac{1}{9}} f(t)
\end{equation}
holds, then $x \notin \mathcal{M}_{t}$ since $u - u(\xi , t) \le O(f(t)) - 15^{\frac{3}{4}} \frac{3}{2} ( \frac{5}{4} )^{-\frac{5}{2}} (\ln \ln t)^{\frac{1}{9}} f(t) < 0$. To make \eqref{move-26Aug23-3} hold, it suffices to ensure
\begin{equation*}
\begin{aligned}
&
\lambda^{-\frac{3}{2}} \Big|  \frac{x-\xi}{\lambda} \Big|^2 \ge (\ln \ln t)^{\frac{1}{8}}
\ln t (\ln t_0)^{6} t^{-\frac{1}{2} \min\{\gamma_1, \gamma_2 \} },
\\
& 
\mbox{which can be satisfied by }
|x-\xi| \ge C_1
[(\ln \ln t)^{\frac{1}{8}}
\ln t (\ln t_0)^{6}]^{\frac{1}{2}} t^{\frac{7}{2} -\frac{7}{4} \gamma_1 -\frac{1}{4} \min\{\gamma_1, \gamma_2 \}}
\end{aligned}
\end{equation*}
with a sufficiently large constant $C_1 >0$. Thus,
\begin{equation}\label{move-26Aug23-14}
|x- \xi| \le [(\ln \ln t)
\ln t (\ln t_0)^{6}]^{\frac{1}{2}} t^{\frac{7}{2} -\frac{7}{4} \gamma_1 -\frac{1}{4} \min\{\gamma_1, \gamma_2 \}}
\mbox{ \ for \ } x\in \mathcal{M}_{t}.
\end{equation}

By Lemma \ref{Psi-maxpoint-lem}, the set of maximum points of $\Psi(\cdot,t)$ is not empty and the lower bound of the location of the maximum points of $\Psi(\cdot,t)$ is given by $\min\{ t^{\frac{1}{2}},  |c_{\sharp}| t^{\frac{\gamma_1 - \gamma_2}{2} + \frac{1}{2}}  \}$ up to a constant multiplicity. By \eqref{move-26Aug23-14}, the upper bound of the location of the maximum points of $u(\cdot, t)$ is given by $|\xi| + [(\ln \ln t)
\ln t (\ln t_0)^{6}]^{\frac{1}{2}} t^{\frac{7}{2} -\frac{7}{4} \gamma_1 -\frac{1}{4} \min\{\gamma_1, \gamma_2 \}}$. By \eqref{move-26Aug23-1}, 
\begin{equation}
|\xi| + [(\ln \ln t)
\ln t (\ln t_0)^{6}]^{\frac{1}{2}} t^{\frac{7}{2} -\frac{7}{4} \gamma_1 -\frac{1}{4} \min\{\gamma_1, \gamma_2 \}}
\lesssim
t^{-\epsilon_1} \min\{ t^{\frac{1}{2}},  |c_{\sharp}| t^{\frac{\gamma_1 - \gamma_2}{2} + \frac{1}{2}}  \}
\end{equation}
with a constant $0< \epsilon_1 \ll 1$.

{\textbf{For (ii).}}
The additional assumption $\gamma_1 > \gamma_2$ implies $[(\ln \ln t)
\ln t (\ln t_0)^{6}]^{\frac{1}{2}} t^{\frac{7}{2} -\frac{7}{4} \gamma_1 -\frac{1}{4} \min\{\gamma_1, \gamma_2 \}} \le |\xi|/2$. Then \eqref{move-26Aug23-12} follows from \eqref{move-26Aug23-14}. By \eqref{move-26Aug23-6}, $|\xi_1| \to \infty$ as $t\to \infty$ if $3\gamma_1 + \gamma_2 \le 7$, which implies the maximum points of $u(\cdot,t)$ goes to infinity as $t\to \infty$. The parameter restrictions \eqref{move-26Aug23-1}, $\gamma_1 > \gamma_2$, $3\gamma_1 + \gamma_2 \le 7$ are solvable. See Mathematica \href{https://raw.githubusercontent.com/qdz-wonderful/Mathematica-codes-for-One-bubble-solution-with-a-moving-maximum-point/refs/heads/main/Moving%20maximum%20point-parameters%20in%20Cor%20(ii).wl}{code link} here.
\end{proof}

\subsection{Proof of Remark \ref{move-26Aug24-2-rmk} (\ref{move26Sep7-4-rmk})}\label{move-26Sep8-2-subsec}

\begin{proof}[Proof of Remark \ref{move-26Aug24-2-rmk} \eqref{move26Sep7-4-rmk}]

When $\gamma_1 = \gamma_2$ and $0<|c_{\sharp}|<1$, we have $\Psi(x,0) \ge (1-|c_{\sharp}|) |x|^{-\gamma_1}$ and then $\Psi > 0$. Recall $u(x,t_0)$ given in \eqref{move-26Sep7-3}. Notice that $(\ln t_0) \lambda(t_0) \ll \sqrt{t_0}$.
Denote $y(t_0)=\frac{x-\xi(t_0)}{\lambda(t_0)}$. There exists a large constant $C_1>0$ such that
\begin{equation*}
15^{\frac{3}{4}} \langle y(t_0) \rangle^{-3} + g_{0}(t_0) \eta( 2y(t_0) / \ln t_0) Z_{0}(y(t_0))
\ge 
15^{\frac{3}{4}} \langle y(t_0) \rangle^{-3}
-
C_1
(\ln t_0) t_0^{3-2\gamma_1} Z_0(y(t_0)) \1_{|y(t_0)| \le \ln t_0} > 0
\end{equation*} 
by $\gamma_1>3/2$ and $t_0\gg 1$.
With $\Psi > 0$, we have $u(x,t_0)>0$, and then $u>0$ by the maximum principle.
\end{proof}

\section{Re-gluing inner linear theory}\label{re-gluing-Sec}

This section presents the re-gluing inner linear theory, which is {\textbf{independent of the other sections}} of this paper. Some symbols, like $\tau_0$, $y$, $R$, $R_0$, $\Psi$, appearing in other sections are abused, but there is no relationship between them. The proof merge the arguments in \cite[Section 7]{Green16JEMS}, \cite[Section 7]{infi4d}, \cite[Section 8]{Wei-Zhang-Zhou2022LLG}, and \cite[Proposition 7.2]{TriHMF2026}. We provide details for completeness and supplement the argument for even symmetry in higher modes. For the first reading, readers may focus on the typical argument for mode $0$.

In this section, we always denote
\begin{equation}
p=\frac{n+2}{n-2},\quad
U(y) 
= [n(n-2)]^{\frac{n-2}{4} } ( 1+|y|^2 )^{-\frac{n-2}{2} }.
\end{equation}
Recall that the linearized operator $\Delta + pU^{p-1}$
has only one positive eigenvalue $\gamma_0>0$ such that
\begin{equation}\label{def-Z0Z0}
\Delta Z_0 + pU^{p-1} Z_0=\gamma_0 Z_0,
\end{equation}
where the corresponding eigenfunction $Z_0 \in L^{\infty}(\R^n)$ is radially symmetric and $|Z_0| \le C \rme^{-c|y|}$ with some positive constants $c, C$. The bounded kernels of $\Delta + pU^{p-1}$ are given by $Z_{i}$, $i\in \overline{1,n+1}$ in \eqref{Zi-def}.

Define the weighted $L^{\infty}$ norm
\begin{equation}
	\|h\|_{v, a}^{\mathcal{R}} :=
	\inf\big\{ C \mid |h(y,\tau)| \le C v(\tau) \langle y \rangle^{-a} \mbox{ holds for all } \tau \in (\tau_0, \tau_1), |y| < \mathcal{R}(\tau) \big\},
\end{equation}
where $a \in \mathbb{R}$, $1\le \tau_0 < \tau_1 \le \infty$, $v=v(\tau) \ge 0$, $\mathcal{R} = \mathcal{R}(\tau) >0$ are some functions defined in $(\tau_0, \tau_1)$. We abuse the symbol $\|h\|_{v, a}^{\infty}$ to denote the norm when $\mathcal{R} = \infty$. We emphasize that in this section, {\textbf{all constants are independent of the choice of $\tau_0, \tau_1$}}, and functions do not take the value $\pm \infty$ unless otherwise specified.

The following coercive estimate is from \cite[Lemma 7.2]{Green16JEMS} when $n\ge 5$. The lower dimensions $n=3,4$ can be handled similarly.
\begin{lemma}\label{eigenvalue-problem} Given an integer $n\ge 3$, there exists a constant $c_0 >0$ such that for all sufficiently large $R$ and all radially symmetric functions $f \in H_0^1(B_{R})$
	with  $\int_{B_{R}} f Z_0 \rmd y = 0$, we have
	\begin{equation}\label{Coer-est}
	c_0  \lambda_{R}  \int_{B_{R}} |f|^2 \rmd y \le \int_{B_{R}} ( |\nabla f|^2 - pU^{p-1}|f|^2 ) \rmd y, 
	\mbox{ \ where \ }
	\lambda_{R} :=
	\begin{cases}
		R^{-2} 
		& \mbox{ \ if \ } n=3
		\\
		(R^2 \ln R)^{-1} 
		& \mbox{ \ if \ } n=4
		\\
		R^{2-n}
		& \mbox{ \ if \ } n\ge 5.
	\end{cases}
	\end{equation}
\end{lemma}

We improve \cite[Lemma 7.3]{Green16JEMS}, \cite[Lemma 7.4]{infi4d} into the following form. Recall \eqref{SH-def} about the orthonormal basis $\mathbf{SH}$ in $L^2(S^{n-1})$ made up of spherical harmonic functions $\Upsilon_{i,j}$. We agree that
\begin{equation}
\begin{aligned}
f(y) = f_1(|y|) \Upsilon_{i,j}(\frac{y}{|y|}) 
&
\mbox{ means that $f$ can be written as}
\\
&
\mbox{ a multiplicity of a radially symmetric function $f_1$ and $\Upsilon_{i,j}$.}
\end{aligned}
\end{equation}

\begin{lemma}\label{chiM-eq-lem} Given an integer $n\ge 3$, $1\le \tau_0 <\tau_1 \le \infty$, functions $R(\tau) \ge 2$, $v(\tau) \ge 0$ defined in $(\tau_0, \tau_1)$, consider
	\begin{equation*}
		\begin{cases}
			\partial_{\tau} \phi = \Delta \phi + V(|y|) \phi + h 
			\mbox{ \ for \ } \tau \in (\tau_0, \tau_1), y\in B_{R},
			\\
			\phi  = 0 \mbox{ \ for \ } \tau \in (\tau_0, \tau_1), y\in \partial B_{R},
			\quad
			\phi(\cdot, \tau_0)= 0
			\mbox{ \ in \ } B_{R(\tau_0)},
		\end{cases}
	\end{equation*}
	where $h$ is bounded locally in time, $V(|y|)$ is a radially symmetric and bounded function. There exists a unique solution $\phi = \phi[h]$ linearly depending on $h$.

$(1).$ If $h = h_1(|y|,\tau) \Upsilon_{i,j}(\frac{y}{|y|})$ with some $\Upsilon_{i,j} \in \mathbf{SH}$, then $\phi = \phi_1(|y|,\tau) \Upsilon_{i,j}(\frac{y}{|y|})$.

$(2).$ If $\int_{S^{n-1}} h(r w,\tau) \Upsilon_{i,j}(w) \rmd w = 0$ with some $\Upsilon_{i,j} \in \mathbf{SH}$ for $\tau \in (\tau_0, \tau_1)$, $r \in (0, R)$, then $\int_{S^{n-1}} \phi(r w,\tau) \Upsilon_{i,j}(w) $ $ \rmd w = 0$ for $\tau \in (\tau_0, \tau_1)$, $r \in (0, R)$.

$(3).$ Suppose that $\|h\|_{v,a}^{R} <\infty$, $a \in \mathbb{R}$, and there exists $g(r)$ such that $\big(\frac{\rmd^2}{\rmd r^2} + \frac{n-1}{r} \frac{\rmd}{\rmd r} + V(r) \big) g(r) = 0$ and $C_1^{-1} \le g(r) \le C_1$ with a constant $C_1 \ge 1$ for $r\in (0,\infty)$; and one of the following cases holds,
	
	Case 1: $\dot{v} \ge 0$ and $ \dot{R} \ge 0$;

	Case 2: $|\dot{v}|  \begin{cases}
		R^{2}  &\mbox{ if } a \in (-\infty,n) \setminus \{2\}
		\\
		R^2 \ln R                   
		&\mbox{ if } a \in \{2, n\}
		\\
		R^{2+\epsilon}
		&\mbox{ if } a>n
	\end{cases}
	\le c v$ and 
	$\begin{cases}
		|\dot{R}| R 
		&
		\mbox{ if } a <n
		\\
		|\dot{R}| R \ln R
		&
		\mbox{ if }  a = n
		\\
		|\dot{R}| R^{1+\epsilon}
		&
		\mbox{ if }  a > n
	\end{cases}
	\le c$ with a constant $\epsilon>0$ and a sufficiently small constant $c>0$. Then there exists a constant $C_2>0$ depending on $n, C_1, a$, and additionally on $\epsilon$ if $a>n$ such that
	\begin{equation}\label{Theta0-def}
		|\phi| \le C_2 \|h \|_{v,a}^{R} v
		\Theta_{R, a}^0(|y|),
		\mbox{ where } \Theta_{R, a}^0(|y|) := 
		\begin{cases}
			R^{2-a} &\mbox{ \ if \ } a<2 
			\\
			\ln R                      
			&\mbox{ \ if \ } a=2 
			\\
			\langle |y| \rangle^{2-a} 
			&  \mbox{ \ if \ } 2<a<n
			\\
			\langle |y| \rangle^{2-n}\ln(|y|+2)   
			&\mbox{ \ if \ } a = n
			\\
			\langle |y| \rangle^{2-n}   
			&\mbox{ \ if \ } a > n.
		\end{cases}
	\end{equation}

$(4).$ If $h$ is even with respect to the $i$-th component of $y$ for some $i \in \{1,2,\dots, n\}$ and all $\tau \in (\tau_0, \tau_1)$, then so is $\phi$.
\end{lemma}

\begin{remark}\label{qd26Mar13-3-rmk}

For convenience of application, we note that when $|\dot{v}| = O(\tau^{-1} v)$, $|\dot{R}| = O(\tau^{-1} R)$, and $R^{2+\epsilon} \ll \tau$ with $\epsilon>0$, Case $2$ always holds.
\end{remark}

\begin{proof}

$(1).$ Plug $\phi = \phi_1(|y|,\tau) \Upsilon_{i,j}(\frac{y}{|y|})$. It suffices to solve
	\begin{equation*}
		\begin{cases}
			\begin{aligned}
				\partial_{\tau} \phi_1(|y|,\tau) = & \Big[ \partial_{|y| |y|} + \frac{n-1}{|y|} \partial_{|y|}
				+
				\frac{- i (n-2+i)}{|y|^2} \Big] \phi_1(|y|,\tau) 
				\\
				&
				+ V(|y|) \phi_1(|y|,\tau) + h_1(|y|,\tau) 
				\mbox{ \ for \ } \tau \in (\tau_0, \tau_1), y\in B_{R},
			\end{aligned}
			\\
			\phi_1(|y|,\tau) = 0 \mbox{ \ for \ } \tau \in (\tau_0, \tau_1), y\in \partial B_{R},
			\quad
			\phi_1(|y|,\tau_0) = 0
			\mbox{ \ for \ } y \in B_{R(\tau_0)}.
		\end{cases}
	\end{equation*}
	By uniqueness, we conclude.

$(2).$	Denote $f(r,\tau) := \int_{S^{n-1}} \phi(r w,\tau) \Upsilon_{i,j}(w) \rmd w$. We multiply $\Upsilon_{i,j}$ and integrate in $S^{n-1}$. Then
	\begin{equation*}
		\begin{cases}
			\partial_{\tau} f
			= 
			\big[ \partial_{rr} + \frac{n-1}{r} \partial_{r} +
			\frac{- i (n-2+i)}{r^2} \big] f 
			+ V(r) f
			\mbox{ \ for \ } \tau \in (\tau_0, \tau_1), r \in (0,R),
			\\
			f(R,\tau) = 0 \mbox{ \ for \ } \tau \in (\tau_0, \tau_1),
			\quad
			f(r,\tau_0) = 0
			\mbox{ \ for \ } r \in (0,R(\tau_0)).
		\end{cases}
	\end{equation*}
	Notice that $|f(0,\tau)| < \infty$. By uniqueness, $f \equiv 0$.

$(3).$ Assume $\|h \|_{v,a}^{R} \ne 0$. Otherwise $\phi\equiv 0$ and the result is trivial.
	
	Denote $r=|y|$, $L [\phi] := \Delta \phi + V(|y|) \phi$, $P[\phi] := L[\phi] + h - \partial_{\tau} \phi $. Set a supersolution $\bar{\phi}_1  := C v(\tau) f_a(r,R)$ with a constant $C>0$ to be determined later, where 
	\begin{equation}
		f_a(r,R) := g(r) \int_r^{R} \frac{\rmd \rho}{g^2(\rho) \rho^{n-1}} \int_{0}^{\rho} g(s)s^{n-1} \langle s \rangle^{-a} \rmd s
        \mbox{ \ satisfies \ } L f_a(r,R) = - \langle r \rangle^{-a}.
	\end{equation}
By patient calculation, for $r\in [0,R]$,
	\begin{equation}
		\langle r\rangle^{a} f_a(r,R) \lesssim 
		\langle r\rangle^{a} 
		\begin{cases}
			R^{2-a} 
			&\mbox{ \ if \ } a<2 
			\\
			\ln R    &\mbox{ \ if \ } a=2 
			\\
			\langle r\rangle^{2-a}  &\mbox{ \ if \ } 2<a<n 
			\\
			\langle r\rangle^{2-n}\ln(r+2) 
			&\mbox{ \ if \ } a=n  \\
			\langle r\rangle^{2-n}  &\mbox{ \ if \ } a>n
		\end{cases}
		\lesssim
		\begin{cases}
			R^{2-a}  
			& \mbox{ \ if \ } a<0 
			\\
			R^{2}  &\mbox{ \ if \ } 0\le a<2  
			\\
			R^2 \ln R                   &\mbox{ \ if \ } a=2 
			\\
			R^2 
			&\mbox{ \ if \ } 2<a<n  
			\\
			R^2\ln R 
			&\mbox{ \ if \ } a=n 
			\\  
			R^{2-n+a}
			&\mbox{ \ if \ } a>n.
		\end{cases}
	\end{equation}
	Direct calculation gives
\begin{equation}\label{qd26Apr24-1}
		\begin{aligned}
			&
			P[\bar{\phi}_1] 
			=  
				- C v\langle r\rangle^{-a} + h - C \dot{v} f_a(r,R) 
				- \frac{C v g(r) \dot{R}}{g^2(R) R^{n-1}}
				\int_0^{R} g(s)s^{n-1}
				\langle s\rangle^{-a}  \rmd s
			\\
			\le \ &  C \langle r\rangle^{-a} \Big( - v + C^{-1} v \|h\|_{v,a}^{R} - \dot{v} \langle r\rangle^{a} f_a(r,R) 
			- \frac{v \langle r\rangle^{a} g(r) \dot{R}}{g^2(R) R^{n-1}}
			\int_0^{R} g(s)s^{n-1}
			\langle s\rangle^{-a}  \rmd s \Big).
		\end{aligned}
	\end{equation}
	If Case $1$ holds, then $ P[\bar{\phi}_1] \le 
	C \langle r\rangle^{-a} ( - v + C^{-1} v \|h\|_{v,a}^{R} ) \le 0 $
	when we take $C=\|h\|_{v,a}^{R}$.

	Hereafter, we always assume Case $2$ holds. For $0\le a\le n $,
	\begin{equation*}
		\begin{aligned}
			&
			|\dot{v} \langle r \rangle^{a} f_a(r,R) | \lesssim |\dot{v}|
			\begin{cases}
				R^{2}  & \mbox{ \ if \ } 0 \le a<2 \mbox{ or } 2<a<n
				\\
				R^2 \ln R                   &\mbox{ \ if \ } a=2 \mbox{ or } n,
			\end{cases}
			\\
			&
			\Big|
			\frac{\langle r \rangle^{a}  g(r) \dot{R}}{g^2(R) R^{n-1}}
			\int_0^{R} g(s)s^{n-1} \langle s\rangle^{-a} \rmd s \Big|
			\sim 
			\Big|
			\frac{\langle r \rangle^{a} \dot{R}}{ R^{n-1}}
			\int_0^{R} s^{n-1} \langle s\rangle^{-a} \rmd s
			\Big|
			\lesssim 
			\begin{cases}
				|\dot{R}| R
				&
				\mbox{ \ if \ } 0\le a <n
				\\
				|\dot{R}| R \ln R
				&
				\mbox{ \ if \ }  a =n.
			\end{cases}
		\end{aligned}
	\end{equation*}
	By Case $2$, we have $P[\bar{\phi}_1] \le C \langle r\rangle^{-a} ( - 2^{-1} v + C^{-1} v \|h\|_{v,a}^{R} ) \le 0$ when we take $C= 2\|h\|_{v,a}^{R}$.

	For $a > n$, we take $\bar{\phi}_2 = C v(\tau) f_b(r,R)$ with $b \in (n, \min\{a, n+\epsilon\}]$.  Same as \eqref{qd26Apr24-1},
	\begin{equation*}
		P[\bar{\phi}_2]
		\le
		C \langle r\rangle^{-b} \Big( - v - \dot{v} \langle r \rangle^{b} f_b(r,R)  
		- 
		\frac{v \langle r \rangle^{b}  g(r) \dot{R} }{g^2(R) R^{n-1}}
		\int_0^{R} g(s)s^{n-1} \langle s\rangle^{-b} \rmd s
		\Big)
		+ 
		v \langle r\rangle^{-a}  \|h\|_{v,a}^{R}.
	\end{equation*}
	Notice that $
	|\dot{v} \langle r \rangle^{b} f_b(r,R) |
	 \lesssim |\dot{v}| R^{b-n+2} \le |\dot{v}| R^{2+\epsilon} \le c v$,
	$
	\big| \frac{\langle r \rangle^{b}  g(r) \dot{R} }{g^2(R) R^{n-1}}
	\int_0^{R} g(s)s^{n-1} \langle s\rangle^{-b} \rmd s \big| 
	 \lesssim  |\dot{R}| R^{b-n+1} \le |\dot{R}| R^{1+\epsilon} \le c
	$ by Case $2$. For $c\ll 1$, $ P[\bar{\phi}_2] \le - 2^{-1} C v \langle r\rangle^{-b} + v \langle r\rangle^{-a}  \|h\|_{v,a}^{R} \le 0 $ when we take $C= 2\|h\|_{v,a}^{R}$.

	For $a<0$, we take $\bar{\phi}_3 := C v(\tau) R^{-a} f_0(r,R)$.  Then
	\begin{equation*}
		\begin{aligned}
			&	
			P[\bar{\phi}_3] 
			=  
				- C v R^{-a} + h - C \dot{v} R^{-a} f_0(r,R) 
				- \frac{C v R^{-a} g(r) \dot{R}}{g^2(R) R^{n-1}}
				\int_0^{R} g(s)s^{n-1}
				\rmd s
				+
				a C v R^{-a-1} \dot{R} f_0(r,R)
			\\
			\le \ & 
			C R^{-a} \Big( -v - \dot{v} f_0(r,R) - 
			\frac{v g(r) \dot{R}}{g^2(R) R^{n-1}}
			\int_0^{R} g(s)s^{n-1}  \rmd s
			+ a v R^{-1} \dot{R} f_0(r,R)
			\Big) 
			+ 
			v \langle r\rangle^{-a}  \|h\|_{v,a}^{R}.
		\end{aligned}
	\end{equation*}
	Therein, $ |\dot{v} f_0(r,R) | \lesssim |\dot{v}| R^{2} $,
	$
	\big| \frac{g(r) \dot{R}}{g^2(R) R^{n-1}}
	\int_0^{R} g(s)s^{n-1}  \rmd s 
	\big| +
	| a R^{-1} \dot{R} f_0(r,R) |
	\lesssim |\dot{R}| R
	$.
	By Case 2, then
	$ P[\bar{\phi}_3] \le 
	- 2^{-1} C v R^{-a} + 
	v \langle r\rangle^{-a}  \|h\|_{v,a}^{R} \le 0 $
	when we take $C= 2^{1-\frac{a}{2}}\|h\|_{v,a}^{R}$, where we use $1 \le R^2$.

(4). It is deduced by uniqueness.
\end{proof}

\subsection{Mode $0$}

The linear theory for mode $0$ without orthogonality.
\begin{lemma}\label{mode0-nonorth}

Given an integer $n\ge 3$, $1\le \tau_0 <\tau_1 \le \infty$, functions $R(\tau) \ge 2$, $v(\tau) \ge 0$ defined in $(\tau_0, \tau_1)$, consider
\begin{equation}\label{m0-eq-nonorth}
		\partial_{\tau} \phi =  \Delta \phi + p U^{p-1} \phi
		+ h
		\mbox{ \ for \ } \tau \in (\tau_0, \tau_1), y \in B_{R},
		\quad
		\phi(\cdot,\tau_0) =  g_0 Z_0(y)
		\mbox{ \ in \ } B_{R(\tau_0)},
\end{equation}
where $h = h_1(|y|,\tau) \Upsilon_{0,1}(\frac{y}{|y|})$, $\| h \|_{v,a}^{R} < \infty$, 
$a \in \mathbb{R}$. Suppose that
\begin{equation}\label{25Mar14-6}
\begin{aligned}
&
|\dot{v}| = O(\tau^{-1} v), \ 
|\dot{R}| = O(\tau^{-1} R), \ 
v, R, \ln R \in \mathbf{AP}((\tau_0,\infty)), \  
R^{2+\epsilon} = O(\tau) \mbox{ with a constant } \epsilon >0,
\\ 
&
\mbox{either Case 1: $\mathbf{P}_1[\lambda_R] > -1 $ or Case 2: $\mathbf{P}_1[\lambda_R] < -1$ and $\mathbf{P}_1[\lambda_R^{-1} ( v \theta_{R,a}^0 )^2] > -1$ holds}  
\end{aligned}
\end{equation}
with 
\begin{equation*}
\lambda_{R} =
\begin{cases}
	R^{-2} 
	& \mbox{ \ if \ } n=3
	\\
	(R^2 \ln R)^{-1} 
	& \mbox{ \ if \ } n=4
	\\
	R^{2-n}
	& \mbox{ \ if \ } n\ge 5,
\end{cases}
\end{equation*}
then for $\inf\limits_{s \ge \tau_0} R(s) \ge C_0$ with a constant $C_0>0$ sufficiently large, there exists a solution  $(\phi, g_0)= ( \phi[h], g_0[h])$ linearly depending on $h$ with the estimate
\begin{equation*}
\langle y\rangle |\nabla \phi|	+|\phi| \le C v \big[ \min\{ \tau^{\frac 12}, \lambda_R^{- \frac 12 } \}  \lambda_R^{-\frac 12} \theta_{R,a}^0  \langle y \rangle^{2-n}    + 	 \Theta_{R, a}^0( |y| ) \big] \|h\|_{v,a}^{R},
\quad
|g_0|\le C v(\tau_0) \theta_{R(\tau_0),a}^0 \| h\|_{v,a}^{R},
\end{equation*}
where
\begin{equation}\label{little-theta0-def}
\theta_{R,a}^0  := 
\begin{cases}
	R^{2-a}  & \mbox{ \ if \ }  a<2 
	\\
	\ln R 
	& \mbox{ \ if \ } a=2 
	\\
	1 & \mbox{ \ if \ } a>2,
\end{cases}
\quad
\Theta_{R, a}^0(|y|) = 
\begin{cases}
	R^{2-a} &\mbox{ \ if \ } a<2 
	\\
	\ln R                      
	&\mbox{ \ if \ } a=2 
	\\
	\langle |y| \rangle^{2-a} 
	&  \mbox{ \ if \ } 2<a<n
	\\
	\langle |y| \rangle^{2-n}\ln(|y|+2)   
	&\mbox{ \ if \ } a = n
	\\
	\langle |y| \rangle^{2-n}   
	&\mbox{ \ if \ } a > n,
\end{cases}
\end{equation}
$C>0$ is a constant independent of $\tau_0, \tau_1, R$.
Moreover, $\phi = \phi_1(|y|,\tau) \Upsilon_{0,1}(\frac{y}{|y|})$.

\end{lemma}

\begin{remark}\label{qd26Mar15-1-rmk}

For convenience of application, we note that given $n\ge 3$, $a\in \mathbb{R}$, there exists a constant $C>0$ such that
$\Theta_{R, a}^0(|y|) \le C  \lambda_R^{-1} \theta_{R,a}^0  \langle y \rangle^{2-n}$ for all $|y|\le R$, $R\ge 2$.

\end{remark}

\begin{proof}

This proof is a refined version of
\cite[pp. 334-336]{Green16JEMS} with some modification for $\mathbf{AP}((\tau_0,\infty))$ class. We give details for completeness. 
\begin{equation}\label{chiM-def}
\begin{aligned}
&
\mbox{Denote $\chi_{M}(|y|)$ as a radially symmetric smooth function satisfying $0\le \chi_{M}(|y|) \le 1$}
\\
&
\mbox{and $\chi_{M}(|y|) = \begin{cases}
1 & \mbox{ if } |y| \le M
\\
0 & \mbox{ if } |y| \ge M + 1
\end{cases}$
with a large constant $M>0$.}
\end{aligned}
\end{equation}
$\frac{\rmd^2}{\rmd r^2} + \frac{n-1}{r} \frac{\rmd}{\rmd r} + p U(r)^{p-1} (1 - \chi_{M}(r))$ has a positive radially symmetric kernel $\mathcal{K}(r)$ satisfying $\mathcal{K}(r) \sim 1$ for $r \in (0,\infty)$. Under the assumption in Lemma \ref{chiM-eq-lem} as well as Remark \ref{qd26Mar13-3-rmk}, $\mathcal{T}_*[h] = (\mathcal{T}_*[h])_{1}(|y|,\tau) \Upsilon_{0,1}(\frac{y}{|y|})$ is given by Lemma \ref{chiM-eq-lem} satisfying
\begin{equation}\label{phi*-h0-est}
\partial_{\tau} \mathcal{T}_*[h]
	= \Delta \mathcal{T}_*[h] 
	+ p U^{p-1} (1- \chi_{M}(|y|) ) \mathcal{T}_*[h]
	+ h
	\mbox{ \ for \ } \tau \in (\tau_0, \tau_1), y \in B_{R},
    \quad
	|\mathcal{T}_*[h]| \lesssim
	v \Theta_{R, a}^0(|y|)
	\|h\|_{v,a}^{R}.
\end{equation}

We decompose
$ \phi = \mathcal{T}_*[h] + \tilde{\phi} $. Then it suffices to solve
\begin{equation*}
	\partial_{\tau} \tilde{\phi}
	= \Delta \tilde{\phi}  
	+ p U^{p-1} \tilde{\phi} + p U^{p-1} \chi_{M}(|y|) \mathcal{T}_*[h] 
	\mbox{ \ for \ } \tau \in (\tau_0, \tau_1), y \in B_{R}.
\end{equation*}
Take $\tilde{\phi} = \tilde{\phi}_1(y,\tau) + g(\tau) Z_0(y)$. Using \eqref{def-Z0Z0}, we will find a solution $(\tilde{\phi}_1, g)$ for the equation
\begin{equation}\label{tildephi-h0}
	\begin{cases}
		\partial_{\tau} \tilde{\phi}_1  = \Delta \tilde{\phi}_1 + pU^{p-1} \tilde{\phi}_1  
		-  \dot{g} Z_0 + \gamma_0 g Z_0
		+ pU^{p-1} \chi_{M}(|y|) \mathcal{T}_*[h] \mbox{ \ for \ } \tau \in (\tau_0, \tau_1), y \in B_{R},
		\\
		\tilde{\phi}_1 = 0 
		\mbox{ \ for \ } \tau \in (\tau_0, \tau_1), y \in \partial B_{R}, \quad
		\tilde{\phi}_1(\cdot,\tau_0) = 0 
			\mbox{ \ in \ }  B_{R(\tau_0)}, \quad
		\int_{B_{R} } 
		\tilde{\phi}_1(y, \tau )  Z_0(y)  \rmd y = 0
		 \mbox{ \ for \ } \tau \in (\tau_0, \tau_1).
	\end{cases}
\end{equation}

Multiplying $Z_{0}$ and integrating by parts, since $\tilde{\phi}_1 = 0$ for $\tau \in (\tau_0, \tau_1), y \in \partial B_{R}$, by \eqref{def-Z0Z0}, we have
\begin{equation*}
	\begin{aligned}
		&
	\partial_{\tau} \int_{B_{R}} \tilde{\phi}_1 Z_0 \rmd y  = 	\int_{B_{R}}\partial_{\tau} \tilde{\phi}_1 Z_0 \rmd y 
		= 
		\gamma_0
		\int_{B_{R}} \tilde{\phi}_1  Z_0 \rmd y + 
		\int_{\partial B_{R}} Z_{0} \partial_{n} \tilde{\phi}_1 \rmd S
\\
& - (\dot{g} -\gamma_0 g)
		\int_{B_{R}}   Z_0^2 \rmd y
		+ \int_{B_{R}}  pU^{p-1} \chi_{M}(|y|) \mathcal{T}_*[h] Z_0 \rmd y.
	\end{aligned}
\end{equation*}

By $\tilde{\phi}_1(\cdot,\tau_0) = 0 $, the orthogonality $\int_{B_{ R} } 
\tilde{\phi}_1 Z_0 \rmd y = 0$ holds for all $\tau>\tau_0$ if and only if
\begin{equation}\label{qd26Apr25-1}
\dot{g} - \gamma_0 g
=
\Big( \int_{B_{R}}   Z_0^2 \rmd y \Big)^{-1}
\Big[
\int_{\partial B_{R}} Z_{0} \partial_{n} \tilde{\phi}_1 \rmd S
+ \int_{B_{R}}  pU^{p-1} \chi_{M}(|y|) \mathcal{T}_*[h] Z_0 \rmd y
\Big].
\end{equation}
We take
\begin{equation*}
\begin{aligned}
g = \ & - \rme^{\gamma_0 \tau} 
		\int_{\tau}^{\tau_1} \rme^{- \gamma_0 s} 
 \Big( \int_{B_{R(s)}}   Z_0^2(y) \rmd y \Big)^{-1}
 \\
& \quad \times
\Big[
\int_{\partial B_{R(s)}} Z_{0}(y) \partial_{n} \tilde{\phi}_1(y,s) \rmd S_{y}
+ \int_{B_{R(s)}}  pU^{p-1}(y) \chi_{M}(|y|) \mathcal{T}_*[h](y,s) Z_0(y) \rmd y
\Big] \rmd s.
\end{aligned}
\end{equation*}

Set
\begin{equation*}
	\|\tilde{\phi}_{1} \|_{w} = \sup\limits_{\tau\in (\tau_0,\tau_1)} 
	\big[\big( \min\{ \tau^{\frac 12}, \lambda_R^{-\frac 12} \} 
	\lambda_R^{-\frac 12}  v \theta_{R,a}^0  \big)(\tau) \big]^{-1}
	\big( \| \tilde{\phi}_{1}(\cdot,\tau) \|_{L^{\infty} (B_{R(\tau)})} + \| \langle \cdot \rangle \nabla \tilde{\phi}_{1}(\cdot,\tau) \|_{L^{\infty} (B_{R(\tau)})} \big).
\end{equation*}

By \eqref{phi*-h0-est}, it is straightforward to get 
\begin{equation}\label{e-t-est}
	\begin{aligned}
	&	|g|
		\lesssim  
		\rme^{\gamma_0 \tau} 
		\int_{\tau}^{\tau_1} \rme^{-\gamma_0 s} 
		\Big( 
		\rme^{-c R(s) }
		\| \nabla \tilde{\phi}_1(\cdot,s) \|_{L^{\infty} (B_{R(s)}  )}
		+ v(s) \theta_{R,a}^0(s) \|h\|_{v,a}^{R}
		\Big) \rmd s
		\\
		\lesssim \ &
\rme^{\gamma_0 \tau} 
\int_{\tau}^{\tau_1} \rme^{-\gamma_0 s} 
\Big( \rme^{-c \inf\limits_{s \ge \tau_0} R(s) }
 \min\{ s^{\frac 12}, \lambda_R^{-\frac 12}(s) \} 
\lambda_R^{-\frac 12}(s)  v(s) \theta_{R,a}^0(s) 
\| \tilde{\phi}_1\|_{w}
+ v(s) \theta_{R,a}^0(s) \|h\|_{v,a}^{R}
\Big) \rmd s	
\\
\lesssim \ &
\rme^{-c \inf\limits_{s \ge \tau_0} R(s)} \min\{ \tau^{\frac 12}, \lambda_R^{-\frac 12} \} 
\lambda_R^{-\frac 12}  v \theta_{R,a}^0
\| \tilde{\phi}_1\|_{w}
+ v \theta_{R,a}^0 \|h\|_{v,a}^{R}
	\end{aligned}
\end{equation}
for some small constant $c>0$, where for last step, we require $v, R, \ln R \in \mathbf{AP}((\tau_0,\infty))$, which implies $ 
s^{\frac 12} 
\lambda_R^{-\frac 12}(s)  v(s) \theta_{R,a}^0(s)$,   $\lambda_R^{-\frac 12}(s) 
\lambda_R^{-\frac 12}(s)  v(s) \theta_{R,a}^0(s)$,
$v(s) \theta_{R,a}^0(s) \in \mathbf{AP}((\tau_0,\infty))$ for all cases.  By \eqref{qd26Apr25-1}, it follows that 
\begin{equation}\label{qd26Mar6-2}
|\dot{g}|
\lesssim \rme^{-c \inf\limits_{s \ge \tau_0} R(s)} \min\{ \tau^{\frac 12}, \lambda_R^{-\frac 12} \} 
\lambda_R^{-\frac 12}  v \theta_{R,a}^0
\| \tilde{\phi}_1\|_{w}
+ v \theta_{R,a}^0 \|h\|_{v,a}^{R}.
\end{equation}

With the choice of $g$, \eqref{tildephi-h0} without $\int_{B_{R} } 
\tilde{\phi}_1 Z_0 \rmd y = 0$ is a linear equation. Local existence and uniqueness imply the global existence and uniqueness of \eqref{tildephi-h0} and that $\tilde{\phi}_{1} = \tilde{\phi}_{1}[h]$ depends on $h$ linearly. Then $g=g[h]$ also depends on $h$ linearly. Since $Z_0$, $pU^{p-1} \chi_{M}(|y|) \mathcal{T}_*[h]$ are radially symmetric, uniqueness and Lemma \ref{chiM-eq-lem}-$(1)$ give that $\tilde{\phi}_{1} = \tilde{\phi}_{1,1}(|y|,\tau) \Upsilon_{0,1}(\frac{y}{|y|})$. By the choice of $g$, $\int_{B_{R} } 
\tilde{\phi}_1 Z_0 \rmd y = 0$ is satisfied automatically.

Multiplying \eqref{tildephi-h0} with $\tilde{\phi}_1$ and integrating by parts, we have
\begin{equation*}
		\frac 12
		\partial_{\tau} \int_{B_{R}} (\tilde{\phi}_1 )^2 \rmd y +
		\int_{B_{R}} (|\nabla \tilde{\phi}_1|^2 - pU^{p-1} (\tilde{\phi}_1)^2 ) \rmd y
		= \int_{B_{R}} pU^{p-1} \chi_{M}(|y|) \mathcal{T}_*[h] \tilde{\phi}_1 \rmd y.
\end{equation*}

By Lemma \ref{eigenvalue-problem}, we get
\begin{equation*}
\frac 12 \partial_{\tau} \int_{B_{R}} (\tilde{\phi}_1 )^2 \rmd y +
		c_0 \lambda_R
		\int_{B_{R}}  (\tilde{\phi}_1)^2 \rmd y
		\le 
		\int_{B_{R}} \frac{4}{c_0 \lambda_R}(pU^{p-1} \chi_{M}(|y|) \mathcal{T}_*[h] )^2  \rmd y
		+ \int_{B_{R}} \frac{c_0 \lambda_R}{4} (\tilde{\phi}_1)^2 \rmd y
\end{equation*}
for some constant $c_0 > 0$. By \eqref{phi*-h0-est}, we get
\begin{equation*}
\frac 12 \partial_{\tau} \int_{B_{R}} (\tilde{\phi}_1 )^2 \rmd y + \frac {c_0 \lambda_R}{4}
		\int_{B_{R}}  (\tilde{\phi}_1)^2 \rmd y
		\lesssim
		 \lambda_R^{-1} (v \theta_{R,a}^0 \| h\|_{v,a}^{R} )^2.
\end{equation*}
Since $\tilde{\phi}_{1}(\cdot,\tau_0)=0$, we have
 \begin{equation*}
 \int_{B_{R}}  (\tilde{\phi}_1)^2  \rmd y 
 \lesssim
 \rme^{- \frac{c_0}{2} \int^{\tau} \lambda_R(u) \rmd u } \int_{\tau_0}^{\tau}
 \rme^{ \frac{c_0}{2} \int^{s} \lambda_R(u) \rmd u } \lambda_R^{-1}(s)  \big( v(s) \theta_{R,a}^0(s) \| h\|_{v,a}^{R} \big)^2 \rmd s
 \lesssim
 \min\{ \tau, \lambda_R^{-1} \} 
 \lambda_R^{-1}  (v \theta_{R,a}^0  \| h\|_{v,a}^{R} )^2,
 \end{equation*}
where for the last step, similar to \cite[Lemma 8.3, (8.23)]{Wei-Zhang-Zhou2022LLG}, we require either Case 1: $\mathbf{P}_1[\lambda_R] > -1$ or Case 2: $\mathbf{P}_1[\lambda_R] < -1$ and $\mathbf{P}_1[\lambda_R^{-1} ( v \theta_{R,a}^0 )^2] > -1$ holds. Since
\begin{equation*}
		\big| - \dot{g} Z_0 + \gamma_0 g Z_0
		+ pU^{p-1} \chi_{M}(|y|) \mathcal{T}_*[h] \big|
		\lesssim 
		\big( \rme^{-c \inf\limits_{s \ge \tau_0} R(s)} \min\{ \tau^{\frac 12}, \lambda_R^{-\frac 12} \} 
		\lambda_R^{-\frac 12}  v \theta_{R,a}^0
		\| \tilde{\phi}_1\|_{w}
		+ v \theta_{R,a}^0 \|h\|_{v,a}^{R} \big) \rme^{-c|y|},
\end{equation*}
by parabolic estimates, we have
\begin{equation*}
\| \tilde{\phi}_{1}(\cdot,\tau) \|_{L^{\infty} (B_{R})}
\lesssim
\min\{ \tau^{\frac 12}, \lambda_R^{-\frac 12} \} 
\lambda_R^{-\frac 12}  v \theta_{R,a}^0 \big( \|h\|_{v,a}^{R}
+
\rme^{-c \inf\limits_{s \ge \tau_0} R(s)}
\| \tilde{\phi}_1\|_{w} \big).
\end{equation*}
By $|\dot{v}| = O(\tau^{-1} v)$, $|\dot{R}| = O(\tau^{-1} R)$, $\langle R \rangle^{2+\epsilon} = O(\tau)$, $\epsilon>0$, and Lemma \ref{qd26Mar7-1-lem}, the spatial decay of $\tilde{\phi}_1$ can be improved to $\langle y \rangle^{2-n}$. By the scaling argument, we have
\begin{equation*}
\langle y \rangle | \nabla \tilde{\phi}_1  |  +	| \tilde{\phi}_1 | \lesssim 
\min\{ \tau^{\frac 12}, \lambda_R^{-\frac 12} \} 
\lambda_R^{-\frac 12}  v \theta_{R,a}^0   
\big( \| h\|_{v,a}^{R} + \rme^{-c \inf\limits_{s \ge \tau_0} R(s)} \| \tilde{\phi}_1\|_{w} \big) \langle y \rangle^{2-n}.
\end{equation*}
For $\inf\limits_{s \ge \tau_0} R(s) \ge C_0$ with a constant $C_0 \gg 1$, we have 
\begin{equation*}
\langle y \rangle 	|\nabla \tilde{\phi}_1   |  +	| \tilde{\phi}_1 | 
	\lesssim 
	\min\{ \tau^{\frac 12}, \lambda_R^{- \frac 12 } \}  \lambda_R^{-\frac 12} v \theta_{R,a}^0   
	  \langle y \rangle^{2-n}  \| h\|_{v,a}^{R}.
\end{equation*}

Repeating the operation in \eqref{e-t-est}, we have $|g|
		\lesssim  v \theta_{R,a}^0 \|h\|_{v,a}^{R}$ and
		 then
\begin{equation}\label{tildephi-est}
| \tilde{\phi} | = |\tilde{\phi}_{1} + g Z_0| \lesssim 
\min\{ \tau^{\frac 12}, \lambda_R^{- \frac 12 } \}  \lambda_R^{-\frac 12}
v \theta_{R,a}^0  \langle y \rangle^{2-n}  \|h\|_{v,a}^{R}.
\end{equation}

We take $g_0 = g(\tau_0)$.
Combining \eqref{phi*-h0-est}, \eqref{tildephi-est},  and the scaling argument, we get the conclusion.
\end{proof}

The linear theory for mode $0$ with orthogonality.
\begin{proposition}\label{mode0-orth-prop}
	
Given an integer $n\ge 3$, $1\le \tau_0 <\tau_1 \le \infty$, functions $R(\tau) \ge 2$, $v(\tau) \ge 0$ defined in $(\tau_0, \tau_1)$, consider
	\begin{equation*}
		\partial_{\tau} \phi = 
		\Delta \phi + pU^{p-1} \phi + h 
		\mbox{ \ for \ } \tau \in (\tau_0, \tau_1), y \in B_{R},
		\quad
		\phi(y,\tau_0) = g_0 Z_{0}(y)
		\mbox{ \ in \ } 
		B_{R(\tau_0)},
	\end{equation*}
	where
	$\| h\|_{v, 2+a}^{R} < \infty$,
	$a>0$, $h = h_1(|y|,\tau) \Upsilon_{0,1}(\frac{y}{|y|})$, and $h$ satisfies the orthogonality condition
	\begin{equation}\label{26Mar14-7}
		\int_{ B_{R} } h(y,\tau) Z_{n+1}(y) \rmd y = 0 \mbox{ \ for \ } \tau \in (\tau_0, \tau_1).
	\end{equation}
Suppose that
\begin{equation*}
	\begin{aligned}
		&
		|\dot{v}| = O(\tau^{-1} v), \ 
		|\dot{R}| = O(\tau^{-1} R), \ 
		v, R, \ln R \in \mathbf{AP}((\tau_0,\infty)), \  
		R^{2+\epsilon} = O(\tau) \mbox{ with a constant } \epsilon >0,
		\\ 
		&
		\mbox{either Case 1: $\mathbf{P}_1[\lambda_R] > -1 $ or Case 2: $\mathbf{P}_1[\lambda_R] < -1$ and $\mathbf{P}_1[\lambda_R^{-1} ( v \theta_{R, \hat{a}_{0}}^0 )^2] > -1$ holds}  
	\end{aligned}
\end{equation*}
with
\begin{equation}
\lambda_{R} =
\begin{cases}
	R^{-2} 
	& \mbox{ \ if \ } n=3
	\\
	(R^2 \ln R)^{-1} 
	& \mbox{ \ if \ } n=4
	\\
	R^{2-n}
	& \mbox{ \ if \ } n\ge 5,
\end{cases}
\quad
\theta_{R,\hat{a}_{0}}^0 = 
\begin{cases}
	R^{2-\hat{a}_{0}}  & \mbox{ \ if \ }  \hat{a}_{0}<2 
	\\
	\ln R 
	& \mbox{ \ if \ } \hat{a}_{0}=2 
	\\
	1 & \mbox{ \ if \ } \hat{a}_{0}>2,
\end{cases}
\quad
\hat{a}_{0} := 
\begin{cases}
	a & \mbox{ if } a \ne n-2
	\\
	(n-2)- & \mbox{ if } a = n-2,
\end{cases}
\end{equation}
then for $\inf\limits_{s \ge \tau_0} R(s) \ge C_0$ with a constant $C_0$ sufficiently large, there exists a solution $(\phi, g_0) = (\phi[h], g_0[h] )$ linearly depending on $h$ with the estimate
	\begin{equation*}
		\langle y\rangle |\nabla \phi| +	|\phi| 
		\le C
		v \big[
		\min\{ \tau^{\frac 12}, \lambda_R^{- \frac 12 } \}  \lambda_R^{-\frac 12}  \theta_{R,\hat{a}_{0}}^0 \langle y \rangle^{-n}  
		+ \Theta_{R, \hat{a}_{0} }^0(|y| )  \langle y\rangle^{-2}
		\big] \|h\|_{v,2+a}^{R},
\quad
		|g_0|\le C
		v(\tau_0) \theta_{R(\tau_0), \hat{a}_{0}}^0 \| h \|_{v,2+a}^{R} 
	\end{equation*}
with
\begin{equation}
\Theta_{R, \hat{a}_{0}}^0(|y|) = 
\begin{cases}
	R^{2-\hat{a}_{0}} &\mbox{ \ if \ } \hat{a}_{0}<2 
	\\
	\ln R                      
	&\mbox{ \ if \ } \hat{a}_{0}=2 
	\\
	\langle |y| \rangle^{2-\hat{a}_{0}} 
	&  \mbox{ \ if \ } 2<\hat{a}_{0}<n
	\\
	\langle |y| \rangle^{2-n}\ln(|y|+2)   
	&\mbox{ \ if \ } \hat{a}_{0} = n
	\\
	\langle |y| \rangle^{2-n}   
	&\mbox{ \ if \ } \hat{a}_{0} > n,
\end{cases}
\end{equation}
where $C>0$ is a constant independent of $\tau_0, \tau_1, R$.
Moreover, $\phi = \phi_1(|y|,\tau) \Upsilon_{0,1}(\frac{y}{|y|})$.
	
\end{proposition}

\begin{proof}
	
	$h(\cdot,\tau)$ is radially symmetric. Denote $r=|y|$, $\|h\| = \|h\|_{v,2+a}^{R}$ for brevity. Consider
	\begin{equation*}
		\Delta H + p U^{p-1} H = \tilde{h}(r,\tau) \mbox{ \ in \ } \mathbb{R}^n, 
	\end{equation*}
	where $\tilde{h}$ is the extension of $h$ as zero outside $B_{R}$. \eqref{26Mar14-7} is rewritten as $ \int_0^{\infty} \tilde{h}(r,\tau) Z_{n+1}(r) r^{n-1} \rmd r = 0 $.
	Take $H(r,\tau)$ as the form
	\begin{equation*}
		H(r,\tau) =  \tilde{Z}_{n+1}(r) \int_0^{r} \tilde{h}(s,\tau)Z_{n+1}(s)s^{n-1} \rmd s 
		- Z_{n+1}(r) \int_0^{r} \tilde{h}(s,\tau) \tilde{Z}_{n+1}(s)s^{n-1} \rmd s
		\mbox{ \ if \ } a\le n-2,
	\end{equation*}
	\begin{equation*}
		H(r,\tau) =  \tilde{Z}_{n+1}(r) \int_0^{r} \tilde{h}(s,\tau)Z_{n+1}(s)s^{n-1} \rmd s 
		+ Z_{n+1}(r) \int_{r}^{\infty} \tilde{h}(s,\tau) \tilde{Z}_{n+1}(s)s^{n-1} \rmd s
		\mbox{ \ if \ } a> n-2,
	\end{equation*}
	where  $\tilde{Z}_{n+1}(r)$ is another kernel satisfying that the Wronskian $W[Z_{n+1},\tilde{Z}_{n+1}]=r^{1-n}$, $\tilde{Z}_{n+1}(r) \sim r^{2-n}$ if $r\rightarrow 0$ and $\tilde{Z}_{n+1}(r) \sim  1$ if $r\rightarrow \infty$. By \eqref{26Mar14-7}, direct calculation, and the scaling argument, we have
	\begin{equation}\label{26Mar14-2}
		\|\partial_{r} H\|_{v,\hat{a}_{0} + 1}^{\infty}
		+
		\|H\|_{v,\hat{a}_{0}}^{\infty} \lesssim
		\|h\|,
	\end{equation}
	where $a>0$ is used to make the spatial decay of $ \tilde{h}(s,\tau) Z_{n+1}(s)s^{n-1}$ faster than $s^{-1}$ for $s\ge 1$. Consider
	\begin{equation}
		\partial_{ \tau} \Phi = \Delta \Phi + p U^{p-1} \Phi + H \mbox{ \ for \ } \tau \in (\tau_0, \tau_1), y \in B_{2 R},
		\quad
		\Phi(\cdot,\tau_0) = \bar{g}_0 Z_0
		\mbox{ \ in \ } B_{2 R(\tau_0)},
	\end{equation}
	where $(\Phi, \bar{g}_0) = (\Phi[H], \bar{g}_0[H])$ is given by Lemma \ref{mode0-nonorth} under the corresponding assumption. By \eqref{26Mar14-2}, the regularity theory in Dong-Escauriaza-Kim \cite[Lemma 4.13]{dong21-C0-para}, and the scaling argument \cite[Proposition 6.2]{Wei-Zhang-Zhou2022LLG}, we have
	\begin{equation}\label{26Mar14-3}
		\langle y \rangle^2 | \nabla^2 \Phi |
		+
		\langle y \rangle |\nabla \Phi| + |\Phi|
		\lesssim 
		v \big[ \min\{ \tau^{\frac 12}, \lambda_R^{- \frac 12 } \}  \lambda_R^{-\frac 12}
		\theta_{R, \hat{a}_{0}}^0  \langle y \rangle^{2-n}    + 	\Theta_{R, \hat{a}_{0}}^0( |y| ) \big]
		\|h\|,
\quad
		|\bar{g}_0| 
		\lesssim v(\tau_0) \theta_{R(\tau_0), \hat{a}_{0}}^0 \| h\|.
	\end{equation}	
Denote $\phi := (\Delta + pU^{p-1}) \Phi$ and $g_0 := \gamma_0 \bar{g}_{0}$. We get the equation for $\phi$, and the estimate is deduced from \eqref{26Mar14-3}.
\end{proof}

The re-gluing inner linear theory for mode $0$.
\begin{proposition}\label{mode0-regluing-prop}

Given an integer $n\ge 3$, $1\le \tau_0 <\tau_1 \le \infty$, consider
	\begin{equation*}
		\partial_{\tau} \phi=\Delta \phi + pU^{p-1}\phi + h + \varrho(\tau) \eta(y) Z_{n+1}(y) 
		\mbox{ \ for \ } \tau \in (\tau_0, \tau_1), y \in B_{R},
        \quad
		\phi(y,\tau_0)= g_0 \eta\Big( \frac{2 y}{R_0(\tau_0)} \Big) Z_0(y) 
		\mbox{ \ in \ } B_{R(\tau_0)},
	\end{equation*}
	where $\|h\|_{v,2+a}^{R}<\infty$, $h = h_1(|y|,\tau) \Upsilon_{0,1}(\frac{y}{|y|})$. Suppose that 
\begin{align}
				&
				\begin{cases}
					a\in (1,2), \epsilon_0 \in (0, a-1) & \mbox{ if } n=3
					\\
					a\in (0,2), \epsilon_0 \in (0, a) & \mbox{ if } n\ge 4
				\end{cases}, 
		\ 
				 4 \le R_0 < R, \ 
				|\dot{v}| = O(\tau^{-1} v), \  |\dot{R}_0| = O(\tau^{-1} R_0),
				\\
				&
				v, R_0, \ln R_0  \in \mathbf{AP}((\tau_0,\infty)),
				\ 
				\mathbf{P}_1[\lambda_{R_0}] > -1,
				\ |\dot{R}| = O(\tau^{-1} R),
				\ 
				R^{2+\epsilon} = O(\tau) 
				\mbox{ with a small constant } \epsilon>0 
                \notag
        \end{align}
with
\begin{equation}
\lambda_{R_0} =
\begin{cases}
	R_0^{-2} 
	& \mbox{ \ if \ } n=3
	\\
	(R_0^2 \ln R_0)^{-1} 
	& \mbox{ \ if \ } n=4
	\\
	R_0^{2-n}
	& \mbox{ \ if \ } n\ge 5,
\end{cases}
\end{equation}
then for $\inf_{s \ge \tau_0} R_0(s) \ge C_0$ with a constant $C_0$ sufficiently large, there exists a solution $(\phi,g_0,\varrho) = (\phi[h],g_0[h],\varrho[h])$ linearly depending on $h$, and satisfying $\phi \in C^{1+\sigma,\frac{1+\sigma}{2}} \big( \{ (y,\tau) \mid \tau \in (\tau_0, \tau_1), y\in B_{R} \} \big)$, $\sigma \in (0,1)$,
	 \begin{equation}
	 		\langle y\rangle |\nabla \phi| + |\phi|
	 		\le D_1 \lambda_{R_0}^{-1} R_0^{2-a} v \langle y\rangle^{ -\min\{ a, n-2\} } \|h\|_{v,2+a}^{R},
\quad
		|g_0| \le D_1
		v(\tau_0) R_0^{2-a}(\tau_0) \|h \|_{v,2+a }^{R},
	\end{equation}
	\begin{equation}
		\varrho(\tau)
		=
		-\Big(\int_{B_{2} } \eta(y) Z_{n+1}^2(y) \rmd y\Big)^{-1} \Big(\int_{B_{R_0} }
		h(y,\tau)  Z_{n+1}(y) \rmd y
		+ \varrho_{1}[h](\tau)  \Big),
	\end{equation}
	where $\varrho_{1}[h](\tau)$ linearly depends on $h$ and satisfies the estimate
	\begin{equation}
		|\varrho_{1}[h](\tau)| \le D_1 v
		R_{0}^{-\epsilon_0}
		\| h \|_{v,2+a}^{R},
	\end{equation}
where $D_1>0$ is a constant independent of $\tau_0, \tau_1, R, R_0$. Moreover, $\phi = \phi_1(|y|,\tau) \Upsilon_{0,1}(\frac{y}{|y|})$.

Under the additional assumption that $C_{R}^{-1} R(\tau) \le R(s) \le C_{R} R(\tau)$, $R_0(s) \le (9 C_{R})^{-1} R(\tau)$ with a constant $C_{R} \ge 1$ for all $\tau \in (\tau_0, \tau_1)$, $s\in [\max\{ \tau_0, \tau/2 \}, \tau]$, given $\alpha \in (0,1)$, then for all $\tau \in (\tau_0, \tau_1)$, $\kappa_1, \kappa_2 \in ( \max\{\tau_0, \tau- \delta_1^2 \}, \tau]$, $\delta_1 \in (0, \min\{ (\frac{\tau}{8})^{\frac{1}{2}}, \frac{R}{9 C_R} \} ]$,  we have
\begin{equation}
		|\varrho_{1}(\kappa_1) - \varrho_{1}(\kappa_2)|
		\le D_2 v
		\big[
		| \kappa_1 - \kappa_2|^{\alpha}
		\delta_1^{- 2 \alpha}
		R_{0}^{-\epsilon_0} 
		(1 + \delta_1^{2})
		+
		R_{0}^{-\epsilon_0-3} | R_0(\kappa_1) - R_0(\kappa_2) | \big] \|h\|_{v,2+a}^{R},
\end{equation}
where $D_2>0$ is a constant independent of $\tau_0, \tau_1, R, R_0$.

\end{proposition}

\begin{remark}\label{qd26Apr28-1-rek}

When $\tau_1 < \infty$, we can use replace $(\tau_0, \tau_1)$ by $(\tau_0, \tau_1]$.

For brevity in application, we note that
$\lambda_{R_0}^{-1} R_0^{2-a} \lesssim  R_0^{n+1}$,
and
$\begin{cases}
\mathbf{P}_1[R_0] < 2^{-1} & \mbox{ if } n=3
\\
\mathbf{P}_1[R_0] < (n-2)^{-1} & \mbox{ if } n \ge 4
\end{cases}$
implies $\mathbf{P}_1[\lambda_{R_0}] > -1$.
\end{remark}

\begin{remark}

The mapping $(\phi,g_0,\varrho) = (\phi[h],g_0[h],\varrho[h])$ depends on $\tau_0, \tau_1, R, R_0$.

\end{remark}

\begin{remark}

We restrict $a\in (0,2)\setminus \{n-2\}$ for the conciseness of the conclusion and partial due to $|U^{p-1}| \sim \langle y \rangle^{-4}$.

\end{remark}

\begin{proof}

Denote $\| h \| = \| h \|_{v,2+a}^{R}$ in this proof. Set 
\begin{equation*}
\phi = \Psi(y,\tau) + \eta(2 y / R_0) \zeta(y,\tau).
\end{equation*}
To solve $\partial_{\tau} \phi=\Delta \phi + pU^{p-1}\phi + h + \varrho \eta Z_{n+1}$ for $\tau \in (\tau_0, \tau_1)$, $y \in B_{R}$, since $4 \le R_0 < R$ implies $\eta =  \eta(2 y / R_0) \eta$, then it suffices to  consider the following gluing system for $(\Psi, \zeta)$
	\begin{equation}\label{1d26Mar8-3}
    \begin{aligned}
&
			\partial_{\tau} \Psi
			= 
			\Delta \Psi 
			+
			J[\Psi,\zeta] \1_{|y|<R}
			\mbox{ for } \tau \in (\tau_0, \tau_1), y \in B_{4 R},
\\
&
			\Psi = 0
			\mbox{ for } \tau \in (\tau_0, \tau_1), y \in \partial B_{4 R},
			\quad
			\Psi(\cdot,\tau_0) = 0
			\mbox{ in } B_{4 R(\tau_0)},
    \end{aligned}
	\end{equation}
	\begin{equation}\label{qd26Mar8-4}
		\partial_{\tau} \zeta
		= 
		\Delta \zeta + pU^{p-1} \zeta + H[\Psi]
		\mbox{ \ for \ } \tau \in (\tau_0, \tau_1), y \in B_{R_0},
	\end{equation}
	where
	\begin{equation}
	\begin{aligned}
&
			J[\Psi,\zeta] :=
			(pU^{p-1}\Psi + h) [
			1- \eta(2 y / R_0)
			]
			+ 
			A[\zeta],
			\\
			&
			A[\zeta] :=
			\zeta \Delta [\eta(2 y / R_0)]  + 2 \nabla \zeta \cdot \nabla  [\eta(2 y / R_0)]   - \zeta \partial_{\tau} [\eta(2 y / R_0)],
\quad
H[\Psi] := p U^{p-1} \Psi
+ h + \varrho \eta Z_{n+1},
	\end{aligned}
	\end{equation}
and we consider \eqref{1d26Mar8-3} in a larger domain $B_{4 R}$ for convenience of the scaling argument for $\nabla \phi$ later.

To meet the orthogonality condition
\begin{equation}\label{qd26Mar17-2}
	\int_{B_{R_0}} H[\Psi](z,\tau) Z_{n+1}(z) \rmd z = 0 \mbox{ \ for \ } \tau \in (\tau_0, \tau_1),
\end{equation}
we take
	\begin{equation}\label{qd26Mar20-3}
	\begin{aligned}
&
		\varrho = \varrho[\Psi](\tau) :=
		C_1 
		\Big(
		\int_{B_{R_0}} h(z,\tau) Z_{n+1}(z) \rmd z 
		+
		\varrho_{1}
		\Big),
		\quad
		C_{1} :=  -\Big( \int_{B_{2} } \eta(z) Z_{n+1}^2(z) \rmd z \Big)^{-1},
\\
&
\varrho_{1} = \varrho_{1}[\Psi](\tau) := \int_{B_{R_0} } pU^{p-1} (z) \Psi(z,\tau) Z_{n+1}(z) \rmd z.
	\end{aligned}
	\end{equation}

The mappings for solving \eqref{1d26Mar8-3} and \eqref{qd26Mar8-4} are given as the following forms
	\begin{equation}\label{qd26Mar8-7}
			\Psi =
			\mathcal{T}_{{\rm{ou}}} [
			J[\Psi,\zeta] \1_{|y|<R}  ](y,\tau),
			\quad
			\zeta = 
			\mathcal{T}_{\rm{in}} [ H[\Psi] ](y,\tau) \mbox{ \ with \ } \zeta(y,\tau_0) = g_0 Z_0(y),
			\quad
			g_0 =
			\mathcal{T}_{g_0}[ H[\Psi] ],
	\end{equation}
	where $\mathcal{T}_{\rm{ou}}$ is the linear mapping given by solving the homogeneous part of \eqref{1d26Mar8-3}; and $\mathcal{T}_{{\rm{in}}}$, $\mathcal{T}_{g_0}$ are given by Proposition \ref{mode0-orth-prop} (with $R(\tau) = R_0(\tau)$) since \eqref{qd26Mar17-2} are satisfied by the choice of $\varrho$.

	Denote the leading term of $H[\Psi]$ as $H_{1} := h+ C_{1} \eta Z_{n+1} \int_{B_{R_0} }
	h(z,\tau) Z_{n+1}(z) \rmd z$. For $a>0$, we have $\| H_{1} \|_{v,2+a}^{R_0} \lesssim \|h\|$. Since $h = h_1(|y|,\tau) \Upsilon_{0,1}(\frac{y}{|y|})$ and $\eta Z_{n+1}$ is radially symmetric, then $H_{1} = (H_{1})_{1}(|y|,\tau) \Upsilon_{0,1}(\frac{y}{|y|})$.
	If $H_1$ satisfies the orthogonality condition in $B_{R_0}$, under the assumption 
\begin{equation}\label{qd26Mar17-1}
\begin{aligned}
&
|\dot{v}| = O(\tau^{-1} v), \  |\dot{R}_0| = O(\tau^{-1} R_0),
\ 
v, R_0, \ln R_0 \in \mathbf{AP}((\tau_0,\infty)),
\\
&  
R_0^{2+\epsilon_1} = O(\tau) 
\mbox{ with } \epsilon_1>0, 
\ 
\mathbf{P}_1[\lambda_{R_0}] > -1,
\
\inf_{s \ge \tau_0} R_0(s) \ge C_0 \gg 1,
\end{aligned}
\end{equation}
since $\mathbf{P}_1[\lambda_{R_0}] > -1$ implies $\lambda_{R_0}^{-\frac{1}{2}} \lesssim \tau^{\frac{1}{2}}$, then for $a\in (0,2)\setminus \{n-2\}$, Proposition \ref{mode0-orth-prop} and Remark \ref{qd26Mar15-1-rmk} give the following apriori estimate
	\begin{equation}\label{qd26Mar20-2}
		\langle y \rangle 
		|\nabla \mathcal{T}_{{\rm{in}}}[H_1]|
		+	|\mathcal{T}_{{\rm{in}}}[H_1] | \le
		D_{\rm{in}} \|h\|  v \lambda_{R_0}^{-1} R_0^{2-a} \langle y \rangle^{-n},
		\quad
		\lambda_{R_0}^{-1} R_0^{2-a} = 
		\begin{cases}
		R_0^{4-a} & \mbox{ if } n=3
		\\
		R_0^{4-a} \ln R_0 & \mbox{ if } n=4
		\\
		R_0^{n-a} & \mbox{ if } n\ge 5,
		\end{cases} 
	\end{equation}
	where  $D_{{\rm{in}}}\ge 1$ is a large constant. Moreover, $\mathcal{T}_{{\rm{in}}}[H_1] = (\mathcal{T}_{{\rm{in}}}[H_1])_{1}(|y|,\tau) \Upsilon_{0,1}(\frac{y}{|y|})$. For this reason, we will solve for $\zeta$ in the space
	\begin{equation*}
		\mathcal{B}_{\rm{in}} := \big\{
		g \mid
		\| g \|_{\rm{in}} \le 
		2D_{\rm{in}} \|h\|, \  g = g_1(|y|,\tau) \Upsilon_{0,1}(\frac{y}{|y|}) \big\}
	\end{equation*}
endowed with the norm
\begin{equation*}
	\| g \|_{\rm{in}} := \inf \big\{ C \mid \langle y \rangle |\nabla g(y,\tau)|  + |g(y,\tau)| \le C  v \lambda_{R_0}^{-1} R_0^{2-a} \langle y \rangle^{-n} \mbox{ for all }  \tau\in (\tau_0, \tau_1), y\in B_{R_0} \big\}.
\end{equation*}
It is ready to get that $\mathcal{B}_{\rm{in}}$ is a complete normed space.

For any $\zeta \in \mathcal{B}_{{\rm{in}}}$, let us estimate $J[0,\zeta] = h [
1- \eta(2 y / R_0)
]
+ 
A[\zeta]$. Under the assumption $a\in (0,2) \setminus \{n-2\}$ and $|\dot{R}_0| = O(R_0^{-1})$, then
\begin{equation}\label{qd26Mar18-1}
		|A[\zeta]| \lesssim \1_{R_0/2 \le |y| \le R_0} \| \zeta \|_{\rm{in}} v 
		\begin{cases}
			\langle y \rangle^{-1-a}
			& \mbox{ \ if \ } n=3
			\\
			\langle y \rangle^{-2-a} \ln R_0 
			& \mbox{ \ if \ } n=4
			\\
			\langle y \rangle^{-2-a}
			& \mbox{ \ if \ } n\ge 5
		\end{cases}
		\lesssim \1_{R_0/2 \le |y| \le R_0}
		\| \zeta \|_{\rm{in}} v 
		R_0^{-\epsilon_0} \langle y \rangle^{-2- 2 a_1},
\end{equation}
where for the last step, we require
\begin{equation}
	\begin{cases}
		a>1, \  0<\epsilon_0<a-1, \  0<a_1 < (a-1-\epsilon_0)/2
		&
		\mbox{ \ if \ } n=3
		\\
		0<\epsilon_0 < a, \  0<a_1 < (a-\epsilon_0)/2
		&
		\mbox{ \ if \ } n\ge 4.
	\end{cases}
\end{equation}
The choice of $a_1$ is prepared for \eqref{qd26Mar19-2} about the contraction mapping property later. Since for $\zeta \in \mathcal{B}_{{\rm{in}}}$, $\zeta = \zeta_1(|y|,\tau) \Upsilon_{0,1}(\frac{y}{|y|})$ and $\eta(\cdot)$ is radially symmetric, we have 
\begin{equation*}
	\begin{aligned}
		&
		\nabla \zeta \cdot \nabla  [\eta(2 y / R_0)]
		= \partial_{|y|} \big[ \zeta_1(|y|,\tau) \Upsilon_{0,1}(\frac{y}{|y|}) \big] \partial_{|y|}  [\eta(2 y / R_0)]
		+
		\Big\langle 
		\nabla_{S^{n-1}} \big[ \zeta_1(|y|,\tau) \Upsilon_{0,1}(\frac{y}{|y|}) \big], \nabla_{S^{n-1}} [\eta(2 y / R_0)]
		\Big\rangle
		\\
		= \ & \Upsilon_{0,1}(\frac{y}{|y|}) \partial_{|y|} [ \zeta_1(|y|,\tau) ] \partial_{|y|}  [\eta(2 |y| / R_0)].
	\end{aligned}
\end{equation*}
It follows that $A[\zeta] = (A[\zeta])_{1}(|y|,\tau) \Upsilon_{0,1}(\frac{y}{|y|})$.	
Also, we have
\begin{equation}\label{qd26Mar18-2}
		| h [ 1- \eta(2 y / R_0) ] | \lesssim 
		\1_{ |y|\ge R_0/2 } v\langle y\rangle^{-2-a } \| h\|
		\lesssim 
		\1_{ |y|\ge R_0/2 } v
		R_{0}^{-\epsilon_0}
		\langle y \rangle^{-2-a_1}
		\| h \|
	\end{equation}
provided
\begin{equation*}
0<\epsilon_0<a, \ 0< a_1 \le a-\epsilon_0.
\end{equation*}
Since $h = h_1(|y|,\tau) \Upsilon_{0,1}(\frac{y}{|y|})$ and $\eta(\cdot)$ is radially symmetric, then $h [ 1- \eta(2 y / R_0) ] = h_1(|y|,\tau) \Upsilon_{0,1}(\frac{y}{|y|}) [ 1- \eta(2 |y| / R_0) ]$.
In sum,
\begin{equation*}
|J[0, \zeta] \1_{|y|<R} | \lesssim D_{\rm{in}} \|h\| v 
R_0^{-\epsilon_0} \langle y \rangle^{-2-a_1}
\mbox{ \ and \ }
J[0, \zeta] \1_{|y|<R} = (J[0, \zeta] \1_{|y|<R})_{1}(|y|,\tau) \Upsilon_{0,1}(\frac{y}{|y|}).
\end{equation*}
We restrict $a_1 \in (0, n-2)$.
Under the assumption 
\begin{equation}\label{qd26Mar21-3}
R\ge 2, 
\ 
v\ge 0,
\ 
|\dot{v}| = O(\tau^{-1} v),
\ 
|\dot{R}_0| = O(\tau^{-1} R_0),
\ 
|\dot{R}| = O(\tau^{-1} R),
\ 
R^{2+\epsilon_1} \ll \tau 
\mbox{ with } 
\epsilon_1 > 0,
\end{equation}
by Lemma \ref{chiM-eq-lem} and Remark \ref{qd26Mar13-3-rmk}, then
	\begin{equation}\label{qd26Mar20-1}
	| \mathcal{T}_{{\rm{ou}}}[J[0, \zeta] \1_{|y|<R} ] |
		\le 
		D_{\rm{ou}}	D_{\rm{in}} \|h\| v
		R_{0}^{-\epsilon_0} \langle y \rangle^{-a_1}
\mbox{ \ for \ } \tau\in (\tau_0, \tau_1), y\in B_{4R}
	\end{equation}
with a large constant $D_{\rm{ou}} \ge 1$. Moreover, $\mathcal{T}_{\rm{ou}}[J[0, \zeta] \1_{|y|<R}] = (\mathcal{T}_{\rm{ou}}[J[0, \zeta] \1_{|y|<R} ])_{1}(|y|,\tau) \Upsilon_{0,1}(\frac{y}{|y|})$.

	This suggests that we solve $\Psi$ in the space
	\begin{equation*}
		\mathcal{B}_{\rm{ou}} := \big\{
		f \mid
		\| f \|_{\rm{ou}} \le
		2 D_{\rm{ou}}	D_{\rm{in}} \|h\|, \ f = f_1(|y|,\tau) \Upsilon_{0,1}(\frac{y}{|y|}) \big\}
	\end{equation*}
endowed with the norm
\begin{equation*}
	\| f \|_{\rm{ou}} := \inf
	\big\{ C \mid  |f(y,\tau)| \le C  v
	R_{0}^{-\epsilon_0} \langle y \rangle^{-a_1}
	\mbox{ for all } \tau\in (\tau_0,\tau_1), y\in B_{4 R} \big\}.
\end{equation*}
$\mathcal{B}_{\rm{ou}}$ is a complete norm space. For any $\Psi \in \mathcal{B}_{\rm{ou}}$, we have $pU^{p-1} \Psi [ 1- \eta(2 y / R_0) ] \1_{|y|<R} = pU^{p-1}(|y|) \Psi_1(|y|,\tau) $ $ \Upsilon_{0,1}(\frac{y}{|y|}) [ 1- \eta(2 |y| / R_0) ] \1_{|y|<R}$, and
\begin{equation}\label{qd26Mar21-1}
	\begin{aligned}
&
		| pU^{p-1} \Psi [ 1- \eta(2 y / R_0) ] \1_{|y|<R} |
		\lesssim 
		\1_{R_0/2 \le |y| < R} \| \Psi \|_{\rm{ou}}  v
		R_{0}^{-\epsilon_0} \langle y \rangle^{-4-a_1} 
		\\
		\lesssim \ & 
		\1_{R_0/2 \le |y| < R} \big( \inf_{s \ge \tau_0} R_0(s) \big)^{-2} \| \Psi \|_{\rm{ou}} v
		R_{0}^{-\epsilon_0} \langle y \rangle^{-2-a_1}.
	\end{aligned}
	\end{equation}
Since $\| \Psi \|_{\rm{ou}} \le 2 D_{\rm{ou}}	D_{\rm{in}} \|h\|$ and $\inf_{s \ge \tau_0} R_0(s) \ge C_0 \gg 1$, by \eqref{qd26Mar21-3} and Lemma \ref{chiM-eq-lem}, we have $ \mathcal{T}_{\rm{ou}}[
		J[\Psi, \zeta] \1_{|y|<R} ]  \in \mathcal{B}_{\rm{ou}} $. For $\Psi \in \mathcal{B}_{\rm{ou}}$ and $a\in (0,2]$, we have $p U^{p-1} \Psi 
		+ \eta Z_{n+1} C_{1} \varrho_{1} = (p U^{p-1} \Psi 
		+ \eta Z_{n+1} C_{1} \varrho_{1})_{1}(|y|,\tau) \Upsilon_{0,1}(\frac{y}{|y|}) $ and
	\begin{equation}\label{qd26Mar21-2}
		\| 
		p U^{p-1} \Psi 
		+ \eta Z_{n+1} C_{1} \varrho_{1} \|_{v,2+a}^{R_0} 
		\lesssim 
		\| \Psi \|_{\rm{ou}} 
		(\inf_{s \ge \tau_0} R_0(s))^{-\epsilon_0}.
	\end{equation}
Under the assumption \eqref{qd26Mar17-1},
by \eqref{qd26Mar17-2} and Proposition \ref{mode0-orth-prop}, since $\| \Psi \|_{\rm{ou}} \le 2 D_{\rm{ou}}	D_{\rm{in}} \|h\|$ and $\inf_{s \ge \tau_0} R_0(s) \ge C_0 \gg 1$, we have
$ \mathcal{T}_{\rm{in}}[H[\Psi] ] \in \mathcal{B}_{\rm{in}} $.

The contraction mapping property can be deduced similarly. Indeed, for any $(\tilde{\Psi}_j, \tilde{\zeta}_j) \in \mathcal{B}_{\rm{ou}} \times \mathcal{B}_{\rm{in}}$, $j=1,2$,
	\begin{equation*}
			( J[\tilde{\Psi}_1, \tilde{\zeta}_1]
			-
			J[\tilde{\Psi}_2, \tilde{\zeta}_2] ) \1_{|y|<R}
=
p U^{p-1} (\tilde{\Psi}_1 - \tilde{\Psi}_2) [
1- \eta(2 y / R_0)
] \1_{|y|<R}
+ 
A[\tilde{\zeta}_1 - \tilde{\zeta}_2] \1_{|y|<R}.
	\end{equation*}
Since
	\begin{equation*}
	\begin{aligned}
&
			\big| p U^{p-1} (\tilde{\Psi}_1 - \tilde{\Psi}_2) [
			1- \eta(2 y / R_0)
			] \1_{|y|<R} \big|
\stackrel{\eqref{qd26Mar21-1}}{\lesssim}
\1_{R_0/2 \le |y| < R} (\inf_{s\ge \tau_0} R_0(s))^{-2}
\| \tilde{\Psi}_1 - \tilde{\Psi}_2 \|_{\rm{ou}} v
R_{0}^{-\epsilon_0} \langle y \rangle^{-2-a_1},
\\
&
| A[\tilde{\zeta}_1 - \tilde{\zeta}_2] |
\stackrel{\eqref{qd26Mar18-1}}{\lesssim} 
\1_{R_0/2 \le |y| \le R_0}
(\inf_{s\ge \tau_0} R_0(s))^{-a_1}
\| \tilde{\zeta}_1 - \tilde{\zeta}_2 \|_{\rm{in}} v 
R_0^{-\epsilon_0} \langle y \rangle^{-2-a_1},
	\end{aligned}
	\end{equation*}
then $|( J[\tilde{\Psi}_1, \tilde{\zeta}_1]
-
J[\tilde{\Psi}_2, \tilde{\zeta}_2] ) \1_{|y|<R}| \lesssim (\inf_{s\ge \tau_0} R_0(s))^{-a_1}
( \| \tilde{\Psi}_1 - \tilde{\Psi}_2 \|_{\rm{ou}} + 
\| \tilde{\zeta}_1 - \tilde{\zeta}_2 \|_{\rm{in}} ) v 
R_0^{-\epsilon_0} \langle y \rangle^{-2-a_1}$. Similar to \eqref{qd26Mar20-1},
\begin{equation}\label{qd26Mar19-2}
\big\| \mathcal{T}_{\rm{ou}} \big[ J[\tilde{\Psi}_1, \tilde{\zeta}_1] \1_{|y|<R} \big] - \mathcal{T}_{\rm{ou}} \big[
J[\tilde{\Psi}_2, \tilde{\zeta}_2] \1_{|y|<R} \big] \big\|_{\rm{ou}}
\lesssim 
(\inf_{s\ge \tau_0} R_0(s))^{-a_1}
\big(
\| \tilde{\Psi}_1 - \tilde{\Psi}_2 \|_{\rm{ou}}
+
\| \tilde{\zeta}_1 - \tilde{\zeta}_2 \|_{\rm{in}}
\big).
\end{equation}
	\begin{equation*}
		H[\tilde{\Psi}_1] - H[\tilde{\Psi}_2]
		=
		p U^{p-1} (\tilde{\Psi}_1 - \tilde{\Psi}_2) + \eta Z_{n+1} C_1
		\int_{B_{R_0} } pU^{p-1} (z) (\tilde{\Psi}_1 - \tilde{\Psi}_2)(z,\tau) Z_{n+1}(z) \rmd z.
	\end{equation*}
It holds that
\begin{equation*}
\int_{B_{R_0}} (H[\tilde{\Psi}_1] - H[\tilde{\Psi}_2])(w,\tau) Z_{n+1}(w) \rmd w = 0,
\quad
| H[\tilde{\Psi}_1] - H[\tilde{\Psi}_2] |
\lesssim  (\inf_{s \ge \tau_0} R_0(s))^{-\epsilon_0} \| \tilde{\Psi}_1 - \tilde{\Psi}_2 \|_{\rm{ou}} v \langle y \rangle^{-2-a},
\end{equation*}
where the second part is similar to \eqref{qd26Mar21-2}.
Similar to \eqref{qd26Mar20-2},
\begin{equation}\label{qd26Mar19-3}\|  \mathcal{T}_{\rm{in}}[H[\tilde{\Psi}_1]] - \mathcal{T}_{\rm{in}}[H[\tilde{\Psi}_2]]  \|_{\rm{in}} \lesssim  (\inf_{s \ge \tau_0} R_0(s))^{-\epsilon_0} \| \tilde{\Psi}_1 - \tilde{\Psi}_2 \|_{\rm{ou}}.
\end{equation}

Combining \eqref{qd26Mar19-2}, \eqref{qd26Mar19-3}, we get the contraction mapping property for \eqref{qd26Mar8-7}. By the Banach fixed-point theorem, we find a solution 
$
(\Psi, \zeta) \in \mathcal{B}_{\rm{ou}} \times \mathcal{B}_{\rm{in}}
$
for \eqref{1d26Mar8-3} and \eqref{qd26Mar8-4}. Denote
\begin{equation*}
\mathcal{X} := \big\{
f \mid
\| f \|_{\rm{ou}} <\infty, \ f = f_1(|y|,\tau) \Upsilon_{0,1}(\frac{y}{|y|}) \big\} \times \big\{
g \mid
\| g \|_{\rm{in}} <\infty, \  g = g_1(|y|,\tau) \Upsilon_{0,1}(\frac{y}{|y|}) \big\}
\end{equation*}  
including $\mathcal{B}_{\rm{ou}} \times \mathcal{B}_{\rm{in}}$. We can repeat the contraction mapping argument in $\mathcal{X}$. Since \eqref{1d26Mar8-3} and \eqref{qd26Mar8-4} is a linear equation system, by the uniqueness in the Banach fixed-point theorem applied in $\mathcal{X}$, $(\Psi, \zeta)$ depends on $h$ linearly.

Hereafter we always will regard $D_{\rm{ou}}$, $D_{{\rm{in}}}$  as general constants. By Proposition \ref{mode0-orth-prop} and \eqref{qd26Mar8-7}, we have
	\begin{equation*}
		|g_0| = |\mathcal{T}_{g_0}[ H[\Psi] ]| \lesssim 
		v(\tau_0) R_0^{2-a}(\tau_0) \|h \|.
	\end{equation*}
\begin{equation*}
	\phi(y,\tau_0) = \Psi(y,\tau_0) + \eta( 2 y / R_0(\tau_0)) \zeta(y,\tau_0)
	= g_0 \eta( 2 y / R_0(\tau_0)) Z_0(y) \mbox{ \ in \ } B_{4 R(\tau_0)}.
\end{equation*}
	By \eqref{qd26Mar20-3} and $\Psi \in \mathcal{B}_{\rm{ou}}$,
	 \begin{equation*}
	 	|\varrho_{1}(\tau)| 
	 	 \lesssim v
	 	R_{0}^{-\epsilon_0}
	 	\|h\|.
	 \end{equation*}

	 By \eqref{qd26Mar18-1} and \eqref{qd26Mar18-2},
	 $ |J[0,\zeta]| \lesssim \|h\| v R_0  \langle y \rangle^{-2-a} $. To improve the spatial decay of $\Psi$, under the assumption \eqref{qd26Mar21-3}, applying Lemma \ref{chiM-eq-lem} to \eqref{1d26Mar8-3} finitely many time, we have
$
	|\Psi| \lesssim
R_0 v \langle y\rangle^{ -\min\{ a, n-2\} } \|h\|
$.
Thus,
\begin{equation*}
|\phi| = | \Psi + \eta(2 y / R_0) \zeta |
\lesssim 
R_0 v \langle y\rangle^{ -\min\{ a, n-2\} } \|h\| + 
\1_{|y| \le R_0} \|h\| v \lambda_{R_0}^{-1}
R_0^{2-a} \langle y \rangle^{-n}
\lesssim  \lambda_{R_0}^{-1} R_0^{2-a} v \langle y\rangle^{ -\min\{ a, n-2\} } \|h\|.
\end{equation*}
Since
$
| (h + \varrho \eta Z_{n+1}) \1_{|y|<R} |  \lesssim \| h \| v \langle y \rangle^{-2-a} \1_{|y|<R}
$, by the scaling argument, we get the estimate of $\nabla \phi$.

Finally, we give the quantitative H\"older continuity of $\varrho_{1}$. For any $\tau \in (\tau_0, \tau_1)$, $\kappa_1, \kappa_2 \in ( \max\{\tau_0, \tau- \delta_1^2 \}, \tau] \subset ( \max\{\tau_0, \tau/2 \}, \tau]$ with $\delta_1 \in (0, \min\{ (\frac{\tau}{8})^{\frac{1}{2}}, \frac{R}{9 C_{R}} \} ]$,
\begin{equation*}
	\begin{aligned}
		&
		\varrho_{1}(\kappa_1) - \varrho_{1}(\kappa_2)
		= \int_{B_{R_0(\kappa_1)}} pU^{p-1} (z) [\Psi(z,\kappa_1) - \Psi(z,\kappa_2)] Z_{n+1}(z) \rmd z
		\\
		& + \int_{B_{R_0(\kappa_1)}} pU^{p-1} (z) \Psi(z,\kappa_2) Z_{n+1}(z) \rmd z
		-
		\int_{B_{R_0(\kappa_2)}} pU^{p-1} (z) \Psi(z,\kappa_2) Z_{n+1}(z) \rmd z.
	\end{aligned}
\end{equation*}
Since $\Psi(\cdot,\kappa_2)$ is defined in $B_{4 R(\kappa_2)}$, and $R_0(\kappa_1) \le (9 C_{R})^{-1} R$, $4 R(\kappa_2) \ge 4 C_{R}^{-1} R$, the integral above is well-defined. Since we are concerned about time decay here, we use 
$ |J[\Psi,\zeta] \1_{|y|<R}|
	\lesssim \|h\| v 
	R_0^{-\epsilon_0} \langle y \rangle^{-2-a_1} $, $
	|\Psi| \lesssim \|h\| v
	R_{0}^{-\epsilon_0} \langle y \rangle^{-a_1} $. By the scaling argument, for $\alpha \in (0,1)$, any $\kappa_1, \kappa_2\in ( \max\{\tau_0, \tau - \delta_1^2\}, \tau]$, $|y| \le (9 C_{R})^{-1} R$, 
\begin{equation*}
	\begin{aligned}
		&
		\frac{|\Psi(y, \kappa_1) -\Psi(y, \kappa_2) |}{| \kappa_1 - \kappa_2|^{\alpha}}
		\lesssim
		(4 \delta_1)^{- 2 \alpha} \| \Psi(y + 4 \delta_1 z , \tau + (4 \delta_1)^2 s)   \|_{L^{\infty} (B(0,\frac 1{2}) \times (-\frac 14,0))  } 
		\\
		&
		+
		(4 \delta_1)^{2- 2 \alpha}
		\|  (J[\Psi,\zeta] \1_{|\cdot|<R})(y + 4 \delta_1 z , \tau + (4 \delta_1)^2 s) \|_{L^{\infty} (B(0,\frac 1{2}) \times (-\frac 14,0))  }
\\
\lesssim \ &
\delta_1^{- 2 \alpha} \|h\| v
R_{0}^{-\epsilon_0} 
+
\delta_1^{2- 2 \alpha}
\|h\| v 
R_0^{-\epsilon_0}
= 
\delta_1^{- 2 \alpha} \|h\| v
R_{0}^{-\epsilon_0}
(1 + \delta_1^{2}),
	\end{aligned}
\end{equation*}
where we use $\delta_1 \in (0, \min\{ (\frac{\tau}{8})^{\frac{1}{2}}, \frac{R}{9 C_{R}} \} ]$; $\tau + (4 \delta_1)^2 s \in [\frac{\tau}{2}, \tau]$; $C_{R}^{-1} R \le R(s_1) \le C_{R} R$ for $s_1 \in [\max\{ \tau_0, \tau/2 \}, \tau]$; $4 R(\tau + (4 \delta_1)^2 s) \ge 4 C_{R}^{-1} R$, $ | y + 4 \delta_1 z | \le (9 C_{R})^{-1} R + 2 \delta_1 
\le C_{R}^{-1} R$; and $v, R_0 \in \mathbf{AP}((\tau_0,\infty))$. It follows that
\begin{equation*}
	\Big|
	\int_{B_{R_0(\kappa_1)}} pU^{p-1} (z) [\Psi(z,\kappa_1) - \Psi(z,\kappa_2)] Z_{n+1}(z) \rmd z
	\Big|
	\lesssim | \kappa_1 - \kappa_2|^{\alpha}
	\delta_1^{- 2 \alpha} \|h\| v
	R_{0}^{-\epsilon_0} 
	(1 + \delta_1^{2}).
\end{equation*}

For the other part in $\varrho_{1}(\kappa_1) - \varrho_{1}(\kappa_2)$, without loss of generality, we can assume $R_0(\kappa_1) \ge R_0(\kappa_2)$. Using $v, R_0 \in \mathbf{AP}((\tau_0,\infty))$, we have
\begin{equation*}
	\begin{aligned}
		& \Big| \int_{ B_{R_0(\kappa_1)} \setminus B_{R_0(\kappa_2)} } pU^{p-1} (z) \Psi(z,\kappa_2) Z_{n+1}(z) \rmd z \Big|
		\lesssim R_0^{-2-n} \int_{ B_{R_0(\kappa_1)} \setminus B_{R_0(\kappa_2)} } \|h\| (v
		R_{0}^{-\epsilon_0})(\kappa_2) \langle R_0 \rangle^{-a_1} \rmd z
		\\
		\sim \ & R_0^{-2-n} \|h\| v
		R_{0}^{-\epsilon_0} \langle R_0 \rangle^{-a_1}
		\big| (R_0(\kappa_1))^n - (R_0(\kappa_2))^n \big|
\sim  \|h\| v
R_{0}^{-a_1-\epsilon_0-3} | R_0(\kappa_1) - R_0(\kappa_2) |.
	\end{aligned}
\end{equation*}
Hence, we get the estimate of $|\varrho_{1}(\kappa_1) - \varrho_{1}(\kappa_2)|$.
\end{proof}

\subsection{Mode $1$}

The linear theory for mode $1$ without orthogonality.
\begin{lemma}\label{M1-nonortho-lem}

Given an integer $n\ge 3$, $1\le \tau_0 <\tau_1 \le \infty$, functions $R(\tau) \ge 2$, $v(\tau) \ge 0$ defined in $(\tau_0, \tau_1)$, consider 
	\begin{equation*}
    \left\{
    \begin{aligned}
    &
			\partial_{\tau} \phi = \Delta \phi + 
			pU^{p-1} \phi + h
			\mbox{ \ for \ } \tau \in (\tau_0, \tau_1), y \in B_{R},
            \\
            &
			\phi = 0 
			\mbox{ \ for \ } \tau \in (\tau_0, \tau_1), y \in \partial B_{R},
	\quad
			\phi(\cdot,\tau_0) = 0
			\mbox{ \ in \ } B_{R(\tau_0)},
        \end{aligned}
        \right.
	\end{equation*}
where $\|h\|_{v,a}^{R}<\infty$, $a\in \mathbb{R}$, $h = h_1(|y|,\tau) \Upsilon_{1,i}(\frac{y}{|y|})$, $i= 1,2,\dots, n$. There exists a unique solution $\phi = \phi[h]$ linearly depending on $h$ of the form
$ \phi = \phi_1(|y|,\tau) \Upsilon_{1,i}(\frac{y}{|y|}) $. Suppose that $|\dot{v}| = O(\tau^{-1} v)$,
$|\dot{R}| = O(\tau^{-1} R)$, then for $\sup_{s\ge \tau_0} s^{-1} R(s)^n \theta_{R(s),a}^1 \le d_0$ with a constant $d_0>0$ sufficiently small and
\begin{equation}
	\theta_{R,a}^1 := 
	\begin{cases}
		R^{1-a}  & \mbox{ \ if \ } a < 1
		\\
		\ln R   &\mbox{ \ if \ }  a = 1
		\\
		1 &\mbox{ \ if \ }  a>1,
	\end{cases}
\end{equation}
there exists a a constant $C>0$ independent of $\tau_0, \tau_1, R$ such that
	\begin{equation}
	\langle y\rangle |\nabla \phi| +	|\phi|
		\le C
		v
		\theta_{R,a}^1  R^n  
		\langle y \rangle^{1-n}
		\| h \|_{v,a}^{R}.
	\end{equation}

\end{lemma}

\begin{proof}

The proof is similar to \cite[Lemma 7.6]{infi4d}, \cite[pp. 336-337]{Green16JEMS}. $ \phi = \phi_1(|y|,\tau) \Upsilon_{1,i}(\frac{y}{|y|}) $ is deduced by Lemma \ref{chiM-eq-lem}-$(1)$.
\end{proof}

The linear theory for mode $1$ with orthogonality.
\begin{proposition}\label{mode1-orth-prop}
	
	Given an integer $n\ge 3$, $1\le \tau_0 <\tau_1 \le \infty$, functions $R(\tau) \ge 2$, $v(\tau) \ge 0$ defined in $(\tau_0, \tau_1)$, consider
	\begin{equation*}
		\partial_{\tau} \phi = 
		\Delta \phi + pU^{p-1} \phi + h 
		\mbox{ \ for \ } \tau \in (\tau_0, \tau_1), y \in B_{R},
\quad
\phi(\cdot, \tau_0) = 0 
\mbox{ \ in \ } B_{R(\tau_0)},
	\end{equation*}
where $\| h\|_{v, 2+a}^{R} < \infty$, $a > -1$, $h = h_1(|y|,\tau) \Upsilon_{1,i}(\frac{y}{|y|})$ for some $i\in\{ 1,2,\dots, n\}$, and $h$ satisfies the orthogonality condition
	\begin{equation}
		\int_{ B_{R} } h(y,\tau) Z_{i}(y) \rmd y = 0 \mbox{ \ for \ } \tau \in (\tau_0, \tau_1).
	\end{equation}
Suppose that $|\dot{v}| = O(\tau^{-1} v)$, $|\dot{R}| = O(\tau^{-1} R)$, then for $\sup_{s\ge \tau_0} s^{-1} R(s)^n \theta_{R(s), a}^1 \le d_0$ with a constant $d_0>0$ sufficiently small and
\begin{equation}
	\theta_{R,a}^1 = 
	\begin{cases}
		R^{1-a}  & \mbox{ \ if \ } a < 1
		\\
		\ln R   &\mbox{ \ if \ }  a = 1
		\\
		1 &\mbox{ \ if \ }  a>1,
	\end{cases}
\end{equation} 
there exists a solution $\phi = \phi[h]$ linearly depending on $h$ with the estimate
\begin{equation}
\langle y \rangle |\nabla \phi| +
	|\phi| \le C v
	\theta_{R, a}^1  R^n  \langle y\rangle^{-1-n}
	\|h\|_{v,2+a}^{R},
\end{equation}
where $C>0$ is a constant independent of $\tau_0, \tau_1, R$. Moreover, $\phi = \phi_1(|y|,\tau) \Upsilon_{1,i}(\frac{y}{|y|})$.
\end{proposition}

\begin{proof}

The proof is similar to \cite[pp. 117-118]{infi4d}, \cite[pp. 341-342]{Green16JEMS} with application of the regularity theory in Dong-Escauriaza-Kim \cite[Lemma 4.13]{dong21-C0-para}. We give details in Section \ref{mode1-orth-pf-sec} for completeness.
\end{proof}

Re-gluing inner linear theory for mode $1$.
\begin{proposition}\label{regluing-mode1-prop}
	
	Given an integer $n\ge 3$, $1\le \tau_0 <\tau_1 \le \infty$, $i\in\{ 1,2,\dots, n\}$, consider
	\begin{equation*}
		\partial_{\tau} \phi=\Delta \phi + pU^{p-1}\phi + h + \varrho(\tau) \eta(y) Z_i(y)
		\mbox{ \ for \ } \tau \in (\tau_0, \tau_1), y \in B_{R}, 
		\quad
		\phi(y,\tau_0) = 0
		\mbox{ \ in \ } B_{R(\tau_0)},
	\end{equation*}
	where $\|h\|_{v,2+a}^{R}<\infty$,  $h = h_1(|y|,\tau) \Upsilon_{1,i}(\frac{y}{|y|})$.
	Suppose that 
	\begin{equation}
	\begin{aligned}
&
		a\in (0,2] \setminus \{n-2\},
		\
		\epsilon_0 \in (0, \min\{a,1\} ),
		\ 
		4 \le R_0 < R,
		\ 
        v\ge 0,
        \ 
		|\dot{v}| = O(\tau^{-1} v), \  |\dot{R}_0| = O(\tau^{-1} R_0),
		\\
		&
		|\dot{R}| = O(\tau^{-1} R),
		\ 
		R^{2+\epsilon} = O(\tau) 
	\mbox{ with a small constant } \epsilon>0,
	\end{aligned}
	\end{equation}
then for $\sup_{s\ge \tau_0} s^{-1} R_0(s)^n \theta_{R_0(s), a}^1 \le d_0$ with a constant $d_0>0$ sufficiently small and
\begin{equation}
	\theta_{R_0,a}^1 = 
	\begin{cases}
		R_0^{1-a}  & \mbox{ \ if \ } a < 1
		\\
		\ln R_0   &\mbox{ \ if \ }  a = 1
		\\
		1 &\mbox{ \ if \ }  a>1,
	\end{cases}
\end{equation}
and $\inf_{s \ge \tau_0} R_0(s) \ge C_0$ with a constant $C_0>0$ sufficiently large, there exists a solution $(\phi, \varrho) = (\phi[h], \varrho[h])$ linearly depending on $h$, and satisfying $\phi \in C^{1+\sigma,\frac{1+\sigma}{2}} \big( \{ (y,\tau) \mid \tau \in (\tau_0, \tau_1), y\in B_{R} \} \big)$, $\sigma \in (0,1)$,
	\begin{equation}
		\langle y\rangle |\nabla \phi| + |\phi|
		\le D_1 \theta_{R_0, a}^1  R_0^n  v \langle y \rangle^{- \min\{ a, n-2 \} } \|h\|_{v,2+a}^{R},
	\end{equation}
	\begin{equation}
		\varrho(\tau)
		=
		-\Big(\int_{B_{2} } \eta(y) Z_{i}^2(y) \rmd y\Big)^{-1} \Big(\int_{B_{R_0} }
		h(y,\tau)  Z_{i}(y) \rmd y
		+ \varrho_{1}[h](\tau)  \Big),
	\end{equation}
	where $\varrho_{1}[h](\tau)$ linearly depends on $h$ and satisfies the estimate
	\begin{equation}
		|\varrho_{1}[h](\tau)| \le D_1 v
		R_{0}^{-\epsilon_0} \| h \|_{v,2+a}^{R},
	\end{equation}
	where $D_1 >0$ is a constant independent of $\tau_0, \tau_1, R, R_0$. Moreover, $\phi = \phi_1(|y|,\tau) \Upsilon_{1,i}(\frac{y}{|y|})$.

	Under the additional assumption that $v, R_0 \in \mathbf{AP}((\tau_0,\infty))$, $C_{R}^{-1} R(\tau) \le R(s) \le C_{R} R(\tau)$, $R_0(s) \le (9 C_{R})^{-1} R(\tau)$ with a constant $C_{R} \ge 1$ for all $\tau \in (\tau_0, \tau_1)$, $s\in [\max\{ \tau_0, \tau/2 \}, \tau]$, then given $\alpha \in (0,1)$, for all $\tau \in (\tau_0, \tau_1)$, $\kappa_1, \kappa_2 \in ( \max\{\tau_0, \tau- \delta_1^2 \}, \tau]$, $\delta_1 \in (0, \min\{ (\frac{\tau}{8})^{\frac{1}{2}}, \frac{R}{9 C_R} \} ]$, we have
	\begin{equation}
		|\varrho_{1}(\kappa_1) - \varrho_{1}(\kappa_2)|
		\le D_2 v
		\big[ | \kappa_1 - \kappa_2|^{\alpha}
		\delta_1^{- 2 \alpha}
		R_{0}^{-\epsilon_0}
		(1 + \delta_1^{2}) +
		R_{0}^{-\epsilon_0-4} | R_0(\kappa_1) - R_0(\kappa_2) | \big] \|h\|_{v,2+a}^{R},
	\end{equation}
	where $D_2>0$ is a constant independent of $\tau_0, \tau_1, R, R_0$.

\end{proposition}

\begin{remark}\label{qd26Apr30-1-rmk}

For brevity in application, we note that for $a\ge 0$, $R_0\ge 4$, we have $\theta_{R_0, a}^1  R_0^n \lesssim R_0^{n+1}$, and then $R_0^{n+1} \ll \tau$ implies $\tau^{-1} R_0^n \theta_{R_0, a}^1 \ll 1$.
\end{remark}

\begin{proof}

The proof is similar to Proposition \ref{mode0-regluing-prop}. We give details for completeness.
Denote $\| h \| = \| h \|_{v,2+a}^{R}$ in this proof. Set 
	\begin{equation*}
		\phi = \Psi(y,\tau) + \eta(2 y / R_0) \zeta(y,\tau).
	\end{equation*}
	To solve $\partial_{\tau} \phi=\Delta \phi + pU^{p-1}\phi + h + \varrho \eta Z_i$ for $\tau \in (\tau_0, \tau_1)$, $y \in B_{R}$, since $4 \le R_0 < R$ implies $\eta =  \eta(2 y / R_0) \eta$, then it suffices to  consider the following gluing system for $(\Psi, \zeta)$
\begin{equation}\label{1d26Mar819-3}
\begin{aligned}
&
		\partial_{\tau} \Psi
		= 
		\Delta \Psi 
		+
		J[\Psi,\zeta] \1_{|y|<R}
		\mbox{ for } \tau \in (\tau_0, \tau_1), y \in B_{4 R},
\\
&
		\Psi = 0
		\mbox{ for } \tau \in (\tau_0, \tau_1), y \in \partial B_{4 R},
		\quad
		\Psi(\cdot,\tau_0) = 0
		\mbox{ in } B_{4 R(\tau_0)},
\end{aligned}
	\end{equation}
\begin{equation}\label{qd26Mar819-4}
		\partial_{\tau} \zeta
		= 
		\Delta \zeta + pU^{p-1} \zeta + H[\Psi]
		\mbox{ \ for \ } \tau \in (\tau_0, \tau_1), y \in B_{R_0},
	\end{equation}
	where
	\begin{equation}
		\begin{aligned}
			&
			J[\Psi,\zeta] :=
			(pU^{p-1}\Psi + h) [
			1- \eta(2 y / R_0)
			]
			+ 
			A[\zeta],
			\\
			&
			A[\zeta] :=
			\zeta \Delta [\eta(2 y / R_0)]  + 2 \nabla \zeta \cdot \nabla  [\eta(2 y / R_0)]   - \zeta \partial_{\tau} [\eta(2 y / R_0)],
			\quad
			H[\Psi] := p U^{p-1} \Psi
			+ h + \varrho \eta Z_{i}.
		\end{aligned}
	\end{equation}

	To meet the orthogonality condition
\begin{equation}\label{qd26Mar23-2}
		\int_{B_{R_0}} H[\Psi](z,\tau) Z_{i}(z) \rmd z = 0 \mbox{ \ for \ } \tau \in (\tau_0, \tau_1),
	\end{equation}
	we take
	\begin{equation}\label{qd26Mar23-3}
		\begin{aligned}
			&
			\varrho = \varrho[\Psi](\tau) :=
			C_1 
			\Big(
			\int_{B_{R_0}} h(z,\tau) Z_{i}(z) \rmd z 
			+
			\varrho_{1}
			\Big),
			\quad
			C_{1} :=  -\Big( \int_{B_{2} } \eta(z) Z_{i}^2(z) \rmd z \Big)^{-1},
			\\
			&
			\varrho_{1} = \varrho_{1}[\Psi](\tau) := \int_{B_{R_0} } pU^{p-1} (z) \Psi(z,\tau) Z_{i}(z) \rmd z.
		\end{aligned}
	\end{equation}

	The mappings for solving \eqref{1d26Mar819-3} and \eqref{qd26Mar819-4} are given as the following forms
	\begin{equation}\label{qd26Mar23-7}
		\Psi =
		\mathcal{T}_{{\rm{ou}}} [
		J[\Psi,\zeta] \1_{|y|<R}  ](y,\tau),
		\quad
		\zeta = 
		\mathcal{T}_{\rm{in}} [ H[\Psi] ](y,\tau) \mbox{ \ with \ } \zeta(y,\tau_0) = 0 \mbox{ \ in \ } B_{R_0(\tau_0)},
	\end{equation}
	where $\mathcal{T}_{\rm{ou}}$ is the linear mapping given by solving the homogeneous part of \eqref{1d26Mar819-3}; and $\mathcal{T}_{\rm{in}}$ is given by Proposition \ref{mode1-orth-prop} (with $R(\tau) = R_0(\tau)$) since \eqref{qd26Mar23-2} are satisfied by the choice of $\varrho$.

	Denote the leading term of $H[\Psi]$ as $H_{1} := h+ C_{1} \eta Z_{i} \int_{B_{R_0} }
	h(z,\tau) Z_{i}(z) \rmd z$. For $a>-1$, we have $\| H_{1} \|_{v,2+a}^{R_0} \lesssim \|h\|$. Since $h = h_1(|y|,\tau) \Upsilon_{1,i}(\frac{y}{|y|})$ and $\eta$ is radially symmetric, then $H_{1} = (H_{1})_{1}(|y|,\tau) \Upsilon_{1,i}(\frac{y}{|y|})$.
	If $H_1$ satisfies the orthogonality condition in $B_{R_0}$, under the assumption
\begin{equation}\label{qd26Mar23-1} 
			R_0 \ge 2,
            \ 
            v\ge 0,
			\ 
            a>-1,
\
			|\dot{v}| = O(\tau^{-1} v), \  |\dot{R}_0| = O(\tau^{-1} R_0),
			\ 
			\tau^{-1} R_0^n \theta_{R_0, a}^1 \le d_0 \ll 1,
	\end{equation}
then Proposition \ref{mode1-orth-prop} gives the apriori estimate
	\begin{equation}\label{qd26Mar20-192}
		\langle y \rangle 
		|\nabla \mathcal{T}_{{\rm{in}}}[H_1]|
		+	|\mathcal{T}_{{\rm{in}}}[H_1] | \le
		D_{\rm{in}} \|h\| v
		\theta_{R_0, a}^1  R_0^n  \langle y\rangle^{-1-n}, 
	\end{equation}
	where  $D_{{\rm{in}}}\ge 1$ is a large constant. Moreover, $\mathcal{T}_{{\rm{in}}}[H_1] = (\mathcal{T}_{{\rm{in}}}[H_1])_{1}(|y|,\tau) \Upsilon_{1,i}(\frac{y}{|y|})$. For this reason, we will solve $\zeta$ in the space
	\begin{equation*}
		\mathcal{B}_{\rm{in}} := \big\{
		g \mid
		\| g \|_{\rm{in}} \le 
		2D_{\rm{in}} \|h\|, \  g = g_1(|y|,\tau) \Upsilon_{1,i}(\frac{y}{|y|}) \big\}
	\end{equation*}
	endowed with the norm
	\begin{equation*}
		\| g \|_{\rm{in}} := \inf \big\{ C \mid \langle y \rangle |\nabla g(y,\tau)|  + |g(y,\tau)| \le C v
		\theta_{R_0, a}^1  R_0^n  \langle y\rangle^{-1-n} \mbox{ for all }  \tau\in (\tau_0, \tau_1), y\in B_{R_0} \big\}.
	\end{equation*}
	It is ready to get that $\mathcal{B}_{\rm{in}}$ is a complete norm space.

	For any $\zeta \in \mathcal{B}_{{\rm{in}}}$, let us estimate $J[0,\zeta] = h [
	1- \eta(2 y / R_0)
	]
	+ 
	A[\zeta]$. Under the assumption $|\dot{R}_0| = O(R_0^{-1})$, then
	\begin{equation}\label{qd26Mar18-191}
		|A[\zeta]| \lesssim
		\1_{R_0/2 \le |y| \le R_0}
		\| \zeta \|_{\rm{in}} v 
		\begin{cases}
			\langle y \rangle^{-a-2}  & \mbox{ \ if \ } a < 1
			\\
			\langle y \rangle^{-3} \ln R_0   &\mbox{ \ if \ }  a = 1
			\\
			\langle y \rangle^{-3} &\mbox{ \ if \ }  a>1
		\end{cases}
		\lesssim
		\1_{R_0/2 \le |y| \le R_0}
		\| \zeta \|_{\rm{in}} v R_0^{-\epsilon_0} \langle y \rangle^{-2-2 a_1},
\end{equation}
where for the last step, we require
\begin{equation}
	0< \epsilon_0 < \min\{a,1\}, \ 0< a_1 < (\min\{a,1\}-\epsilon_0)/2.
\end{equation}
Since for $\zeta \in \mathcal{B}_{{\rm{in}}}$, $\zeta = \zeta_1(|y|,\tau) \Upsilon_{1,i}(\frac{y}{|y|})$ and $\eta(\cdot)$ is radially symmetric, we have 
	\begin{equation*}
		\nabla \zeta \cdot \nabla  [\eta(2 y / R_0)] = \Upsilon_{1,i}(\frac{y}{|y|}) \partial_{|y|} [ \zeta_1(|y|,\tau) ] \partial_{|y|}  [\eta(2 |y| / R_0)].
	\end{equation*}
	It follows that $A[\zeta] = (A[\zeta])_{1}(|y|,\tau) \Upsilon_{1,i}(\frac{y}{|y|})$. Also, we have
\begin{equation}\label{qd26Mar18-192}
		| h [ 1- \eta(2 y / R_0) ] | \lesssim 
		\1_{ |y| \ge R_0/2 } v\langle y\rangle^{-2-a } \| h\|
		\lesssim 
		\1_{ |y| \ge R_0/2 } v
		R_{0}^{-\epsilon_0}
		\langle y \rangle^{-2-a_1}
		\| h \|
	\end{equation}
	provided
	\begin{equation*}
		0<\epsilon_0<a, \ 0< a_1 \le a-\epsilon_0.
	\end{equation*}
	Since $h = h_1(|y|,\tau) \Upsilon_{1,i}(\frac{y}{|y|})$ and $\eta(\cdot)$ is radially symmetric, then $h [ 1- \eta(2 y / R_0) ] = h_1(|y|,\tau) \Upsilon_{1,i}(\frac{y}{|y|}) [ 1- \eta(2 |y| / R_0) ]$.
	Thus, we conclude
	\begin{equation*}
		|J[0, \zeta] \1_{|y|<R} | \lesssim D_{\rm{in}} \|h\| v 
		R_0^{-\epsilon_0} \langle y \rangle^{-2-a_1}
		\mbox{ \ and \ }
		J[0, \zeta] \1_{|y|<R} = (J[0, \zeta] \1_{|y|<R})_{1}(|y|,\tau) \Upsilon_{1,i}(\frac{y}{|y|}).
	\end{equation*}
	We restrict $a_1 \in (0, n-2)$.
	Under the assumption 
\begin{equation}\label{qd26Mar21-193}
		R\ge 2, 
		\
        v\ge 0,
        \ 
		|\dot{v}| = O(\tau^{-1} v),
		\ 
		|\dot{R}_0| = O(\tau^{-1} R_0),
		\ 
		|\dot{R}| = O(\tau^{-1} R),
		\ 
		R^{2+\epsilon_1} \ll \tau 
		\mbox{ with } 
		\epsilon_1 > 0,
	\end{equation}
	by Lemma \ref{chiM-eq-lem} and Remark \ref{qd26Mar13-3-rmk}, then
	\begin{equation}\label{qd26Mar2019-1}
		| \mathcal{T}_{{\rm{ou}}}[J[0, \zeta] \1_{|y|<R}] |
		\le 
		D_{\rm{ou}}	D_{\rm{in}} \|h\| v
		R_{0}^{-\epsilon_0} \langle y \rangle^{-a_1}
		\mbox{ \ for \ } \tau\in (\tau_0, \tau_1), y\in B_{4R}
	\end{equation}
	with a large constant $D_{\rm{ou}} \ge 1$. Moreover, $\mathcal{T}_{\rm{ou}}[J[0, \zeta] \1_{|y|<R} ] = (\mathcal{T}_{\rm{ou}}[J[0, \zeta] \1_{|y|<R} ])_{1}(|y|,\tau) \Upsilon_{1,i}(\frac{y}{|y|})$.

	This suggests that we solve $\Psi$ in the space
	\begin{equation*}
		\mathcal{B}_{\rm{ou}} := \big\{
		f \mid
		\| f \|_{\rm{ou}} \le
		2 D_{\rm{ou}}	D_{\rm{in}} \|h\|, \ f = f_1(|y|,\tau) \Upsilon_{1,i}(\frac{y}{|y|}) \big\}
	\end{equation*}
	endowed with the norm
	\begin{equation*}
		\| f \|_{\rm{ou}} := \inf
		\big\{ C \mid  |f(y,\tau)| \le C  v
		R_{0}^{-\epsilon_0} \langle y \rangle^{-a_1}
		\mbox{ for all } \tau\in (\tau_0,\tau_1), y\in B_{4 R} \big\}.
	\end{equation*}
	$\mathcal{B}_{\rm{ou}}$ is a complete norm space.
	For any $\Psi \in \mathcal{B}_{\rm{ou}}$, we have $pU^{p-1} \Psi [ 1- \eta(2 y / R_0) ] \1_{|y|<R} = pU^{p-1}(|y|) \Psi_1(|y|,\tau) $ $ \Upsilon_{1,i}(\frac{y}{|y|}) [ 1- \eta(2 |y| / R_0) ] \1_{|y|<R}$, and
\begin{equation}\label{qd26Mar21-191}
		\begin{aligned}
			&
			| pU^{p-1} \Psi [ 1- \eta(2 y / R_0) ] \1_{|y|<R} |
			\lesssim 
			\1_{R_0/2 \le |y| < R} \| \Psi \|_{\rm{ou}}  v
			R_{0}^{-\epsilon_0} \langle y \rangle^{-4-a_1} 
			\\
			\lesssim \ & 
			\1_{R_0/2 \le |y| < R} \big( \inf_{s \ge \tau_0} R_0(s) \big)^{-2} \| \Psi \|_{\rm{ou}} v
			R_{0}^{-\epsilon_0} \langle y \rangle^{-2-a_1}.
		\end{aligned}
	\end{equation}
	Since $\| \Psi \|_{\rm{ou}} \le 2 D_{\rm{ou}}	D_{\rm{in}} \|h\|$ and $\inf_{s \ge \tau_0} R_0(s) \ge C_0 \gg 1$, by \eqref{qd26Mar21-193} and Lemma \ref{chiM-eq-lem}, we have $ \mathcal{T}_{\rm{ou}}[
	J[\Psi, \zeta] \1_{|y|<R} ]  \in \mathcal{B}_{\rm{ou}} $. For $\Psi \in \mathcal{B}_{\rm{ou}}$ and $a\in (0,2]$, we have $p U^{p-1} \Psi 
	+ \eta Z_{i} C_{1} \varrho_{1} = (p U^{p-1} \Psi 
	+ \eta Z_{i} C_{1} \varrho_{1})_{1}(|y|,\tau) \Upsilon_{1,i}(\frac{y}{|y|}) $ and
	\begin{equation}\label{qd26Mar21-192}
		\| 
		p U^{p-1} \Psi 
		+ \eta Z_{i} C_{1} \varrho_{1} \|_{v,2+a}^{R_0} 
		\lesssim 
		\| \Psi \|_{\rm{ou}} 
		(\inf_{s \ge \tau_0} R_0(s))^{-\epsilon_0}.
	\end{equation}
	Under the assumption \eqref{qd26Mar23-1},
	by \eqref{qd26Mar23-2} and Proposition \ref{mode1-orth-prop}, since $\| \Psi \|_{\rm{ou}} \le 2 D_{\rm{ou}}	D_{\rm{in}} \|h\|$ and $\inf_{s \ge \tau_0} R_0(s) \ge C_0 \gg 1$, we have
	$ \mathcal{T}_{\rm{in}}[H[\Psi] ] \in \mathcal{B}_{\rm{in}} $.

	The contraction mapping property can be deduced similarly. Indeed, for any $(\tilde{\Psi}_j, \tilde{\zeta}_j) \in \mathcal{B}_{\rm{ou}} \times \mathcal{B}_{\rm{in}}$, $j=1,2$,
\begin{equation*}
			( J[\tilde{\Psi}_1, \tilde{\zeta}_1]
			-
			J[\tilde{\Psi}_2, \tilde{\zeta}_2] ) \1_{|y|<R}
			=
			p U^{p-1} (\tilde{\Psi}_1 - \tilde{\Psi}_2) [
			1- \eta(2 y / R_0)
			] \1_{|y|<R}
			+ 
			A[\tilde{\zeta}_1 - \tilde{\zeta}_2] \1_{|y|<R}.
	\end{equation*}
	Since
	\begin{equation*}
		\begin{aligned}
			&
			\big| p U^{p-1} (\tilde{\Psi}_1 - \tilde{\Psi}_2) [
			1- \eta(2 y / R_0)
			] \1_{|y|<R} \big|
			\stackrel{\eqref{qd26Mar21-191}}{\lesssim}
			\1_{R_0/2 \le |y| < R} (\inf_{s\ge \tau_0} R_0(s))^{-2}
			\| \tilde{\Psi}_1 - \tilde{\Psi}_2 \|_{\rm{ou}} v
			R_{0}^{-\epsilon_0} \langle y \rangle^{-2-a_1},
			\\
			&
			| A[\tilde{\zeta}_1 - \tilde{\zeta}_2] |
			\stackrel{\eqref{qd26Mar18-191}}{\lesssim} 
			\1_{R_0/2 \le |y| \le R_0}
			(\inf_{s\ge \tau_0} R_0(s))^{-a_1}
			\| \tilde{\zeta}_1 - \tilde{\zeta}_2 \|_{\rm{in}} v 
			R_0^{-\epsilon_0} \langle y \rangle^{-2-a_1},
		\end{aligned}
	\end{equation*}
	then $|( J[\tilde{\Psi}_1, \tilde{\zeta}_1]
	-
	J[\tilde{\Psi}_2, \tilde{\zeta}_2] ) \1_{|y|<R}| \lesssim (\inf_{s\ge \tau_0} R_0(s))^{-a_1}
	( \| \tilde{\Psi}_1 - \tilde{\Psi}_2 \|_{\rm{ou}} + 
	\| \tilde{\zeta}_1 - \tilde{\zeta}_2 \|_{\rm{in}} ) v 
	R_0^{-\epsilon_0} \langle y \rangle^{-2-a_1}$. Similar to \eqref{qd26Mar2019-1},
	\begin{equation}\label{qd26Mar19-192}
		\big\| \mathcal{T}_{\rm{ou}} \big[ J[\tilde{\Psi}_1, \tilde{\zeta}_1] \1_{|y|<R} \big] - \mathcal{T}_{\rm{ou}} \big[
		J[\tilde{\Psi}_2, \tilde{\zeta}_2] \1_{|y|<R} \big] \big\|_{\rm{ou}}
		\lesssim 
		(\inf_{s\ge \tau_0} R_0(s))^{-a_1}
		\big(
		\| \tilde{\Psi}_1 - \tilde{\Psi}_2 \|_{\rm{ou}}
		+
		\| \tilde{\zeta}_1 - \tilde{\zeta}_2 \|_{\rm{in}}
		\big).
	\end{equation}
	\begin{equation*}
		H[\tilde{\Psi}_1] - H[\tilde{\Psi}_2]
		=
		p U^{p-1} (\tilde{\Psi}_1 - \tilde{\Psi}_2) + \eta Z_{i} C_1
		\int_{B_{R_0} } pU^{p-1} (z) (\tilde{\Psi}_1 - \tilde{\Psi}_2)(z,\tau) Z_{i}(z) \rmd z.
	\end{equation*}
	It holds that
	\begin{equation*}
		\int_{B_{R_0}} (H[\tilde{\Psi}_1] - H[\tilde{\Psi}_2])(w,\tau) Z_{i}(w) \rmd w = 0,
		\quad
		| H[\tilde{\Psi}_1] - H[\tilde{\Psi}_2] |
		\lesssim  (\inf_{s \ge \tau_0} R_0(s))^{-\epsilon_0} \| \tilde{\Psi}_1 - \tilde{\Psi}_2 \|_{\rm{ou}} v \langle y \rangle^{-2-a},
	\end{equation*}
	where the second part is similar to \eqref{qd26Mar21-192}.
	Similar to \eqref{qd26Mar20-192},
	\begin{equation}\label{qd26Mar19-193}
		\|  \mathcal{T}_{\rm{in}}[H[\tilde{\Psi}_1]] - \mathcal{T}_{\rm{in}}[H[\tilde{\Psi}_2]]  \|_{\rm{in}} \lesssim  (\inf_{s \ge \tau_0} R_0(s))^{-\epsilon_0} \| \tilde{\Psi}_1 - \tilde{\Psi}_2 \|_{\rm{ou}}.
	\end{equation}

	Combining \eqref{qd26Mar19-192}, \eqref{qd26Mar19-193}, we get the contraction mapping property for \eqref{qd26Mar23-7}. By the Banach fixed-point theorem, we find a solution 
	$
	(\Psi, \zeta) \in \mathcal{B}_{\rm{ou}} \times \mathcal{B}_{\rm{in}}
	$
	for \eqref{1d26Mar819-3} and \eqref{qd26Mar819-4}. Denote
	\begin{equation*}
		\mathcal{X} := \big\{
		f \mid
		\| f \|_{\rm{ou}} <\infty, \ f = f_1(|y|,\tau) \Upsilon_{1,i}(\frac{y}{|y|}) \big\} \times \big\{
		g \mid
		\| g \|_{\rm{in}} <\infty, \  g = g_1(|y|,\tau) \Upsilon_{1,i}(\frac{y}{|y|}) \big\}
	\end{equation*}  
	including $\mathcal{B}_{\rm{ou}} \times \mathcal{B}_{\rm{in}}$. We can repeat the contraction mapping argument in $\mathcal{X}$. Since \eqref{1d26Mar819-3} and \eqref{qd26Mar819-4} is a linear equation system, by the uniqueness in the Banach fixed-point theorem applied in $\mathcal{X}$, $(\Psi, \zeta)$ depends on $h$ linearly.
	\begin{equation*}
		\phi(y,\tau_0) = \Psi(y,\tau_0) + \eta( 2 y / R_0(\tau_0)) \zeta(y,\tau_0)
		= 0 \mbox{ \ in \ } B_{4 R(\tau_0)}.
	\end{equation*}
Hereafter we always will regard $D_{\rm{ou}}$, $D_{{\rm{in}}}$  as general constants. By \eqref{qd26Mar23-3} and $\Psi \in \mathcal{B}_{\rm{ou}}$,
	\begin{equation*}
		|\varrho_{1}(\tau)| 
		 \lesssim v
		R_{0}^{-\epsilon_0} \|h\|.
	\end{equation*}

Recall \eqref{qd26Mar18-191} and \eqref{qd26Mar18-192}.  Since
$
\1_{R_0/2 \le |y| \le R_0} \langle y \rangle^{-3} \sim \1_{R_0/2 \le |y| \le R_0} R_0^{a-1} \langle y \rangle^{-2-a}
$,
then
$
|A[\zeta]| \lesssim $ $
	\1_{R_0/2 \le |y| \le R_0} $ $
	\|h\| v  R_0^{\max\{\epsilon_2, a-1\}} \langle y \rangle^{-2-a}
$
with $0< \epsilon_2 \ll 1$. It follows that
$
|J[0,\zeta]|
\lesssim \|h\| v  R_0^{\max\{\epsilon_2, a-1\}} \langle y \rangle^{-2-a}
$.
To improve the spatial decay of $\Psi$, under the assumption \eqref{qd26Mar21-193}, when $a\in (0, n-2)$, applying Lemma \ref{chiM-eq-lem} to \eqref{1d26Mar819-3} finitely many time, we have
$
|\Psi| \lesssim
\|h\| v  R_0^{\max\{\epsilon_2, a-1\}} \langle y \rangle^{-a}
$.
Thus,
\begin{equation*}
	|\phi| = | \Psi + \eta(2 y / R_0) \zeta |
	\lesssim 
	\|h\| v  R_0^{\max\{\epsilon_2, a-1\}} \langle y \rangle^{-a} + 
	\1_{|y| \le R_0} \|h\| v
	\theta_{R_0, a}^1  R_0^n  \langle y\rangle^{-1-n}
\lesssim   \|h\| \theta_{R_0, a}^1  R_0^n  v \langle y \rangle^{-a}.
\end{equation*}

When $a>n-2 \ge 1$, we have
$
	|A[\zeta]| \lesssim \1_{R_0/2 \le |y| \le R_0} \|h\| v R_0^{(n+\epsilon_3)-3} \langle y \rangle^{-(n+\epsilon_3)}
$ with $0< \epsilon_3 \ll 1$.

It follows that
$ |J[0,\zeta]|
	\lesssim \|h\| v R_0^{(n+\epsilon_3)-3} \langle y \rangle^{-(n+\epsilon_3)}
$. By the assumption \eqref{qd26Mar21-193}, when $a > n-2$, applying Lemma \ref{chiM-eq-lem} to \eqref{1d26Mar819-3} finitely many time, we have $|\Psi| \lesssim \|h\| v R_0^{(n+\epsilon_3)-3} \langle y \rangle^{2-n}$. Thus,
\begin{equation*}
		|\phi| = | \Psi + \eta(2 y / R_0) \zeta |
		\lesssim 
		\|h\| v R_0^{(n+\epsilon_3)-3} \langle y \rangle^{2-n} + 
		\1_{|y| \le R_0} \|h\| v
		\theta_{R_0, a}^1  R_0^n  \langle y\rangle^{-1-n}
		\lesssim  \|h\| \theta_{R_0, a}^1  R_0^n v \langle y \rangle^{2-n}.
\end{equation*}

In sum, $|\phi| \lesssim \|h\| \theta_{R_0, a}^1  R_0^n  v \langle y \rangle^{- \min\{ a, n-2 \} }$. Since
$
| (h + \varrho \eta Z_{i}) \1_{|y|<R} |  \lesssim \| h \| v \langle y \rangle^{-2-a} \1_{|y|<R}
$, by the scaling argument, we get the estimate of $\nabla \phi$.

The quantitative estimate of $ \varrho_{1}(\kappa_1) - \varrho_{1}(\kappa_2) $ is similar to Proposition \ref{mode0-regluing-prop} with slight modification that $|Z_{i}| \lesssim \langle y \rangle^{1-n}$ for the second part of $\varrho_{1}(\kappa_1) - \varrho_{1}(\kappa_2)$.
\end{proof}

\subsection{Higher modes}

Given $0< \tau_0 <\tau_1 \le \infty$, a function $\mathcal{R} = \mathcal{R}(\tau) >0$, denote
\begin{equation}
\begin{aligned}
{\mathbf{SO}}_{\tau_0, \tau_1}^{\mathcal{R}} := \Big\{ f(y,\tau) \ \Big| \ & 
	\int_{S^{n-1}} f(r w, \tau) \Upsilon_{i,j}(w) \rmd w =0
	\mbox{ \ for \ } \tau \in (\tau_0, \tau_1),
    \
    r\in (0, \mathcal{R}),
    \\
    &
	(i,j) \in \{ (0,1), (1,k) \mid k=1,2,\dots,n \} \Big\}.
\end{aligned}
\end{equation}
Obviously, given $0< R_1(\tau) \le R_2(\tau)$, we have $ {\mathbf{SO}}_{\tau_0, \tau_1}^{R_2}
\subset
{\mathbf{SO}}_{\tau_0, \tau_1}^{R_1} $.

\begin{lemma}\label{phi-perp-lemma}
	
	Given an integer $n\ge 3$, $1\le \tau_0 <\tau_1 \le \infty$, functions $R(\tau) \ge 2$, $v(\tau) \ge 0$ defined in $(\tau_0, \tau_1)$, consider
	\begin{equation*}
			\partial_{\tau} \phi 
			= 
			\Delta \phi + p U^{p-1} \phi + h 
			\mbox{ \ for \ } \tau \in (\tau_0, \tau_1), y \in B_{R},
			\quad
			\phi =  0
			\mbox{ \ for \ } \tau \in (\tau_0, \tau_1), y \in \partial B_{R},
	\quad
			\phi(\cdot,\tau_0) = 0 
			\mbox{ \ in \ } B_{R(\tau_0)},
	\end{equation*}
	where $\|h\|_{v,a}^{R} <\infty$, $a\in \mathbb{R}$, $h \in {\mathbf{SO}}_{\tau_0, \tau_1}^{R}$. There exists a unique solution $\phi = \phi[h]$ linearly depending on $h$, and satisfying $\phi \in {\mathbf{SO}}_{\tau_0, \tau_1}^{R}$.

If $h$ is even with respect to the $i$-th component of $y$ for some $i \in \{1,2,\dots, n\}$ and all $\tau \in (\tau_0, \tau_1)$, then so is $\phi$.

    Suppose that $|\dot{v}| = O(\tau^{-1} v)$, $|\dot{R}| = O(\tau^{-1} R)$, $v, R, \ln R \in \mathbf{AP}((\tau_0,\infty))$, $\mathbf{P}_1[R] < 2^{-1}$, for $\sup_{s\ge \tau_0} s^{-1} R(s)^{2 + \epsilon} \le d_0$ with a constant $\epsilon > 0$ and a sufficiently small constant $d_0>0$, then it holds that	\begin{equation}\label{qd26Mar8-1}
\langle y\rangle |\nabla \phi| 	+	|\phi| \le C 
		v
		\big(
		\Theta_{R, a}^0(|y|) 
		+
		\theta_{R, a}^0  R 	\langle  y \rangle^{2-n}
		\big)
		\|h\|_{v,a}^{R}
	\end{equation}
with
\begin{equation}
\Theta_{R, a}^0(|y|) = 
\begin{cases}
	R^{2-a} &\mbox{ \ if \ } a<2 
	\\
	\ln R                      
	&\mbox{ \ if \ } a=2 
	\\
	\langle |y| \rangle^{2-a} 
	&  \mbox{ \ if \ } 2<a<n
	\\
	\langle |y| \rangle^{2-n}\ln(|y|+2)   
	&\mbox{ \ if \ } a = n
	\\
	\langle |y| \rangle^{2-n}   
	&\mbox{ \ if \ } a > n,
\end{cases}
\quad
\theta_{R,a}^0  = 
\begin{cases}
	R^{2-a}  & \mbox{ \ if \ }  a<2 
	\\
	\ln R 
	& \mbox{ \ if \ } a=2 
	\\
	1 & \mbox{ \ if \ } a>2,
\end{cases}
\end{equation}
where
$C>0$ is a constant independent of $\tau_0, \tau_1, R$.

\end{lemma}

\begin{proof}

The proof is similar to \cite[pp. 114-115]{infi4d}, \cite[pp. 337-338]{Green16JEMS} with some modification for $\mathbf{AP}((\tau_0,\infty))$ class. We give details in Section \ref{high-mode-pf-sec} for completeness.
\end{proof}

The re-gluing inner linear theory for higher modes.
\begin{proposition}\label{regluing-higher-mode-prop}
	
	Given an integer $n\ge 4$, $1\le \tau_0 <\tau_1 \le \infty$, consider
	\begin{equation*}
		\partial_{\tau} \phi=\Delta \phi + pU^{p-1}\phi + h 
		\mbox{ \ for \ } \tau \in (\tau_0, \tau_1), y \in B_{R}, 
		\quad
		\phi(y,\tau_0)= 0
		\mbox{ \ in \ } B_{R(\tau_0)},
	\end{equation*}
	where $\|h\|_{v,2+a}^{R}<\infty$, $h \in {\mathbf{SO}}_{\tau_0, \tau_1}^{R}$. Suppose that
	\begin{equation*}
		a\in (0,2],
		\ 
	4 \le R_0 < R, \ 
		|\dot{v}| = O(\tau^{-1} v), \  
		v \in \mathbf{AP}((\tau_0,\infty)),
		\ 
		R^{2+\epsilon} = O(\tau) \mbox{ with a constant } \epsilon>0,
		\
		|\dot{R}| = O(\tau^{-1} R),
	\end{equation*}
then for the constant $R_0 \ge C_0$ with a constant $C_0>0$ sufficiently large, there exists a solution $\phi = \phi[h] \in {\mathbf{SO}}_{\tau_0, \tau_1}^{R}$ linearly depending on $h$, and satisfying $\phi \in C^{1+\sigma,\frac{1+\sigma}{2}} \big( \{ (y,\tau) \mid \tau \in (\tau_0, \tau_1), y\in B_{R} \} \big)$, $\sigma \in (0,1)$,
	\begin{equation}
		\langle y\rangle |\nabla \phi| + |\phi|
		\le C R_0 v \langle y\rangle^{ -\min\{ a, n-3\} } \|h\|_{v,2+a}^{R},
	\end{equation}
where $C>0$ is a constant independent of $\tau_0, \tau_1, R, R_0$.

Moreover, if $h$ is even with respect to the $i$-th component of $y$ for some $i \in \{1,2,\dots, n\}$ and all $\tau \in (\tau_0, \tau_1)$, then $\phi$ can be chosen to inherit the same even symmetry.
	
\end{proposition}

\begin{proof}
	
The proof is similar to Proposition \ref{mode0-regluing-prop}. We give details for completeness. Without loss of generality, we assume that $h$ is evenly symmetric about $y_k$ for $k\in \{1,2,\dots, n\}$. The other cases, including no even symmetry assumption or multiple even symmetry assumptions about spatial variables of $h$, can be handled similarly. Denote $\| h \| = \| h \|_{v,2+a}^{R}$ in this proof. Set 
	\begin{equation*}
		\phi = \Psi(y,\tau) + \eta(2 y / R_0) \zeta(y,\tau)
	\end{equation*}
with a large constant $R_0$. To solve $\partial_{\tau} \phi=\Delta \phi + pU^{p-1}\phi + h$ for $\tau \in (\tau_0, \tau_1)$, $y \in B_{R}$, since $4 \le R_0 < R$ implies $\eta =  \eta(2 y / R_0) \eta$, then it suffices to  consider the following gluing system for $(\Psi, \zeta)$
\begin{equation}\label{1d26Mar408-3}
\begin{aligned}
&
		\partial_{\tau} \Psi
		= 
		\Delta \Psi 
		+
		J[\Psi,\zeta] \1_{|y|<R}
		\mbox{ for } \tau \in (\tau_0, \tau_1), y \in B_{4 R},
        \\
        &
		\Psi = 0
		\mbox{ for } \tau \in (\tau_0, \tau_1), y \in \partial B_{4 R},
		\quad
		\Psi(\cdot,\tau_0) = 0
		\mbox{ in } B_{4 R(\tau_0)},
\end{aligned}
	\end{equation}
	\begin{equation}\label{qd26Mar408-4}
		\partial_{\tau} \zeta
		= 
		\Delta \zeta + pU^{p-1} \zeta + H[\Psi]
		\mbox{ \ for \ } \tau \in (\tau_0, \tau_1), y \in B_{R_0},
	\end{equation}
	where
	\begin{equation}
		\begin{aligned}
			&
			J[\Psi,\zeta] :=
			(pU^{p-1}\Psi + h) [
			1- \eta(2 y / R_0)
			]
			+ 
			A[\zeta],
			\\
			&
			A[\zeta] :=
			\zeta \Delta [\eta(2 y / R_0)]  + 2 \nabla \zeta \cdot \nabla  [\eta(2 y / R_0)],
			\quad
			H[\Psi] := p U^{p-1} \Psi
			+ h.
		\end{aligned}
	\end{equation}

	The mappings for solving \eqref{1d26Mar408-3} and \eqref{qd26Mar408-4} are given as the following forms
	\begin{equation}\label{qd26Mar408-7}
		\Psi =
		\mathcal{T}_{{\rm{ou}}} [
		J[\Psi,\zeta] \1_{|y|<R}  ](y,\tau),
		\quad
		\zeta = 
		\mathcal{T}_{\rm{in}} [ H[\Psi] ](y,\tau) \mbox{ \ with \ } \zeta(y,\tau_0) = 0 \mbox{ \ in \ } B_{R_0(\tau_0)},
	\end{equation}
	where $\mathcal{T}_{\rm{ou}}$ is the linear mapping given by solving the homogeneous part of \eqref{1d26Mar408-3}; and $\mathcal{T}_{{\rm{in}}}$ is given by Lemma \ref{phi-perp-lemma} (with $R(\tau) = R_0$). The leading term of $H[\Psi]$ is $h$. Obviously, $\|h\|_{v,2+a}^{R_0} \le \|h\|$. Since $h \in  {\mathbf{SO}}_{\tau_0, \tau_1}^{R}$ is evenly symmetric about $y_k$, under the assumption
	\begin{equation}
		R_0\ge 2, \ 
        v\ge 0,
        \ 
			|\dot{v}| = O(\tau^{-1} v), \  
			v \in \mathbf{AP}((\tau_0,\infty)),
			\ 
			\tau^{-1} R_0^{2+\epsilon_1} \ll 1 \mbox{ with } \epsilon_1>0,
	\end{equation}
then for $a > 0$, Lemma \ref{phi-perp-lemma} gives the following apriori estimate
	\begin{equation}\label{qd26Mar4020-2}
		\langle y \rangle 
		|\nabla \mathcal{T}_{\rm{in}}[h]|
		+	|\mathcal{T}_{\rm{in}}[h] | \le
		D_{\rm{in}} \|h\| v (\langle y \rangle^{-a} + R_0 \langle  y \rangle^{2-n})
        \mbox{ \ for \ } |y|<R_0, 
	\end{equation}
where $D_{{\rm{in}}}\ge 1$ is a large constant. Moreover, $\mathcal{T}_{\rm{in}}[h] \in {\mathbf{SO}}_{\tau_0, \tau_1}^{R_0}$ is evenly symmetric about $y_k$. For this reason, we will solve for $\zeta$ in the space
	\begin{equation*}
		\mathcal{B}_{\rm{in}} := \{
		g \mid
		\| g \|_{\rm{in}} \le 
		2D_{\rm{in}} \|h\|, \ g \in {\mathbf{SO}}_{\tau_0, \tau_1}^{R_0} \mbox{ is evenly symmetric about $y_k$} \}
	\end{equation*}
	endowed with the norm
	\begin{equation*}
		\| g \|_{\rm{in}} := \inf \big\{ C \mid \langle y \rangle |\nabla g(y,\tau)|  + |g(y,\tau)| \le C v (\langle y \rangle^{-a} + R_0 \langle  y \rangle^{2-n}) \mbox{ for all }  \tau\in (\tau_0, \tau_1), y\in B_{R_0} \big\}.
	\end{equation*}
	It is ready to get that $\mathcal{B}_{\rm{in}}$ is a complete norm space.

	For any $\zeta \in \mathcal{B}_{{\rm{in}}}$, let us estimate $J[0,\zeta] = h [
	1- \eta(2 y / R_0)
	]
	+ 
	A[\zeta]$. Under the assumption
\begin{equation*}
a>0, \quad 
	n>3, \quad 
	0 < \epsilon_0 < \min\{a, n-3\}, \quad 
	0<a_1 \le (\min\{a, n-3\} - \epsilon_0)/2,
\end{equation*}
then
\begin{equation}\label{qd26Mar4018-1}
		|A[\zeta]|
		\lesssim
		\1_{R_0/2 \le |y| \le R_0}
		\| \zeta \|_{\rm{in}} v \langle y \rangle^{-2-\min\{a, n-3\} }
		\lesssim
		\1_{R_0/2 \le |y| \le R_0}
		\| \zeta \|_{\rm{in}} v R_0^{-\epsilon_0} \langle y \rangle^{-2-2a_1}.
\end{equation}
Since for $\zeta \in \mathcal{B}_{{\rm{in}}}$ and $\eta(\cdot)$ is radially symmetric, for $(i,j) \in \{ (0,1), (1,k) \mid k=1,2,\dots,n \}$, we have
\begin{equation*}
	\begin{aligned}
		&
		\int_{S^{n-1}} \big( \nabla \zeta \cdot \nabla  [\eta(2 y / R_0)] \big)(|y| w, \tau) \Upsilon_{i,j}(w) \rmd w
		\\
		= \ & \int_{S^{n-1}}
		\Big( \partial_{|y|} \big[ \zeta(|y| w, \tau) \big] \partial_{|y|}  [\eta(2 |y| / R_0)]
		+
		\Big\langle 
		\nabla_{S^{n-1}} \big[ \zeta(|y| w, \tau) \big], \nabla_{S^{n-1}} [\eta(2 |y| / R_0)]
		\Big\rangle \Big)
		\Upsilon_{i,j}(w) \rmd w
		\\
		&
		\mbox{since $\eta$ is radially symmetric, we have $\nabla_{S^{n-1}} [\eta(2 |y| / R_0)] = 0$. Then}
		\\
		= \ & \partial_{|y|} [\eta(2 |y| / R_0)] \partial_{|y|} \int_{S^{n-1}} \zeta(|y| w, \tau) \Upsilon_{i,j}(w) \rmd w
		= 0.
	\end{aligned}
\end{equation*} 
It follows that $A[\zeta] \in {\mathbf{SO}}_{\tau_0, \tau_1}^{R}$ is evenly symmetric about $y_k$. Also, we have
\begin{equation}\label{qd26Mar4018-2}
		| h [ 1- \eta(2 y / R_0) ] | \lesssim 
		\1_{ |y|\ge R_0/2 } v\langle y\rangle^{-2-a } \| h\|
		\lesssim 
		\1_{ |y|\ge R_0/2 } v
		R_{0}^{-\epsilon_0}
		\langle y \rangle^{-2-a_1}
		\| h \|
	\end{equation}
	provided
	\begin{equation*}
		0<\epsilon_0<a, \ 0< a_1 \le a-\epsilon_0.
	\end{equation*}
	Since $\eta(\cdot)$ is radially symmetric, then $h [ 1- \eta(2 y / R_0) ] \in {\mathbf{SO}}_{\tau_0, \tau_1}^{R}$ is evenly symmetric about $y_k$.
	In sum,
	\begin{equation*}
		|J[0, \zeta] \1_{|y|<R} | \lesssim D_{\rm{in}} \|h\| v 
		R_0^{-\epsilon_0} \langle y \rangle^{-2-a_1}
		\mbox{ \ and \ }
		J[0, \zeta] \1_{|y|<R} \in {\mathbf{SO}}_{\tau_0, \tau_1}^{4 R}
        \mbox{ is evenly symmetric about $y_k$}.
	\end{equation*}
	We restrict $a_1 \in (0, n-2)$.
	Under the assumption 
	\begin{equation}\label{qd26Mar4021-3}
		R\ge 2, 
		\ 
        v\ge 0,
        \ 
		|\dot{v}| = O(\tau^{-1} v),
		\ 
		|\dot{R}| = O(\tau^{-1} R),
		\ 
		R^{2+\epsilon_1} \ll \tau 
		\mbox{ with } 
		\epsilon_1 > 0,
	\end{equation}
	by Lemma \ref{chiM-eq-lem} and Remark \ref{qd26Mar13-3-rmk}, then
	\begin{equation}\label{qd26Mar4020-1}
		| \mathcal{T}_{{\rm{ou}}}[J[0, \zeta] \1_{|y|<R} ] |
		\le 
		D_{\rm{ou}}	D_{\rm{in}} \|h\| v
		R_{0}^{-\epsilon_0} \langle y \rangle^{-a_1}
		\mbox{ \ for \ } \tau\in (\tau_0, \tau_1), y\in B_{4R}
	\end{equation}
	with a large constant $D_{\rm{ou}} \ge 1$. Moreover, $\mathcal{T}_{\rm{ou}}[J[0, \zeta] \1_{|y|<R}] \in {\mathbf{SO}}_{\tau_0, \tau_1}^{4 R}$ is evenly symmetric about $y_k$. This suggests that we solve $\Psi$ in the space
	\begin{equation*}
		\mathcal{B}_{\rm{ou}} := \big\{
		f \mid
		\| f \|_{\rm{ou}} \le
		2 D_{\rm{ou}}	D_{\rm{in}} \|h\|, \ f \in {\mathbf{SO}}_{\tau_0, \tau_1}^{4 R} 
        \mbox{ is evenly symmetric about $y_k$}
        \big\}
	\end{equation*}
	endowed with the norm
	\begin{equation*}
		\| f \|_{\rm{ou}} := \inf
		\big\{ C \mid  |f(y,\tau)| \le C  v
		R_{0}^{-\epsilon_0} \langle y \rangle^{-a_1}
		\mbox{ for all } \tau\in (\tau_0,\tau_1), y\in B_{4 R} \big\}.
	\end{equation*}
	$\mathcal{B}_{\rm{ou}}$ is a complete norm space.

	For any $\Psi \in \mathcal{B}_{\rm{ou}}$, since $U, \eta$ are radially symmetric, we have that $pU^{p-1} \Psi [ 1- \eta(2 y / R_0) ] \1_{|y|<R} \in {\mathbf{SO}}_{\tau_0, \tau_1}^{4 R}$ is evenly symmetric about $y_k$, and
	\begin{equation}\label{qd26Mar4021-1}
			| pU^{p-1} \Psi [ 1- \eta(2 y / R_0) ] \1_{|y|<R} |
			\lesssim 
			\1_{R_0/2 \le |y| < R} \| \Psi \|_{\rm{ou}}  v
			R_{0}^{-\epsilon_0} \langle y \rangle^{-4-a_1} 
			\lesssim  
			\1_{R_0/2 \le |y| < R} R_0^{-2} \| \Psi \|_{\rm{ou}} v
			R_{0}^{-\epsilon_0} \langle y \rangle^{-2-a_1}.
	\end{equation}
	Since $\| \Psi \|_{\rm{ou}} \le 2 D_{\rm{ou}}	D_{\rm{in}} \|h\|$ and $R_0 \gg 1$, by \eqref{qd26Mar4021-3} and Lemma \ref{chiM-eq-lem}, we have $ \mathcal{T}_{\rm{ou}}[
	J[\Psi, \zeta] \1_{|y|<R} ]  \in \mathcal{B}_{\rm{ou}} $.  For $\Psi \in \mathcal{B}_{\rm{ou}}$ and $a\in (0,2]$, we have that $p U^{p-1} \Psi \in {\mathbf{SO}}_{\tau_0, \tau_1}^{4 R}$ is evenly symmetric about $y_k$, and
	\begin{equation}\label{qd26Mar4021-2}
		\| p U^{p-1} \Psi \|_{v,2+a}^{R_0} 
		\lesssim 
		\| \Psi \|_{\rm{ou}} 
		R_0^{-\epsilon_0}.
	\end{equation}
By Lemma \ref{phi-perp-lemma}, since $\| \Psi \|_{\rm{ou}} \le 2 D_{\rm{ou}}	D_{\rm{in}} \|h\|$ and $R_0 \gg 1$, we have
	$ \mathcal{T}_{\rm{in}}[H[\Psi] ] \in \mathcal{B}_{\rm{in}} $.

	The contraction mapping property can be deduced similarly. Indeed, for any $(\tilde{\Psi}_j, \tilde{\zeta}_j) \in \mathcal{B}_{\rm{ou}} \times \mathcal{B}_{\rm{in}}$, $j=1,2$,
	\begin{equation*}
			( J[\tilde{\Psi}_1, \tilde{\zeta}_1]
			-
			J[\tilde{\Psi}_2, \tilde{\zeta}_2] ) \1_{|y|<R}
			= 
			p U^{p-1} (\tilde{\Psi}_1 - \tilde{\Psi}_2) [
			1- \eta(2 y / R_0)
			] \1_{|y|<R}
			+ 
			A[\tilde{\zeta}_1 - \tilde{\zeta}_2] \1_{|y|<R}.
	\end{equation*}
	Since
	\begin{equation*}
		\begin{aligned}
			&
			\big| p U^{p-1} (\tilde{\Psi}_1 - \tilde{\Psi}_2) [
			1- \eta(2 y / R_0)
			] \1_{|y|<R} \big|
			\stackrel{\eqref{qd26Mar4021-1}}{\lesssim}
			\1_{R_0/2 \le |y| < R} R_0^{-2}
			\| \tilde{\Psi}_1 - \tilde{\Psi}_2 \|_{\rm{ou}} v
			R_{0}^{-\epsilon_0} \langle y \rangle^{-2-a_1},
			\\
			&
			| A[\tilde{\zeta}_1 - \tilde{\zeta}_2] |
			\stackrel{\eqref{qd26Mar4018-1}}{\lesssim} 
			\1_{R_0/2 \le |y| \le R_0}
			R_0^{-a_1}
			\| \tilde{\zeta}_1 - \tilde{\zeta}_2 \|_{\rm{in}} v 
			R_0^{-\epsilon_0} \langle y \rangle^{-2-a_1},
		\end{aligned}
	\end{equation*}
	then $|( J[\tilde{\Psi}_1, \tilde{\zeta}_1]
	-
	J[\tilde{\Psi}_2, \tilde{\zeta}_2] ) \1_{|y|<R}| \lesssim R_0^{-a_1}
	( \| \tilde{\Psi}_1 - \tilde{\Psi}_2 \|_{\rm{ou}} + 
	\| \tilde{\zeta}_1 - \tilde{\zeta}_2 \|_{\rm{in}} ) v 
	R_0^{-\epsilon_0} \langle y \rangle^{-2-a_1}$. Similar to \eqref{qd26Mar4020-1},
	\begin{equation}\label{qd26Mar4019-2}
		\big\| \mathcal{T}_{\rm{ou}} \big[ J[\tilde{\Psi}_1, \tilde{\zeta}_1] \1_{|y|<R} \big] - \mathcal{T}_{\rm{ou}} \big[
		J[\tilde{\Psi}_2, \tilde{\zeta}_2] \1_{|y|<R} \big] \big\|_{\rm{ou}}
		\lesssim 
		R_0^{-a_1}
		\big(
		\| \tilde{\Psi}_1 - \tilde{\Psi}_2 \|_{\rm{ou}}
		+
		\| \tilde{\zeta}_1 - \tilde{\zeta}_2 \|_{\rm{in}}
		\big).
	\end{equation}
\begin{equation*}
		| H[\tilde{\Psi}_1] - H[\tilde{\Psi}_2] | =
		| p U^{p-1} (\tilde{\Psi}_1 - \tilde{\Psi}_2) |
		\lesssim R_0^{-\epsilon_0} \| \tilde{\Psi}_1 - \tilde{\Psi}_2 \|_{\rm{ou}} v \langle y \rangle^{-2-a},
	\end{equation*}
Similar to \eqref{qd26Mar4020-2},
	\begin{equation}\label{qd26Mar4019-3}
		\|  \mathcal{T}_{\rm{in}}[H[\tilde{\Psi}_1]] - \mathcal{T}_{\rm{in}}[H[\tilde{\Psi}_2]]  \|_{\rm{in}} \lesssim  R_0^{-\epsilon_0} \| \tilde{\Psi}_1 - \tilde{\Psi}_2 \|_{\rm{ou}}.
	\end{equation}

	Combining \eqref{qd26Mar4019-2}, \eqref{qd26Mar4019-3}, we get the contraction mapping property for \eqref{qd26Mar408-7}. By the Banach fixed-point theorem, we find a solution 
	$
	(\Psi, \zeta) \in \mathcal{B}_{\rm{ou}} \times \mathcal{B}_{\rm{in}}
	$
	for \eqref{1d26Mar408-3} and \eqref{qd26Mar408-4}. Denote
	\begin{equation*}
    \begin{aligned}
		\mathcal{X} := \ & \big\{
		f \mid
		\| f \|_{\rm{ou}} <\infty, \ f \in {\mathbf{SO}}_{\tau_0, \tau_1}^{4 R} \mbox{ is evenly symmetric about $y_k$} \big\} 
        \\
        &
        \times \big\{
		g \mid
		\| g \|_{\rm{in}} <\infty, \  g \in {\mathbf{SO}}_{\tau_0, \tau_1}^{R_0} \mbox{ is evenly symmetric about $y_k$} \big\}
           \end{aligned}
	\end{equation*}  
	including $\mathcal{B}_{\rm{ou}} \times \mathcal{B}_{\rm{in}}$. We can repeat the contraction mapping argument in $\mathcal{X}$. Since \eqref{1d26Mar408-3} and \eqref{qd26Mar408-4} is a linear equation system, by the uniqueness in the Banach fixed-point theorem applied in $\mathcal{X}$, $(\Psi, \zeta)$ depends on $h$ linearly.

	Hereafter we always will regard $D_{\rm{ou}}$, $D_{{\rm{in}}}$  as general constants. By \eqref{qd26Mar408-7}, we have
	\begin{equation*}
		\phi(y,\tau_0) = \Psi(y,\tau_0) + \eta( 2 y / R_0) \zeta(y,\tau_0)
		= 0 \mbox{ \ in \ } B_{4 R(\tau_0)}.
	\end{equation*}

	By \eqref{qd26Mar4018-1} and \eqref{qd26Mar4018-2},
	$ |J[0,\zeta]| \lesssim \|h\| v \langle y \rangle^{-2-\min\{a, n-3\} }$. To improve the spatial decay of $\Psi$, under the assumption \eqref{qd26Mar4021-3}, applying Lemma \ref{chiM-eq-lem} to \eqref{1d26Mar408-3} finitely many time, we have
	$
	|\Psi| \lesssim v \langle y\rangle^{ -\min\{ a, n-3\} } \|h\|
	$. Thus,
	\begin{equation*}
		|\phi| = | \Psi + \eta(2 y / R_0) \zeta |
		\lesssim 
		v \langle y\rangle^{ -\min\{ a, n-3\} } \|h\| + 
		\1_{|y| \le R_0} \|h\| v (\langle y \rangle^{-a} + R_0 \langle  y \rangle^{2-n})
		\lesssim R_0 v \langle y\rangle^{ -\min\{ a, n-3\} } \|h\|.
	\end{equation*}
	Since
	$
	| h \1_{|y|<R} | \le \| h \| v \langle y \rangle^{-2-a} \1_{|y|<R}
	$, by the scaling argument, we get the estimate of $\nabla \phi$.
\end{proof}

\appendix

\section{Proof of Lemma \ref{qd25Oct22-3-lem}}\label{lem-proof-26June10-sec}

We first present a useful inequality. See \cite[Lemma B.1 (1)]{WYZZ2024} for instance.
\begin{lemma}
	
	Given $a \in \mathbb{R}$, $b<0$, $r_0>0$, then for any $r>r_0$,
	\begin{equation}\label{23Oct07-1}
		\big( C(a,b,r_0) \big)^{-1}  r^a  \rme^{br} \le \int_r^\infty  x^a \rme^{bx} \rmd x
		\le  C(a,b,r_0)  r^a  \rme^{br}
	\end{equation}
	with a constant $C(a,b,r_0)>1$ depending on $a, b, r_0$.
	
\end{lemma}

\begin{proof}[Proof of Lemma \ref{qd25Oct22-3-lem}]

	First, we have 
	\begin{equation}\label{qd25Oct22-4}
		\int_{\mathbb{R}^n}	\rme^{-C_0 | z- w |^2 } |w|^b  \rmd w \sim \1_{|z| \le 1} + |z|^{b} \1_{|z| > 1}.
	\end{equation}
	Indeed, for the upper bound,
	\begin{equation*}
		\begin{aligned}
			&
			\int_{\mathbb{R}^n}	\rme^{-C_0 | z- w |^2 } |w|^b  \rmd w
			= \Big( \int_{|w| \le |z|/2} + \int_{ |z|/2 < |w| \le 2|z| } + \int_{ |w| > 2|z| } \Big)	\rme^{-C_0 | z- w |^2 } |w|^b  \rmd w
			\\
			\lesssim \ & \rme^{- 4^{-1} C_0 |z|^2 } \int_{|w| \le |z|/2} |w|^b  \rmd w
			+ 
			|z|^b \int_{ |z- w| \le 3|z| } \rme^{-C_0 | z- w |^2 } \rmd w
			+ 
			\int_{ |w| > 2|z| } \rme^{- 4^{-1} C_0 |w|^2 } |w|^b  \rmd w
			\\
			\sim \ & \rme^{- 4^{-1} C_0 |z|^2 } \int_{0}^{|z|/2} r^{b+n-1}  \rmd r
			+ 
			|z|^b \int_{0}^{3|z|} \rme^{-C_0 r^2 } r^{n-1} \rmd r
			+ 
			\int_{2|z|}^{\infty} \rme^{- 4^{-1} C_0 r^2 } r^{b+n-1}  \rmd r
			\\
			\stackrel{b>-n}{\sim} \ & \rme^{- 4^{-1} C_0 |z|^2 } |z|^{b+n}
			+ 
			|z|^b
			\big( |z|^n \1_{|z| \le 1} + \1_{|z| > 1} \big)
			+ \1_{|z| \le 1} + |z|^{b+n-2} \rme^{-C_0 |z|^2} \1_{|z| > 1}
			\\
			\sim \ & \1_{|z| \le 1} + |z|^{b} \1_{|z| > 1},
		\end{aligned}
	\end{equation*}
	where we use that for $A >0$, then
	\begin{equation*}
    \int_{2|z|}^{\infty} \rme^{- A r^2 } r^{b+n-1}  \rmd r
			\sim  \int_{4|z|^2}^{\infty} \rme^{- A a } a^{\frac{b+n}{2}-1} \rmd a
			\stackrel{b>-n, \eqref{23Oct07-1}}{\sim} \1_{|z| \le 1} + |z|^{b+n-2} \rme^{- 4 A |z|^2} \1_{|z| > 1}.
	\end{equation*}
	For the lower bound, similarly,
	\begin{equation*}
		\begin{aligned}
			&
			\int_{\mathbb{R}^n}	\rme^{-C_0 | z- w |^2 } |w|^b  \rmd w
            \\
			\gtrsim \ & \rme^{- \frac{9}{4} C_0 |z|^2 } \int_{|w| \le |z|/2} |w|^b  \rmd w
			+ 
			|z|^b \int_{ |z|/2 < |w| \le 2|z| } \rme^{-C_0 | z- w |^2 } \rmd w
			+ 
			\int_{ |w| > 2|z| } \rme^{- \frac{9}{4} C_0 |w|^2 } |w|^b  \rmd w
			\\
			\gtrsim \ & \rme^{- \frac{9}{4} C_0 |z|^2 } \int_{|w| \le |z|/2} |w|^b  \rmd w
			+ 
			|z|^b \int_{ |z-w| < |z|/2 } \rme^{-C_0 | z- w |^2 } \rmd w
			+ 
			\int_{ |w| > 2|z| } \rme^{- \frac{9}{4} C_0 |w|^2 } |w|^b  \rmd w
			\\
			\sim \ & \rme^{- \frac{9}{4} C_0 |z|^2 } \int_{0}^{|z|/2} r^{b+n-1}  \rmd r
			+ 
			|z|^b \int_{0}^{|z|/2} \rme^{-C_0 r^2 } r^{n-1} \rmd r
			+ 
			\int_{2|z|}^{\infty} \rme^{- \frac{9}{4} C_0 r^2 } r^{b+n-1}  \rmd r
			\\
			\stackrel{b>-n}{\sim} \ & \rme^{- \frac{9}{4} C_0 |z|^2 } |z|^{b+n}
			+ 
			|z|^b
			\big( |z|^n \1_{|z| \le 1} + \1_{|z| > 1} \big)
			+ \1_{|z| \le 1} + |z|^{b+n-2} \rme^{- 9 C_0 |z|^2} \1_{|z| > 1}
			\\
			\sim \ & \1_{|z| \le 1} + |z|^{b} \1_{|z| > 1}.
		\end{aligned}
	\end{equation*}
	Then,
	\begin{equation*}
		\begin{aligned}
			&
			\int_{\mathbb{R}^n}	\rme^{-C_0 \frac{|x-z|^2}{t} } |z|^{b} \rmd z
			= t^{\frac{b+n}{2}} \int_{\mathbb{R}^n}	\rme^{-C_0 | t^{-\frac{1}{2}} x- w |^2 } |w|^b  \rmd w
			\\
			\stackrel{\eqref{qd25Oct22-4}}{\sim} \ & t^{\frac{b+n}{2}}
			\big( \1_{|t^{-\frac{1}{2}} x| \le 1} + |t^{-\frac{1}{2}} x|^{b} \1_{|t^{-\frac{1}{2}} x| > 1} \big)
			= t^{\frac{b+n}{2}} \1_{|x| \le t^{\frac{1}{2}}} + t^{\frac{n}{2}} |x|^{b} \1_{|x| > t^{\frac{1}{2}}}.
		\end{aligned}
	\end{equation*}
\end{proof}

\section{Proof of (\ref{move-26Sep2-2})}

\eqref{move-26Sep2-2} follows directly from the following lemma.
\begin{lemma}\label{move-26Spe2-1-lem}

Given an integer $n > 0$, constants $a > -n$, $C_1 >0$, and $\mathbf{m} = (m_1,m_2,\dots, m_n) \in \mathbb{N}^n$, suppose
\begin{equation*}
|\partial_{z}^{\mathbf{k}} f(z)| \le C_1 |z|^{a - \| \mathbf{k} \|_{\ell_1}}
\mbox{ \ for all $z\in \mathbb{R}^n$, $\mathbf{k} = (k_1,k_2,\dots, k_n) \in \mathbb{N}^n$ satisfying $k_i \le m_i$, $i=1,2,\dots, n$},  
\end{equation*}
then there exist a constant $C>0$ depending on $n, a, \mathbf{m}$ such that
\begin{equation*}
\Big| \partial_{x}^{\mathbf{m}} \Big[ ( 4\pi t )^{-\frac n2} \int_{\mathbb{R}^n} 
	\rme^{-\frac{|x-z|^2}{4t} } f(z) \rmd z \Big] \Big|
\le C C_1 |x|^{a - \| \mathbf{m} \|_{\ell_1}}
\mbox{ \ for \ } |x| \ge t^{\frac{1}{2}} >0.
\end{equation*}
\end{lemma}

\begin{proof}

Without loss of generality, we take $C_1=1$ in the proof. It suffices to estimate
\begin{align*}
&
P_1 :=  (4\pi t)^{-\frac n2} \int_{\mathbb{R}^n} 
\partial_{x}^{\mathbf{m}} \Big( \rme^{-\frac{|x-z|^2}{4t} } \Big) f(z) \eta\Big(\frac{z}{9^{-1} t^{\frac{1}{2}}} \Big) \rmd z,
\\
&
P_2 :=  (4\pi t)^{-\frac n2} \int_{\mathbb{R}^n} 
\partial_{x}^{\mathbf{m}} \Big( \rme^{-\frac{|x-z|^2}{4t} } \Big) f(z) \Big[ 1 - \eta\Big(\frac{z}{9^{-1} t^{\frac{1}{2}}} \Big) \Big] \rmd z.
\end{align*}
For $P_1$,
\begin{equation*}
\begin{aligned}
&
|P_1| \lesssim t^{-\frac{n}{2} - \frac{1}{2}\|\mathbf{m}\|_{\ell_1} } \int_{\mathbb{R}^n} \rme^{-\frac{|x-z|^2}{5 t}} |z|^{a} \eta\Big(\frac{z}{9^{-1} t^{\frac{1}{2}} } \Big) \rmd z
\\
\le \ &
t^{-\frac{n}{2} - \frac{1}{2}\|\mathbf{m}\|_{\ell_1} } \rme^{- (\frac{7}{9})^2 \frac{|x|^2}{5t} } \int_{\mathbb{R}^n} |z|^{a} \eta\Big(\frac{z}{9^{-1} t^{\frac{1}{2}} } \Big) \rmd z
\stackrel{a>-n}{\lesssim}
t^{\frac{1}{2} (a - \|\mathbf{m}\|_{\ell_1}) } \rme^{- (\frac{7}{9})^2 \frac{|x|^2}{5t} }
\end{aligned}
\end{equation*}
since for $|z| \le \frac{2}{9} t^{\frac{1}{2}}$, $|x| \ge t^{\frac{1}{2}}$, we have $\frac{7}{9} |x| \le |x - z| \le \frac{11}{9} |x|$.

For $P_2$,
\begin{align*}
&
P_2 = (4\pi t)^{-\frac n2}
(-1)^{-\|\mathbf{m}\|_{\ell_1}} \int_{\mathbb{R}^n} 
\partial_{z}^{\mathbf{m}} \Big( \rme^{-\frac{|x-z|^2}{4t} } \Big) f(z) \Big[ 1 - \eta\Big(\frac{z}{9^{-1} t^{\frac{1}{2}}} \Big) \Big] \rmd z
\\
= \ & (4\pi t)^{-\frac n2} \int_{\mathbb{R}^n} \rme^{-\frac{|x-z|^2}{4t} } 
\partial_{z}^{\mathbf{m}} \Big\{
f(z) \Big[ 1 - \eta\Big(\frac{z}{9^{-1} t^{\frac{1}{2}}} \Big) \Big]
\Big\} \rmd z
\\
= \ & (4\pi t)^{-\frac n2} \int_{\mathbb{R}^n} \rme^{-\frac{|x-z|^2}{4t} } 
\Big\{ \sum_{\mathbf{i}, \mathbf{j} \in \mathbb{N}^n, \mathbf{i} + \mathbf{j} = \mathbf{m}} 
\partial_{z}^{\mathbf{i}} f(z) \partial_{z}^{\mathbf{j}} \Big[ 1 - \eta\Big(\frac{z}{9^{-1} t^{\frac{1}{2}}} \Big) \Big]
\Big\} \rmd z.
\end{align*}
Then
\begin{equation*}
\begin{aligned}
&
|P_2|
\lesssim t^{-\frac n2} \int_{\mathbb{R}^n} \rme^{-\frac{|x-z|^2}{4t} } 
\Big( |z|^{a - \|\mathbf{m}\|_{\ell_1}}
\1_{|z| \ge \frac{1}{9} t^{\frac{1}{2}}}
+ \sum_{\mathbf{i}, \mathbf{j} \in \mathbb{N}^n, \,  \mathbf{i} + \mathbf{j} = \mathbf{m}, \, \mathbf{j} \ne 0} |z|^{a - \|\mathbf{i}\|_{\ell_1}}
t^{-\frac{1}{2} \|\mathbf{j}\|_{\ell_1} } \1_{\frac{1}{9} t^{\frac{1}{2}} \le |z| \le \frac{2}{9} t^{\frac{1}{2}}}
\Big) \rmd z
\\
\sim \ & t^{-\frac n2} \int_{\mathbb{R}^n} \rme^{-\frac{|x-z|^2}{4t} } |z|^{a - \|\mathbf{m}\|_{\ell_1} } \big( 
\underbrace{\1_{\frac{1}{9} t^{\frac{1}{2}} \le |z| \le \frac{1}{2} |x|}}_{=: P_{21}}
+
\underbrace{\1_{\frac{1}{2} |x| < |z| \le 2|x|}}_{=: P_{22}} 
+
\underbrace{\1_{|z| > 2|x|}}_{=: P_{23}}
\big) \rmd z.
\end{aligned}
\end{equation*}
Therein,
\begin{equation*}
\begin{aligned}
&
P_{21} 
\le t^{-\frac n2} \rme^{-\frac{|x|^2}{16 t} }\int_{\mathbb{R}^n} |z|^{a -\|\mathbf{m}\|_{\ell_1}} 
\1_{ \frac{1}{9} t^{\frac{1}{2}} \le |z| \le \frac{1}{2} |x|} \rmd z 
\sim 
t^{-\frac n2} \rme^{-\frac{|x|^2}{16 t} }\int_{\frac{1}{9} t^{\frac{1}{2}}}^{\frac{1}{2} |x|} r^{a+n -\|\mathbf{m}\|_{\ell_1}-1}  \rmd r
\\
\lesssim \ & t^{-\frac n2} \rme^{-\frac{|x|^2}{16 t} }
\begin{cases}
t^{\frac{1}{2}(a+n -\|\mathbf{m}\|_{\ell_1})}
& \mbox{ \ if \ } a < \|\mathbf{m}\|_{\ell_1} - n
\\
\ln(\frac{9}{2} t^{-\frac{1}{2}} |x| ) 
& \mbox{ \ if \ } a = \|\mathbf{m}\|_{\ell_1} - n
\\
|x|^{a+n -\|\mathbf{m}\|_{\ell_1}} & \mbox{ \ if \ } a > \|\mathbf{m}\|_{\ell_1} -n,
\end{cases}
\end{aligned}
\end{equation*}
\begin{align*}
& P_{22}
\sim t^{-\frac n2} |x|^{a -\|\mathbf{m}\|_{\ell_1} } \int_{\mathbb{R}^n} \rme^{-\frac{|x-z|^2}{4t} }  
\1_{\frac{1}{2} |x| < |z| \le 2|x|} \rmd z
\\
\le \ & t^{-\frac n2} |x|^{a -\|\mathbf{m}\|_{\ell_1}} \int_{|x-z| \le 3|x|} \rme^{-\frac{|x-z|^2}{4t} } \rmd z 
\sim
t^{-\frac n2} |x|^{a -\|\mathbf{m}\|_{\ell_1}} \int_{0}^{3|x|} \rme^{-\frac{r^2}{4t} } r^{n-1} \rmd r
\\
\sim \ & |x|^{a -\|\mathbf{m}\|_{\ell_1}} \int_{0}^{\frac{9}{4} \frac{|x|^2}{t}} \rme^{-w} w^{\frac{n}{2} - 1}  \rmd w
\stackrel{n>0, |x| \ge t^{\frac{1}{2}}}{\sim}  |x|^{a -\|\mathbf{m}\|_{\ell_1}},
\end{align*}
which is greater than the upper bounds of $|P_1|$ and $P_{21}$ since
$
t^{\frac{1}{2} (a - \|\mathbf{m}\|_{\ell_1}) } \rme^{- (\frac{7}{9})^2 \frac{|x|^2}{5t} }
\lesssim |x|^{a -\|\mathbf{m}\|_{\ell_1}}
$ is true for $|x| \ge t^{\frac{1}{2}}$, and
\begin{equation*}
\begin{aligned}
&
t^{-\frac n2} \rme^{-\frac{|x|^2}{16 t} }
\begin{cases}
t^{\frac{1}{2}(a+n -\|\mathbf{m}\|_{\ell_1})}
& \mbox{ \ if \ } a < \|\mathbf{m}\|_{\ell_1} - n
\\
\ln(\frac{9}{2} t^{-\frac{1}{2}} |x| ) 
& \mbox{ \ if \ } a = \|\mathbf{m}\|_{\ell_1} - n
\\
|x|^{a+n -\|\mathbf{m}\|_{\ell_1}} & \mbox{ \ if \ } a > \|\mathbf{m}\|_{\ell_1} - n
\end{cases}
\lesssim |x|^{a -\|\mathbf{m}\|_{\ell_1}}
\\
\Leftrightarrow \ & \rme^{-\frac{|x|^2}{16 t} }
\begin{cases}
(t^{-1} |x|^2)^{\frac{1}{2}(-a + \|\mathbf{m}\|_{\ell_1})}
& \mbox{ \ if \ } a < \|\mathbf{m}\|_{\ell_1} - n
\\
(t^{-1} |x|^2)^{\frac{n}{2}}
\ln(\frac{9}{2} t^{-\frac{1}{2}} |x| ) 
& \mbox{ \ if \ } a = \|\mathbf{m}\|_{\ell_1} - n
\\
(t^{-1} |x|^2)^{\frac{n}{2}} & \mbox{ \ if \ } a > \|\mathbf{m}\|_{\ell_1} - n
\end{cases}
\lesssim 1,
\end{aligned}
\end{equation*}
which is true for $|x| \ge t^{\frac{1}{2}}$.
\begin{equation*}
\begin{aligned}
& P_{23}
\le t^{-\frac n2} \int_{\mathbb{R}^n} \rme^{-\frac{|z|^2}{16 t} } |z|^{a -\|\mathbf{m}\|_{\ell_1}} 
\1_{|z| > 2|x|} \rmd z
\sim t^{-\frac n2} \int_{2|x|}^{\infty} \rme^{-\frac{r^2}{16 t} } r^{a+n - \|\mathbf{m}\|_{\ell_1}-1} \rmd r
\\
\sim \ & t^{\frac{1}{2}(a -\|\mathbf{m}\|_{\ell_1}) } \int_{\frac{|x|^2}{4t}}^{\infty} \rme^{-w}  w^{\frac{1}{2} (a+n -\|\mathbf{m}\|_{\ell_1} -2)} \rmd w
\stackrel{|x| \ge t^{\frac{1}{2}}, \eqref{23Oct07-1}}{\sim}
t^{\frac{1}{2}(a -\|\mathbf{m}\|_{\ell_1}) } \rme^{-\frac{|x|^2}{4t}} ( t^{-1} |x|^2 )^{\frac{1}{2} (a+n -\|\mathbf{m}\|_{\ell_1} -2)},
\end{aligned}
\end{equation*}
which is smaller than the upper bound of $P_{22}$ for $|x| \ge t^{\frac{1}{2}}$. In sum, we get the conclusion.
\end{proof}

\section{Spatial decay improvement lemma}

The following lemma improves the spatial decay of solutions by the comparison theorem.
\begin{lemma}\label{qd26Mar7-1-lem}
	
	Given an integer $n\ge 3$, $0< t_0 <t_1 \le \infty$, $R(t)>0, v(t) \ge 0$ defined in $(t_0, t_1)$, consider
	\begin{equation*}
		\partial_{t} f = \Delta f + V(x,t) f + h(x,t) \mbox{ \ for \ } t\in (t_0, t_1), x\in B_{R},
		\quad
		f=0
		\mbox{ \ for \ } t\in (t_0, t_1), x\in \partial B_{R}.
	\end{equation*}
	Suppose that $\| f(\cdot,t) \|_{L^{\infty}(B_{R})} \le v$, $|h(x,t)| \le v \langle x \rangle^{-a_1}$ with a constant $a_1>n$; $|V(x,t)| \le C_{V} \langle x \rangle^{-2-a_2}$ with positive constants $C_{V}, a_2$; $|f(x,t_0)| \le v(t_0) \langle x \rangle^{a_3}$ with a constant $a_3 \le 2-n$;
	\begin{equation*}
		\begin{aligned}
			&
			\mbox{either Case 1: $\dot{v} \ge 0$}
			\\
			&
			\mbox{or Case 2: $|\dot{v}| \le C_{v} t^{-1} v$, $\langle R \rangle^{2+\epsilon} \le C_{R} t$ with positive constants $C_{v}, \epsilon, C_{R}$}
		\end{aligned}
	\end{equation*}
	holds, then
	\begin{equation*}
		|f| \le C v \langle x \rangle^{2-n}
	\end{equation*}
	with a constant $C>0$ independent of $t_0, t_1$.
\end{lemma}

\begin{proof}
	
	For Case 2, set
	$ f_1(x,t) := D v (-\Delta)^{-1} ( \langle x \rangle^{-b} ) $
	with $D>1$ to be determined later and $b \in (n, \min\{ a_2+n, n+\epsilon, a_1 \})$. Straightforward calculation gives
	$ C_1^{-1} \langle x \rangle^{2-n} \le (-\Delta)^{-1} ( \langle x \rangle^{-b} ) \le C_1 \langle x \rangle^{2-n} $ for some $C_1>1$. 
	\begin{equation*}
		\begin{aligned}
			&
			\Delta f_1 + V f_1 - \partial_{t} f_1 + h
			\\
			= \ & - D v \langle x \rangle^{-b} + V D v (-\Delta)^{-1} ( \langle x \rangle^{-b} ) - D \dot{v} (-\Delta)^{-1} ( \langle x \rangle^{-b} ) + h
			\\
			\le \ & - D v \langle x \rangle^{-b} + C_{V} \langle x \rangle^{-2-a_2} D v C_1 \langle x \rangle^{2-n} + D C_{v} t^{-1} v C_1 \langle x \rangle^{2-n} + v \langle x \rangle^{-a_1}
			\\
			\le \ & - D v \langle x \rangle^{-b} + C_{V} D v C_1 \langle x \rangle^{-a_2-n} + D C_{v} t^{-1} v \langle R \rangle^{2+\epsilon} C_1 \langle x \rangle^{-n-\epsilon} + v \langle x \rangle^{-a_1}
			\\
			\le \ &  D v \langle x \rangle^{-b} \big( - 1 + C_{V} C_1 \langle x \rangle^{b-a_2-n} + C_{v} C_{R} C_1 \langle x \rangle^{b -n-\epsilon} + D^{-1} \langle x \rangle^{b-a_1} \big).
		\end{aligned}
	\end{equation*}
	There exists a constant $R_{1} \gg 1$ such that for $|x| \ge R_{1}$, we have $ - 1 + C_{V} C_1 \langle x \rangle^{b-a_2-n} + C_{v} C_{R} C_1 \langle x \rangle^{b -n-\epsilon} + D^{-1} \langle x \rangle^{b-a_1} <0$. If $R \le R_1$, then $|f| \le v \langle R_1 \rangle^{n-2} \langle x \rangle^{2-n}$. If $R > R_1$, we take $D = D(R_1) \gg 1$ to make
$ f_1(x,t_0) \ge D v(t_0) C_1^{-1} \langle x \rangle^{2-n} \ge v(t_0) \langle x \rangle^{a_3} $
and $f_1(x,t) |_{|x|=R_1} \ge v$. Then $f_1$ is a barrier function in $t\in (t_0, t_1), R_1 \le |x| \le R$. The estimate in $|x|< R_1$ is trivial. Hence, we get the conclusion. Case 1 can be proved similarly. 
\end{proof}

\section{Proof of Lemma \ref{qd26Apr17-5-lem}}\label{ext-lem-proof}

\begin{proof}[Proof of Lemma \ref{qd26Apr17-5-lem}]

It is ready to get \eqref{qd26May2-3}.
\begin{align*}
			&
			g_{i}(y, t[\dot{\lambda}_{1,i}](\tau) ) \1_{\tau_0 \le \tau \le  \tau[\dot{\lambda}_{1,i}](T)}
			+
			g_{i}(y,T) \1_{\tau > \tau[\dot{\lambda}_{1,i}](T)}
			\\
			& -
			\big[
			g_{\infty}(y,t[\dot{\lambda}_{1,\infty}](\tau)) \1_{\tau_0 \le \tau \le  \tau[\dot{\lambda}_{1,\infty}](T)}
			+ g_{\infty}(y,T) \1_{\tau > \tau[\dot{\lambda}_{1,\infty}](T)}
			\big]
			\\
			= \ & g_{i}(y, t[\dot{\lambda}_{1,i}](\tau) ) \1_{\tau \ge \tau_0, \tau \le  \tau[\dot{\lambda}_{1,i}](T), \tau \le  \tau[\dot{\lambda}_{1,\infty}](T)}
			+
			g_{i}(y, t[\dot{\lambda}_{1,i}](\tau) ) \1_{\tau \ge \tau_0, \tau \le  \tau[\dot{\lambda}_{1,i}](T), \tau > \tau[\dot{\lambda}_{1,\infty}](T)}
			\\
			&
			+
			g_{i}(y,T) \1_{\tau > \tau[\dot{\lambda}_{1,i}](T), \tau > \tau[\dot{\lambda}_{1,\infty}](T)}
			+
			g_{i}(y,T) \1_{\tau > \tau[\dot{\lambda}_{1,i}](T), \tau \le \tau[\dot{\lambda}_{1,\infty}](T)}
			\\
			& -
			\big[
			g_{\infty}(y,t[\dot{\lambda}_{1,\infty}](\tau)) \1_{\tau \ge \tau_0, \tau \le  \tau[\dot{\lambda}_{1,\infty}](T), \tau \le  \tau[\dot{\lambda}_{1,i}](T)}
			+
			g_{\infty}(y,t[\dot{\lambda}_{1,\infty}](\tau)) \1_{\tau \ge \tau_0, \tau \le  \tau[\dot{\lambda}_{1,\infty}](T), \tau >  \tau[\dot{\lambda}_{1,i}](T)}
			\\
			&
			+ g_{\infty}(y,T) \1_{\tau > \tau[\dot{\lambda}_{1,\infty}](T), \tau >  \tau[\dot{\lambda}_{1,i}](T)}
			+ g_{\infty}(y,T) \1_{\tau > \tau[\dot{\lambda}_{1,\infty}](T), \tau \le   \tau[\dot{\lambda}_{1,i}](T)}
			\big]
			\\
			= \ & 
			\big[
			g_{i}(y, t[\dot{\lambda}_{1,i}](\tau) )
			-
			g_{\infty}(y, t[\dot{\lambda}_{1,i}](\tau) )
			+
			g_{\infty}(y, t[\dot{\lambda}_{1,i}](\tau) )
			-
			g_{\infty}(y,t[\dot{\lambda}_{1,\infty}](\tau))
			\big] \1_{\tau \ge \tau_0, \tau \le  \tau[\dot{\lambda}_{1,i}](T), \tau \le  \tau[\dot{\lambda}_{1,\infty}](T)}
			\\
			&
			+
			\big[
			g_{i}(y, t[\dot{\lambda}_{1,i}](\tau) )
			-
			g_{\infty}(y,T)
			\big] \1_{\tau \le  \tau[\dot{\lambda}_{1,i}](T), \tau > \tau[\dot{\lambda}_{1,\infty}](T)}
			\\
			&
			+
			\big[
			g_{i}(y,T)
			-
			g_{\infty}(y,T)
			\big] \1_{\tau > \tau[\dot{\lambda}_{1,i}](T), \tau > \tau[\dot{\lambda}_{1,\infty}](T)}
			\\
			&
			+
			\big[ 
			g_{i}(y,T) - g_{\infty}(y,t[\dot{\lambda}_{1,\infty}](\tau))
			\big] \1_{\tau > \tau[\dot{\lambda}_{1,i}](T), \tau \le \tau[\dot{\lambda}_{1,\infty}](T)}.
    \end{align*}
	Therein, for the 1st part,
	\begin{equation*}
		\big[
		g_{i}(y, t[\dot{\lambda}_{1,i}](\tau) )
		-
		g_{\infty}(y, t[\dot{\lambda}_{1,i}](\tau) )
		\big] \1_{\tau \ge \tau_0, \tau \le  \tau[\dot{\lambda}_{1,i}](T), \tau \le  \tau[\dot{\lambda}_{1,\infty}](T)} \rightsquigarrow 0 \mbox{ \ in \ } \tau \in [\tau_0,\infty), y\in \Omega
	\end{equation*}
	since $t[\dot{\lambda}_{1,i}](\tau) \le t[\dot{\lambda}_{1,i}](\tau[\dot{\lambda}_{1,i}](T)) = T$, and the convergence assumption of $g_{i}$.
	\begin{equation*}
		\big[
		g_{\infty}(y, t[\dot{\lambda}_{1,i}](\tau) )
		-
		g_{\infty}(y,t[\dot{\lambda}_{1,\infty}](\tau))
		\big] \1_{\tau \ge \tau_0, \tau \le  \tau[\dot{\lambda}_{1,i}](T), \tau \le  \tau[\dot{\lambda}_{1,\infty}](T)} \rightsquigarrow 0 
		\mbox{ \ in \ } \tau \in [\tau_0,\infty), y\in \Omega
	\end{equation*}
	by \eqref{qd26Apr16-1}, $\tau \le  \tau[\dot{\lambda}_{1,\infty}](T) \stackrel{\eqref{tau-est}}{\le} C_{tm} T^{2\gamma_1 -3}$, the assumption $\lim_{i\to \infty} \|\dot{\lambda}_{1,i} - \dot{\lambda}_{1,\infty}\|_{\dot{\lambda}_1} = 0$, and \eqref{qd26May2-3}.
	
	For the 2nd part,
	\begin{equation*}
		\begin{aligned}
			&
			\big[
			g_{i}(y, t[\dot{\lambda}_{1,i}](\tau) )
			-
			g_{\infty}(y,T)
			\big] \1_{\tau \le  \tau[\dot{\lambda}_{1,i}](T), \tau > \tau[\dot{\lambda}_{1,\infty}](T)}
			\\
			= \ & \big[
			g_{i}(y, t[\dot{\lambda}_{1,i}](\tau))
			-
			g_{\infty}(y, t[\dot{\lambda}_{1,i}](\tau))
			+
			g_{\infty}(y, t[\dot{\lambda}_{1,i}](\tau))
			-
			g_{\infty}(y,T)
			\big] \1_{\tau \le  \tau[\dot{\lambda}_{1,i}](T), \tau > \tau[\dot{\lambda}_{1,\infty}](T)}.
		\end{aligned}
	\end{equation*}
	Here,
	\begin{equation*}
		\big[
		g_{i}(y, t[\dot{\lambda}_{1,i}](\tau))
		-
		g_{\infty}(y, t[\dot{\lambda}_{1,i}](\tau))
		\big] \1_{\tau \le  \tau[\dot{\lambda}_{1,i}](T), \tau > \tau[\dot{\lambda}_{1,\infty}](T)}
		\rightsquigarrow 0 
		\mbox{ \ in \ } \tau \in [\tau_0,\infty), y\in \Omega
	\end{equation*}
	since $t[\dot{\lambda}_{1,i}](\tau) \le t[\dot{\lambda}_{1,i}](\tau[\dot{\lambda}_{1,i}](T)) = T$, and the convergence assumption of $g_{i}$.
	\begin{equation*}
		\begin{aligned}
			&
			\big[
			g_{\infty}(y, t[\dot{\lambda}_{1,i}](\tau))
			-
			g_{\infty}(y,T)
			\big] \1_{\tau \le  \tau[\dot{\lambda}_{1,i}](T), \tau > \tau[\dot{\lambda}_{1,\infty}](T)}
			\\
			= \ & \big[
			g_{\infty}(y, t[\dot{\lambda}_{1,i}](\tau))
			-
			g_{\infty}(y,t[\dot{\lambda}_{1,i}](\tau[\dot{\lambda}_{1,i}](T)))
			\big] \1_{\tau \le  \tau[\dot{\lambda}_{1,i}](T), \tau > \tau[\dot{\lambda}_{1,\infty}](T)}
			\rightsquigarrow 0 
			\mbox{ \ in \ } \tau \in [\tau_0,\infty), y\in \Omega
		\end{aligned}
	\end{equation*}
by \eqref{qd26May2-3} and
	\begin{equation*}
		\big| t[\dot{\lambda}_{1,i}](\tau) - t[\dot{\lambda}_{1,i}](\tau[\dot{\lambda}_{1,i}](T)) \big|
		\le
		\big| t[\dot{\lambda}_{1,i}](\tau[\dot{\lambda}_{1,\infty}](T)) - t[\dot{\lambda}_{1,i}](\tau[\dot{\lambda}_{1,i}](T)) \big|
		\to 0,
	\end{equation*}
    which is deduced by
	$\tau[\dot{\lambda}_{1,i}](T) \to \tau[\dot{\lambda}_{1,\infty}](T)$ by \eqref{qd26Apr17-3} and the assumption $\lim_{i\to \infty} \|\dot{\lambda}_{1,i} - \dot{\lambda}_{1,\infty}\|_{\dot{\lambda}_1} = 0$; and  $\tau_1 - \tau_2 \sim \int_{t[\dot{\lambda}_{1}](\tau_2) }^{t[\dot{\lambda}_{1}](\tau_1) } (s^{2-\gamma_1} )^{-2} \rmd s$ for $\tau_1, \tau_2 \in [\tau_0, \infty)$ by \eqref{qd26May2-1}, where the ``$\sim$'' is independent of $\tau_1, \tau_2$ and the choice of $\dot{\lambda}_{1}$.

	For the 3rd part, by the convergence assumption of $g_{i}$
	\begin{equation*}
		\big[
		g_{i}(y,T)
		-
		g_{\infty}(y,T)
		\big] \1_{\tau > \tau[\dot{\lambda}_{1,i}](T), \tau > \tau[\dot{\lambda}_{1,\infty}](T)}
		\rightsquigarrow 0 
		\mbox{ \ in \ } \tau \in [\tau_0,\infty), y\in \Omega.
	\end{equation*}
	
	For the 4th part,
	\begin{equation*}
		\begin{aligned}
			&
			\big[ 
			g_{i}(y,T) - g_{\infty}(y,t[\dot{\lambda}_{1,\infty}](\tau))
			\big] \1_{\tau > \tau[\dot{\lambda}_{1,i}](T), \tau \le \tau[\dot{\lambda}_{1,\infty}](T)}
			\\
			= \ & \big[ 
			g_{i}(y,T) - g_{\infty}(y,T) + g_{\infty}(y,T) - g_{\infty}(y,t[\dot{\lambda}_{1,\infty}](\tau))
			\big] \1_{\tau > \tau[\dot{\lambda}_{1,i}](T), \tau \le \tau[\dot{\lambda}_{1,\infty}](T)}.
		\end{aligned}
	\end{equation*}
	Here, $g_{i}(y,T) - g_{\infty}(y,T) \rightsquigarrow 0$ in $y\in \Omega$.
	\begin{equation*}
		\begin{aligned}
			& \big[ g_{\infty}(y,T) - g_{\infty}(y,t[\dot{\lambda}_{1,\infty}](\tau))
			\big] \1_{\tau > \tau[\dot{\lambda}_{1,i}](T), \tau \le \tau[\dot{\lambda}_{1,\infty}](T)}
			\\
			= \ & \big[ g_{\infty}(y,t[\dot{\lambda}_{1,\infty}]( \tau[\dot{\lambda}_{1,\infty}](T) ) ) - g_{\infty}(y,t[\dot{\lambda}_{1,\infty}](\tau))
			\big] \1_{\tau > \tau[\dot{\lambda}_{1,i}](T), \tau \le \tau[\dot{\lambda}_{1,\infty}](T)}
			\rightsquigarrow 0 
			\mbox{ \ in \ } \tau \in [\tau_0,\infty), y\in \Omega
		\end{aligned}
	\end{equation*}
	by \eqref{qd26May2-3} and
	$ \big| \tau[\dot{\lambda}_{1,\infty}](T) - \tau \big|
		\le
		\big| \tau[\dot{\lambda}_{1,\infty}](T) - \tau[\dot{\lambda}_{1,i}](T) \big| \stackrel{\eqref{qd26Apr17-3}}{\to} 0$ with the assumption $\lim_{i\to \infty} \|\dot{\lambda}_{1,i} - \dot{\lambda}_{1,\infty}\|_{\dot{\lambda}_1} = 0$.
\end{proof}

\section{Proof of Proposition \ref{mode1-orth-prop}}\label{mode1-orth-pf-sec}

\begin{proof}[Proof of Proposition \ref{mode1-orth-prop}]

Denote $r=|y|$. The orthogonality condition is equivalent to
\begin{equation*}
	\int_0^{R} h_1(r,\tau) U_r(r) r^{n-1} \rmd r = 0
	\mbox{ \ for \ } 
	\tau \in (\tau_0, \tau_1), 
\end{equation*}
where $U_r(r) = [n(n-2)]^{\frac{n-2}{4}} (2-n)r(1+r^2)^{-\frac n2}$. Denote $\tilde{h}_1$ as the zero extension of $h_1$ outside $B_{R}$. Obviously, $\|\tilde{h}_1\|_{v,2+a}^{\infty} = \|h_1\|_{v,2+a}^{R}$.
Denote $\mathcal{L}_{1}[\phi_1]:=
(\partial_{rr} + \frac{n-1}{r} \partial_{r}
- \frac{n-1}{r^2} + pU^{p-1}) \phi_1$. Let $H = H_1(|y|,\tau) \Upsilon_{1,i}(\frac{y}{|y|})$ satisfying $\mathcal{L}_{1} H_{1} +  \tilde{h}_1  = 0 $ for $r>0$. For $r>0$, $H_1$ is given by
\begin{equation*}
	\begin{aligned}
		H_1(r,\tau) = \ &  U_r(r)
		\int_0^r 
		\frac{1}{\rho^{n-1} U_r(\rho)^2 } 
		\int_\rho^\infty 
		\tilde{h}_1(s,\tau) U_r(s) s^{n-1} \rmd s
		\rmd \rho  \mbox{ \ for \ } -1 < a\le n-1,
		\\
		H_1(r,\tau) = \ &  -U_r(r)
		\int_r^{\infty} 
		\frac{1}{\rho^{n-1} U_r(\rho)^2 } 
		\int_\rho^\infty 
		\tilde{h}_1(s,\tau) U_r(s) s^{n-1} \rmd s
		\rmd \rho  
		\mbox{ \ for \ } a > n-1,
	\end{aligned}
\end{equation*}
where $a>-1$ is used to guarantee that the spatial decay of $\tilde{h}_1(s,\tau) U_r(s) s^{n-1}$ is faster than $s^{-1}$ for $s \ge 1$. Using the orthogonality condition and then the scaling argument, we have
\begin{equation*}
\|\partial_{r} H_1 \|_{v, \hat{a}_{1} + 1 }^{\infty} +	\|H_1 \|_{v, \hat{a}_{1} }^{\infty} \lesssim \| h_1 \|_{v,2+a}^{R},
\end{equation*}
where $
\hat{a}_{1} := 
\begin{cases}
	a
	&
	\mbox{ if }  a\ne n-1
	\\
	(n-1)-
	&
	\mbox{ if } a=n-1
\end{cases}
$. It holds that $\Delta H + pU^{p-1} H = - \tilde{h}_1(|y|,\tau) \Upsilon_{1,i}(\frac{y}{|y|})$.
Consider 
\begin{equation*}
	\begin{cases}
		\partial_\tau \Phi = \Delta \Phi + pU^{p-1} \Phi + H
		\mbox{ \ for \ } \tau \in (\tau_0, \tau_1), y \in B_{2R},
		\\
		\Phi = 0 
		\mbox{ \ for \ } \tau \in (\tau_0, \tau_1), y \in \partial B_{2R}
		\quad	\Phi(\cdot,\tau_0) = 0 
		\mbox{ \ in \ }  B_{2R(\tau_0)}.
	\end{cases}
\end{equation*}
Under the assumption $|\dot{v}| = O(\tau^{-1} v)$, $|\dot{R}| = O(\tau^{-1} R)$, and $\tau^{-1} R^n \theta_{R, \hat{a}_{1}}^1 \le d_0 \ll 1$, 
by Lemma \ref{M1-nonortho-lem} and similar to the argument of \eqref{26Mar14-3}, we find a solution $\Phi=\Phi[H]$ linearly depending on $H$, and satisfying
\begin{equation*}
\langle y \rangle^2 | \nabla^2 \Phi |
+
\langle y \rangle |\nabla \Phi| + |\Phi| \lesssim 
	v
	\theta_{R, \hat{a}_{1} }^1  R^n  \langle y\rangle^{1-n}
	\|h_1\|_{v,2+a}^{R},
\quad
\Phi = \Phi_1(|y|,\tau) \Upsilon_{1,i}(\frac{y}{|y|}).
\end{equation*}
We take $\phi = (\Delta + pU^{p-1}) \Phi$. It follows that
\begin{equation*}
\begin{aligned}
&
\phi
= 
\Upsilon_{1,i}(\frac{y}{|y|}) \Big( \frac{\rmd^2}{\rmd r^2} + \frac{n-1}{r} \frac{\rmd}{\rmd r} - \frac{n-1}{r^2} +
pU^{p-1}(r) \Big) \Phi_1(r,\tau),
\quad
\phi(\cdot, \tau_0) = 0 
\mbox{ \ in \ } B_{2R(\tau_0)},
\\
&
\langle y \rangle |\nabla \phi| +
	|\phi| \lesssim v
	\theta_{R, \hat{a}_{1} }^1  R^n  \langle y\rangle^{-1-n}
	\|h_1\|_{v,2+a}^{R}.
\end{aligned}
\end{equation*}
Since $\theta_{R, \hat{a}_{1}}^1 = \theta_{R, a}^1$, $\|h_1\|_{v,2+a}^{R} \sim \|h\|_{v,2+a}^{R}$, we get the result.
\end{proof}

\section{Proof of Lemma \ref{phi-perp-lemma}}\label{high-mode-pf-sec}

\begin{proof}[Proof of Lemma \ref{phi-perp-lemma}]

The existence and uniqueness of the solution $\phi = \phi[h]$ linearly depending on $h$ are guaranteed by the classical parabolic theory.  Since $h \in {\mathbf{SO}}_{\tau_0, \tau_1}^{R}$, Lemma \ref{chiM-eq-lem}-$(2)$ implies $\phi \in {\mathbf{SO}}_{\tau_0, \tau_1}^{R}$.

By Lemma \ref{chiM-eq-lem}-$(4)$, $\phi$ can inherits the even symmetry of $h$.

Recall $\chi_{M}$ given in \eqref{chiM-def}. Under the assumption $R\ge 2$, $v\ge 0$, $|\dot{v}| = O(\tau^{-1} v)$, $|\dot{R}| = O(\tau^{-1} R)$, $R^{2 + \epsilon_1} \ll \tau$ with $\epsilon_1>0$, by Remark \ref{qd26Mar13-3-rmk}, $\mathcal{T}_*[h]$ is given by Lemma \ref{chiM-eq-lem} and satisfies
	\begin{equation}
		\begin{cases}
			\partial_{\tau} \mathcal{T}_*[h]
			= 
			\Delta \mathcal{T}_*[h] + p U^{p-1} (1-\chi_{M}(|y|)) \mathcal{T}_*[h]
			 + h \mbox{ \ for \ } \tau \in (\tau_0, \tau_1), y \in B_{R}, 
			\\
			\mathcal{T}_*[h] = 
			0 \mbox{ \ for \ } \tau \in (\tau_0, \tau_1), y \in \partial B_{R},
			\quad
			\mathcal{T}_*[h](\cdot,\tau_0) = 0 
			\mbox{ \ in \ } B_{R(\tau_0)},
		\end{cases}
	\end{equation}
\begin{equation}\label{phi*-hperp-estimate}
|\mathcal{T}_*[h]| \lesssim
	v \Theta_{R, a}^0(|y|) 
	\|h\|_{v,a}^{R}.
\end{equation}
And Lemma \ref{chiM-eq-lem}-$(2)$ implies that $\mathcal{T}_*[h]$ satisfies 
$\mathcal{T}_*[h] \in {\mathbf{SO}}_{\tau_0, \tau_1}^{R}$. Set $\tilde{\phi} = \phi - \mathcal{T}_*[h]$. Then $\tilde{\phi}$ satisfies 
$\tilde{\phi} \in {\mathbf{SO}}_{\tau_0, \tau_1}^{R}$
and
	\begin{equation}\label{phi-tilde-perpeq}
		\begin{cases}
			\partial_{\tau}
			\tilde{\phi} = 
			\Delta \tilde{\phi} + p U^{p-1} \tilde{\phi}
			 + p U^{p-1} \chi_{M}(|y|) \mathcal{T}_*[h] \mbox{ \ for \ } \tau \in (\tau_0, \tau_1), y \in B_{R},
			 \\
			\tilde{\phi} = 
			0 \mbox{ \ for \ } \tau \in (\tau_0, \tau_1), y \in \partial B_{R},
			\quad	\tilde{\phi}(\cdot,\tau_0) = 0 
			\mbox{ \ in \ } B_{R(\tau_0) }.
		\end{cases}
	\end{equation}
Multiplying \eqref{phi-tilde-perpeq} by $\tilde{\phi}$ and integrating both sides, we have
	\begin{equation*}
	\frac{1}{2} \partial_{\tau} \int_{B_{R  }} \tilde{\phi}^2 \rmd y
	+
	\int_{B_R} ( |\nabla \tilde{\phi}|^2 - p U^{p-1} \tilde{\phi}^2 ) \rmd y
	=
	\int_{B_{R  }} p U^{p-1}(y) \chi_{M}(|y|) \mathcal{T}_*[h] \tilde{\phi} \rmd y.
	\end{equation*}
Since $\tilde{\phi}$ satisfies $\tilde{\phi} \in {\mathbf{SO}}_{\tau_0, \tau_1}^{R}$, by \cite[Lemma 7.8]{infi4d} and H\"older inequality, we have
	\begin{equation*}
	\frac 12 \partial_{\tau} \int_{B_{R  }} \tilde{\phi}^2 \rmd y
	+ (n+ \frac 12) \int_{B_{R  } } \frac{ \tilde{\phi}^2}{|y|^2} \rmd y 
	\le
	\frac 12 \int_{B_{R  }} \Big( p U^{p-1}(y) \chi_{M}(|y|) \mathcal{T}_*[h] |y|\Big)^2 \rmd y.
	\end{equation*}
	Then, by \eqref{phi*-hperp-estimate}, we have
	\begin{equation*}
		\partial_{\tau} \int_{B_{R }} \tilde{\phi}^2 \rmd y
		+  R^{-2}
		\int_{B_{R  } } |\tilde{\phi}|^2 \rmd y 
		\lesssim 
		( v  \theta_{R, a}^0  )^2
		(\|h\|_{v,a}^{R})^2.
	\end{equation*}
	Since $\tilde{\phi}(\cdot,\tau_0) = 0 $, similar to \cite[(8.23) in Lemma 8.3]{Wei-Zhang-Zhou2022LLG}, under the assumption $v, R, \ln R \in \mathbf{AP}((\tau_0,\infty))$, $\mathbf{P}_1[R^{-2}] > -1$,  we have
	\begin{equation*}
	\begin{aligned}
&
	\int_{B_{R }} \tilde{\phi}^2 \rmd y
	\lesssim 
	\rme^{-\int^{\tau} R^{-2}(u) \rmd u}
	\int_{\tau_0}^{\tau} 
	\rme^{ \int^{s} R^{-2}(u) \rmd u} 
	(v(s) \theta_{R,a}^{0}(s))^2 \rmd s (\|h\|_{v,a}^{R})^2  
	\\
	\lesssim \ &
	\min\{ \tau, R^2 \} (v \theta_{R,a}^{0})^2 (\|h\|_{v,a}^{R})^2
	\lesssim
	( v\theta_{R, a}^0 R 
	\|h\|_{v,a}^{R} )^2.
		\end{aligned}
	\end{equation*}

Since $|p U^{p-1} \chi_{M}(|y|) \mathcal{T}_*[h]| \lesssim v \theta_{R,a}^{0} \chi_{M}(|y|) \|h\|_{v,a}^{R}$, then $\| \tilde{\phi} (\cdot,\tau) \|_{L^{\infty}(B_{R})} \lesssim v\theta_{R, a}^0 R 
\|h\|_{v,a}^{R}$. Under the assumption $|\dot{v}| = O(t^{-1} v)$, $\langle R \rangle^{2+\epsilon_1} = O(t)$ with $\epsilon_1>0$, applying Lemma \ref{qd26Mar7-1-lem} to \eqref{phi-tilde-perpeq}, we have $ |\tilde{\phi}| 
		\lesssim 
		v \theta_{R, a}^0   R 
		\langle y \rangle^{2-n}
		\|h\|_{v,a}^{R} $.
Combining \eqref{phi*-hperp-estimate}, and the scaling argument, we get \eqref{qd26Mar8-1}.
\end{proof}

\section{Acknowledgements}

We appreciate the meaningful discussion with Professor Yifu Zhou, Professor Kihyun Kim's suggestions to our manuscript.
The first author is partly supported by
Grant-in-Aid for Young Scientists (B) No. 26800065.


\end{document}